\documentclass[12pt]{article}
\usepackage{amsmath}
\usepackage{graphicx}
\usepackage{enumerate}
\usepackage{natbib}
\usepackage{titlesec}
\usepackage{url} 

\usepackage{xr}
\usepackage{amsfonts}
\usepackage{amsmath}
\usepackage{amsthm}
\usepackage{import}
\usepackage{enumerate}
\usepackage[dvipsnames]{xcolor}
\usepackage{authblk, bm, bbm, soul}
\usepackage{amssymb}
\usepackage{enumitem}
\usepackage{algorithm}
\usepackage{algpseudocode}

\usepackage{multirow}
\usepackage{capt-of}
\usepackage{booktabs}
\usepackage{hyperref}
\hypersetup{colorlinks=true, citecolor=blue, linkcolor = blue, urlcolor = blue}
\usepackage{cleveref}
\usepackage{stackrel}
\usepackage{makecell}

\usepackage{setspace}

\DeclareMathOperator*{\cov}{Cov}
\DeclareMathOperator*{\op}{op}
\DeclareMathOperator*{\tr}{Tr}

\DeclareMathOperator*{\argmin}{arg\,min}
\DeclareMathOperator*{\argmax}{arg\,max}

\newtheorem{theorem}{Theorem}[]
\newtheorem{definition}{Definition}[]

\newtheorem{assumption}{Assumption}
\newtheorem{proposition}[theorem]{Proposition}
\newtheorem{corollary}{Corollary}
\newtheorem{lemma}[theorem]{Lemma}
\theoremstyle{remark}
\newtheorem{remark}{Remark}

\allowdisplaybreaks

\newcommand{\blind}{1}

\begin{document}

\def\spacingset#1{\renewcommand{\baselinestretch}%
{#1}\small\normalsize} \spacingset{1}

%%%%%%%%%%%%%%%%%%%%%%%%%%%%%%%%%%%%%%%%%%%%%%%%%%%%%%%%%%%%%%%%%%%%%%%%%%%%%%

\if1\blind
{
  \title{\bf Optimal estimation in private distributed functional data analysis}
  \author{Gengyu Xue\\
    Department of Statistics, University of Warwick, Coventry, UK\\
    Zhenhua Lin\\
    Department of Statistics, National University of Singapore, Singapore\\
    and\\
    Yi Yu \\
    Department of Statistics, University of Warwick, Coventry, UK}
  \maketitle
} \fi

\if0\blind
{
  \bigskip
  \bigskip
  \bigskip
  \begin{center}
    {\LARGE\bf Optimal estimation in private distributed functional data analysis}
\end{center}
  \medskip
} \fi

\bigskip
\begin{abstract}
We systematically investigate the preservation of differential privacy in functional data analysis, beginning with functional mean estimation and extending to varying coefficient model estimation. Our work introduces a distributed learning framework involving multiple servers, each responsible for collecting several sparsely observed functions. This hierarchical setup introduces a mixed notion of privacy.  Within each function, user-level differential privacy is applied to $m$ discrete observations. At the server level, central differential privacy is deployed to account for the centralized nature of data collection.  Across servers, only private information is exchanged, adhering to federated differential privacy constraints.  To address this complex hierarchy, we employ minimax theory to reveal several fundamental phenomena: from sparse to dense FDA, from user-level to central and federated differential privacy costs, and the intricate interplay between different regimes of FDA and privacy preservation. To the best of our knowledge, this is the first study to rigorously examine functional data estimation under multiple privacy constraints. Our theoretical findings are complemented by efficient private algorithms and extensive numerical evidence, providing a comprehensive exploration of this challenging problem.
\end{abstract}

\noindent%
{\it Keywords:} Differential privacy; Federated learning; Functional data analysis; Functional mean estimation; Varying coefficient model.
\vfill

\newpage
\spacingset{1.9} 
\section{Introduction}
Recent advancements in modern technology have increased the prevalence of functional data in practical applications. Modeling data as samples of random functions has offered invaluable insights in various applications including neuroscience \citep[e.g.][]{dai2019age}, climatology \citep[e.g.][]{fraiman2014detecting}, medicine \citep[e.g.][]{crawford2020predicting} and others. We refer readers to \citet{wang2016functional} for a comprehensive review.

In tandem with the benefit, the increasing availability of data brings challenges in safeguarding individual privacy and maintaining public trust. A concrete route to handle these challenges is through differential privacy \citep[DP,][]{dwork2006calibrating} ---a mathematical notion which enjoys various applications, such as Google \citep[e.g.][]{song2021evading} and the US Census Bureau \citep[e.g.][]{uscensus}.

Preserving privacy is, in layman's words, achieved by injecting noise into raw data at the cost of statistical accuracy. The other side of this coin is the inevitable loss of statistical accuracy. It is therefore essential to understand the fundamental costs of privacy, consequently offering advice to policymakers when deciding which privacy notion should be implemented. A line of work is dedicated to quantifying costs of privacy in different problems, e.g.~mean and variance estimation \citep[e.g.][]{brown2023fast}, linear regression \citep[e.g.][]{cai2021cost} and non-parametric regression \citep[e.g.][]{cai2024optimal}. In this paper, we start with functional mean estimation - an indisputable core task in functional data analysis (FDA) - and extend to varying coefficient model (VCM) estimation, in a distributed learning setup with multiple~servers.

In terms of the privacy framework, the distributed data scenario, along with the intrinsic challenges of 
discretely observed FDA calls for a mixed notion of DP. Each server centrally collects multiple sparsely-observed functions, subject to central DP (CDP, \Cref{def_dp}). In view of the discretely observed nature of each function, user-level DP (\Cref{remark_user_level}) is imposed at the function level.  Across servers only private information is transmitted, naturally summoning local DP (LDP, \Cref{def_ldp}). We formulate the above framework as Federated DP (FDP, \Cref{def_fdp}) and
our goal is to understand the fundamental costs of this complex privacy notion in FDA via minimax risks, accompanied by open-source algorithms achieving optimal rates.  New phase transition phenomena in problems involving sparsely-observed functional data are revealed as a consequence. For the sake of space, a comprehensive literature review is deferred to \Cref{section_appendix_literature} in the online supplementary material.
% while comparisons with related results will be discussed in the sequel.

\subsection{List of contributions}

We propose a novel framework to study distributed functional data under a hierarchical differential privacy constraint, using two canonical problems (functional mean and VCM estimation) as examples, and the non-distributed CDP counterpart as an intermediate result - of independent interest. This framework builds a rigorous platform to study wide-ranging applications with timely needs.  For instance, in response to HIPAA \citep{HHS_HIPAA_Privacy_Rule_2025}, collecting sensitive healthcare data should be subject to privacy constraints.  Our framework is ideal for scenarios where data are collected from patients across different hospitals \citep[e.g.][]{li2020multi}, or for joint projects across different nations \citep[e.g.][]{azizi2023comparison}.

Within this framework, we propose a novel variant of the Gaussian mechanism (\Cref{section_mean_cdp_gaussian}), namely the anisotropic Gaussian mechanism, tailored for the inherent feature of functional data. Compared to the standard Gaussian mechanism, we demonstrate that our mechanism achieves the same privacy guarantees with substantially less loss of accuracy for functional data.  Built upon this mechanism, we propose mini-batch gradient descents based algorithms adhering to DP constraints, supported by both theoretical (Sections \ref{section_mean_cdp}, \ref{section_mean_fdp} and \ref{section_vcm_fdp}) and numerical (\Cref{section_numerical}) guarantees.

We further justify the minimax optimality of our proposed methods, by deriving matching lower bounds.  The upper and lower bounds together, quantify the fundamental costs of different privacy notions, summarised in \Cref{table_summary_minimax}. This opens the door for decision makers to make informed decisions on which privacy notions to implement in practice. 

At a high level, all the results collected in \Cref{table_summary_minimax} can be regarded as
\vspace{-2em}
\[
    \mbox{non-private sparse rate} \vee \mbox{private sparse rate} \vee \mbox{non-private dense rate} \vee \mbox{private dense rate}, \vspace{-1.5em}
\]
where the sparse and dense regimes are determined by $m$, the sampling frequency.  This general structure echoes both the existing literature in FDA \citep[e.g.][]{cai2011optimal,zhang2016from} and DP \citep[e.g.][]{cai2021cost,cai2023score,cai2024optimal}, where the sparse-dense and the private-non-private transitions are mostly studied separately. The presence of both transitions incurs additional interplay of FDA and DP, leading to a more complex transition phenomenon with four different transition boundaries; see \Cref{fig_phase_transition}.  

\begin{table}[!h]
\caption{A summary of main results (homogenous design in FDP), where $n$ is the number of functions, $m$ is the number of observations per function, $S$ is the number of servers, $d$ is the dimensionality of VCM, $\alpha$ is the smooth parameter, and $\epsilon, \delta$ are privacy parameters.}
\begin{center}
\begin{tabular}{lll} \hline
Problem   & Privacy & Minimax rates \\ \hline
\multirow{2}{*}{\makecell[c]{Func.~mean\\[-0.2em]
(Secs.~\ref{section_mean_cdp} \& \ref{section_mean_fdp})}} & CDP     & $ (nm)^{-\frac{2\alpha}{2\alpha+1}} + (n^2m\epsilon^2)^{-\frac{\alpha}{\alpha+1}}+n^{-1} + (n^2\epsilon^2)^{-1}$                                                         \\
                                            & FDP     & $(Snm)^{-\frac{2\alpha}{2\alpha+1}}  + (Sn^2m\epsilon^2)^{-\frac{\alpha}{\alpha+1}}+(Sn)^{-1} + (Sn^2\epsilon^2)^{-1}$                                \\
\multirow{2}{*}{\makecell[c]{VCM\\[-0.2em]
(Sec.~\ref{section_vcm_fdp})}}  & CDP     & $d(nm)^{-\frac{2\alpha}{2\alpha+1}} + d^{\frac{2\alpha+1}{\alpha+1}}(n^2m\epsilon^2)^{-\frac{\alpha}{\alpha+1}} +dn^{-1} +d^2(n^2\epsilon^2)^{-1}$           \\
                                            & FDP     & $d(Snm)^{-\frac{2\alpha}{2\alpha+1}} + d^{\frac{2\alpha+1}{\alpha+1}}(Sn^2m\epsilon^2)^{-\frac{\alpha}{\alpha+1}} + d(Sn)^{-1} + d^2(Sn^2\epsilon^2)^{-1}$\\ \hline
\end{tabular}

\label{table_summary_minimax}
\end{center}
\end{table}

\vspace{-3em}

\subsection{Notation} \label{section_notation}
Throughout the paper, we adopt the following notation. For a positive integer $a$, denote $[a] = \{1, \ldots, a\}$.  Let $\lceil a \rceil$ be the smallest integer greater than or equal to $a$ and $\lfloor a \rfloor$ be the greatest integer less than or equal to $a$. For $a,b \in \mathbb{R}$, let $a \vee b = \max\{a,b\}$ and $a\wedge b = \min\{a,b\}$. For any set $S$, let $|S|$ denote its cardinality. For any sets $S_1$ and $S_2$, denote $S_1\sqcup S_2$ the disjoint union. For $v \in \mathbb{R}^p$, let $\|v\|_1$, $\|v\|_2$ and $\|v\|_{\infty}$ be $\ell_1$-, $\ell_2$- and $\ell_{\infty}$-norms. Given a sequence of positive truncation radii $R= \{R_{\ell}\}_{\ell=1}^p$, define the entrywise projection $\Pi^{\mathrm{entry}}_{R}[v] \in~\mathbb{R}^p$ as $(\Pi^{\mathrm{entry}}_{R}[\nu])_\ell = \nu_\ell\min\{1, R_\ell/|\nu_\ell|\},$ for all~$\ell \in [p]$. For a function $f:[0,1] \rightarrow \mathbb{R}$, denote $\|f\|_{\infty} = \sup_{s\in [0,1]}|f(s)|$ and $\|f\|_{L^2} = \{\int_{0}^1 f^2(s)\,\mathrm{d}s\}^{1/2}$. Let $L^2([0,1])$ denote the space of square-integrable functions on $[0,1]$. For $f,g \in L^2([0,1])$, denote the inner product by $\langle f, g \rangle_{L^2} = \int_{0}^1 f(s)g(s)\;\mathrm{d}s$. For a sequence of positive numbers $\{a_n\}$ and a sequence of random variables $\{X_n\}$, denote $X_n = O_p(a_n)$ if $\lim_{M \rightarrow \infty}\lim\sup_n \mathbb{P}(|X_n|\geq M a_n)=0$. For two sequences of positive numbers $\{a_n\}$ and $\{b_n\}$, denote $a_n \lesssim b_n$, $a_n \gtrsim b_n$, and $a_n \asymp b_n$, if there exists some constants $c,C > 0$ such that $a_n/b_n \leq C$, $b_n/a_n \leq C$ and $c \leq a_n/b_n \leq C$. Write $a_n\lesssim_{\log} b_n$, $a_n \asymp_{\log} b_n$ and $a_n =_{\log} b_n$, if $a_n\lesssim b_n$, $a_n \asymp b_n$ and $a_n = b_n$ up to poly-logarithmic factors. For an $\mathbb{R}$-valued random variable $X$ and $k \in \{1,2\}$, let $\|X\|_{\psi_k}$ denotes the Orlicz-$\psi_k$ norm, i.e.~$\|X\|_{\psi_k} = \inf \{t > 0: \mathbb{E}[\exp(\{|X|/t\}^k)]\leq 2\}$.

\section{Differential privacy notions and minimax risks} \label{section_DP_notation}
We start with the central DP \citep[e.g.][]{dwork2006calibrating}, where a trusted server exists. 

Formally speaking, given a dataset $D$, a privacy mechanism $Q(\cdot|D)$ is a conditional distribution of the private information given data $D$. Let $Z \in \mathcal{Z}$ be the privatized data and $\sigma(\mathcal{Z})$ be the sigma-algebra on $\mathcal{Z}$. We require the privacy mechanism to satisfy the following. 
\begin{definition}[Central differential privacy, CDP] \label{def_dp}
    For $\epsilon > 0$ and $\delta \geq 0$, the privacy mechanism $Q$ is said to satisfy $(\epsilon,\delta)$-CDP if $ Q(Z \in A|D) \leq \exp(\epsilon)Q(Z \in A|D')+ \delta$,
    for all $A \in \sigma(\mathcal{Z})$ and $D$ and $D'$ that differ by at most one data entry, denoted by $D\sim D'$.
\end{definition}
In the above definition, the parameter $\epsilon$ controls the strength of the privacy constraint - the smaller~$\epsilon$, the stronger the constraint. For our discussion on minimax optimality, we focus on the case when $\epsilon \in (0,1)$, while our algorithms and upper bounds on the estimation errors of their outputs are applicable for any $\epsilon >0$. The parameter $\delta \geq 0 $ controls the level of privacy leakage. A small $\delta$ is desirable \citep[e.g.~Section 2.3 in][]{dwork2014algorithmic}.

Without a trusted central data curator, a more stringent DP notion is introduced, namely local DP. Let $\{X_i\}_{i=1}^n \subset \mathcal{X}^n$ be raw data and $\{Z_i\}_{i=1}^n \subset \mathcal{Z}^n$ be privatized output generated from data~$\{X_i\}_{i=1}^n$ using some privacy mechanism $Q=\{Q_i\}_{i=1}^n$. 
\begin{definition}[Local differential privacy, LDP] \label{def_ldp}
    For $\epsilon > 0$ and $\delta \geq 0$, the privacy mechanism $Q=~\{Q_i\}_{i=1}^n$ is said to satisfy $(\epsilon,\delta)$-LDP if $ Q_i(Z_i \in A|X_i = x_i) \leq \exp(\epsilon) Q_i(Z_i \in A|X_i= x'_i) + \delta$,
    for all $A\in \sigma(\mathcal{Z})$ and all $x_i, x_i' \in \mathcal{X}$.
\end{definition}

In a distributed setting, FDP is introduced to preserve privacy in distributed data. In this paper, we consider a setting where $N = \sum_{s=1}^S n_s$ samples of data are distributed across $S$ servers with the $s$th server holding $n_s$ samples. Formally speaking, given datasets $\{D_s\}_{s=1}^S = \{\{d_{s,i}\}_{i=1}^{n_s}\}_{s=1}^S$ distributed across $S$ servers, we consider a $T$-round privacy mechanism $(\bm{\epsilon},\bm{\delta},T)$-FDP. The same FDP framework has also been considered in \citet{li2024federated}.

For $T, S \in \mathbb{Z}_+$ and any $s\in [S]$, let $\{D_{s}^t\}_{t=1}^T$ form a partition of $D_s$, i.e.~$D_s = D_{s}^1 \sqcup \cdots \sqcup D_{s}^T$.  For $t \in [T]$, let $\mathcal{I}_s^t$ be the index set of $D_s^t$ with $|\mathcal{I}_s^t| = b_s^t$.  In the $t$th round, private information $Z_s^t$ is produced based on raw data $D_s^t$ and private information gained in the previous rounds.  The private  $Z_s^t$ is then transmitted to the central server and processed private information is sent back to servers for next round updates. In this protocol, no raw information is exchanged among servers and only privatized information aggregated by the central server is shared across local ones. Let $M^{(t)} = \{\{Z_s^t\}_{s=1}^S, M^{(t-1)}\}$ be the private information cumulated in the first $t$ rounds and $M^0 =\emptyset$. The collection of privacy mechanisms $Q = \{Q_s^t\}_{s=1,t=1}^{S,T}$ is required to satisfy the following $(\bm{\epsilon},\bm{\delta},T)$-FDP constraint.

\begin{definition}[Federated differential privacy, FDP] \label{def_fdp}
    For $S, T \in \mathbb{Z}_+$, denote $(\bm{\epsilon}, \bm{\delta}) = \{(\epsilon_s, \delta_s)\}_{s=1}^S$ the collection of privacy parameters where $\epsilon_s >0$ and $\delta_s \geq 0$, $s\in [S]$. A privacy mechanism $Q$ satisfies $(\bm{\epsilon}, \bm{\delta},T)$-FDP, if for any $s\in [S]$ and $t\in [T]$, the transcript $Z_s^t \in \mathcal{Z}$ shared across servers satisfies an $(\epsilon_s,\delta_s)$-CDP constraint, i.e.~$Q_s^t(Z_{s}^t\in A|M^{(t-1)},D^t_{s}) \leq \exp(\epsilon_s)Q_{s}^t(Z_{s}^t\in A|M^{(t-1)}, (D^t_{s})') + \delta_s $
    for all $D^t_{s}$ and $(D^t_{s})'$ that differ in at most one data entry, i.e.~$\sum_{i \in \mathcal{I}_{s}^t} \mathbbm{1}\{d_{s, i} \neq d'_{s, i}\} \leq 1$, and for all $A \in \sigma(\mathcal{Z})$.
\end{definition}

To evaluate the cost of privacy, we utilize the minimax framework \citep[e.g.][]{cai2021cost}.  Let $\mathcal{P}$ be the family of distributions generating the observation and $\mathcal{Q}$ the set of mechanisms satisfying the DP constraints in one of Definitions \ref{def_dp}, \ref{def_ldp} and \ref{def_fdp}. The minimax risk is defined as $\inf_{Q \in \mathcal{Q}}\inf_{\widetilde{\theta}}\sup_{P \in \mathcal{P}} \mathbb{E}_{P, Q}\|\theta(P) - \widetilde{\theta}\|_{L^2}^2$,
where $\theta(P)$ is the parameter of interest, the inner infimum is over all privatized estimators, the outer infimum is over all privacy mechanisms and the risk is the expectation with respect to both $P$ and $Q$.

\vspace{-1em}
\section{Functional mean estimation with CDP} \label{section_mean_cdp}

In this section, we consider the problem of functional mean estimation under the CDP constraint (\Cref{def_dp}).  This serves as the foundation for the estimation problems under more involving FDP constraint (see \Cref{def_fdp}) studied in \Cref{section_mean_fdp}.

\vspace{-1em}
\subsection{Problem setup} \label{section_mean_cdp_setup}

Let the raw data $\{(X_j^{(i)}, Y_j^{(i)})\}_{i=1, j =1}^{n,m}$ be generated from the model 
\vspace{-1em}
\begin{equation} \label{mean_model_obs}
    Y_j^{(i)} = \mu^*(X_j^{(i)}) + U^{(i)}(X_j^{(i)}) + \xi_{ij},
    \vspace{-1em}
\end{equation}
where $\{X_j^{(i)}\}_{i=1, j =1}^{n,m} \subset [0, 1]$ denote the discrete observation grids, $\mu^*(\cdot): [0,1] \rightarrow \mathbb{R}$ denotes the deterministic mean function, $\{U^{(i)}(\cdot): [0,1] \rightarrow \mathbb{R}\}_{i=1}^{n}$ denote a collection of random functions, with $\mathbb{E}\{U^{(i)}(s)\} = 0$ and covariance function $\Sigma^*(s,t) = \mathbb{E}\{U^{(i)}(s)U^{(i)}(t)\}$ for all $s, t \in [0,1]$ and~$i \in [n]$, and $\{\xi_{ij}\}_{i=1, j=1}^{n,m}$ are a sequence of mean-zero measurement errors.

Given $n$ discretely observed functions, our goal is to output a private functional mean estimator under the CDP constraint in \Cref{def_dp}. Since each function comprises $m$ discretely-observed univariate observations, changing one function potentially affects $m$ observations. Namely, for two data sets $\{(x_j^{(i)}, y_j^{(i)})\}_{i=1, j =1}^{n,m}$ and $\{((x')_j^{(i)}, (y')_j^{(i)})\}_{i=1, j =1}^{n,m}$, we say that they differ by one entry if $\sum_{i=1}^n \mathbbm{1}\{\{(x_j^{(i)}, y_j^{(i)})\}_{j =1}^{m} \neq \{((x')_j^{(i)}, (y')_j^{(i)})\}_{j =1}^{m}\} =1$.

\begin{remark} [User-level DP] \label{remark_user_level}
This DP notation above is in fact the \emph{user-level} DP coined in the literature \citep[e.g.][]{dwork2010differential}.  For user-level DP, each ``user'' (each function in our context) possesses multiple repeated measurements.  The DP is measured by changing a whole collection of possessions of a user, but the interest of the statistical inference lies in the distribution of each of these repeated measurements.  This is in contrast with the so-called \emph{event-level} DP, which is measured by changing one observation, rather than a whole collection of a user. We refer readers to \citet{levy2021learning} and \citet{kent2024rate} for more detailed discussions on user-level CDP and LDP respectively. 
\end{remark}

For functions involved in the modeling, we assume that they lie in a Sobolev space. 
\begin{definition}[Sobolev space] \label{def_sobolev}
    Let $\alpha > 1$ and $C_\alpha >0$ be absolute constants and $\{\phi_\ell\}_{\ell\in \mathbb{N}+}$ be a collection of orthonormal basis functions of $L^2([0,1])$. For a given sequence of non-negative constants $\{\tau_\ell\}_{\ell \in \mathbb{N}_+}$ associated with $\{\phi_\ell\}_{\ell\in \mathbb{N}_+}$, the Sobolev space $\mathcal{W}(\alpha, C_\alpha)$ is  
    \begin{align*}
        \mathcal{W}(\alpha, C_\alpha) = \bigg\{f:[0,1]\rightarrow\mathbb{R}: f = \sum_{\ell=1}^\infty \phi_\ell\langle f, \phi_\ell \rangle_{L^2},\; \sum_{\ell=1}^\infty (\tau_\ell)^{2\alpha}\langle f, \phi_\ell \rangle_{L^2}^2 \leq C_\alpha^2 < \infty \bigg\}.
    \end{align*}
\end{definition}
\vspace{-0.5cm}
In this work, we consider specifically the Fourier basis $\Phi = \{\phi_\ell\}_{\ell \in \mathbb{N}_+}$, i.e.~for $\ell \in \mathbb{Z}_+$,
\vspace{-0.5cm}
\begin{align} \label{eq_fourier_basis}
    \phi_1(t) =1, \; \phi_{2\ell}(t) = \sqrt{2}\cos(2\ell \pi t)\; \text{ and } \; \phi_{2\ell+1}(t) =\sqrt{2}\sin(2\ell\pi t),
    \vspace{-1.5cm}
\end{align}
\vspace{-0.5cm}
and the coefficient sequence $\tau_\ell = 2 \lfloor \ell/2 \rfloor$. The space $\mathcal{W}(\alpha, C_\alpha)$ can then be characterized~by
\begin{align}\label{eq_sobolev_eq1}
    \sum_{\ell=1}^\infty \tau_\ell^2\langle f, \phi_\ell \rangle_{L^2}^2  \leq \sum_{\ell=1}^\infty \ell^{2\alpha}\langle f, \phi_\ell \rangle_{L^2}^2\leq C_\alpha^2/\pi^{2\alpha} < \infty,
    \vspace{-2em}
\end{align}
where $C_\alpha >0$ is an absolute constant controlling the radius of the ellipsoid. Condition~\eqref{eq_sobolev_eq1} can be rewritten as the Sobolev class of functions concerning functions having $\alpha$th order squared-integrable weak derivatives \citep[e.g.~Proposition 1.14 in][]{tsybakov2008introduction}. More discussions on \Cref{def_sobolev} can be found in \Cref{section_appendix_discussion_sobolev} in the online supplementary material. In our paper, the boundedness property of the Fourier basis further allows us to work within a compact space, making it a convenient choice for DP.

\subsection{Anisotropic Gaussian mechanism} \label{section_mean_cdp_gaussian}

Gaussian mechanism \citep[e.g.][]{dwork2014algorithmic} is, arguably, the most popular privacy mechanism considered to achieve $(\epsilon,\delta)$-CDP, when $\delta > 0$.  To be specific, for $r \in \mathbb{Z}_+$ and an algorithm $f: \mathcal{\mathcal{D}}\rightarrow \mathbb{R}^r$ mapping a dataset taking values in $\mathcal{D}$ to $\mathbb{R}^r$, define the $\ell_2$-sensitivity of $f$ as $\Delta_2(f) = \sup_{D \sim D'} \|f(D)-f(D')\|_2$, where $D \sim D'$ is defined in \Cref{def_dp}. For an algorithm $f$ such that $\Delta_2(f) <~\infty$, the standard Gaussian mechanism outputs $M(D)=f(D)+Z$, where $Z \sim N_r(0,\sigma_0^2I)$ with $\sigma_0^2 = 2\log(1.25/\delta)\{\Delta_2(f)\}^2/\epsilon^2$. Although this mechanism is conceptually straightforward, substantial evidence indicates that it can achieve optimal performance in a range of classical statistical problems \citep[e.g.][]{cai2021cost,cai2023score}.  By adding the same level of noise to each entry, we, however, inherently overlook the unique behavior of individual entries---particularly in the FDA, where the output of interest often exhibits an entry-wise decaying rate. 

With this insight, we consider a new variant of the Gaussian mechanism, by adding anisotropic Gaussian noise tailored to the original data distribution.  With the same notation above, denote the sensitivity vector $\Delta f \in \mathbb{R}^r$ with the $\ell$th entry $\Delta f_\ell= \sup_{D\sim D'} |f_\ell(D)-f_\ell(D')|$, $\ell \in [r]$.  We have the following guarantee. 

\begin{lemma}[Anisotropic Gaussian mechanism] \label{l_gaussian_mechanism}
    Let $f: \mathcal{\mathcal{D}}\rightarrow \mathbb{R}^r$ be an algorithm such that $\|\Delta f\|_{\infty} < \infty$. The mechanism $M(D) = f(D) + Z$ achieves $(\epsilon,\delta)$-CDP for any $\epsilon, \delta > 0$ satisfying that $4\log(2/\delta) \geq \epsilon$, where $Z \in \mathbb{R}^r$ follows from a multivariate Gaussian distribution $N_r(0,\Sigma)$ with $\Sigma = \mathrm{diag}(\sigma_1^2, \ldots,\sigma_r^2)$ and $\sigma_\ell^2 = 4\log(2/\delta)\Delta f_\ell \|\Delta f\|_1/\epsilon^2$, $\ell \in [r]$.
\end{lemma}

%The proof of \Cref{l_gaussian_mechanism} is presented in Appendix \ref{section_appendix_gaussian_mechanism}. 
To achieve $(\epsilon,\delta)$-CDP, by the Hanson--Wright inequality \citep[e.g.~Theorem 6.2.1 in][]{vershynin2018high}, we can show that the standard Gaussian mechanism requires the added noise in squared $\ell_2$-norm to be of order $r\sigma_0^2$, up to poly-logarithmic factors.  As for the anisotropic Gaussian mechanism detailed in \Cref{l_gaussian_mechanism}, the added noise is of order $\sum_{\ell=1}^r \sigma_r^2$, up to logarithmic factors.  It follows from the Cauchy--Schwarz inequality~that
\begin{align*}
    \sum_{\ell=1}^r \sigma_r^2 \asymp_{\log} \Big(\sum_{\ell=1}^r \Delta f_\ell \Big)\|\Delta f\|_1/\epsilon^2 = \|\Delta f\|_1^2/\epsilon^2 \leq r\|\Delta f\|_2^2/\epsilon^2 \asymp_{\log} r\sigma_0^2,
\end{align*}
where the last identity holds provided that $\|\Delta f\|_2 \asymp \Delta_2(f)$. We, hence, see that under mild conditions, the anisotropic Gaussian mechanism comes with a smaller amount of variance under the same privacy guarantees, especially for a heterogeneous sensitivity~vector~$\Delta f$.

\subsection{Centrally private functional mean estimation}\label{section_mean_cdp_up}

In traditional FDA, without privacy concerns, it is a common practice to estimate mean functions via basis expansion \citep[e.g.][]{yao2006penalized, lin2021basis}.  It starts with solving the optimization problem $\widehat{a} = \argmin_{a \in \mathbb{R}^r} \; [(nm)^{-1}\sum_{i=1}^n\sum_{j=1}^m \big\{Y_{j}^{(i)} - a^{\top}\Phi_r(X_{j}^{(i)})\big\}^2]$,
where $r \in \mathbb{N}_+$ is the number of selected basis, $a \in \mathbb{R}^r$ and $\Phi_r(\cdot) = \big(\phi_1(\cdot), \ldots, \phi_r(\cdot)\big)^{\top}$ is the function formed by the leading $r$ bases. The final non-private estimator is then $\widehat{\mu}(\cdot) = \Phi_r^\top(\cdot)\widehat{a}$. 

To adhere to the CDP constraint, we propose a differentially private $\widetilde{a}$ via a noisy mini-batch gradient descent method detailed in \Cref{algorithm_mean} and consequently a differentially private functional mean estimator $\widetilde{\mu}(\cdot) = \Phi_r^\top (\cdot)\widetilde{a}$.

\begin{algorithm}
    \caption{Differentially private mean function estimation} \label{algorithm_mean}
    \begin{algorithmic}[1]
        \Require Data $\{(X_j^{(i)}, Y_j^{(i)})\}_{i=1, j =1}^{n,m}$, number of basis $r$, step size $\rho$, number of iterations~$T$, constant $C_R$, privacy parameter $\epsilon$, $\delta$, initialization $a^0$, failure probability~$\eta$. 
        \State Set $b = \lfloor n/T\rfloor$ and $R_\ell = C_R\big\{\sqrt{m^{-1}\log^2(n/\eta)}+\ell^{-\alpha}\big\}$, $\ell \in [r]$.
        \For{$t = 0 ,\ldots, T-1$}
        \State Set $\tau_t = bt$. Generate $w_t \sim N(0, \Sigma)$, where $\Sigma =\mathrm{diag}(\sigma_1^2, \ldots, \sigma_r^2)$ with $\sigma_\ell^2 = 16\log(2/\delta)R_\ell \sum_{k=1}^r R_k/(b^2\epsilon^2)$, $\ell \in [r]$.
        \State \hspace{-0.25cm} $a^{t+1} = \Pi^{*}_{\mathcal{A}}\Big\{ a^{t} -  \Big[\frac{\rho}{b}\sum_{i=1}^b\Pi^{\mathrm{entry}}_{R}\Big[\frac{1}{m}\sum_{j=1}^m \Phi_r(X^{(\tau_t+i)}_j)\big\{\Phi^{\top}_r(X^{(\tau_t+i)}_j)a^{t}-Y^{(\tau_t+i)}_j \big\}\Big] + w_t\Big]\Big\}$.
        \EndFor
        \Ensure $\widetilde{a} = a^T$ and $\widetilde{\mu} = \Phi_r^\top a^T$.
    \end{algorithmic}
\end{algorithm}

\begin{remark}[The projection $\Pi^*_{\mathcal{A}}$]\label{remark_projection_mean} With $\Pi_R^{\text{entry}}$ defined in \Cref{section_notation},
for any vector $a \in \mathbb{R}^r$, we further denote $\Pi_{\mathcal{A}}^*(a)$ the projection mapping $a$ to the closest point (in the $\ell_2$-norm sense) in $\mathcal{A} = \{\Bar{a} \in \mathbb{R}^r:\Phi_r^\top \Bar{a} \in \mathcal{W}(\alpha, C_{\alpha})\}$. The projection operator $\Pi^*_{\mathcal{A}}$ is used to ensure that in each iteration $t$, the estimator $a^t$ resides inside the parameter space of interest. For any $v \in \mathbb{R}^r$, $\Pi^*_{\mathcal{A}}(v)$ is equivalent to solving the convex optimization problem: $\argmin_{a\in \mathbb{R}^r} \|a - v\|_2^2 $ such that $\sum_{\ell=1}^r \ell^{2\alpha} a_{\ell}^2 \leq C_{\alpha}^2/\pi^{2\alpha}$.
    % \begin{align*}
    %     \quad \text{s.t.~}.
    % \end{align*}
\end{remark}

\begin{remark}[Varying sampling frequency $m$]
We focus on the case when the sampling frequency $m$ is the same across each function.  \Cref{algorithm_mean} can accommodate subject-varying sampling frequencies straightforwardly by adjusting the weights $(bm)^{-1}$ and privacy-related quantities correspondingly. For the minimax optimality, we remark that even for the non-private FDA this is yet studied thoroughly \citep[e.g.][]{zhang2016from, zhang2018optimal}.
\end{remark}

%To show theoretical guarantees for the output of Algorithm \ref{algorithm_mean}, we list a few assumptions.
\begin{assumption}[Sampling mechanism]\label{a_sample}
    The observation grids $\{X_{j}^{(i)}\}_{i=1,j=1}^{n,m}$ are independently sampled from a common density function $f_X:[0,1]\rightarrow \mathbb{R}_+$.  For an absolute constant $L > 1$, assume that $f_X\in \mathcal{W}(\alpha,C_\alpha)$ and $0 < 1/L \leq f_X(x) < L < \infty$ for any $x\in [0,1]$.
\end{assumption}

\begin{assumption} \label{a_model} The observations $\{(X_j^{(i)}, Y_j^{(i)})\}_{i=1, j =1}^{n,m}$ are from Model \eqref{mean_model_obs} and in addition, we assume the following holds. \textbf{(a)}. Mean function. The mean function $\mu^* \in \mathcal{W}(\alpha,C_\alpha)$. \textbf{(b)}. Functional noise. The sequence of functional noise $\{U^{(i)}(\cdot)\}_{i=1}^n$ are independent and uniformly sub-Gaussian, i.e.~for any $i \in [n]$, $x \in [0,1]$, it holds that $\|U^{(i)}(x)\|_{\psi_2} \leq C_{U}$, where $C_U >0$ is an absolute constant. We assume that $U^{(i)} \in \mathcal{W}(\alpha,C_\alpha)$ almost surely and $\{U^{(i)}\}_{i=1}^n$ are independent of $\{X_j^{(i)}\}_{i=1,j=1}^{n,m}$. \textbf{(c)}. Measurement error.  The sequence of measurement error $\{\xi_{ij}\}_{i=1,j=1}^{n,m}$are independently distributed sub-Gaussian variables, i.e.~for any $i \in [n]$, $j\in [m]$, it holds that $\|\xi_{ij}\|_{\psi_2} \leq C_{\xi}$, where $C_\xi >0$ is an absolute constant. We further assume that $\{\xi_{ij}\}_{i=1,j=1}^{n,m}$ are independent of $\{U^{(i)}\}_{i=1}^n$ and $\{X_{j}^{(i)}\}_{i=1,j=1}^{n,m}$.
\end{assumption}

Assumption \ref{a_sample} provides details of the sampling mechanism in Model \eqref{mean_model_obs}. We require that the sampling density belongs to a Sobolev class and is bounded away from zero and infinity. This type of assumption is commonly used in both FDA \citep[e.g.][]{cai2011optimal, zhang2016from} and nonparametric statistics \citep[e.g.][]{tsybakov2008introduction} literature to ensure that the observation grids  sufficiently spread out over the interval $[0,1]$. 

Assumption \ref{a_model}(a)~characterizes the smoothness of the true mean function. In Assumptions \ref{a_model}(b)~and \ref{a_model}(c), the tail behavior of both the functional noise and measurement errors are regulated. Similar sub-Gaussian assumptions are adopted in the FDA literature \citep[e.g.][]{cai2024transfer}. In Assumption \ref{a_model}(b), we further regulate the smoothness of functional noise, which is essential to control the sensitivity of the gradient in Algorithm \ref{algorithm_mean}, hence determining the level of noise used in the anisotropic Gaussian mechanism to preserve privacy. This assumption, at a high level, can be achieved by imposing smoothness conditions on the covariance function $\Sigma^*$ \citep[e.g.][]{steinwart2019convergence,henderson2024sobolev}. Examples include random process with the Mat\'ern covariance function of order $\alpha+\zeta$ for some small $\zeta>0$. More discussions on Assumption \ref{a_model} are in \Cref{section_appendix_discussion_assumption_2} in the online supplementary material.

%With the above assumptions, we provide the theoretical guarantee in Theorem \ref{thm_mean_upper}.
\begin{theorem} \label{thm_mean_upper}
    Let $\{(X_j^{(i)}, Y_j^{(i)})\}_{i=1, j =1}^{n,m}$ be from \eqref{mean_model_obs} satisfying Assumptions \ref{a_sample} and \ref{a_model}.
    \begin{enumerate}[leftmargin=*]
        \item \label{thm_mean_upper_rate-123} With $\epsilon, \delta > 0$ satisfying $4\log(2/\delta) \geq \epsilon$, Algorithm \ref{algorithm_mean} is $(\epsilon,\delta)$-CDP as in \Cref{def_dp}.

        \item \label{thm_mean_upper_rate-1} Initialize the algorithm with $a^0 = 0$, step size $\rho \in (0, L^{-1})$ being an absolute constant with $L$ in \Cref{a_sample} and $r>0$. Suppose that $nm \gtrsim L^2\{r+ \log(T/\eta)\}\log(n)$ and $T = \lceil C_1\log(n)\rceil$ for an absolute constant $C_1 >0$.  For any $\eta < 1/10$, it holds with probability at least $1-10\eta$ that
        \begin{align*}
            & \|\widetilde{\mu} - \mu^*\|_{L^2}^2 \lesssim \eta^{-1}\big\{n^{-1}\log(n)+r(nm)^{-1}\log(n)+ r^{-2\alpha}\big\} \nonumber \\
            & \hspace{1cm} +\big\{(n^2\epsilon^2)^{-1}+r^2(n^2m\epsilon^2)^{-1}\log^2(n/\eta)\big\}\log^2(n)\log(\log(n)/\eta)\log(1/\delta).
        \end{align*}        
        \item \label{thm_mean_upper_rate} Moreover, if $r \asymp_{\log}  (nm)^{1/(2\alpha+1)} \wedge (n^2m\epsilon^2)^{1/(2\alpha+2)} \wedge n^{1/(2\alpha)} \wedge (n^2\epsilon^2)^{1/(2\alpha)}$,
        % \begin{align}\label{eq-r-opt-choise-thm2}
        %     r \asymp_{\log}  (nm)^{\frac{1}{2\alpha+1}} \wedge (n^2m\epsilon^2)^{\frac{1}{2\alpha+2}} \wedge n^{\frac{1}{2\alpha}} \wedge (n^2\epsilon^2)^{\frac{1}{2\alpha}},
        % \end{align}
        then 
        \begin{align} \label{eq_minimax_mean_cdp}
             \|\widetilde{\mu} - \mu^*\|_{L^2}^2 =_{\log} O_p\Big\{(nm)^{-\frac{2\alpha}{2\alpha+1}} +(n^2m\epsilon^2)^{-\frac{\alpha}{\alpha+1}} + n^{-1} + (n^2\epsilon^2)^{-1}\Big\}.
        \end{align}
    \end{enumerate}
\end{theorem}

\Cref{thm_mean_upper}.\ref{thm_mean_upper_rate-123} ensures the privacy constraints provided that that $\delta$ is sufficiently small---in favour of us considering that $\delta$ indicates the privacy leakage probability.  \Cref{thm_mean_upper}.\ref{thm_mean_upper_rate-1} shows that, up to poly-logarithmic factors, the upper bound is of the form
    \begin{align} \label{eq_match_rate_22} 
        \|\widetilde{\mu} - \mu^*\|_{L^2}^2 & =_{\log} O_p\Big(\frac{r}{nm} +\frac{r^2}{n^2m\epsilon^2} + \frac{1}{n} + \frac{1}{n^2\epsilon^2} + r^{-2\alpha}\Big) = O_p\{(i)+(ii)+(iii)+(iv)+(v)\}, 
    \end{align}
where the five terms in \eqref{eq_match_rate_22} correspond to the (i) variance in the non-private sparse regime, (ii) variance in the private sparse regime, (iii) variance in the non-private dense regime, (iv) variance in the private dense regime and (v) squared bias due to approximation. 

This echoes common patterns in FDA and DP literature separately. In FDA literature, especially those based on basis expansion type methods \citep[e.g.][]{lin2021basis}, the estimation error is the summation of squared bias due to only using finite $r$ basis to approximate, i.e.~term (v), and variance of estimating an $r$-dimensional coefficient vector, i.e.~terms (i)-(iv).

In DP literature, added noise for privacy preservation (i.e.~terms (ii) and (iv)) and intrinsic data randomness (i.e.~terms (i) and (iii)) come in tandem into estimation error.

With the general results in \Cref{thm_mean_upper}.\ref{thm_mean_upper_rate-1}, in \Cref{thm_mean_upper}.\ref{thm_mean_upper_rate}, we show  that
\begin{align} 
    & \|\widetilde{\mu} - \mu^*\|_{L^2}^2 =_{\log}  O_p\Big\{\inf_r\Big(\frac{r}{nm} \vee  r^{-2\alpha}\Big)+\inf_r\Big(\frac{r^2}{n^2m\epsilon^2} \vee r^{-2\alpha}\Big) + \inf_r\Big(\frac{1}{n}  \vee r^{-2\alpha}\Big) \nonumber \\
    & \hspace{4cm} +  \inf_r\Big(\frac{1}{n^2\epsilon^2} \vee r^{-2\alpha}\Big)\Big\} = (I) + (II) + (III) + (IV). \nonumber 
\end{align}
Letting $r_{\mathrm{I}}$, $r_{\mathrm{II}}$, $r_{\mathrm{III}}$ and $r_{\mathrm{IV}}$ be the values of $r$ minimizing each term separately, an upper bound is achieved by choosing $r = \min\{r_{\mathrm{I}}, r_{\mathrm{II}}, r_{\mathrm{III}}, r_{\mathrm{IV}}\}$, i.e.~as per in \Cref{thm_mean_upper}.\ref{thm_mean_upper_rate}.

We conclude this subsection with a few remarks on Algorithm \ref{algorithm_mean} and Theorem \ref{thm_mean_upper}.

\begin{remark}[Sample-splitting] \label{remark_sample_splitting}
    In \Cref{algorithm_mean}, we operate with different batches of data in different iterations.  Sample splitting enforces the independence of data when analyzing the sensitivity of gradients of interest.  This is also crucial in constructing algorithms in the FDP framework discussed in Sections \ref{section_mean_fdp} and \ref{section_vcm_fdp}.  Despite the reasonable concern of accuracy loss due to the sample splitting, we show in \Cref{section_mean_cdp_low} that \Cref{algorithm_mean} is optimal. Further discussions can be found in \Cref{section_appendix_discussion_sample_splitting} in the online supplementary material.
\end{remark} 

\begin{remark}[Privacy guarantee]
    To show that Algorithm \ref{algorithm_mean} satisfies the CDP constraint, it suffices to show that in each iteration, an $(\epsilon, \delta)$-CDP is guaranteed. Given that fresh data batches are utilized in each iteration, the overall privacy guarantee for Algorithm \ref{algorithm_mean} follows from the parallel composition property \citep[e.g.][]{smith2021making}. The privacy guarantee for~$\widetilde{\mu}$ then follows by the post-processing property \citep[e.g.][]{dwork2014algorithmic}.
\end{remark}

\begin{remark}[Gaussian mechanism]
    We use anisotropic Gaussian mechanism in Algorithm \ref{algorithm_mean}. The desirable term $(n^2\epsilon^2)^{-1}$ in \eqref{eq_minimax_mean_cdp}, arising from controlling the privacy-preservation noise, consolidates its significance. In our analysis, as a consequence of smoothness properties of sampling density $f_X$, mean function $\mu^*$ and random noise function $U$, we can control entry-wise truncation levels $R_{\ell}$ by its estimation variance plus the coefficients decay rates,  $\ell^{-\alpha}$, from projecting $\alpha$-Sobolev functions to the $\ell$th Fourier basis in \eqref{eq_fourier_basis}. The dimension-independent rate of $(n^2\epsilon^2)^{-1}$ (i.e.~independent of $r$) is a direct consequence of summable property of $\ell^{-\alpha}$, and is not achievable with the standard Gaussian~mechanism.

\end{remark}

\subsection{Optimality and minimax lower bounds}\label{section_mean_cdp_low}
To understand the optimality of \Cref{thm_mean_upper}, we present a minimax lower bound below. Additional discussions on minimax risks and more phase transition phenomena can be found in Sections \ref{section_appendix_disscussion_mean_cdp} and \ref{section_appendix_phase} in the online supplementary material. 

\begin{theorem} \label{thm_mean_lower}
    Denote $\mathcal{P}_X$ the class of sampling distributions satisfying Assumption \ref{a_sample} and~$\mathcal{P}_Y$ the class of distributions for noisy observations satisfying Assumptions \ref{a_model}.  Suppose that $\delta \sqrt{\log(1/\delta)}\lesssim n^{-(2\alpha+3)/(2\alpha+2)}m^{-(8\alpha+9)/(4\alpha+4)}\epsilon^{-1/(2\alpha+2)}$.
    It then holds that 
    \begin{align} \label{eq_minimax_original}
        \underset{Q \in \mathcal{Q}_{\epsilon, \delta}}{\inf} \underset{\widetilde{\mu}}{\inf} \underset{P_X \in \mathcal{P}_X, P_Y \in \mathcal{P}_Y}{\sup} \mathbb{E}_{P_X, P_Y, Q}\|\widetilde{\mu}- \mu^*\|_{L^2}^2 \gtrsim  (nm)^{-\frac{2\alpha}{2\alpha+1}} \vee (n^2m\epsilon^2)^{-\frac{\alpha}{\alpha+1}} \vee n^{-1} \vee (n\epsilon)^{-2},
    \end{align}
    where $\mathcal{Q}_{\epsilon, \delta}$ is any mechanism satisfying $(\epsilon,\delta)$-CDP defined in Definition \ref{def_dp} with $\epsilon \in (0,1)$.
\end{theorem}

Up to poly-logarithmic factors, the lower bound in \Cref{thm_mean_lower} matches the upper bound in \Cref{thm_mean_upper}.\ref{thm_mean_upper_rate}.  At a high level, the lower bound is obtained by analysing two distinct setups: one where the random function is a constant random function and another where the mean function can be represented by a linear combination of finite basis functions. In the former case, the problem is essentially reduced to a univariate mean estimation problem and yields a lower bound $n^{-1}+(n^2\epsilon^2)^{-1}$ \citep[e.g.][]{cai2021cost}. In the latter case, applying the score attack technique introduced in \citet{cai2023score}, we obtain the remaining terms.  

Note that \citet{mirshani2019formal} and \cite{lin2023differentially} also study functional mean estimation under CDP constraint, but only handles fully observed functional data.

We exploit the minimax rate to unveil a number of new phase transition phenomena. An analysis of all possible regimes is depicted in \Cref{fig_phase_transition}. More details are deferred to \Cref{section_appendix_phase} in the online supplementary material.

The first regime $0 < m \lesssim n^{1/(2\alpha)}$ corresponds to the sparse regime in traditional FDA when there are no privacy constraints \citep[e.g.][]{cai2011optimal}. In this regime, the non-private rate only consists of the sparse term.
% ; for the sparse rates, it depends on the privacy level.  
Once $m \gtrsim n^{1/(2\alpha)}$, as seen in the middle and right panels, non-private rates are only in the dense regime, appearing in the term $n^{-1}$.  It is interesting to see that the rightmost panel implies that $m \asymp n^{1/\alpha}$ serves as a private sparse-dense transition boundary.

From a DP-oriented standpoint, the dense and sparse rates together determine the dense-sparse phase transition boundaries for non-private and private cases, $m \asymp n^{1/(2\alpha)}$ and $m \asymp (n^2\epsilon^2)^{1/\alpha}$, respectively. In the high privacy regime $0 < \epsilon \lesssim n^{-1/2}$, the rate is dominated by private rates only, with a phase transition from private sparse to private dense regimes at $m\asymp (n^2\epsilon^2)^{1/\alpha}$. Once $\epsilon~\gtrsim~n^{-1/2}$, only the non-private phase transition can be observed when the sampling frequency $m$ exceeds $n^{1/\alpha}$, while the private phase transition is obscured by the fact that $n^{1/\alpha} \lesssim (n^2\epsilon^2)^{1/\alpha}$.

In addition to phase transition phenomena, one key motivation for studying the minimax rates is to quantify the cost of privacy.

As for the dense rates, $n^{-1}$ and $(n^2\epsilon^2)^{-1}$ exhibit the same non-private-vs.-CDP parametric rate pattern in the literature \citep[e.g.][]{cai2021cost}.  
    
The sparse rates, as shown in \eqref{eq_minimax_mean_cdp}, are $(nm)^{-2\alpha/(2\alpha+1)}$ and $(n^2m\epsilon^2)^{-\alpha/(\alpha+1)}$. The cost of preserving privacy is reflected in a reduction of the effective sample size from $nm$ to $n^2m\epsilon^2$ and a loss in the exponent. The effect on $n$ changes from $n$ to $n^2\epsilon^2$ is consistent with the CDP property, while the change of exponent from $2\alpha/(2\alpha+1)$ to $\alpha/(\alpha+1)$ aligns with the CDP literature on non-parametric statistics \citep[e.g.][]{cai2023score}. A unique feature in private sparse FDA is that one function consists of $m$ observations, suggesting the user-level DP (see \Cref{remark_user_level} in \Cref{section_mean_cdp_setup}). This explains the effect on $m$ changes from $m$ to $m\epsilon^2$, which is of the same pattern in the user-level CDP literature \citep[e.g.][]{levy2021learning}.  

\begin{figure}[!htbp]
    \centering
    \includegraphics[width=0.85\linewidth]{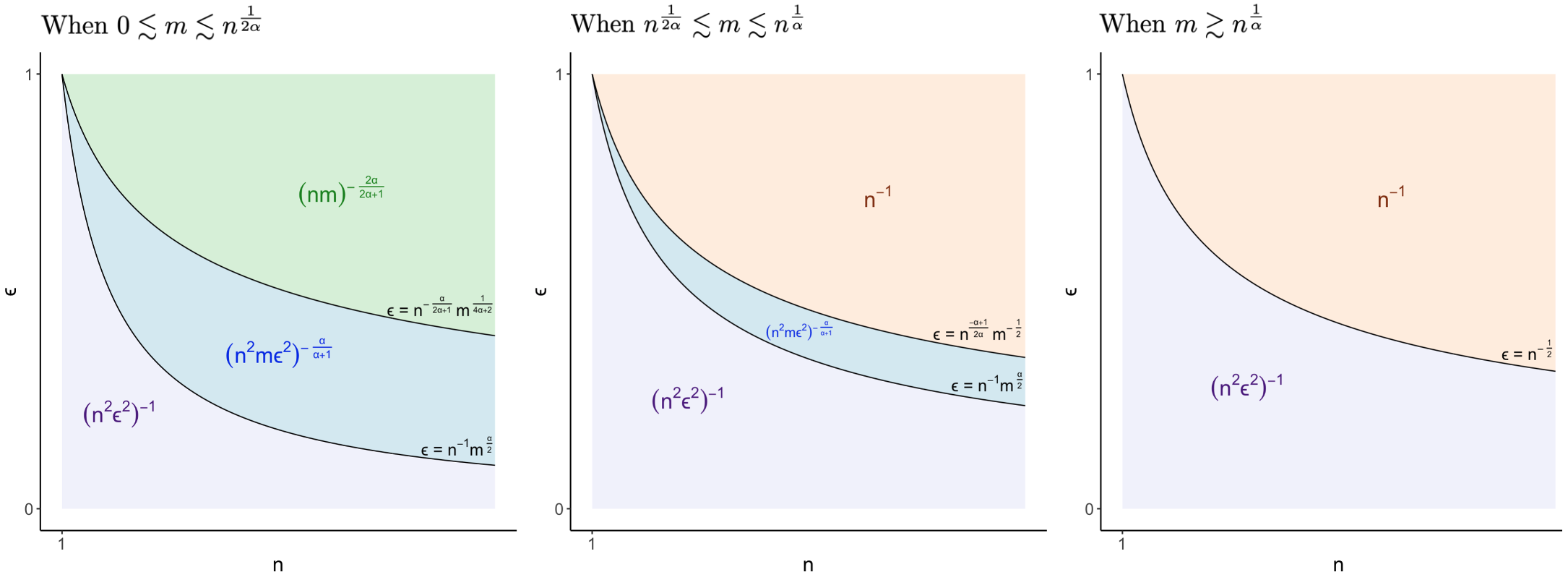}
    \caption{An illustration of phase transition phenomena for functional mean estimation under CDP constraint. All the rates are up to poly-logarithmic factors.}
    \label{fig_phase_transition}
\end{figure}

\vspace{-2em}

\section{Functional mean estimation with FDP}\label{section_mean_fdp}
In this section, we consider the case when $\{(X_j^{(i)}, Y_j^{(i)})\}_{i=1, j =1}^{N,m}$ from Model \eqref{mean_model_obs} are observed across $S$ servers. For any $s\in [S], i\in [n_s]$ and $j\in [m]$, we rewrite the model~as
\vspace{-1em}
\begin{equation}\label{fdp_mean_obs}
     Y_j^{(s,i)} = \mu^*(X_j^{(s,i)}) + U^{(s,i)}(X_j^{(s,i)}) + \xi_{s,ij},
     \vspace{-1em}
\end{equation}
with the additional index $s$ indicating the server and $N = \sum_{s \in [S]}n_s$. 

\vspace{-2em}
\subsection{Federated private functional mean estimation} \label{section_mean_fdp_up}
With the building block \Cref{algorithm_mean}, a functional mean estimation algorithm adhering to the $(\bm{\epsilon},\bm{\delta},T)$-FDP constraint is presented in \Cref{algorithm_fdp_mean}. It gathers private gradients from every single server obtained from \Cref{algorithm_mean}, with a total of $T$-round interactions.  Theoretical guarantees are collected in \Cref{fdp_thm_mean_up}.

\begin{algorithm}[!htbp]
    \caption{Federate differentially private mean function estimation} \label{algorithm_fdp_mean}
    \begin{algorithmic}[1]
        \Require Data $\{\{(X^{(s,i)}_j, Y^{(s,i)}_j)\}_{i=1,j=1}^{n_s,m}\}_{s=1}^S$, number of basis $r$, step size $\rho$, number of iterations $T$, weights $\{\nu_s\}_{s=1}^S$, constant $C_R$, privacy parameters $\epsilon$, $\delta$, initialization $a^0$, failure probability $\eta$.
        \State Set $b_s = \lfloor n_s/T \rfloor$, $s
        \in [S]$, $N = \sum_{s=1}^S n_s$ and $R_\ell = C_R\{\sqrt{m^{-1}\log^2(N/\eta)}+\ell^{-\alpha}\}$, $\ell \in [r]$.
        \For{$t = 0 ,\ldots, T-1$}
        \For{$s = 1, \ldots, S$}
        \State Set $\tau_{s,t} = tb_s$. Generate $w_{s,t} \in \mathbb{R}^{r}$ with $w_{s,t} \sim N(0, \Sigma_s)$, where $\Sigma_s =\mathrm{diag}(\sigma_{s,1}^2, \ldots, \sigma_{s,r}^2)$ and for any $\ell \in [r], \sigma_{s,\ell}^2 = 16\log(2/\delta_s)R_\ell(\sum_{k=1}^r R_k)/(b_s^2\epsilon_s^2)$;
        \State $M^{t}_s =   \frac{1}{b_s}\sum_{i=1}^{b_s}\Pi^{\mathrm{entry}}_{R}\Big[\frac{1}{m}\sum_{j=1}^m \Phi_r(X^{(s,\tau_{s,t}+i)}_j)\big\{\Phi^{\top}_r(X^{(s,\tau_{s,t}+i)}_j)a^{t}-Y^{(s,\tau_{s,t}+i)}_j \big\}\Big] + w_{s,t}$.
        \EndFor
        \State $a^{t+1} = \Pi^{*}_{\mathcal{A}}\{a^t - \rho\sum_{s=1}^S \nu_s M_s^t\}$.
        \EndFor
        \Ensure $\widetilde{a} = a^T$ and $\widetilde{\mu} = \Phi_r^\top \widetilde{a}$.
    \end{algorithmic}
\end{algorithm}

\begin{theorem} \label{fdp_thm_mean_up} 
    Let $\{(X_j^{(s,i)}, Y_j^{(s,i)})\}_{i, j, s = 1}^{n_s,m, S}$ be from Model \eqref{fdp_mean_obs} satisfying Assumptions~\ref{a_sample}~and~\ref{a_model}.
    \begin{enumerate}[leftmargin=*]
        \item With input $\epsilon_s, \delta_s > 0$ such that $4\log(2/\delta_s) \geq \epsilon_s$ for any $s\in [S]$, Algorithm \ref{algorithm_fdp_mean} satisfies the $(\bm{\epsilon}, \bm{\delta},T)$-FDP defined in Definition \ref{def_fdp}.

        \item Initialize \Cref{algorithm_fdp_mean} with $a^0 = 0$, step size $\rho \in (0, L^{-1})$ being an absolute constant with $L$ in \Cref{a_sample}, and $r >0$. Let the number of iterations $T\asymp \log(\sum_{s=1}^S n_s)$, and for each $s\in [S]$, let the weight $\nu_s \asymp_{\log} u_s(\sum_{s=1}^S u_s)^{-1}$, with $u_s = \{r(n_sm)^{-1} \vee r^2(n_s^2m\epsilon_s^2)^{-1} \vee n_s^{-2}\vee (n_s\epsilon)^{-2}\}^{-1}$.
        % $$ \nu_s = u_s\Big(\sum_{s=1}^S u_s\Big)^{-1}, \quad \mbox{with } u_s \asymp_{\log}  \Big(\frac{r}{n_s m} \vee \frac{r^2}{n_s^2 m \epsilon_s^2} \vee \frac{1}{n_s^2} \vee \frac{1}{n_s^2 \epsilon_s^2}\Big)^{-1}.$$
        Suppose that
        \begin{align} \label{fdp_thm_mean_up_r}
            r\log^2(Tr/\eta) \lesssim \bigg(\sum_{s=1}^S \frac{T\nu_s^2}{n_sm}\bigg)^{-1}  \quad \text{and} \quad r\log(Tr/\eta) \lesssim \Big(\sup_{s\in[S]} \frac{T\nu_s}{n_sm}\Big)^{-1}.
        \end{align}
        It then holds that
        \begin{align} \label{fdp_thm_mean_up_hete}
            \|\widetilde{\mu} - \mu^*\|_{L^2}^2 =_{\log} O_P\Big[r^2\Big\{\sum_{s=1}^S (rn_sm \wedge n_s^2m\epsilon_s^2 \wedge r^2n_s \wedge r^2n_s^2\epsilon_s^2)\Big\}^{-1} + r^{-2\alpha}\Big].
        \end{align}

        \item If we additionally assume that in a homogeneous setting,  where for any $s\in [S]$, the number of observations~$n_s$ and the privacy budgets $(\epsilon_s, \delta_s)$ are of the same order, i.e.~$n_s \asymp n$, $\epsilon_s \asymp \epsilon$ and $\delta_s \asymp \delta$, then  by selecting the number of basis $r$ as $r \asymp_{\log}  (Snm)^{1/(2\alpha+1)} \wedge (Sn^2m\epsilon^2)^{1/(2\alpha+2)} \wedge (Sn)^{1/(2\alpha)} \wedge (Sn^2\epsilon^2)^{1/(2\alpha)}$,
        % \begin{align} \label{fdp_thm_mean_up_homo_r}
        %     r \asymp_{\log}  (Snm)^{\frac{1}{2\alpha+1}} \wedge (Sn^2m\epsilon^2)^{\frac{1}{2\alpha+2}} \wedge (Sn)^{\frac{1}{2\alpha}} \wedge (Sn^2\epsilon^2)^{\frac{1}{2\alpha}}, 
        % \end{align}
        it holds that
        \begin{align} \label{fdp_thm_mean_up_homo}
            \|\widetilde{\mu} - \mu^*\|_{L^2}^2 =_{\log} O_P\Big\{(Snm)^{-\frac{2\alpha}{2\alpha+1}} + (Sn^2m\epsilon^2)^{-\frac{\alpha}{\alpha+1}}  + (Sn)^{-1} + (Sn^2\epsilon^2)^{-1}\Big\}.
        \end{align}
    \end{enumerate}
\end{theorem}

\begin{remark}[Assumptions on $r$] \label{remark_mean_fdp_r}
    The assumptions on $r$ in \eqref{fdp_thm_mean_up_r} both reduce to $r \lesssim_{\log} Snm$ in a homogeneous setting, automatically satisfied by construction. In the heterogeneous setting, it requires a mild lower bound on $r$ but not contradicting its optimal choice, provided that $\epsilon$ is not too large. See \Cref{remark_fdp_mean_r} in the online supplementary material for details.
\end{remark}

Theorem \ref{fdp_thm_mean_up} presents the privacy guarantee and an upper bound on the estimation error for Algorithm \ref{algorithm_fdp_mean}. In each iteration of Algorithm \ref{algorithm_fdp_mean}, a fresh batch of data is used from every server, and the overall procedure fits into the FDP mechanism defined in Definition~\ref{def_fdp}. Note when $S=1$, i.e.~the single-server scenario, the upper bound in \eqref{fdp_thm_mean_up_homo} recovers the one in \eqref{eq_minimax_mean_cdp}. 

It can be seen from \Cref{fdp_prop_mean_up} in the online supplementary material and the single-server result in \eqref{eq_match_rate_22}, that with a large probability, the upper bound \eqref{fdp_thm_mean_up_hete} is of the form
\begin{align*}
    & \|\widetilde{\mu} - \mu^*\|_{L^2}^2 \lesssim_{\log} r^{-2\alpha} + \sum_{s = 1}^S \nu_s^2\Big\{r(n_sm)^{-1} + r^2(n_s^2 m \epsilon^2_s)^{-1} + n_s^{-1} + (n_s \epsilon_s)^{-1} \Big\}  \\ 
    \asymp \;& \mbox{ squared bias}+ \sum_{s = 1}^S (\mbox{weight of server }s)^2\times (\mbox{variance of est. }r\mbox{-dim. object in server }s).
\end{align*}
A Lagrange multiplier argument leads to the choice of weights $\nu_s$ 
% that
% \[
%     \nu_s \asymp_{\log} \frac{1}{\frac{r}{n_sm} + \frac{r^2}{n_s^2 m \epsilon^2_s} + \frac{1}{n_s} + \frac{1}{n^2_s \epsilon^2_s}} \Bigg\{\sum_{t = 1}^S \frac{1}{\frac{r}{n_tm} + \frac{r^2}{n_t^2 m \epsilon^2_t} + \frac{1}{n_t} + \frac{1}{n^2_t \epsilon^2_t}}\Bigg\}^{-1},
% \]
which is the contribution of the variance yielded from the server $s$ to the harmonic mean of variances. As a consequence, the overall estimation error is of the form
%\vskip -1.5cm
\vspace{-1em}
\begin{align*}
    (S^{-1}\times \mbox{ harmonic mean of var from all local servers}) \, \vee \, \mbox{squared bias from approximation}.
    % \vspace{-10em}
\end{align*}
\vskip -.5cm
This echos findings in federated nonparametric estimation problems in \cite{cai2024optimal}.

\subsection{Optimality and minimax lower bounds} \label{section_mean_fdp_low}
%A minimax lower bound is now in order. 
\begin{theorem} \label{fdp_thm_mean_low} 
    Denote $\mathcal{P}_X$ the class of sampling distributions satisfying Assumption \ref{a_sample} and $\mathcal{P}_Y$ the class of distributions for noisy observations satisfying Assumptions \ref{a_model}.
    \begin{enumerate}[leftmargin=*]
        \item (Heterogeneous design.)\label{fdp_thm_mean_low_hete} Suppose that $\epsilon_s \in (0,1)$ and $\delta_s\log(1/\delta_s)  \lesssim r_0^{-1}m^{-3}\epsilon_s^2$ for any $s \in [S]$ with $r_0 \in \mathbb{Z}_+$ satisfying 
        $r_0^{2\alpha + 2} \asymp \sum_{s=1}^S  \{r_0n_sm \wedge T^{-1}n_s^2m\epsilon_s^2\} $.  It then holds that 
        \begin{align}  \notag
            & \underset{Q \in \mathcal{Q}^T_{\bm{\epsilon}, \bm{\delta}}}{\inf} \underset{\widetilde{\mu}}{\inf} \underset{P_X \in \mathcal{P}_X, P_Y \in \mathcal{P}_Y}{\sup} \mathbb{E}_{P_X, P_Y, Q}\|\widetilde{\mu}- \mu^*\|_{L^2}^2 \\ \label{fdp_thm_mean_low_eq1}
            \gtrsim\; &  \sup_{r \in \mathbb{Z}_+} \Big[r^2\Big\{\sum_{s=1}^S  (rn_sm \wedge T^{-1}n_s^2m\epsilon_s^2) + r^{2\alpha + 2}\Big\}^{-1}\Big] \vee \Big\{\sum_{s=1}^S (n_s \wedge T^{-1}n_s^2\epsilon_s^2)\Big\}^{-1},
        \end{align}
        where $\mathcal{Q}^T_{\bm{\epsilon}, \bm{\delta}}$ is any mechanisms satisfying $(\bm{\epsilon},\bm{\delta},T)$-FDP defined in Definition \ref{def_fdp}.

        \item (Homogeneous design.)\label{fdp_thm_mean_low_homo} Suppose additionally, for any $s\in [S]$, the number of observations~$n_s$ and the privacy budgets $(\epsilon_s, \delta_s)$ are of the same order, i.e.~$n_s \asymp n$, $\epsilon_s \asymp \epsilon$ and $\delta_s \asymp \delta$. When $\delta\log(1/\delta)\lesssim \{S^{-1/(2\alpha+1)}n^{-1/(2\alpha+1)}m^{-(6\alpha+4)/(2\alpha+1)}\epsilon^2\} \wedge \{S^{-1/(2\alpha+2)}n^{-1/(\alpha+1)}$ $\times m^{-(6\alpha+7)/(2\alpha+2)}\epsilon^{(2\alpha+1)/(\alpha+1)}\}$, 
        % \[
        %     \delta\log(1/\delta)\lesssim \Big\{S^{-\frac{1}{2\alpha+1}}n^{-\frac{1}{2\alpha+1}}m^{\frac{-6\alpha-4}{2\alpha+1}}\epsilon^2\Big\} \wedge \Big\{S^{-\frac{1}{2\alpha+2}}n^{-\frac{1}{\alpha+1}}m^{-\frac{6\alpha+7}{2\alpha+2}}\epsilon^{\frac{2\alpha+1}{\alpha+1}}\Big\},
        % \]
        it  holds for any $T \in [\min_{s\in[S]} n_s]$~that
        \begin{align} \notag
            & \underset{Q \in \mathcal{Q}^T_{\bm{\epsilon}, \bm{\delta}}}{\inf} \underset{\widetilde{\mu}}{\inf} \underset{P_X \in \mathcal{P}_X, P_Y \in \mathcal{P}_Y}{\sup} \mathbb{E}_{P_X, P_Y, Q}\|\widetilde{\mu}- \mu^*\|_{L^2}^2 \\ \label{fdp_thm_mean_low_eq2}
            & \hspace{3cm} \gtrsim  (Snm)^{-\frac{2\alpha}{2\alpha+1}} \vee (Sn^2m\epsilon^2)^{-\frac{\alpha}{\alpha+1}} \vee (Sn)^{-1} \vee (Sn^2\epsilon^2)^{-1}.
        \end{align}
        \end{enumerate}
\end{theorem}

\vskip -.5cm
The core of the proof is an application of the Van-Trees inequality \citep[e.g.][]{gill1995applications}. To the best of our knowledge, \citet{cai2024optimal} is the first to adopt the Van-Trees method in FDP minimax lower bound proofs. In our paper, rather than considering a non-interactive (i.e.~$T = 1$) mechanism as in \citet{cai2024optimal}, our FDP definition covers a more general class allowing interaction (i.e.~$T \in \mathbb{Z}_+$). 

In detail, let $Z^{(s)} = \{Z_s^t\}_{t=1}^{T}$ represent the collection of privatized transcripts released from server~$s$ to the central server over $T$ iterations and let $I_{Z^{(1)}, \ldots, Z^{(S)}}$ denote its Fisher information. To control $I_{Z^{(1)}, \ldots, Z^{(S)}}$, which is the key quantity of interest in the Van-Trees inequality, \citet{cai2024optimal} relies on the fact that $\{Z^{(s)}\}_{s=1}^S$ are mutually independent, allowing the Fisher information to be decomposed as $\sum_{s=1}^S I_{Z^{(s)}}$. In our setting, this independence no longer holds due to the interaction.  However, the use of a fresh batch of data at each iteration $t$ ensures conditional independence, leading to the decomposition of $\sum_{s=1}^S \sum_{t=1}^T I_{Z_t^{(s)}|M^{(t-1)}}$. The rest of the proof will then follow by considering two individual
setups as the ones used in the proof of Theorem~\ref{thm_mean_lower}. 

In the homogeneous setting when $n_s, \epsilon_s$ and $\delta_s$ across all servers are of the same order, \eqref{fdp_thm_mean_up_homo} and \eqref{fdp_thm_mean_low_eq2} show that \Cref{algorithm_fdp_mean} achieves the optimal convergence rate up to poly-logarithmic factors. In the heterogeneous setting, the lower bound \eqref{fdp_thm_mean_low_eq1} comprises two terms representing minimax rates for (i) estimating an $r$-dimensional object (sparse rates) and (ii) estimating a univariate object (dense rates). Due to this artifact, the final rate is chosen by only matching $r^{-2\alpha}$ with the related term corresponding to the sparse rates. 

This contrasts with the upper bound \eqref{fdp_thm_mean_up_hete}, where the variance contributed by each server accounts for all four scenarios: non-private sparse, private sparse, non-private dense and private dense.  These terms guide the choice of the number of basis functions $r$ to balance the squared bias term $r^{-2\alpha}$ with the harmonic mean of the variances from all servers. This results in a gap between the lower and upper bounds in the heterogeneous case.

Regarding the gap, we have the following conjectures in connection to the existing literature.  
In \cite{cai2024optimal}, a federated nonparametric problem is studied, where they also approximate an infinite-dimensional space by a finite-dimensional one.  Although they also consider heterogeneity across servers, we have the additional discretization of functions, i.e.~the parameter $m$ with respect to $n_s$, which introduces further heterogeneity.  The minimax rate in \cite{cai2024optimal} is obtained by matching the squared approximation bias with the overall harmonic mean of variances, as we have in the upper bound.  This suggests that our upper bound might be tight but not the lower bound.

An attractive feature of FDP is that it serves as an intermediate privacy model between CDP and LDP.  In the CDP setting, there is a central server which accesses $Sn$ samples of non-private data, while in LDP, discrete observations from every function are privatized before sending to a central server. The FDP framework bridges the gap between these approaches by allowing $S$ local trusted servers to access the raw data. Notably, the FDP framework encompasses the CDP framework when $S=1$ and the LDP framework when $n=1$. Comparing the CDP rate in \eqref{eq_minimax_original} with the sample size $Sn$  and FDP rate in \eqref{fdp_thm_mean_low_eq2} further confirms this behavior. In particular, the effective sample size in the private term under FDP is reduced by a factor of $S$ compared to that in CDP, while it is larger by a factor of $n$ compared with that under LDP.

More discussions on the involvement of $T$ in the lower bound and further justifications of the cost of privacy are deferred to \Cref{section_appendix_disscussion_mean_fdp} in the online supplementary material.
\vskip -.5cm

\section{Varying coefficient model estimation with FDP} \label{section_vcm_fdp}
In this section, we further consider estimating the functional regression coefficients in varying coefficient models (VCM), building upon the results in Sections~\ref{section_mean_cdp} and \ref{section_mean_fdp}. In the exposition, we will omit repetitive details and focus solely on the differences.

Suppose the data $\{(X^{(i)}_j, \bm{G}^{(i)}, Y^{(i)}_j\}_{i=1,j=1}^{N,m}$ are generated from a VCM, specifically, 
\vspace{-2em}
\begin{align} \label{simple_vcm_model_obs}
    Y^{(i)}_j = \bm{G}^{(i)\top}\bm{\beta}^*(X^{(i)}_j) + \xi_{ij},
    \vspace{-2em}
\end{align}
\vspace{-0.2em}
where $\{Y^{(i)}_j\}_{i=1,j=1}^{N,m}$ represent the real-valued response, $\{X^{(i)}_j\}_{i=1,j=1}^{N,m}$ denote the discrete observation grids, $\bm{G}^{(i)}=\{G_0^{(i)}, G_1^{(i)}, \ldots, G_d^{(i)}\}^\top \in \mathbb{R}^{d+1}$ denote the random design vector, $\bm{\beta}^*(\cdot) = \{\beta^*_0(\cdot), \beta^*_1(\cdot), \ldots, \beta^*_d(\cdot)\}^{\top} \in \mathbb{R}^{d+1}$ denote the $(d+1)$-dimensional coefficient function, and $\{\xi_{ij}\}_{i=1,j=1}^{N,m}$ is a sequence of centered measurement errors with $\mathbb{E}(\xi_{ij}) = 0$, $i \in [N]$ and $j \in [m]$. 
% We note that an intercept term can be accommodated by letting $G_0 = 1$. 
Consider a distributed setting where $N$ observations are distributed across $S$ servers, with $N = \sum_{s=1}^S n_s$.  For $s\in[S], i\in [n_s]$ and $j\in[m]$, we rewrite the model as
\vskip -.5cm
\begin{equation}\label{fdp_vcm_obs}
     Y^{(s,i)}_j = \bm{G}^{(s,i)\top}\bm{\beta}^*(X^{(s,i)}_j) + \xi_{s,ij}.
\end{equation}
\vskip -.5cm
Recall that the functional mean estimation under CDP in \Cref{section_mean_cdp} is a building block for that under FDP in \Cref{section_mean_fdp}.  The narrative is the same for VCM, but for presentation, we defer the VCM estimation under CDP to \Cref{section_appendix_vcm_cdp} in the online supplementary material, and focus on the FDP setting.

\vspace{-1.5em}

\subsection{Federated private varying coefficient model estimation} \label{section_vcm_fdp_up}
Basis expansion type methods are also widely used in VCM estimation \citep[e.g.][]{huang2002varying, huang2004polynomial,wang2008variable}. Given observations $\{(X^{(i)}_j, \bm{G}^{(i)}, Y^{(i)}_j\}_{i=1,j=1}^{N,m}$ generated from Model \eqref{simple_vcm_model_obs}, for any $k \in \{0\} \cup [d]$, a non-private estimator of $\beta_k$ is given by $\widehat{\beta}_k = \Phi^{\top}_r\widehat{b}_k$, where~$r \in \mathbb{Z}_+$ is a pre-specified tuning parameter, 
\begin{align*}
\vspace{-2em}
     (\widehat{b}_0^\top,\ldots,\widehat{b}_d^\top)^{\top}= \widehat{B}=\underset{B \in \mathbb{R}^{r(d+1)}}{\argmin}\frac{1}{nm}\sum_{i=1}^N\sum_{j=1}^m\Big\{Y_{j}^{(i)} - \bm{G}^{(i)\top}\widetilde{\Phi}_r^\top(X_{j}^{(i)})B\Big\}^2,
\end{align*}
\vspace{-0.1em}
and $\widetilde{\Phi}_r \in \mathbb{R}^{r(d+1)\times (d+1)}$ is the matrix-valued function with diagonal block $\Phi_r$. To adhere the FDP constraint, we propose a noisy gradient descent algorithm to obtain a differentially private $\widetilde{B}$ and the final private estimator $\widetilde{\beta}_k$ is constructed by post-processing: $\widetilde{\beta}_k(\cdot) = \Phi_r^\top(\cdot)\widetilde{b}_k$. The Algorithm shares much of the same spirit as \Cref{algorithm_fdp_mean}, 
% by aggregating private gradients from different servers and by gradient descents to estimate the coefficients. 
and we defer it to \Cref{algorithm_fdp_vcm} in \Cref{section_appendix_discussion_vcm_fdp} in the online supplementary material.

\begin{assumption} \label{simple_vcm_a_model} The observations $\{(X^{(i)}_j, \bm{G}^{(i)}, Y^{(i)}_j)\}_{i=1,j=1}^{N,m}$ are generated from Model~\eqref{simple_vcm_model_obs} and assume the following holds. \textbf{(a)}. Model design. Assume that the sequence of predictors $\{\bm{G}^{(i)}\}_{i=1}^N$ are independent and identically distributed.  Denote $\{\lambda_l\}_{l = 0}^d$ eigenvalues of $\mathbb{E}[\bm{G}^{(1)}(\bm{G}^{(1)})^\top]$. Assume there exists an absolute constant $C_\lambda>1$ such that $0 < 1/C_\lambda \leq \lambda_0 \leq \cdots \leq \lambda_d \leq C_\lambda <~\infty$. \textbf{(b)}. Random predictor. Assume that $\bm{G}^{(1)}$ have bounded entries, i.e.~there exists an absolute constant $C_g > 0$ such that $\|\bm{G}^{(1)}\|_\infty \leq C_g$ almost surely. Assume that $\{\bm{G}^{(i)}\}_{i=1}^N$ are independent of $\{X_j^{(i)}\}_{i=1,j=1}^{N,m}$. \textbf{(c)}. Coefficient function. Assume that $\beta_k(\cdot) \in \mathcal{W}(\alpha,C_\alpha)$ and $\bm{G}^{(i)\top}\bm{\beta}^*(\cdot) \in \mathcal{W}(\alpha,C_\alpha)$ almost surely for any $k \in \{0\} \cup[d]$ and $i \in [N]$.  Assume that there exists an absolute constant $C_b >0$ such that $\|\bm{G}^{(i)\top}\bm{\beta}^*\|_{\infty} \leq C_{b}$ almost surely, $i \in [N]$. \textbf{(d)}. Measurement error. Assume that $\{\xi_{ij}\}_{i=1,j=1}^{N,m}$ is a collection of independent random variables with  $\|\xi_{ij}\|_{\psi_2} \leq C_{\xi}$, where $C_\xi >0$ is an absolute constant. We further assume that $\{\xi_{ij}\}_{i=1,j=1}^{N,m}$ are independent of $\{\bm{G}^{(i)}\}_{i=1}^n$ and $\{X_{j}^{(i)}\}_{i=1,j=1}^{N,m}.$

\end{assumption}

Assumptions \ref{simple_vcm_a_model}(a)~and \ref{simple_vcm_a_model}(b)~are common assumptions on the random design matrix in VCM estimation; see for example \citet{huang2002varying,huang2004polynomial}, \citet{wang2008variable} and \citet{he2018dimensionality}, to name but a few.  Assumption \ref{simple_vcm_a_model}(c)~consists of a set of technical assumptions on the behavior of~$d+1$ functional coefficients. We regulate the smoothness of each by assuming Sobolev.  The tail behaviors of measurement errors are characterized in Assumption~\ref{simple_vcm_a_model}(d). 

Denote $\|\widetilde{\bm{\beta}}- \bm{\beta}^*\|_{L^2}^2 = \sum_{k=0}^d \|\widetilde{\beta}_k- \beta_k^*\|_{L^2}^2$. We present theoretical properties in Theorem~\ref{fdp_thm_vcm_up}.
\begin{theorem} \label{fdp_thm_vcm_up}
    Let $\{\{(X^{(s,i)}_j, \bm{G}^{(s,i)}, Y^{(s,i)}_j)\}_{i=1,j=1}^{n_s,m}\}_{s=1}^S$ be from Model \eqref{fdp_vcm_obs} satisfying Assumptions \ref{a_sample} and \ref{simple_vcm_a_model}.
    \begin{enumerate}[leftmargin=*]
        \item With input $\epsilon_s, \delta_s > 0$ such that $4\log(2/\delta_s) \geq \epsilon_s$ for any $s\in [S]$, Algorithm \ref{algorithm_fdp_vcm} satisfies $(\bm{\epsilon}, \bm{\delta},T)$-FDP defined in Definition \ref{def_fdp}.

        \item  \label{fdp_thm_vcm_up_2} Initialize \Cref{algorithm_fdp_vcm} with $B^0 = 0$, step size $\rho \in (0, (C_{\lambda}L)^{-1})$ being an absolute constant with $L$ and $C_{\lambda}$ defined in Assumptions \ref{a_sample} and \ref{simple_vcm_a_model}(a), and $r >0$.   Let the number of iterations $T\asymp \log(\sum_{s=1}^S n_s)$, and for each $s \in [S]$ let the weight $ \nu_s = u_s/(\sum_{s=1}^S u_s)$ with $u_s\asymp_{\log}\{d n_s^{-1} \vee dr(n_sm)^{-1} \vee d^2(n_s^2 \epsilon_s^2)^{-1} \vee d^2r^2(n_s^2 m \epsilon_s^2)^{-1}\}^{-1}$.
        % \begin{align*}
        %     \nu_s = \frac{u_s}{\sum_{s=1}^S u_s} \quad \mbox{with } u_s\asymp_{\log} \frac{1}{\frac{d}{n_s} \vee \frac{dr}{n_sm} \vee \frac{d^2}{n_s^2 \epsilon_s^2} \vee \frac{d^2r^2}{n_s^2 m \epsilon_s^2}}.
        % \end{align*}
        Suppose that
        % \label{fdp_thm_vcm_up_r}
        \begin{align*} 
            r\log^2(Tr/\eta)\lesssim \Big(\sum_{s=1}^S \frac{T\nu_s^2d}{n_sm}\Big)^{-1},\, \log^2(Tr/\eta)\sum_{s=1}^S \frac{T\nu_s^2d}{n_s} \lesssim 1, \, r\log(Tr/\eta)\lesssim \Big(\sup_{s\in [S] }\frac{T\nu_s d}{n_s}\Big)^{-1}.
        \end{align*}
        It then holds that $ \|\widetilde{\bm{\beta}}- \bm{\beta}^*\|_{L^2}^2 =_{\log} O_p[d^2r^2\{\sum_{s=1}^S (drn_sm)\wedge (n_s^2m\epsilon_s^2) \wedge (dr^2n_s) \wedge (r^2n_s^2\epsilon_s^2)\}^{-1}\\ + dr^{-2\alpha}].$
        % \begin{align*}
        %     \|\widetilde{\bm{\beta}}- \bm{\beta}^*\|_{L^2}^2 =_{\log} O_p\Big[d^2r^2\Big\{\sum_{s=1}^S (drn_sm)\wedge (n_s^2m\epsilon_s^2) \wedge (dr^2n_s) \wedge (r^2n_s^2\epsilon_s^2)\Big\}^{-1}+ dr^{-2\alpha}\Big].
        % \end{align*}

        \item If we additionally assume that in a homogeneous setting, where for any $s\in [S]$, the number of observations~$n_s$ and the privacy budgets $(\epsilon_s, \delta_s)$ are of the same order, i.e.~$n_s \asymp n$, $\epsilon_s \asymp \epsilon$ and $\delta_s \asymp \delta$. By selecting the number of basis $r$ as $r \asymp_{\log}  (Snm)^{1/(2\alpha+1)} \wedge (d^{-1}Sn^2m\epsilon^2)^{1/(2\alpha+2)} \wedge  (Sn)^{1/(2\alpha)} \wedge (d^{-1}Sn^2\epsilon^2)^{1/(2\alpha)}$,
        it holds that
        \begin{align} \label{fdp_thm_vcm_up_homo}
            \|\widetilde{\bm{\beta}}- \bm{\beta}^*\|_{L^2}^2 =_{\log} O_p\Big\{ d(Snm)^{-\frac{2\alpha}{2\alpha+1}}+d^{\frac{2\alpha+1}{\alpha+1}}(Sn^2m\epsilon^2)^{-\frac{\alpha}{\alpha+1}}+ d(Sn)^{-1}  +d^2(Sn^2\epsilon^2)^{-1}\Big\}.
        \end{align}
    \end{enumerate}
\end{theorem}

Theorem \ref{fdp_thm_vcm_up} shows the privacy guarantee and estimation error for Algorithm \ref{algorithm_fdp_vcm}. Similar to the discussion in \Cref{remark_mean_fdp_r}, the condition required for $r$ in \Cref{fdp_thm_vcm_up}.\ref{fdp_thm_vcm_up_2} is mild. See  \Cref{remark_vcm_fdp_r} in the online supplementary material for details.

\subsection{Optimality and minimax lower bound}\label{section_vcm_fdp_low}
We present a minimax lower bound in \Cref{fdp_thm_vcm_low}.
\begin{theorem} \label{fdp_thm_vcm_low}
    Denote $\mathcal{P}_X$ the class of sampling distributions satisfying Assumption \ref{a_sample}, and $\mathcal{P}_Y$, $\mathcal{P}_G$ the classes of distributions for observations satisfying Assumption \ref{simple_vcm_a_model}.
    \begin{enumerate}[leftmargin=*]
        \item (Heterogeneous design.) \label{fdp_thm_vcm_low_hete} Suppose $\epsilon_s \in (0,1)$ and $\delta_s\log(1/\delta_s) \lesssim r_0^{-1}d^{-1}m^{-1}\epsilon_s^2$ for all $s\in [S]$ with $r_0 \in \mathbb{Z}_+$ satisfying $dr_0^{2\alpha+2} \asymp \sum_{s=1}^S \{r_0dn_sm \wedge T^{-1}n_s^2m\epsilon_s^2\}$. It holds that
        \begin{align} \notag
            &\underset{Q \in \mathcal{Q}^T_{\bm{\epsilon}, \bm{\delta}}}{\inf}  \underset{\widetilde{\bm{\beta}}}{\inf} \underset{\substack{P_X \in \mathcal{P}_X, P_Y \in \mathcal{P}_Y\\P_G \in \mathcal{P}_G}}{\sup} \mathbb{E}_{P_X, P_Y, P_G, Q}\|\widetilde{\bm{\beta}}- \bm{\beta}^*\|_{L^2}^2 \\ \label{fdp_thm_vcm_low_eq1}
           \gtrsim \;& \sup_{r \in \mathbb{Z}_+} \Big[ r^2d^2\Big\{\sum_{s=1}^S (rdn_sm \wedge T^{-1}n_s^2m\epsilon_s^2) + dr^{2\alpha+2}\Big\}^{-1}\Big]\vee d^2\Big\{\sum_{s=1}^S (dn_s\wedge T^{-1}n_s^2\epsilon_s^2)\Big\}^{-1},
            \vspace{-2em}
        \end{align}
        where $\mathcal{Q}^T_{\bm{\epsilon}, \bm{\delta}}$ is any mechanisms satisfying $(\bm{\epsilon},\bm{\delta},T)$-FDP defined in Definition \ref{def_fdp}.

        \item (Homogeneous design.) Suppose for any $s\in [S]$, the number of observations $n_s$ and the privacy budgets $\epsilon_s$ and $\delta_s$ are of the same order, i.e.~$n_s \asymp n$, $\epsilon_s \asymp \epsilon$ and $\delta_s \asymp \delta$. Suppose that $ \delta\log(\delta) \lesssim d^{-1}(Sn)^{-1/(2\alpha+1)}m^{-(2\alpha+2)/(2\alpha+1)}\epsilon^2 \wedge d^{-(2\alpha+1)/(2\alpha+2)}S^{-1/(2\alpha+2)}n^{-1/(\alpha+1)}$\\ $\times m^{-(2\alpha+3)/(2\alpha+2)}\epsilon^{(2\alpha+1)/(\alpha+1)}$.
        % \begin{align*}
        %     \delta\log(\delta) \lesssim d^{-1}S^{-\frac{1}{2\alpha+1}}n^{-\frac{1}{2\alpha+1}}m^{\frac{-2\alpha-2}{2\alpha+1}}\epsilon^2 \wedge d^{\frac{-2\alpha-1}{2\alpha+2}}S^{-\frac{1}{2\alpha+2}}n^{-\frac{1}{\alpha+1}}m^{\frac{-2\alpha-3}{2\alpha+2}}\epsilon^{\frac{2\alpha+1}{\alpha+1}}.
        % \end{align*}
        It holds for any $T \in [\min_{s\in [S]}n_s]$ that 
        \begin{align} \notag
            \underset{Q \in \mathcal{Q}^T_{\bm{\epsilon}, \bm{\delta}}}{\inf} &\underset{\widetilde{\bm{\beta}}}{\inf} \underset{\substack{P_X \in \mathcal{P}_X, P_Y \in \mathcal{P}_Y\\P_G \in \mathcal{P}_G}}{\sup} \mathbb{E}_{P_X, P_Y, P_G, Q}\|\widetilde{\bm{\beta}}- \bm{\beta}^*\|_{L^2}^2\\ \label{fdp_thm_vcm_low_homo}
            &\gtrsim  d(Snm)^{-\frac{2\alpha}{2\alpha+1}}\vee  d^{\frac{2\alpha+1}{\alpha+1}}(Sn^2m\epsilon^2)^{-\frac{\alpha}{\alpha+1}} \vee d(Sn)^{-1}  \vee d^2(Sn^2\epsilon^2)^{-1}.
        \end{align}        
    \end{enumerate}
\end{theorem}

The core idea behind Theorem \ref{fdp_thm_vcm_low} is an application of the Van-Trees inequality \citep[e.g.][]{gill1995applications} to two classes of distributions: one that each entry of $\bm{\beta}$ is a linear combination of finite basis functions, and the other that they are constant functions. These two result in the first and second terms in \eqref{fdp_thm_vcm_low_eq1}, respectively.
% Detailed proofs are deferred to \Cref{section_appendix_vcm_fdp_low}.

The upper bound \eqref{fdp_thm_vcm_up_homo} and the lower bound \eqref{fdp_thm_vcm_low_homo} together show that Algorithm \ref{algorithm_fdp_vcm} achieves optimal rate of convergence, up to poly-logarithmic factors, under a homogeneous setting. We acknowledge that there is a gap between the minimax lower and upper bounds under the heterogeneous setting, as in the discussion in Section \ref{section_mean_fdp_low}. We conjecture that the lower bound in \Cref{fdp_thm_vcm_low}.\ref{fdp_thm_vcm_low_hete} is not tight.  

Most of the discussions can be carried from those in Sections~\ref{section_mean_cdp} and \ref{section_mean_fdp}, and we focus on the effect of dimension $d$. We first note that even when $d = 1$, the VCM problem differs from the functional mean estimation in \eqref{mean_model_obs} due to the absence of random functional noise in the VCM, but, notably, \Cref{fdp_thm_vcm_low} recovers the rate in \Cref{fdp_thm_mean_low} when $d \asymp 1$.  

The first term in \eqref{fdp_thm_vcm_low_eq1} reflects the rate of estimating an $rd$-dimensional object and the second for a $d$-dimensional one. The privacy preservation imposes a severe curse of dimensionality compared to the non-private case, resulting in inflation by squaring the dimensions.  This is consistent with the observations in the DP literature \citep[e.g.][]{cai2021cost, li2024federated} and the discussions in Sections~\ref{section_mean_cdp_low} and \ref{section_mean_fdp_low}. To the best of our knowledge, this is the first study of estimation in VCM under the DP constraint.

\section{Numerical experiments}\label{section_numerical}
We conduct numerical experiments to demonstrate the feasibility of \Cref{algorithm_mean} and support our theoretical findings in \Cref{section_mean_cdp}. Detailed setups and more numerical results are deferred to \Cref{section_appendix_numerical} in the online supplementary material.

\subsection{Simulated data analysis} \label{section_numerical_simulated}
We simulated data from Model \eqref{mean_model_obs} with $\{X_{j}^{(i)}\}_{i=1,j=1}^{n,m} \stackrel{\text{i.i.d.}}{\sim} \text{Uniform}[0,1]$ and $U$ being a mean-zero Gaussian process with Mat\'{e}rn covariance function. The measurement errors are sampled as $\{\xi_{ij}\}_{i=1,j=1}^{n,m} \stackrel{\text{i.i.d.}}{\sim} N(0, 0.25)$. We aim to estimate the mean function, which is constructed as $ \mu^*(x) = 4/5 + 3/5\cos(2\pi x)+ 2/3\sin(2 \pi x)$. We vary the privacy budget $\epsilon \in \{0.5,0.6,0.7,0.8,0.9,1\}\cup \{3,4,5,6,7,8\}$
and fix the privacy leakage parameter $\delta = 10^{-3}$. 

Two separate functional mean estimation problems are considered and for each true mean function, two studies are carried out to investigate the effects of $m$ and $n$ individually. We carry out $100$ Monte Carlo experiments for each setting and report the mean and standard error of $\|\widetilde{\mu}-\mu^*\|_{L^2}^2$, where $\widetilde{\mu}$ is the privatized estimator obtained by Algorithm \ref{algorithm_mean}.

The simulation results are collected in Figures \ref{plot_simulated_mean_1}, with more deferred to \Cref{section_appendix_numerical_simulated} in the online supplementary material. The estimation error decreases in all plots as $\epsilon$ increases. Focusing on each curve in plots~(a) and (b) in the top panel, the estimation error is improved by increasing sampling frequency $m$. However, once $m$ reaches approximately $12$, additional increases in $m$ result in only minor improvements. This observation is consistent with our theoretical findings in \Cref{thm_mean_upper}. As $m$ reaches above~$12$, a phase transition occurs and the convergence rate becomes mostly independent of $m$. Focusing on the plots in (c) and (d) on the bottom panel, we can see that the estimators perform better as $n$ increases. 

\begin{figure}[!htbp]
    \centering
    \includegraphics[width=0.78\linewidth]{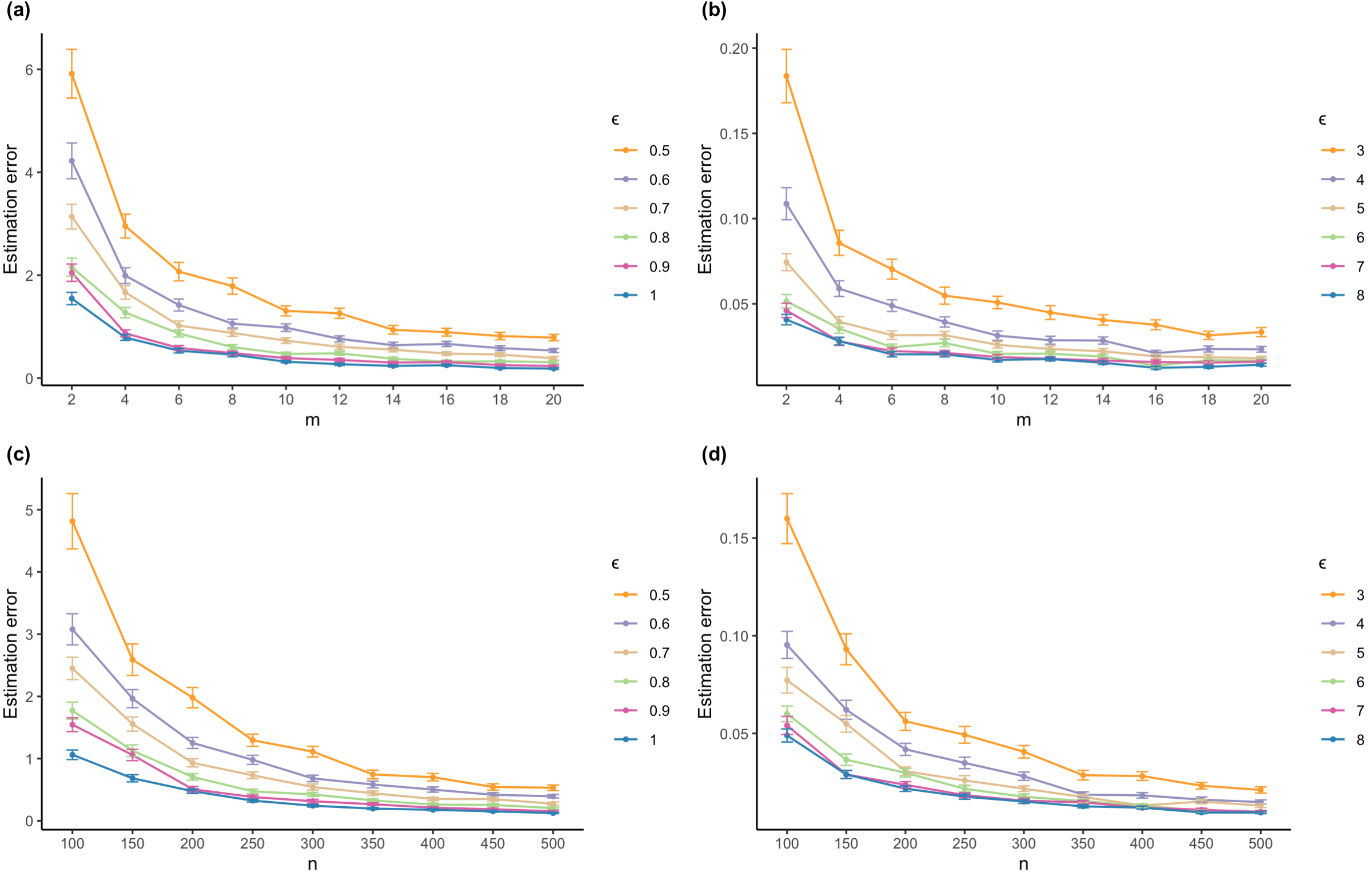}
    \caption{Simulation results for functional mean estimation. (a) and (b): Results as $m$ varies; (c) and (d): Results as $n$ varies.}
    \label{plot_simulated_mean_1}
\end{figure}

\subsubsection{Phase transition phenomenon}
To better understand the phase transition between private and non-private regimes, we present \Cref{fig_phase_large_m} to demonstrate the phase transition phenomenon from the private dense regime to the non-private dense regime. For each sample size $n \in \{200, 400, \ldots, 3600\}$, under the same setting, we run experiments with $m = n^{1/\alpha}$ and $\epsilon \in \{0.1, 0.2, \ldots, 8\}$. In \Cref{fig_phase_large_m}(a), for illustration purposes, we only report the cases when $\epsilon \in \{1,3,4,8\}$. The $x$-axis in \Cref{fig_phase_large_m}(a) is given by $\log(n)$, while the $y$-axis is given by $\log(\text{mean error})$ over $200$ iterations. We further include \Cref{fig_phase_large_m}(b), which is the right-most panel of \Cref{fig_phase_transition} and illustrates the theoretical phase transition in the corresponding dense~regime.

In \Cref{fig_phase_large_m}(a), for each value of $\epsilon$, we plot the best linear fit for $\log(\text{error})$ versus $\log(n)$. For $\epsilon=1$ and $\epsilon=8$, the curves exhibit a single dominant linear trend over the displayed range of $n$, and therefore only one line is fitted in each case. In contrast, for $\epsilon=3$ and $\epsilon=4$, the curves display two distinct linear trends with noticeably different gradients, and we thus fit two separate lines to reflect the two regimes. To facilitate comparison across regimes, we use different colours and line types to group fitted lines with similar gradients.

Ignoring all terms of the order $\log\{\log(n)\}$, \Cref{fig_phase_large_m}(b) suggests that, theoretically, the following holds: in private dense regime, we have $\log(\text{error}) = - 2\log(n) - 2\log(\epsilon) + \log(C_1)$, while in non-private dense regime $\log(\text{error}) = - \log(n) + \log(C_2)$, where $C_1, C_2 >0$ are unknown absolute constants. Therefore, a phase transition from the private to the non-private regime is reflected by a change of gradient from $-2$ to $-1$. 

In our numerical experiments, the fitted gradients coincide with the theoretical results. In particular, the average fitted gradients for the private dense regime is $-1.81$, while the average fitted gradient in the non-private dense regime is $-1.06$. Moreover, as $\epsilon$ increases, the phase transition occurs at smaller values of $n$, echoing the behaviour suggested by \Cref{fig_phase_large_m}(b) by comparing the horizontal dotted lines for $\epsilon =3$ and $\epsilon =4$.

\begin{figure}[!htbp]
  \centering
  \includegraphics[width=0.8\linewidth]{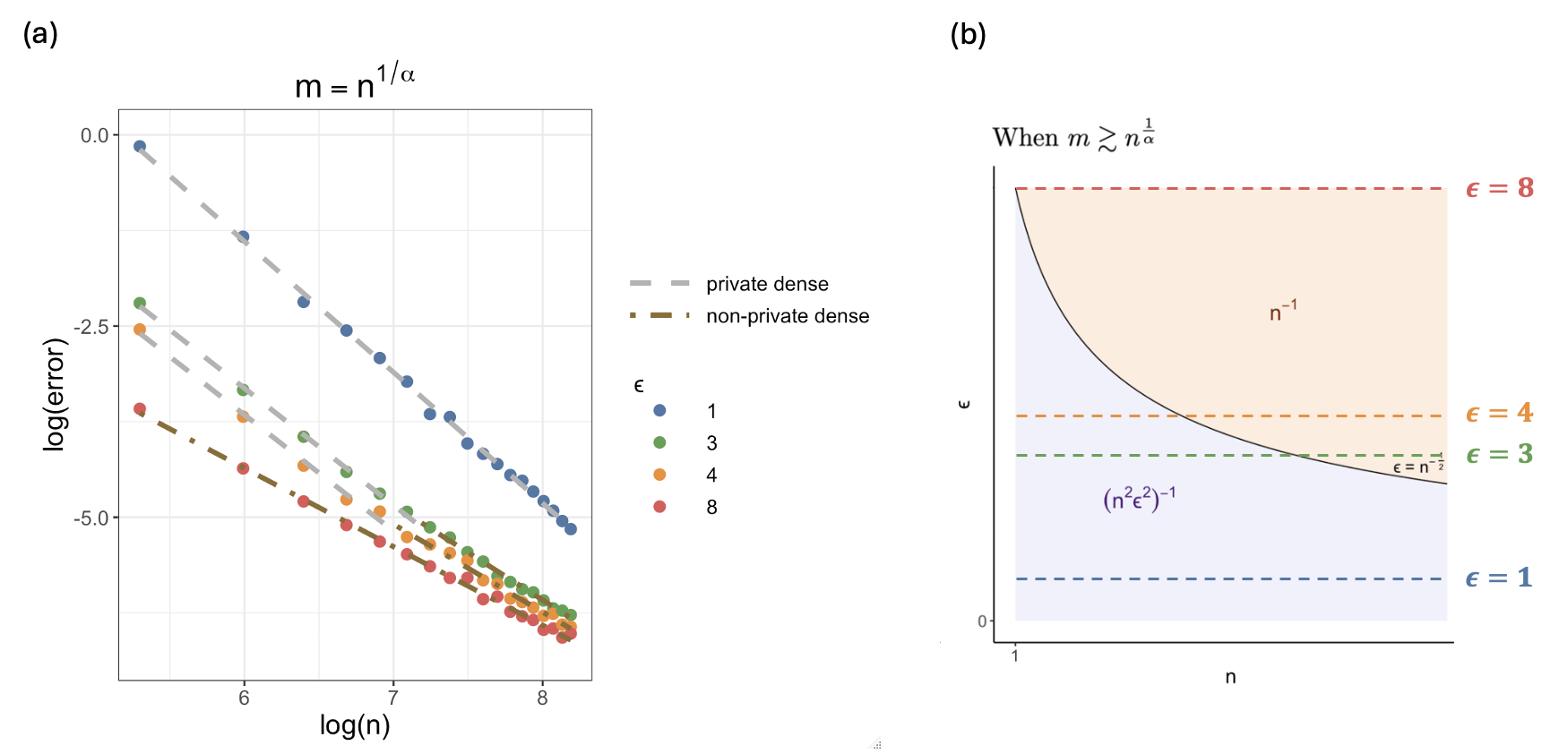}
  \caption{Plots for phase transition in the dense regime when $n \in [200,3600]$ and $m = n^{1/\alpha}$.}
  \label{fig_phase_large_m}
\end{figure}

\subsection{Real data example} \label{section_real_data}
We apply private functional mean estimation on the data from the Study of Women's Health Across the Nation (SWAN) Series \footnote{Accessible at: \url{https://www.icpsr.umich.edu/web/ICPSR/series/253}.}. This dataset has also been previously studied in the FDA literature \citep[e.g.][]{zhang2022nonparametric}.

We apply \Cref{algorithm_mean} to analyze the trend of the average level of estradiol $\{Y^{(i)}_j\}_{i=1,j=1}^{n,m}$ over age $\{X^{(i)}_j\}_{i=1,j=1}^{n,m}$ for middle-aged women with data from Visits 5 to 10, while adhering to privacy constraints. Estradiol is a major female hormone and plays a vital role in regulating reproductive cycles. Studying changes in estradiol levels in middle-aged women is essential, as its significant decline contributes to perimenopause and menopause. 

After excluding individuals with missing data points and abnormal estradiol levels which are defined as those with an average estradiol level exceeding 500 pg/ml, we model the data according to \eqref{mean_model_obs}. To be specific, individuals' estradiol curves over age are treated as realizations of a random function. The age range is $[46,63]$, the average estradiol range in pg/mL is $[3.8,484.4]$ and the data dimension is given by $n=678$ and $m=6$.  

We normalize values of age and average estradiol level by setting the minimum and maximum values to zero and one respectively. To handle non-periodicity, we use the Fourier extension technique as recommended by \cite{lin2021basis} due to its computational~stability.

To compare the effect of privacy budget $\epsilon$ on the estimation error, we randomly sample $1/3$ of data as training data in each iteration with $n_{\text{train}}= 226$, and leave the rest $2/3$ of data as testing data with $n_{\text{test}}= 452$. Non-private estimation with the ordinary gradient descent algorithm is firstly performed on the training data to get a non-private estimator $\widehat{\mu}_{\text{train}}$. Private functional mean estimation \Cref{algorithm_mean} is then summoned to the testing data to obtain a private estimator $\widetilde{\mu}_{\text{test}}$. We report the mean and standard error of  $\|\widetilde{\mu}_{\text{test}} - \widehat{\mu}_{\text{train}}\|_{L^2}$ over $100$ iterations in Figures \ref{plot_real_function}(a) and \ref{plot_real_function}(b). We also plot the means and $90 \%$ bands over 100 runs of the privacy mechanism with various $\epsilon$ over the full data in \Cref{plot_real_function}(c) to capture the randomness in the mechanism resulting from the random noise~added.

We see that in Figures \ref{plot_real_function}(a) and \ref{plot_real_function}(b), the estimation error decreases as $\epsilon$ increases, indicating the larger $\epsilon$, the smaller the noise induced by privacy preservation. This behaviour is also observed in \Cref{plot_real_function}(c). The non-private mean function reveals that, on average, estradiol levels in middle-aged women decrease as women pass the age of $50$, which aligns with the general findings in the existing literature \citep[e.g.][]{lephart2018review}. However, while the mean of the private estimator closely approximates the non-private estimator, individual estimators are less precise in capturing this trend, as illustrated by the shaded region. As expected, we observe that less deviances from the non-private estimator occur for a larger privacy budget, echoing theoretical results discussed in \Cref{section_mean_cdp}.

\begin{figure}[!htbp]
    \centering
    \includegraphics[width = 1\linewidth]{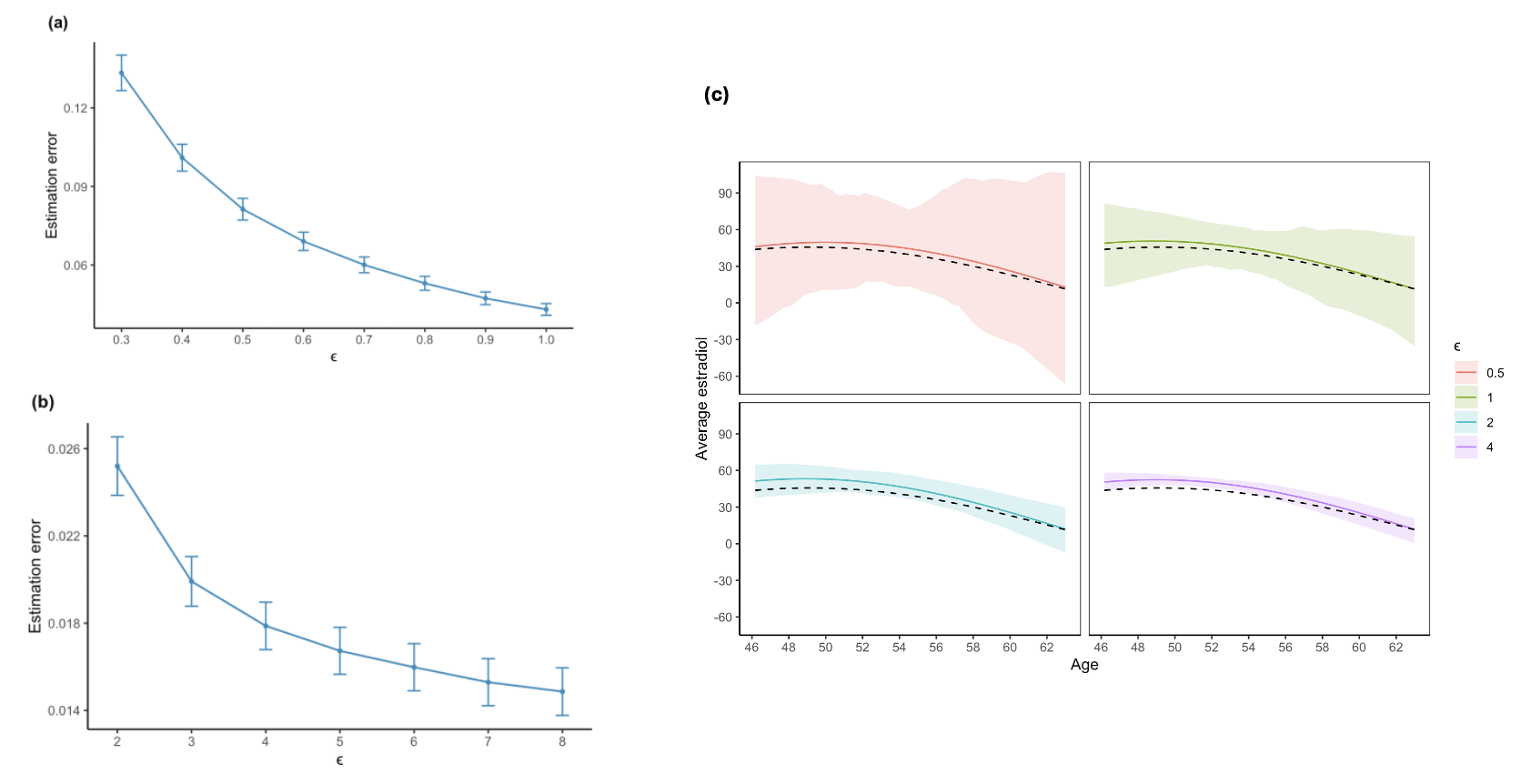}
    \caption{(a) and (b): Estimation results in the study of the average level of estradiol over age, with $\epsilon$ being the privacy budget. (c): Means and $90 \%$ bands of functional mean estimators over 100 repetitions of the privacy mechanism with various values of $\epsilon$. The dotted line represents the non-private estimator obtained from the ordinary gradient descent algorithm.}
    \label{plot_real_function}
\end{figure}

\section{Conclusions} \label{section_conclusion}
In this paper, we study optimal estimations in functional mean estimation and VCM coefficient function estimation under various DP frameworks. 
% To the best of our knowledge, these studies are the first time seen in the literature. 
By deriving matching upper and lower bounds (up to poly-logarithmic factors) on the minimax rates, we systematically examine the inevitable costs under CDP and FDP constraints and reveal new phase transition phenomena as both the sampling frequency and privacy budget vary.

We envisage several potential extensions. Even though we allow the existence of heterogeneity across servers, we claim optimality only under the homogeneous design in Sections~\ref{section_mean_fdp} and \ref{section_vcm_fdp}. We conjecture some further technical improvements in the construction of distribution classes during the application of the Van-Trees inequality may close the gap. 

In some scenarios, data from different servers may differ in distribution.  Transfer learning, as a research area, is dedicated to leveraging distinct but similar information across servers \citep[e.g.][]{cai2024transfer}. There has been some recent work attempting to study differentially private transfer learning, see for example \citet{li2024federated} and \citet{auddy2024minimax}. It would be interesting to further investigate optimality in distributed FDA with distributional shifts under the FDP framework.

Another potential extension of our work is to explore the data adaptivity in the privacy mechanism. In the design of our noisy gradient descent, independent noise calibrated to the global sensitivity based on the worst-case scenario is added across coordinates. However, in many real-life problems, it is possible to adapt to a given dataset and add noise based on data-dependent local sensitivity, for example, the Propose-Test-Release framework in \citet{dwork2009differential} and its application to mean and covariance estimation in \citet{brown2023fast}. In the problem of linear regression, with the intuition that the parameter output in each iteration is expected to get closer to the true value as the algorithm progresses, the data-adaptive gradient descent has been explored in works such as \citet{varshney2022nearly} and \citet{liu2023near}. Deriving practical data-adaptive algorithms and theoretical properties under our settings remains an intriguing area for further investigation. 

\section*{Funding}
This work was supported by EPSRC under Grant EP/Z531327/1; the Philip Leverhulme Prize; and MoE and NUS AcRF Tier 1 grant under Grant A-8002518-00-00.

\section*{Disclosure Statement}
The authors report there are no competing interests to declare.

\bibliographystyle{apalike}
\bibliography{ref}

@article{cai2011optimal,
author = {T. Tony Cai and Ming Yuan},
title = {{Optimal estimation of the mean function based on discretely sampled functional data: Phase transition}},
volume = {39},
journal = {The Annals of Statistics},
number = {5},
publisher = {Institute of Mathematical Statistics},
pages = {2330 -- 2355},
year = {2011}
}

@book{vershynin2018high,
  title={High-dimensional probability: An introduction with applications in data science},
  author={Vershynin, Roman},
  volume={47},
  year={2018},
  publisher={Cambridge university press}
}

@article{vershynin2010introduction,
  title={Introduction to the non-asymptotic analysis of random matrices},
  author={Vershynin, Roman},
  journal={arXiv preprint arXiv:1011.3027},
  year={2010}
}

@article{cai2021cost,
  title={The cost of privacy: Optimal rates of convergence for parameter estimation with differential privacy},
  author={Cai, T Tony and Wang, Yichen and Zhang, Linjun},
  journal={The Annals of Statistics},
  volume={49},
  number={5},
  pages={2825--2850},
  year={2021},
  publisher={Institute of Mathematical Statistics}
}

@article{quan2024optimal,
  title={Optimal one-pass nonparametric estimation under memory constraint},
  author={Quan, Mingxue and Lin, Zhenhua},
  journal={Journal of the American Statistical Association},
  volume={119},
  number={545},
  pages={285--296},
  year={2024},
  publisher={Taylor \& Francis}
}

@article{cai2023score,
  title={Score attack: A lower bound technique for optimal differentially private learning},
  author={Cai, T Tony and Wang, Yichen and Zhang, Linjun},
  journal={arXiv preprint arXiv:2303.07152},
  year={2023}
}

@article{zhang2007statistical,
author = {Jin-Ting Zhang and Jianwei Chen},
title = {{Statistical inferences for functional data}},
volume = {35},
journal = {The Annals of Statistics},
number = {3},
publisher = {Institute of Mathematical Statistics},
pages = {1052 -- 1079},
year = {2007},
}

@inproceedings{mirshani2019formal,
  title={Formal privacy for functional data with gaussian perturbations},
  author={Mirshani, Ardalan and Reimherr, Matthew and Slavkovi{\'c}, Aleksandra},
  booktitle={International Conference on Machine Learning},
  pages={4595--4604},
  year={2019},
  organization={PMLR}
}

@article{zhang2016from,
author = {Xiaoke Zhang and Jane-Ling Wang},
title = {{From sparse to dense functional data and beyond}},
volume = {44},
journal = {The Annals of Statistics},
number = {5},
publisher = {Institute of Mathematical Statistics},
pages = {2281 -- 2321},
year = {2016}
}

@article{hall2013differential,
  title={Differential privacy for functions and functional data},
  author={Hall, Rob and Rinaldo, Alessandro and Wasserman, Larry},
  journal={The Journal of Machine Learning Research},
  volume={14},
  number={1},
  pages={703--727},
  year={2013},
  publisher={JMLR. org}
}

@article{varshney2022nearly,
  title={(Nearly) Optimal Private Linear Regression via Adaptive Clipping},
  author={Varshney, Prateek and Thakurta, Abhradeep and Jain, Prateek},
  journal={arXiv preprint arXiv:2207.04686},
  year={2022}
}

@article{dwork2014algorithmic,
  title={The algorithmic foundations of differential privacy},
  author={Dwork, Cynthia and Roth, Aaron and others},
  journal={Foundations and Trends{\textregistered} in Theoretical Computer Science},
  volume={9},
  number={3--4},
  pages={211--407},
  year={2014},
  publisher={Now Publishers, Inc.}
}

@article{behzadan2021multiplication,
  title={Multiplication in Sobolev spaces, revisited},
  author={Behzadan, Ali and Holst, Michael},
  journal={Arkiv f{\"o}r Matematik},
  volume={59},
  number={2},
  pages={275--306},
  year={2021},
  publisher={Lehigh University Bethlehem, Penn., USA}
}

@article{henderson2024sobolev,
  title={Sobolev regularity of Gaussian random fields},
  author={Henderson, Iain},
  journal={Journal of Functional Analysis},
  volume={286},
  number={3},
  pages={110241},
  year={2024},
  publisher={Elsevier}
}

@article{steinwart2019convergence,
  title={Convergence types and rates in generic Karhunen-Loeve expansions with applications to sample path properties},
  author={Steinwart, Ingo},
  journal={Potential Analysis},
  volume={51},
  number={3},
  pages={361--395},
  year={2019},
  publisher={Springer}
}

@book{tsybakov2008introduction,
  title={Introduction to Nonparametric Estimation},
  author={Alexandre B. Tsybakov},
  year={2008},
  publisher={Springer Science \& Business Media}
}

@article{cai2024optimal,
  title={Optimal Federated Learning for Nonparametric Regression with Heterogeneous Distributed Differential Privacy Constraints},
  author={Cai, T Tony and Chakraborty, Abhinav and Vuursteen, Lasse},
  journal={arXiv preprint arXiv:2406.06755},
  year={2024}
}

@article{tropp2015introduction,
  title={An introduction to matrix concentration inequalities},
  author={Tropp, Joel A and others},
  journal={Foundations and Trends{\textregistered} in Machine Learning},
  volume={8},
  number={1-2},
  pages={1--230},
  year={2015},
  publisher={Now Publishers, Inc.}
}

@article{mihoc2003fisher,
  title={Fisher's information measures and truncated normal distributions (II)},
  author={Mihoc, Ion and F{\u{a}}tu, Cristina Ioana},
  journal={Revue d'analyse num{\'e}rique et de th{\'e}orie de l'approximation},
  volume={32},
  number={2},
  pages={177--186},
  year={2003}
}

@article{smith2021making,
  title={Making the most of parallel composition in differential privacy},
  author={Smith, Josh and Asghar, Hassan Jameel and Gioiosa, Gianpaolo and Mrabet, Sirine and Gaspers, Serge and Tyler, Paul},
  journal={arXiv preprint arXiv:2109.09078},
  year={2021}
}

@article{gill1995applications,
author = {Richard D. Gill and Boris Y. Levit},
title = {{Applications of the van Trees inequality: a Bayesian Cramér-Rao bound}},
volume = {1},
journal = {Bernoulli},
number = {1-2},
publisher = {Bernoulli Society for Mathematical Statistics and Probability},
pages = {59 -- 79},
year = {1995},
}

@article{dai2019age,
  title={Age-dynamic networks and functional correlation for early white matter myelination},
  author={Dai, Xiongtao and M{\"u}ller, Hans-Georg and Wang, Jane-Ling and Deoni, Sean CL},
  journal={Brain Structure and Function},
  volume={224},
  pages={535--551},
  year={2019},
  publisher={Springer}
}

@article{fraiman2014detecting,
  title={Detecting trends in time series of functional data: A study of Antarctic climate change},
  author={Fraiman, Ricardo and Justel, Ana and Liu, Regina and Llop, Pamela},
  journal={Canadian Journal of Statistics},
  volume={42},
  number={4},
  pages={597--609},
  year={2014},
  publisher={Wiley Online Library}
}

@article{crawford2020predicting,
  title={Predicting clinical outcomes in glioblastoma: an application of topological and functional data analysis},
  author={Crawford, Lorin and Monod, Anthea and Chen, Andrew X and Mukherjee, Sayan and Rabad{\'a}n, Ra{\'u}l},
  journal={Journal of the American Statistical Association},
  volume={115},
  number={531},
  pages={1139--1150},
  year={2020},
  publisher={Taylor \& Francis}
}

@article{wang2016functional,
  title={Functional data analysis},
  author={Wang, Jane-Ling and Chiou, Jeng-Min and M{\"u}ller, Hans-Georg},
  journal={Annual Review of Statistics and its application},
  volume={3},
  pages={257--295},
  year={2016},
  publisher={Annual Reviews}
}

@inproceedings{song2021evading,
  title={Evading the curse of dimensionality in unconstrained private glms},
  author={Song, Shuang and Steinke, Thomas and Thakkar, Om and Thakurta, Abhradeep},
  booktitle={International Conference on Artificial Intelligence and Statistics},
  pages={2638--2646},
  year={2021},
  organization={PMLR}
}

@inproceedings{dwork2006calibrating,
  title={Calibrating noise to sensitivity in private data analysis},
  author={Dwork, Cynthia and McSherry, Frank and Nissim, Kobbi and Smith, Adam},
  booktitle={Theory of Cryptography: Third Theory of Cryptography Conference, TCC 2006, New York, NY, USA, March 4-7, 2006. Proceedings 3},
  pages={265--284},
  year={2006},
  organization={Springer}
}

@article{butucealocal2020,
author = {Cristina Butucea and Amandine Dubois and Martin Kroll and Adrien Saumard},
title = {{Local differential privacy: Elbow effect in optimal density estimation and adaptation over Besov ellipsoids}},
volume = {26},
journal = {Bernoulli},
number = {3},
pages = {1727 -- 1764},
year = {2020},
doi = {10.3150/19-BEJ1165}
}

@article{lin2021basis,
  title={Basis expansions for functional snippets},
  author={Lin, Zhenhua and Wang, Jane-Ling and Zhong, Qixian},
  journal={Biometrika},
  volume={108},
  number={3},
  pages={709--726},
  year={2021},
  publisher={Oxford University Press}
}

@article{zhu2012multivariate,
author = {Hongtu Zhu and Runze Li and Linglong Kong},
title = {{Multivariate varying coefficient model for functional responses}},
volume = {40},
journal = {The Annals of Statistics},
number = {5},
pages = {2634 -- 2666},
year = {2012}
}

@article{hastie1993varying,
  title={Varying-coefficient models},
  author={Hastie, Trevor and Tibshirani, Robert},
  journal={Journal of the Royal Statistical Society Series B: Statistical Methodology},
  volume={55},
  number={4},
  pages={757--779},
  year={1993},
  publisher={Oxford University Press}
}

@article{klopp2015sparse,
author = {Olga Klopp and Marianna Pensky},
title = {{Sparse high-dimensional varying coefficient model: Nonasymptotic minimax study}},
volume = {43},
journal = {The Annals of Statistics},
number = {3},
pages = {1273 -- 1299},
year = {2015}
}

@article{chen2018inference,
  title={Inference of high-dimensional linear models with time-varying coefficients},
  author={Chen, Xiaohui and He, Yifeng},
  journal={Statistica Sinica},
  pages={255--276},
  year={2018},
  publisher={JSTOR}
}

@article{wei2011variable,
  title={Variable selection and estimation in high-dimensional varying-coefficient models},
  author={Wei, Fengrong and Huang, Jian and Li, Hongzhe},
  journal={Statistica Sinica},
  volume={21},
  number={4},
  pages={1515},
  year={2011},
  publisher={NIH Public Access}
}

@article{wang2008variable,
  title={Variable selection in nonparametric varying-coefficient models for analysis of repeated measurements},
  author={Wang, Lifeng and Li, Hongzhe and Huang, Jianhua Z},
  journal={Journal of the American Statistical Association},
  volume={103},
  number={484},
  pages={1556--1569},
  year={2008},
  publisher={Taylor \& Francis}
}

@article{huang2004polynomial,
  title={Polynomial spline estimation and inference for varying coefficient models with longitudinal data},
  author={Huang, Jianhua Z and Wu, Colin O and Zhou, Lan},
  journal={Statistica Sinica},
  pages={763--788},
  year={2004},
  publisher={JSTOR}
}

@article{zhang2015varying,
  title={Varying-coefficient additive models for functional data},
  author={Zhang, Xiaoke and Wang, Jane-Ling},
  journal={Biometrika},
  volume={102},
  number={1},
  pages={15--32},
  year={2015},
  publisher={Oxford University Press}
}

@article{fan2003adaptive,
  title={Adaptive varying-coefficient linear models},
  author={Fan, Jianqing and Yao, Qiwei and Cai, Zongwu},
  journal={Journal of the Royal Statistical Society Series B: Statistical Methodology},
  volume={65},
  number={1},
  pages={57--80},
  year={2003},
  publisher={Oxford University Press}
}

@article{duchi2018minimax,
  title={Minimax optimal procedures for locally private estimation},
  author={Duchi, John C and Jordan, Michael I and Wainwright, Martin J},
  journal={Journal of the American Statistical Association},
  volume={113},
  number={521},
  pages={182--201},
  year={2018},
  publisher={Taylor \& Francis}
}

@article{geyer2017differentially,
  title={Differentially private federated learning: A client level perspective},
  author={Geyer, Robin C and Klein, Tassilo and Nabi, Moin},
  journal={arXiv preprint arXiv:1712.07557},
  year={2017}
}

@article{zhang2024differentially,
  title={Differentially Private Federated Learning: Servers Trustworthiness, Estimation, and Statistical Inference},
  author={Zhang, Zhe and Nakada, Ryumei and Zhang, Linjun},
  journal={arXiv preprint arXiv:2404.16287},
  year={2024}
}

@article{zhou2023differentially,
  title={On differentially private federated linear contextual bandits},
  author={Zhou, Xingyu and Chowdhury, Sayak Ray},
  journal={arXiv preprint arXiv:2302.13945},
  year={2023}
}

@article{lowy2021private,
  title={Private federated learning without a trusted server: Optimal algorithms for convex losses},
  author={Lowy, Andrew and Razaviyayn, Meisam},
  journal={arXiv preprint arXiv:2106.09779},
  year={2021}
}

@article{auddy2024minimax,
  title={Minimax And Adaptive Transfer Learning for Nonparametric Classification under Distributed Differential Privacy Constraints},
  author={Auddy, Arnab and Cai, T Tony and Chakraborty, Abhinav},
  journal={arXiv preprint arXiv:2406.20088},
  year={2024}
}

@inproceedings{awan2019benefits,
  title={Benefits and pitfalls of the exponential mechanism with applications to hilbert spaces and functional pca},
  author={Awan, Jordan and Kenney, Ana and Reimherr, Matthew and Slavkovi{\'c}, Aleksandra},
  booktitle={International Conference on Machine Learning},
  pages={374--384},
  year={2019},
  organization={PMLR}
}

@article{lin2023differentially,
  title={Pure Differential Privacy for Functional Summaries via a Laplace-like Process},
  author={Lin, Haotian and Reimherr, Matthew},
  journal={arXiv preprint arXiv:2309.00125},
  year={2023}
}

@article{reimherr2019elliptical,
  title={Elliptical perturbations for differential privacy},
  author={Reimherr, Matthew and Awan, Jordan},
  journal={Advances in Neural Information Processing Systems},
  volume={32},
  year={2019}
}

@inproceedings{dwork2010differential,
  title={Differential privacy under continual observation},
  author={Dwork, Cynthia and Naor, Moni and Pitassi, Toniann and Rothblum, Guy N},
  booktitle={Proceedings of the forty-second ACM symposium on Theory of computing},
  pages={715--724},
  year={2010}
}

@article{levy2021learning,
  title={Learning with user-level privacy},
  author={Levy, Daniel and Sun, Ziteng and Amin, Kareem and Kale, Satyen and Kulesza, Alex and Mohri, Mehryar and Suresh, Ananda Theertha},
  journal={Advances in Neural Information Processing Systems},
  volume={34},
  pages={12466--12479},
  year={2021}
}

@article{kent2024rate,
  title={Rate Optimality and Phase Transition for User-Level Local Differential Privacy},
  author={Kent, Alexander and Berrett, Thomas B and Yu, Yi},
  journal={arXiv preprint arXiv:2405.11923},
  year={2024}
}

@article{cai2024transfer,
author = {T. Tony Cai and Dongwoo Kim and Hongming Pu},
title = {{Transfer learning for functional mean estimation: Phase transition and adaptive algorithms}},
volume = {52},
journal = {The Annals of Statistics},
number = {2},
publisher = {Institute of Mathematical Statistics},
pages = {654 -- 678},
year = {2024}
}

@article{yao2006penalized,
  title={Penalized spline models for functional principal component analysis},
  author={Yao, Fang and Lee, Thomas CM},
  journal={Journal of the Royal Statistical Society Series B: Statistical Methodology},
  volume={68},
  number={1},
  pages={3--25},
  year={2006},
  publisher={Oxford University Press}
}

@article{zhang2018optimal,
  title={Optimal weighting schemes for longitudinal and functional data},
  author={Zhang, Xiaoke and Wang, Jane-Ling},
  journal={Statistics \& Probability Letters},
  volume={138},
  pages={165--170},
  year={2018},
  publisher={Elsevier}
}

@article{li2024federated,
  title={Federated Transfer Learning with Differential Privacy},
  author={Li, Mengchu and Tian, Ye and Feng, Yang and Yu, Yi},
  journal={arXiv preprint arXiv:2403.11343},
  year={2024}
}

@article{huang2002varying,
  title={Varying-coefficient models and basis function approximations for the analysis of repeated measurements},
  author={Huang, Jianhua Z and Wu, Colin O and Zhou, Lan},
  journal={Biometrika},
  volume={89},
  number={1},
  pages={111--128},
  year={2002},
  publisher={Oxford University Press}
}

@article{he2018dimensionality,
  title={Dimensionality reduction and variable selection in multivariate varying-coefficient models with a large number of covariates},
  author={He, Kejun and Lian, Heng and Ma, Shujie and Huang, Jianhua Z},
  journal={Journal of the American Statistical Association},
  volume={113},
  number={522},
  pages={746--754},
  year={2018},
  publisher={Taylor \& Francis}
}

@article{zhang2022nonparametric,
  title={Nonparametric covariance estimation for mixed longitudinal studies, with applications in midlife women’s health},
  author={Zhang, Anru R and Chen, Kehui},
  journal={Statistica Sinica},
  volume={32},
  number={1},
  pages={345--365},
  year={2022},
  publisher={JSTOR}
}

@article{chaudhuri2013stability,
  title={A stability-based validation procedure for differentially private machine learning},
  author={Chaudhuri, Kamalika and Vinterbo, Staal A},
  journal={Advances in neural information processing systems},
  volume={26},
  year={2013}
}

@inproceedings{dwork2009differential,
  title={Differential privacy and robust statistics},
  author={Dwork, Cynthia and Lei, Jing},
  booktitle={Proceedings of the forty-first annual ACM symposium on Theory of computing},
  pages={371--380},
  year={2009}
}

@inproceedings{brown2023fast,
  title={Fast, sample-efficient, affine-invariant private mean and covariance estimation for subgaussian distributions},
  author={Brown, Gavin and Hopkins, Samuel and Smith, Adam},
  booktitle={The Thirty Sixth Annual Conference on Learning Theory},
  pages={5578--5579},
  year={2023},
  organization={PMLR}
}

@article{liu2023near,
  title={Near optimal private and robust linear regression},
  author={Liu, Xiyang and Jain, Prateek and Kong, Weihao and Oh, Sewoong and Suggala, Arun Sai},
  journal={arXiv preprint arXiv:2301.13273},
  year={2023}
}

@article{hall2006properties,
  title={On properties of functional principal components analysis},
  author={Hall, Peter and Hosseini-Nasab, Mohammad},
  journal={Journal of the Royal Statistical Society Series B: Statistical Methodology},
  volume={68},
  number={1},
  pages={109--126},
  year={2006},
  publisher={Oxford University Press}
}

@article{lephart2018review,
  title={A review of the role of estrogen in dermal aging and facial attractiveness in women},
  author={Lephart, Edwin D},
  journal={Journal of cosmetic dermatology},
  volume={17},
  number={3},
  pages={282--288},
  year={2018},
  publisher={Wiley Online Library}
}

@article{li2024federatedpca,
  title={Federated {PCA} and Estimation for Spiked Covariance Matrices: Optimal Rates and Efficient Algorithm},
  author={Li, Jingyang and Cai, T. Tony and Xia, Dong and Zhang, Anru R.},
  journal={arXiv preprint arXiv:2411.15660},
  year={2024}
}

@misc{appledp,
  author = {{Apple Differential Privacy Team}},
  title = {Learning with Privacy at Scale},
  howpublished = {\url{https://docs-assets.developer.apple.com/ml-research/papers/learning-with-privacy-at-scale.pdf}},
  year = {2023}
}

@misc{uscensus,
  author = {{United States Census Bureau}},
  title = {Census Bureau Sets Key Parameters to Protect Privacy in 2020 Census Results},
  year = 2021,
  howpublished = {\url{https://www.census.gov/newsroom/press-releases/2021/2020-census-key-parameters.html}}
}

@article{delaigle2021estimating,
  title={Estimating the covariance of fragmented and other related types of functional data},
  author={Delaigle, Aurore and Hall, Peter and Huang, Wei and Kneip, Alois},
  journal={Journal of the American Statistical Association},
  volume={116},
  number={535},
  pages={1383--1401},
  year={2021},
  publisher={Taylor \& Francis}
}

@inproceedings{huang2023federated,
  title={Federated linear contextual bandits with user-level differential privacy},
  author={Huang, Ruiquan and Zhang, Huanyu and Melis, Luca and Shen, Milan and Hejazinia, Meisam and Yang, Jing},
  booktitle={International Conference on Machine Learning},
  pages={14060--14095},
  year={2023},
  organization={PMLR}
}

@misc{HHS_HIPAA_Privacy_Rule_2025,
  author       = {{U.S. Department of Health and Human Services}},
  title        = {Summary of the HIPAA Privacy Rule},
  year         = {2025},
  howpublished = {\url{https://www.hhs.gov/hipaa/for-professionals/privacy/laws-regulations/index.html}}
}

@article{azizi2023comparison,
  title={A comparison of synthetic data generation and federated analysis for enabling international evaluations of cardiovascular health},
  author={Azizi, Zahra and Lindner, Simon and Shiba, Yumika and Raparelli, Valeria and Norris, Colleen M and Kublickiene, Karolina and Herrero, Maria Trinidad and Kautzky-Willer, Alexandra and Klimek, Peter and Gisinger, Teresa and others},
  journal={Scientific reports},
  volume={13},
  number={1},
  pages={11540},
  year={2023},
  publisher={Nature Publishing Group UK London}
}

@article{li2020multi,
  title={Multi-site fMRI analysis using privacy-preserving federated learning and domain adaptation: ABIDE results},
  author={Li, Xiaoxiao and Gu, Yufeng and Dvornek, Nicha and Staib, Lawrence H and Ventola, Pamela and Duncan, James S},
  journal={Medical image analysis},
  volume={65},
  pages={101765},
  year={2020},
  publisher={Elsevier}
}

\newpage
\appendix
\section*{Appendices}
In this supplementary material, we collect all technical details. Connections with relevant literature is systematically discussed in \Cref{section_appendix_literature}. More discussions of definitions, model assumptions and theoretical results are included in \Cref{section_appendix_additional_discussion}. Additional algorithms and theoretical guarantees for VCM estimation under CDP and FDP are presented in \Cref{section_appendix_vcm}. Phase transition phenomena for functional mean and VCM estimations under CDP and FDP are presented in \Cref{section_appendix_phase}. Detailed setups for numerical experiments in Section~\ref{section_numerical} are collected in \Cref{section_appendix_numerical}. The proofs of results in \Cref{section_mean_cdp} to \Cref{section_vcm_fdp} are provided in \Cref{section_appendix_mean_cdp} to \Cref{section_appendix_vcm_fdp} respectively. Some technical lemmas are listed in \Cref{section_appendix_technical}.

Throughout the supplementary material, unless specifically stated otherwise, let $C_1, C_2,\\ \ldots > 0$ denote absolute constants whose values may vary from place to place. For any vector $v \in \mathbb{R}^{p}$, let $\|v\|_{\op}$ denote the operator norm and for $\ell \in [p]$, let $v_\ell$ denotes its $\ell$th entry. For a matrix $A \in \mathbb{R}^{p \times p}$, let $\|A\|_{\op}$ denote the operator norm and $(A)_{ij}$ denote the $(i,j)$th entry.

\section{Connections with relevant literature} \label{section_appendix_literature}
There are three key ingredients in our analysis: Functional mean estimation, VCM and DP. We discuss the relevant literature separately in this subsection.
% More discussions on and comparisons of results will be presented in the sequel.

Functional mean estimation is a fundamental problem in FDA, interesting \emph{per se} and serving as the foundation for other statistical tasks in FDA, e.g.~functional principal component analysis.  Depending on the observation scheme, it can be categorized into (a) fully-observed functions \citep[e.g.][]{hall2006properties}, (b) fragmented-observed functions \citep[e.g.][]{delaigle2021estimating, lin2021basis} and (c) sparsely-observed functions \citep[e.g.][]{cai2011optimal, zhang2016from}. In this paper, we focus on the sparsely-observed case when the observational grids are independently sampled.  For this case, two estimation methods are commonly used in the literature: pre-smoothing and pooling. The former involves smoothing each curve separately and then averages, achieving $\sqrt{n}$ convergence rate ($n$ is the number of functions) only with sufficiently dense data \citep[e.g.][]{zhang2007statistical}.  The latter, pooling, can be adapted to both sparsely and densely observed functional data by utilizing all available information across subjects.

Varying coefficient models are useful extensions of the classical linear regression model due to their flexibility in capturing the heterogeneity of regression coefficients. 
Since the debut in \citet{hastie1993varying}, a wide range of VCMs have been extensively studied. For example, \citet{zhu2012multivariate} propose a multivariate VCM and systematically study the uniform convergence rate and asymptotic distribution of their proposed estimators.  \citet{zhang2015varying} study a varying coefficient additive model for functional data to capture the non-linear relationship between the response and covariates.  Sparse high-dimensional VCMs are studied in \citet{wei2011variable}, \citet{klopp2015sparse} and \citet{chen2018inference}.  Many methods have been developed for coefficient function estimation; see for example the basis expansion and/or spline methods \citep[e.g.][]{huang2004polynomial, wang2008variable} and methods based on local linear smoother \citep[e.g.][]{fan2003adaptive}.

Differential privacy (DP) is, arguably, the most prevailing privacy framework used in both academia and industry \citep{dwork2006calibrating}. In this paper, we focus on four different variants of DP, namely central differential privacy (CDP), local differential privacy (LDP), federated differential privacy (FDP) and user-level DP. CDP assumes that there exists a trusted central server that can securely store and process raw data from customers. In the recent literature, the trade-off between CDP and statistical accuracy has been studied in numerous settings, for example, low-dimensional regression \citep[e.g.][]{varshney2022nearly}, high-dimensional regression \citep[e.g.][]{cai2021cost} and nonparametric regression \citep[e.g.][]{cai2023score}.  The presence of a central server, however, may still raise privacy concerns if the server becomes compromised. To address this issue, a local version of privacy constraint, LDP, where data are privatized first before being sent to any processors, is proposed and systematically studied in the literature \citep[see e.g.][]{duchi2018minimax}.

Inspired by federated learning (FL), in the setting where data are distributed over multiple servers, the framework of FDP is proposed to preserve overall privacy. 
% At a high level, FDP can be viewed as offering a CDP-type guarantee within each server and LDP across servers.  
FDP has been extensively studied and appeared under different names in the literature. For example, \citet{geyer2017differentially} and \citet{zhang2024differentially} study the client-sided DP and differentially private FL in a setup where there exists a trusted central server aggregating non-private information from each server before sending back the privatized transcript. When there is no trusted central server, this problem is addressed under the framework of silo-level LDP in \citet{zhou2023differentially}, inter-Silo Record-Level DP in \citet{lowy2021private}, FDP in \citet{li2024federated}, and distributed DP in \citet{li2024federatedpca}, to name but a few. A setting where no information is sent from untrusted central server back to local servers for updates is discussed in \citet{cai2024optimal} and \citet{auddy2024minimax}. 

User-level DP is a variant on another dimension, discussing a stricter way to privatize when each user possesses multiple observations, while the statistical learning interests lie in the distribution of each observation.  
Though being a relatively new concept, user-level DP has been studied under the CDP in \cite{levy2021learning}, LDP in \cite{kent2024rate} and a federated
bandits framework in \cite{huang2023federated}, to name but a few. 

Private functional data analysis, especially when the private outputs are functions, has also received substantial attention, primarily within the CDP framework. \citet{hall2013differential}
propose a Gaussian mechanism for releasing a finite number of point-wise evaluations of functions. Building on this, \citet{mirshani2019formal} extend the previous Gaussian mechanism and enable the release of a broad class of functional estimators. Expanding on these ideas, \citet{reimherr2019elliptical} and \citet{lin2023differentially} generalize the Gaussian mechanism to privacy mechanisms using centred elliptical processes and Laplace-like processes. 
Additionally, \citet{awan2019benefits} employ the exponential mechanism to address privacy in functional principal component analysis. In the case when a finite sum of basis can well approximate infinite-dimensional functional space, methods leveraging the post-processing properties of DP have also been proposed, as seen in works like \cite{butucealocal2020} and \citet{cai2023score} for the study on local differentially private density estimation and central differentially private non-parametric regression respectively. In our paper, we build upon the latter idea, leveraging the post-processing property to tackle a more involved problem in differentially private FDA.

\section{Additional discussions} \label{section_appendix_additional_discussion}

\subsection{Additional discussions for Definition \ref{def_sobolev}} \label{section_appendix_discussion_sobolev}
To better understand \Cref{def_sobolev} in the main paper, we present another definition of Sobolev spaces.  
\begin{definition} \label{def-sob-standard}
    \begin{itemize}
        \item [(i)] For $\alpha \in \mathbb{N}$, the Sobolev space $\mathcal{H}^{\alpha}$ is 
        \[
            \mathcal{H}^{\alpha} = \left\{f: [0, 1] \to \mathbb{R}:\, f^{(\beta)} \in L^2([0, 1]), \ \forall |\beta| \leq \alpha \right\}.
        \]
        \item [(ii)] For non-integer $\alpha > 0$, the Sobolev space $\mathcal{H}^{\alpha}$ is        
        \[
            \mathcal{H}^{\alpha} = \left\{f: [0, 1] \to \mathbb{R}:\, \int_{[0, 1]} \big(1 + |x|^2\big)^{\alpha} \big|\hat{f}(x)\big|^2\,\mathrm{d}x < \infty\right\},
        \]
        where $\hat{f}(\cdot)$ is the Fourier transform of $f(\cdot)$.
    \end{itemize}
\end{definition}

For convenience, we have also copied our \Cref{def_sobolev} here.
\begin{definition} \label{def-sob-ours}
    Let $\alpha > 1$ and $C_\alpha >0$ be absolute constants and $\{\phi_\ell\}_{\ell\in \mathbb{N}+}$ be a collection of orthonormal basis functions of $L^2([0,1])$. For a given sequence of non-negative constants $\{\tau_\ell\}_{\ell \in \mathbb{N}_+}$ associated with $\{\phi_\ell\}_{\ell\in \mathbb{N}_+}$, the Sobolev space $\mathcal{W}(\alpha, C_\alpha)$ is  
    \begin{align*}
        \mathcal{W}(\alpha, C_\alpha) = \bigg\{f:[0,1]\rightarrow\mathbb{R}: f = \sum_{\ell=1}^\infty \phi_\ell\langle f, \phi_\ell \rangle_{L^2},\; \sum_{\ell=1}^\infty (\tau_\ell)^{2\alpha}\langle f, \phi_\ell \rangle_{L^2}^2 \leq C_\alpha^2 < \infty \bigg\}.
    \end{align*}
\end{definition}

\medskip
\noindent \textbf{Equivalence between the two definitions.}  The constraints in \Cref{def-sob-ours} imply that
\[
    \big|\langle f, \phi_\ell \rangle_{L^2}\big| \lesssim \tau_{\ell}^{-\alpha},
\]
which means that $\alpha$ controls how strongly large $\ell$ coefficients are penalised.  The larger $\alpha$ is, the faster decay is required for $\langle f, \phi_\ell \rangle_{L^2}$.  With our choices of the Fourier basis and the sequence $\{\tau_{\ell}\}$, we are now to show that 
\[
    \sum_{\ell = 1}^{\infty} \ell^{2\alpha} \big| \hat{f}(\ell)\big|^2 < \infty \Leftrightarrow f \in \mathcal{H}^{\alpha}.
\]

Considering only the Fourier basis, \Cref{def-sob-standard}(ii) can be written as
\begin{equation}\label{eq-sob-def-alt}
    \mathcal{H}^{\alpha} = \left\{f: [0, 1] \to \mathbb{R}: \, \sum_{\ell \in \mathbb{Z}} \big(1 + |\ell|^2\big)^{\alpha} \big|\hat{f}(\ell)\big|^2 < \infty\right\}.
\end{equation}
Note that for large $\ell$, it holds $(1 + \ell^2)^{\alpha} \asymp \ell^{2\alpha}$.  The difference between the constraints in \eqref{eq-sob-def-alt} and \Cref{def-sob-ours} only occurs finitely many low frequencies, so these two classes are equivalent.

\medskip
\noindent \textbf{How $\alpha$ in our definition regulates the derivatives.}  We explain using the case when $\alpha$ is an integer only.  For the non-integer case, the same arguments holds with the definition of fractional derivatives.

When existing, the $\beta$th derivative of $f$ is that
\[
    f^{(\beta)}(x) = \sum_{\ell \in \mathbb{Z}} (2\pi \iota \ell)^{\beta} \hat{f}(\ell) e^{2\pi \iota \ell x}.
\]
It follows with Parseval's identity that
\[
    \big\|f^{(\beta)}\big|_{L^2}^2 = \sum_{\ell \in \mathbb{Z}} \big|(2\pi \iota \ell)^\beta\big|^2 \big|\hat{f}(\ell)\big|^2,
\]
which holds due to the fact that
\[
    \widehat{f^{(\beta)}}(\ell) = (2\pi \iota \ell)^{\beta} \hat{f}(\ell).
\]
Since $|\iota| = 1$, we have that
\[
    \big|(2\pi \iota \ell)^\beta\big|^2 = (2\pi)^{2\beta} |\ell|^{2\beta},
\]
and consequently that
\[
    \|f^{(\beta)}\|_{L^2}^2 = (2\pi)^{2\beta} \sum_{\ell \in \mathbb{Z}} |\ell|^{2\beta} |\hat{f}(\ell)|^2,
\]
which concludes our claims.

\subsection{Additional discussions for sample splitting} \label{section_appendix_discussion_sample_splitting}
In Algorithms \ref{algorithm_mean}, \ref{algorithm_fdp_mean} and \ref{algorithm_fdp_vcm}, we performed sample splitting. To supplement \Cref{remark_sample_splitting}, below, we further discuss this in detail from a few angles.

\medskip
\noindent \textbf{Why we conduct sample splitting in the paper.}  In the differentially private stochastic gradient descent (DP-SGD) step, each iteration summons a fresh batch of data.  In each iteration, the independence between the batch of data and the current estimators significantly simplifies the proofs in upper and lower bounds.  This sample splitting does not affect the performance of our algorithm, as demonstrated by its minimax optimality.

\medskip
\noindent \textbf{Upper bounds still hold without sample splitting.} Without splitting data, one can reuse the same samples in different iterations in the DP-SGD.   The upper-bound guarantees for all algorithms can be extended by reanalysing the gradient bounds under data reuse. For example, in the proof of the theoretical guarantee of \Cref{algorithm_mean} (CDP), the key modification happens in the control of  $\mathcal{E}_2$ in \Cref{l_mean_upper_event} in the supplementary material, in which $a_t$ cannot be treated as fixed for rounds after time~$t$. Therefore, we consider the following modified event
    \begin{align*}
        \mathcal{E} = \Big\{\Pi^{\mathrm{entry}}_{R}\Big[&\frac{1}{m}\sum_{j=1}^m \Phi_r(X^{(i)}_j)\big\{\Phi^{\top}_r(X^{(i)}_j)a-Y^{(i)}_j \big\}\Big]\\
        &= \frac{1}{m}\sum_{j=1}^m \Phi_r(X^{(i)}_j)\big\{\Phi^{\top}_r(X^{(i)}_j)a-Y^{(i)}_j \big\}, \; \forall i \in [n],\; \; a \in \mathcal{A}\Big\}.
    \end{align*}
    Compared with the proof of \Cref{l_mean_upper_event}, to control $\mathcal{E}$, the only term involving $a$ that requires a refined entry-wise upper-bound analysis is the term $|\frac{1}{m}\sum_{j=1}^m \phi_\ell(X^{(i)}_j)\Phi^{\top}_r(X^{(i)}_j)a|$, $\ell \in [r]$.
    We provide a sketch of the proof below, where the high level idea is to exploit the Sobolev ellipsoid structure of the parameter space $\mathcal A=\{a\in\mathbb R^r:\sum_{k=1}^r k^{2\alpha}a_k^2\le C\}$ as given in \Cref{def-sob-ours} (i.e.~\Cref{def_sobolev} in the paper).
        
    For any $i \in [n], \ell \in [r]$ and $k\in[r]$, denote $W^{(i)}_{\ell,k,j}
    := \phi_\ell(X^{(i)}_j)\phi_k(X^{(i)}_j)$, $j=1,\ldots,m$,
    and we further write $(Z_{\ell,i})_k
    = \frac{1}{m}\sum_{j=1}^m (W^{(i)}_{\ell,k,j}-\mathbb{E} W^{(i)}_{\ell,k,j})$.  With the above notation, by Cauchy--Schwarz inequality, we have that,
    \begin{align*}
    &\sup_{a\in\mathcal A} \Bigg|\frac{1}{m}\sum_{j=1}^m \Big[\phi_\ell(X^{(i)}_j)\Phi^{\top}_r(X^{(i)}_j)a - \mathbb{E}\Big\{\phi_\ell(X^{(i)}_j)\Phi^{\top}_r(X^{(i)}_j)a\Big\}\Big]\Bigg|\\
    =\;&
    \sup_{a\in\mathcal A}\Big|\sum_{k=1}^r a_k (Z_{\ell,i})_k\Big|
    =
    \sup_{a\in\mathcal A}\Big|\sum_{k=1}^r \{k^\alpha a_k\}\,\{k^{-\alpha} (Z_{\ell,i})_k\}\Big| \\
    \leq \;&
    \sup_{a\in\mathcal A}\Big(\sum_{k=1}^r k^{2\alpha}a_k^2\Big)^{1/2}
    \Big\{\sum_{k=1}^r k^{-2\alpha}(Z_{\ell,i})_k^2\Big\}^{1/2} \leq \sqrt{C}\Big(\sum_{k=1}^r k^{-2\alpha}(Z_{\ell,i})_k^2\Big)^{1/2},
    \end{align*}
    where the last inequality follows from the definition of $\mathcal{A}$. In the rest of the proof, it suffices to present a large probability upper bound on $(\sum_{k=1}^r k^{-2\alpha}(Z_{\ell,i})_k^2)^{1/2}$. Since each summand in $(Z_{\ell,i})_k$ is bounded due to the boundedness property of the Fourier basis, it holds from the Hoeffding inequality and union bound arguments that 
    \begin{equation} \label{eq_upper}
        \mathbb{P}\Big\{\max_{i\in[n]}\max_{\ell\in[r]}\max_{k\in[r]}
    |(Z_{\ell,i})_k| \lesssim \sqrt{\frac{\log(nr^2/\eta)}{m}} \Big\} \geq 1-\eta.
    \end{equation}
    Suppose the event in \eqref{eq_upper} holds, and substitute the upper bound in $\sum_{k=1}^r k^{-2\alpha}(Z_{\ell,i})_k^2$, it then holds with probability at least $1-\eta$ that
    \begin{align*}
       \max_{i\in[n]}\max_{\ell\in[r]}\Big\{\sum_{k=1}^r k^{-2\alpha}(Z_{\ell,i})_k^2\Big\}^{1/2} \leq \Big\{\max_{i\in[n]}\max_{\ell\in[r]}\max_{k\in[r]}
    |(Z_{\ell,i})_k|^2 \cdot \sum_{k=1}^r k^{-2\alpha}\Big\}^{1/2} \lesssim \sqrt{\frac{\log(nr^2/\eta)}{m}},
    \end{align*}
    where the last inequality follows from the fact that $\sum_{k=1}^r k^{-2\alpha} \lesssim 1$ whenever $\alpha >1/2$. Moreover, the expectation term can be controlled by an upper bound of order $\ell^{-\alpha}$ following a similar argument as the one in the proof of \Cref{l_mean_upper_event}. Therefore, the rate of truncation radius $R$ will only inflate by at most a poly-logarithmic factor.

     In addition to the change of rates in $R$, we would also like to remark that since each iteration processes an entire epoch of the dataset, appropriate composition rules for differential privacy must be applied to avoid privacy leakage. This leads to a different noise calibration compared to the sample-splitting setting.
    
    \medskip
    \noindent \textbf{Challenges in the lower bound proof.} In terms of the lower-bound control, while the analysis for algorithms under CDP remains unaffected in this setting, deriving the corresponding minimax lower bounds under the new definition of FDP with data reuse appears substantially more challenging. We conjecture that the same rate would follow but we cannot provide proofs.

\subsection{Additional discussions for Assumption \ref{a_model}} \label{section_appendix_discussion_assumption_2}
Assumption \ref{a_model}(a)~characterizes the smoothness of the true mean function. It offers a suitable framework to ensure the theoretical performance of our algorithm. We remark that this assumption can be extended to other smoothness classes, such as Besov spaces, with minor theoretical adjustments. 

In Assumptions \ref{a_model}(b)~and \ref{a_model}(c), the tail behaviors of both the functional noise and measurement errors are regulated. Similar sub-Gaussian assumptions are adopted in the FDA literature \citep[e.g.][]{cai2024transfer}. Following a similar treatment by using different concentration inequalities, these tail behaviors could be relaxed to capture scenarios with heavier-tailed distributions (e.g.~sub-Weibull distributions).

In Assumption \ref{a_model}(b), we further regulate the smoothness of functional noise. This is essential to control the sensitivity of the gradient in Algorithm \ref{algorithm_mean}, hence determining the suitable level of noise used in the anisotropic Gaussian mechanism to preserve privacy. This assumption, at a high-level, can be achieved by imposing a smoothness condition on the covariance function $\Sigma^*$. As discussed in Corollary 4.5 in \citet{steinwart2019convergence}, ensuring that the sample path resides in $\mathcal{W}(\alpha,C_{\alpha})$ almost surely is equivalent to that the corresponding RKHS generated from $\Sigma^*$ can be continuously embedded into $\mathcal{W}(\alpha+1/2+\zeta, C)$ for some small $\zeta >0$ and absolute constant $C>0$. Examples include random process with the Mat\'ern covariance function of order $\alpha+\zeta$. In the case when $U(\cdot)$ is a Gaussian process, equivalent conditions in terms of spectral structures and Mercer's decomposition can be found in Proposition 4.4 in \citet{henderson2024sobolev}.

\subsection{Additional discussions for Section \ref{section_mean_cdp_low}} \label{section_appendix_disscussion_mean_cdp}
We account for the most general setting in Theorems \ref{thm_mean_upper} and \ref{thm_mean_lower} where we allow all parameters, including $m, \epsilon$ and $\delta$, to vary as functions of $n$.  We consequently unveil phase transition phenomena from sparse to dense regimes in FDA, private to non-private regimes in DP and their interplay.

\Cref{fig_phase_transition} anatomizes the parameter space with an FDA flavor by first investigating sampling frequency ranges.  From a DP-oriented standpoint, we list the parameter space in \Cref{table_transition_boundary_5}.
\begin{table}[!htbp]
\centering
\caption{Cost of privacy for central private functional mean estimation. All results are up to poly-logarithmic factors.}
\begin{tabular}{lll}\hline
\multicolumn{2}{c}{Regimes} & Rates\\ \hline
\multirow{2}{*}{High privacy: $0 < \epsilon \lesssim n^{-\frac{1}{2}}$} & $m \lesssim  (n^2\epsilon^2)^{\frac{1}{\alpha}}$ & $ (n^2m\epsilon^2)^{-\frac{\alpha}{\alpha+1}}$ \\
& $m \gtrsim (n^2\epsilon^2)^{\frac{1}{\alpha}} $ & $(n^2\epsilon^2)^{-1}$  \\ \hline
\multirow{3}{*}{Low privacy: $\epsilon \gtrsim n^{-\frac{1}{2}}$} & $m \lesssim n^{\frac{1}{2\alpha}}$ & $(nm)^{-\frac{2\alpha}{2\alpha + 1}} + (n^2m\epsilon^2)^{-\frac{\alpha}{\alpha+1}}$ \\ 
& $n^{\frac{1}{2\alpha}} \lesssim m \lesssim n^{\frac{1}{\alpha}}$ & $n^{-1} + (n^2m\epsilon^2)^{-\frac{\alpha}{\alpha+1}}$ \\ 
& $m \gtrsim  n^{\frac{1}{\alpha}}$ & $n^{-1} $ \\ \hline
\end{tabular}
\label{table_transition_boundary_5}
\end{table}

\subsection{Additional discussions for Section \ref{section_mean_fdp_low}} \label{section_appendix_disscussion_mean_fdp}
\textbf{The involvement of $T$ in the lower bound.} In the heterogeneous case, the lower bound depends on $T$. While in the upper bound - such as in \Cref{algorithm_fdp_mean} - $T$ plays a role in the estimator determining the number of iterations required for a mini-batch gradient descent algorithm, in the lower bound, $T$ arises from the class of privacy mechanisms considered. Specifically, the FDP definition in \Cref{def_fdp} accounts for a class of $T$-round interactive mechanisms. In the homogeneous case, the lower bound in \eqref{fdp_thm_mean_low_eq1} is an increasing function of $T$. Consequently, the lower bound with $T=1$ serves as a valid lower bound for all $T$, resulting in the $T$-independent lower bound in \eqref{fdp_thm_mean_low_eq2}.   The condition that $T \in [\min_{s \in [S]} n_s]$ is to eliminate the case that there are rounds with empty contributions from a certain server.  While this condition could be relaxed with more complex arguments, we impose it here for simplicity.

\noindent \textbf{Cost of privacy.} Similar to the case of CDP, the cost of privacy is reflected in the change of effective sample size from $Snm$ to $Sn^2m\epsilon^2$ and a loss in the exponent in the sparse regime due to a severe curse of dimensionality; and from $Sn$ to $Sn^2\epsilon^2$ in the dense regime. In our setup, user-level CDP is ensured for the $n$ functions within each server, while LDP is enforced across $S$ servers. The adjustment of the effective sample size to $Sn^2m\epsilon^2$ further confirms this unique user-level behavior if we compare this with the rate in \citet{cai2024optimal} (see \Cref{table_comparision_mean_fdp}), highlighting the distinctive capability of FDP to provide a practical framework to study a mixture of DP notations. Our rate in the dense regime aligns with the minimax rate of univariate mean estimation under FDP constraint in \citet{li2024federated} (see \Cref{table_comparision_mean_fdp}).
\begin{table}[!htbp]
\caption{Comparison of minimax rates for functional mean estimation in the setting of Theorem \ref{fdp_thm_mean_up}, univariate mean estimation with sample size $n$ and non-parametric regression for $\alpha$-Sobolev regression function with sample size $nm$ under item-level FDP with $S$ servers. Rates are presented up to constants and poly-logarithmic factors.}
\centering
\begin{tabular}{lll}
\hline
Setting                    & Minimax rate                                                                                         & Reference                                         \\ \hline
Func mean est & $(Snm)^{-\frac{2\alpha}{2\alpha+1}}+ (Sn^2m\epsilon^2)^{-\frac{\alpha}{\alpha+1}}  + (Sn)^{-1} + (Sn^2\epsilon^2)^{-1}$ & Theorem \ref{fdp_thm_mean_up}                                           \\
Uni mean est & $(Sn)^{-1}+(Sn^2\epsilon^2)^{-1}$                                                                        & \citet{li2024federated}              \\
Non-para reg  & $(Snm)^{-\frac{2\alpha}{2\alpha+1}} +(Sn^2m^2\epsilon^2)^{-\frac{\alpha}{\alpha+1}}$                                     & \citet{cai2024optimal}             \\ \hline
\end{tabular}
\label{table_comparision_mean_fdp}
\end{table}

For a better illustration of the fact that FDP serves in between CDP and LDP, we summarize the rates under different privacy constraints in Table~\ref{table_mean_dp}. From top to bottom, the problems become increasingly challenging, highlighting the fundamental differences among these constraints. In particular, the effective sample size in the privacy term under FDP is reduced by a factor of $S$ compared to that in CDP, while it is larger by a factor of $n$ compared with that under LDP.

\begin{table}[!htbp]
\centering
\begin{tabular}{lll}
\hline
Privacy    & Minimax rate & Reference \\ \hline
Non &    $(Snm)^{-\frac{2\alpha}{2\alpha+1}}+(Sn)^{-1}$          &    \citet{cai2011optimal}       \\
CDP        & $(Snm)^{-\frac{2\alpha}{2\alpha+1}} +(S^2n^2m\epsilon^2)^{-\frac{\alpha}{\alpha+1}} + (Sn)^{-1} + (S^2n^2\epsilon^2)^{-1}$ & Theorem \ref{thm_mean_upper}  \\
FDP       & $(Snm)^{-\frac{2\alpha}{2\alpha+1}}+ (Sn^2m\epsilon^2)^{-\frac{\alpha}{\alpha+1}}  + (Sn)^{-1} + (Sn^2\epsilon^2)^{-1}$ & Theorem \ref{fdp_thm_mean_up}          \\
LDP        &  $(Snm)^{-\frac{2\alpha}{2\alpha+1}}+ (Snm\epsilon^2)^{-\frac{\alpha}{\alpha+1}}  + (Sn)^{-1} + (Sn\epsilon^2)^{-1}$ & Theorem \ref{fdp_thm_mean_up}  \\ \hline
\end{tabular}
\caption{Minimax rates for functional mean estimation (up to poly-logarithmic factors) with $Sn$ numbers of observations and $m$ numbers of grids under various privacy frameworks.}
\label{table_mean_dp}
\end{table}

A new phase transition from sparse to dense and private to non-private regimes, similar to that discussed in Section \ref{section_mean_cdp_low}, can be observed here. We further discuss this in \Cref{section_appendix_phase}.

\section{Varying coefficient model estimation} \label{section_appendix_vcm}
\subsection{Additional details for \texorpdfstring{\Cref{section_vcm_fdp}}{}} \label{section_appendix_discussion_vcm_fdp}
In this subsection, we present the algorithm used in VCM estimation under FDP constraints. For any vector $B \in \mathbb{R}^{r(d+1)}$, we let $\Pi_{\mathcal{B}}^*(B)$ denote the projection of $B$ to the closest point (in $\ell_2$-norm sense) in $\mathcal{B} = \{\Bar{B} \in \mathbb{R}^{r(d+1)}: \widetilde{\Phi}_r^\top\Bar{B}$ satisfies Assumption \ref{simple_vcm_a_model}(c)$\}$. See Remark \ref{remark_projection_vcm} for details. 

\begin{algorithm}[!htbp] 
    \caption{Federate differentially private coefficient function estimation}
    \label{algorithm_fdp_vcm}
    \begin{algorithmic}[1]
        \Require Data $\{\{(X^{(s,i)}_j, \bm{G}^{(s,i)}, Y^{(s,i)}_j)\}_{i=1,j=1}^{n_s,m}\}_{s=1}^S$, Fourier basis $\Phi$, number of basis $r$, step size $\rho$, number of iterations $T$, weight $\{\nu_s\}_{s=1}^S$, constant for truncation $C_R$, privacy parameter $\epsilon$, $\delta$, initialization $B^0$, failure probability $\eta$.
        \State Set batch size $b_s = \lfloor n_s/T\rfloor$ for any $s\in [S]$, $N = \sum_{s=1}^S n_s$ and the entry-wise truncation radius $R_{h}~=C_R[\sqrt{m^{-1}\log(N/\eta)}+\{h-r (\lceil h/r\rceil-1)\}^{-\alpha}]$ for any $h \in [r(d+1)]$.
        \For{$t = 0 ,\ldots, T-1$}
        \For{$s = 1,\ldots,S$}
        \State Set $\tau_{s,t} = tb_s$. Generate $w_{s,t} \in \mathbb{R}^{r(d+1)}$ with $w_{s,t} \sim N(0, \Sigma_s)$, where $\Sigma_s =\mathrm{diag}(\sigma_{s,1}^2, \ldots,  \sigma_{s,r(d+1)}^2)$ and for any $h \in [r(d+1)]$, $ \sigma_{h}^2 = 16\log(2/\delta_s)R_{h}(\sum_{a=1}^{r(d+1)} R_{a})/(b_s^2\epsilon_s^2)$;
        \State $M_s^{t} = \frac{1}{b_s}\sum_{i=1}^{b_s}\Pi^{\mathrm{entry}}_{R}\Big[\frac{1}{m}\sum_{j=1}^m \widetilde{\Phi}_r(X_{j}^{(s,\tau_{s,t}+i)})\bm{G}^{(s,\tau_{s,t}+ i)}$\\
        \hspace{6.5cm}$\Big\{\bm{G}^{(s,\tau_{s,t}+ i)\top}\widetilde{\Phi}_r^\top(X_{j}^{(s,\tau_{s,t}+i)})B^t-Y_{j}^{(s,\tau_{s,t}+i)} \Big\}\Big]+w_{s,t}$.
        \EndFor
        \State $B^{t+1} = \Pi^{*}_{\mathcal{B}}\{B^t - \rho\sum_{s=1}^S \nu_sM_s^{t}\}$.
        \EndFor
        \State Let $\{\widetilde{b}_k\}_{k = 0}^d = \{b^T_k\}_{k = 0}^d$, where $((b^T_0)^\top,\ldots,(b_d^T)^\top)^{\top} = B^{T}$.
        \Ensure $\widetilde{B} = ((\widetilde{b}_0)^\top, \ldots, (\widetilde{b}_d)^\top)^{\top}$ and $\widetilde{\bm{\beta}} = (\widetilde{\beta}_0, \ldots, \widetilde{\beta}_d)^\top$, where $\widetilde{\beta}_k= \Phi_r^\top \widetilde{b}_k$.
    \end{algorithmic}
\end{algorithm}

\begin{remark} [Discussion on $\Pi^*_{\mathcal{B}}$]\label{remark_projection_vcm}
    The projection operator $\Pi^*_{\mathcal{B}}$ in Algorithms \ref{algorithm_fdp_vcm} and \ref{algorithm_vcm} is used to ensure that in each iteration $t\in \{0\}\cup [T-1]$, the estimator $B^{t+1} = \{(b_0^{t+1})^\top, \ldots, (b_d^{t+1})^\top\}^\top$ resides inside the parameter space of interest. Note that in the case when $d = O(1)$, then truncating using $\Pi^*_{\mathcal{B}}$ is equivalent to applying the projection operator $\Pi^*_{\mathcal{A}}$ in Remark \ref{remark_projection_mean} to $b^{t+1}_k$ for all $k \in \{0\}\cup [d]$. In the case when $d$ is a function of $n$, then in each iteration $t$, $\Pi^{*}_{\mathcal{B}}\{B^t - \rho\sum_{s=1}^S \nu_s M_s^{t}\}$ is equivalent to solving the following convex optimization problem:
    \begin{align*}
        &\hspace{-2cm}\argmin_{B = (b_0^\top, \ldots, b_d^\top)^\top \in \mathbb{R}^{r(d+1)}} \Big\|B- \Big(B^t - \rho\sum_{s=1}^S \nu_sM_s^{t}\Big)\Big\|_2^2\\
        \text{s.t.} \quad &\sum_{\ell=1}^r \ell^{2\alpha}(b_{k})^2_{\ell}\leq C_{\alpha}^2/\pi^{2\alpha}, \quad \forall k \in\{0\}\cup [d], \quad \text{and}\\
        &\sum_{\ell=1}^r \ell^{2\alpha}\Big\{\sum_{k=0}^d G_h^{(s, \tau_{s,t}+ i)}(b_k)_\ell\Big\}^2 \leq C_{\alpha}^2/\pi^{2\alpha},\\
        &\Big|\sum_{k=0}^d  G_k^{(s,\tau_{s,t}+ i)} \sum_{\ell=1}^r (b_k)_\ell \phi_{\ell}(X_{j}^{(s, \tau_{s,t}+i)})\Big| \leq C_b,\quad \forall s\in [S], \; i \in [b_s].
    \end{align*}
\end{remark}

Detailed rates (up to poly-logarithmic factors) with phase transition phenomena are presented below. 
\begin{itemize}
    \item \textbf{Regime 1:} When $0 < m \lesssim (Sn)^{\frac{1}{2\alpha}}$, it holds that 
    $$\|\widetilde{\bm{\beta}}- \bm{\beta}^*\|_{L^2}^2 =O_p\Big\{ d(Snm)^{-\frac{2\alpha}{2\alpha+1}}+d^{\frac{2\alpha+1}{\alpha+1}}(Sn^2m\epsilon^2)^{-\frac{\alpha}{\alpha+1}}+d^2(Sn^2\epsilon^2)^{-1}\Big\};$$

    \item \textbf{Regime 2:} When $(Sn)^{\frac{1}{2\alpha}} \lesssim m \lesssim (Sn)^{\frac{1}{\alpha}}$, it holds that 
    $$\|\widetilde{\bm{\beta}}- \bm{\beta}^*\|_{L^2}^2 = O_p\Big\{d^{\frac{2\alpha+1}{\alpha+1}}(Sn^2m\epsilon^2)^{-\frac{\alpha}{\alpha+1}}+ d(Sn)^{-1}  +d^2(Sn^2\epsilon^2)^{-1}\Big\}; \; \text{and}$$

    \item \textbf{Regime 3:} When $m \gtrsim (Sn)^{\frac{1}{\alpha}}$, it holds that 
    $$\|\widetilde{\bm{\beta}}- \bm{\beta}^*\|_{L^2}^2 = O_p\Big\{d(Sn)^{-1}  +d^2(Sn^2\epsilon^2)^{-1}\Big\}.$$
\end{itemize}

\subsection{Centrally private VCM estimation} \label{section_appendix_vcm_cdp}
In this section, we present details for VCM estimation under CDP. The proposed differentially private coefficient function estimator is produced by Algorithm~\ref{algorithm_vcm}. Minimax lower and upper bounds are presented in Corollaries \ref{corollary_vcm_lower} and \ref{corollary_vcm_upper} respectively.

% \begin{equation*}
%    \widetilde{\Phi}_r = \left(\begin{array}{cccc}
% \Phi_r& 0_r & \cdots & 0_r \\
% 0_r & \Phi_r &  & 0_r \\
% \vdots &  & \ddots &  \\
% 0_r & 0_r & \cdots &  \Phi_r
% \end{array}\right),
% \end{equation*}

\begin{algorithm}[!htbp] 
    \caption{Differentially private coefficient function estimation}
    \label{algorithm_vcm}
    \begin{algorithmic}[1]
        \Require Data $\{(X^{(i)}_j, \bm{G}^{(i)}, Y^{(i)}_j)\}_{i=1,j=1}^{n,m}$, Fourier basis $\Phi$, number of basis $r$, step size $\rho$, number of iterations $T$, constant for truncation $C_R$, privacy parameter $\epsilon$, $\delta$, initialization $B^0$, failure probability~$\eta$.
        \State Set batch size $b = \lfloor n/T\rfloor$ and set $R_{h}~=C_R[\sqrt{m^{-1}\log(n/\eta)}+\{h-r (\lceil h/r\rceil-1)\}^{-\alpha}]$ as the entry-wise truncation radius for any $h \in [r(d+1)]$.
        \For{$t = 0 ,\ldots, T-1$}
        \State Set $\tau_t = bt$;
        \State Generate $w_t \in \mathbb{R}^{r(d+1)}$ with $w_t \sim N(0, \Sigma)$, where $\Sigma =\mathrm{diag}(\sigma_1^2, \ldots,  \sigma_{r(d+1)}^2)$ and for any $h~\in~[r(d+1)]$, $ \sigma_{h}^2 = 16\log(2/\delta)R_{h}(\sum_{a=1}^{r(d+1)} R_{a})/(b^2\epsilon^2)$;
        \State $B^{t+1} = \Pi^{*}_{\mathcal{B}} \Big\{ B^{t} - \rho \Big[\frac{1}{b}\sum_{i=1}^b\Pi^{\mathrm{entry}}_{R}\Big[\frac{1}{m}\sum_{j=1}^m \widetilde{\Phi}_r(X_{j}^{(\tau_t+i)})\bm{G}^{(\tau_t+ i)}$\\
        \hspace{7cm}$\Big\{\bm{G}^{(\tau_t+ i)\top}\widetilde{\Phi}_r^\top(X_{j}^{(\tau_t+i)})B^t-Y_{j}^{(\tau_t+i)} \Big\}\Big] + w_t\Big]\Big\}$.
        \EndFor
        \Ensure $B^{T} = ((b^T_0)^\top,\ldots,(b_d^T)^\top)^{\top}$ and $\widetilde{\bm{\beta}} = (\widetilde{\beta}_0, \ldots, \widetilde{\beta}_d)^\top$, where $\widetilde{\beta}_k= \Phi_r^\top b^T_k$.
    \end{algorithmic}
\end{algorithm}

\begin{corollary}\label{corollary_vcm_lower}
    Denote $\mathcal{P}_X$ the class of sampling distributions satisfying Assumption \ref{a_sample} and $\mathcal{P}_Y$, $\mathcal{P}_G$ the classes of distributions for observations satisfying Assumption \ref{simple_vcm_a_model}. Suppose that 
    $$\delta \log(1/\delta)\lesssim d^{-1}n^{-\frac{1}{2\alpha+1}}m^{\frac{-2\alpha-2}{2\alpha+1}}\epsilon^2 \wedge d^{\frac{-2\alpha-1}{2\alpha+2}}n^{-\frac{1}{\alpha+1}}m^{\frac{-2\alpha-3}{2\alpha+2}}\epsilon^{\frac{2\alpha+1}{\alpha+1}},$$ 
    then it holds 
    \begin{align*}
        \underset{Q \in \mathcal{Q}_{\epsilon, \delta}}{\inf} \underset{\widetilde{\bm{\beta}}}{\inf} \underset{\substack{P_X \in \mathcal{P}_X, P_Y \in \mathcal{P}_Y\\P_G \in \mathcal{P}_G}}{\sup} &\mathbb{E}_{P_X, P_Y, P_G, Q}\|\widetilde{\bm{\beta}}- \bm{\beta}^*\|_{L^2}^2 \\
        &\gtrsim dn^{-1} \vee d^2(n\epsilon)^{-2} \vee d(nm)^{-\frac{2\alpha}{2\alpha+1}} \vee d^{\frac{2\alpha+1}{\alpha+1}}(n^2m\epsilon^2)^{-\frac{2\alpha}{2\alpha+2}}.
    \end{align*}
    where $\mathcal{Q}_{\epsilon, \delta}$ the collection of all mechanisms satisfying $(\epsilon,\delta)$-CDP defined in Definition \ref{def_dp} with $\epsilon \in (0,1)$.
\end{corollary}

\begin{proof}[Proof of Corollary \ref{corollary_vcm_lower}]
    In the case when $S=1$ and $T=1$, by Theorem \ref{fdp_thm_vcm_low} we have that 
    \begin{align*}
    &\underset{Q \in \mathcal{Q}_{\epsilon, \delta}}{\inf} \underset{\widetilde{\bm{\beta}}}{\inf} \underset{\substack{P_X \in \mathcal{P}_X, P_Y \in \mathcal{P}_Y\\P_G \in \mathcal{P}_G}}{\sup} \mathbb{E}_{P_X, P_Y, P_G, Q}\|\widetilde{\bm{\beta}}- \bm{\beta}^*\|_{L^2}^2 \\
    \geq & \frac{d^2}{(n^2\epsilon^2) \wedge (dn)} + \frac{r^2d^2}{ (n^2m\epsilon^2 \wedge rdnm) + dr^{2\alpha+2}}\\
    \geq & \frac{d}{n} \vee \frac{d^2}{n^2\epsilon^2} \vee \Big\{\Big(\frac{rd}{nm} \vee \frac{r^2d^2}{n^2m\epsilon^2}\Big)\wedge dr^{-2\alpha} \Big\}\\
    \asymp & dn^{-1} \vee d(nm)^{-\frac{2\alpha}{2\alpha+1}}\vee d^2n^{-2}\epsilon^{-2} \vee d^{\frac{2\alpha+1}{\alpha+1}}(n^2m\epsilon^2)^{-\frac{\alpha}{\alpha+1}}.
\end{align*}
\end{proof}

\begin{corollary}\label{corollary_vcm_upper}
    Let $\{(X^{(i)}_j, \bm{G}^{(i)}, Y^{(i)}_j)\}_{i=1,j=1}^{n,m}$ be generated from Model \eqref{simple_vcm_model_obs} satisfying Assumptions \ref{a_sample} and \ref{simple_vcm_a_model}, then the following holds.
    \begin{enumerate}
        \item Algorithm \ref{algorithm_vcm} satisfies $(\epsilon, \delta)$-CDP defined in Definition \ref{def_dp}.
        
        \item Initialising the algorithm with $B^0 = 0$ with step size $\rho \in (0, (C_{\lambda}L)^{-1})$ being an absolute constant with $L$ and $C_{\lambda}$ defined in Assumptions \ref{a_sample} and \ref{simple_vcm_a_model}(a). Suppose 
        \begin{align*}
            n \gtrsim \log(n)\log^2(Trd/\eta)d , \; nm \gtrsim \log(n)\log^2(Trd/\eta)rd  \; \mathrm{and} \; n \gtrsim \log(n)\log(Trd/\eta)rd ,
        \end{align*}
        and $T = \lceil C_1\log(n)\rceil$ for an absolute constant $C_1 >0$. It then holds with probability at least $1-7\eta$ that
        \begin{align*}
            &\|\widetilde{\bm{\beta}}- \bm{\beta}^*\|_{L^2}^2 = \sum_{k=0}^d \|\widetilde{\beta}_k- \beta_k\|_{L^2}^2 \\
            \lesssim \; &\frac{1}{\eta}\Big(\frac{d\log(n)}{n} + \frac{dr\log(n)}{nm} + dr^{-2\alpha}\Big) \\
            & \hspace{1cm} + \Big\{ \frac{d^2}{n^2\epsilon^2}+\frac{d^2r^2\log(n/\eta)}{n^2m\epsilon^2} \Big\}\log^2(n)\log(\log(n)/\eta)\log(1/\delta),
        \end{align*}
        for a small $\eta \in (0,1/7)$.

        \item In addition, if we pick the number of basis $r$ at
        \begin{align*}
            r \asymp_{\log} n^{\frac{1}{2\alpha}} \wedge (nm)^{\frac{1}{2\alpha+1}} \wedge d^{-\frac{1}{2\alpha}}(n^2\epsilon^2)^{\frac{1}{2\alpha}} \wedge d^{-\frac{1}{2\alpha+2}}(n^2m\epsilon^2)^{\frac{1}{2\alpha+2}},
        \end{align*}
        then we have that
        \begin{align*}
            \|\widetilde{\bm{\beta}}- \bm{\beta}^*\|_{L^2}^2 =_{\log}O_p\Big\{dn^{-1} \vee d^2(n\epsilon)^{-2} \vee d(nm)^{-\frac{2\alpha}{2\alpha+1}} \vee d^{\frac{2\alpha+1}{\alpha+1}}(n^2m\epsilon^2)^{-\frac{2\alpha}{2\alpha+2}}\Big\}.
        \end{align*}
    \end{enumerate}
\end{corollary}

\section{Phase transitions boundaries} \label{section_appendix_phase}
In this section, we present new phase transition boundaries for functional mean estimation and VCM estimation under both CDP and FDP.

\begin{table}[!htbp]
\centering
\caption{Phase transition phenomena for central private functional mean estimation. All results are up to poly-logarithmic factors.}
\begin{tabular}{lllll} \hline
& Boundaries & Regimes & Rates & Remarks \\  \hline
\multirow{2}{*}{(A)} & \multirow{2}{*}{$m \asymp (n^2\epsilon^2)^{\frac{1}{\alpha}}$} & Sparse & $(n^2m\epsilon^2)^{-\frac{\alpha}{\alpha+1}}$ & \multirow{2}{*}{Private rates} \\ 
&& Dense & $(n^2\epsilon^2)^{-1}$ & \\ \hline
\multirow{2}{*}{(B)} & \multirow{2}{*}{$m \asymp n^{\frac{1}{2\alpha}}$} & Sparse & $(nm)^{-\frac{2\alpha}{2\alpha+1}}$ & \multirow{2}{*}{Non-private rates} \\ 
& & Dense & $n^{-1}$ &  \\ \hline
\multirow{2}{*}{(C)} & \multirow{2}{*}{$\epsilon \asymp n^{-\frac{\alpha}{2\alpha+1}}m^{\frac{1}{4\alpha+2}}$} & High privacy & $(n^2m\epsilon^2)^{-\frac{\alpha}{\alpha+1}}$ & \multirow{2}{*}{Sparse rates} \\
&& Low privacy & $(nm)^{-\frac{2\alpha}{2\alpha+1}}$ &  \\ \hline
\multirow{2}{*}{(D)} & \multirow{2}{*}{$\epsilon \asymp n^{-\frac{1}{2}}$} & High privacy & $(n^2\epsilon^2)^{-1}$ & \multirow{2}{*}{Dense rates} \\ 
& & Low privacy & $n^{-1}$ & \\ \hline
\multirow{2}{*}{(E)} & \multirow{2}{*}{$\epsilon \asymp n^{\frac{-\alpha+1}{2\alpha}}m^{-\frac{1}{2}}$} & Private \& sparse & $(n^2m\epsilon^2)^{-\frac{\alpha}{\alpha+1}}$ & \\
& & Non-private \& dense & $n^{-1}$  & \\ \hline        
\end{tabular}
\label{table_transition_boundary}
\end{table}

In \Cref{table_transition_boundary}, (A) and (B) characterize the effect of the sampling frequency, in the private and non-private cases respectively.  (C) and (D) show how the fundamental difficulties change as the privacy constraints change.  In (E), we present another phase transition phenomenon, as an interplay of the FDA and DP, from the harder end (private and sparse) to the easier end (non-private and dense)  of the spectrum.

Similarly, we demonstrate phase transition boundaries for functional mean estimation under FDP and VCM estimation under both CDP and FDP in Tables \ref{table_transition_boundary_fdp_mean} to \ref{table_transition_boundary_fdp_vcm} respectively.

\begin{table}[!htbp]
\caption{Phase transition phenomena for federated private functional mean estimation. All results are up to poly-logarithmic factors.}
\centering
\begin{tabular}{llll}
\hline
Boundaries & Regimes & Rates & Remarks \\ \hline
\multirow{2}{*}{$m \asymp (Sn^2\epsilon^2)^{\frac{1}{\alpha}}$} & Sparse & $(Sn^2m\epsilon^2)^{-\frac{\alpha}{\alpha+1}}$ & \multirow{2}{*}{Private rates} \\ & Dense & $(Sn^2\epsilon^2)^{-1}$ & \\ \hline
\multirow{2}{*}{$m \asymp (Sn)^{\frac{1}{2\alpha}}$} & Sparse & $(Snm)^{-\frac{2\alpha}{2\alpha+1}}$ & \multirow{2}{*}{Non-private rates} \\ 
&  Dense & $(Sn)^{-1}$ &  \\ \hline
\multirow{2}{*}{$\epsilon \asymp S^{\frac{1}{4\alpha+2}}n^{-\frac{\alpha}{2\alpha+1}}m^{\frac{1}{4\alpha+2}}$} & High privacy & $(Sn^2m\epsilon^2)^{-\frac{\alpha}{\alpha+1}}$ & \multirow{2}{*}{Sparse rates} \\ 
& Low privacy & $(Snm)^{-\frac{2\alpha}{2\alpha+1}}$ &  \\ \hline
\multirow{2}{*}{$\epsilon \asymp n^{-\frac{1}{2}}$} & High privacy & $(Sn^2\epsilon^2)^{-1}$ & \multirow{2}{*}{Dense rates} \\ 
&  Low privacy & $(Sn)^{-1}$ & \\ \hline
\multirow{2}{*}{$\epsilon \asymp S^{\frac{1}{2\alpha}}n^{\frac{-\alpha+1}{2\alpha}}m^{-\frac{1}{2}}$} & Private \& sparse & $(Sn^2m\epsilon^2)^{-\frac{\alpha}{\alpha+1}}$ & \\ 
& Non-private \& dense & $(Sn)^{-1}$  &  \\  \hline         
\end{tabular}
\label{table_transition_boundary_fdp_mean}
\end{table}

\begin{table}[!htbp]
\caption{Phase transition phenomena for central private VCM estimation. All results are up to poly-logarithmic factors.}
\centering
\begin{tabular}{llll}
\hline
Boundaries & Regimes & Rates & Remarks \\ \hline
\multirow{2}{*}{$m \asymp d^{-\frac{1}{\alpha}}(n^2\epsilon^2)^{\frac{1}{\alpha}}$} & Sparse & $d^{\frac{2\alpha+1}{\alpha+1}}(n^2m\epsilon^2)^{-\frac{\alpha}{\alpha+1}}$ & \multirow{2}{*}{Private rates} \\ 
& Dense & $d^2(n^2\epsilon^2)^{-1}$ & \\ \hline
\multirow{2}{*}{$m \asymp n^{\frac{1}{2\alpha}}$} & Sparse & $d(nm)^{-\frac{2\alpha}{2\alpha+1}}$ & \multirow{2}{*}{Non-private rates} \\ 
& Dense & $dn^{-1}$ &  \\ \hline
\multirow{2}{*}{$\epsilon \asymp d^{\frac{1}{2}}n^{-\frac{\alpha}{2\alpha+1}}m^{\frac{1}{4\alpha+2}}$} & High privacy & $d^{\frac{2\alpha+1}{\alpha+1}}(n^2m\epsilon^2)^{-\frac{\alpha}{\alpha+1}}$ & \multirow{2}{*}{Sparse rates} \\ 
& Low privacy & $d(nm)^{-\frac{2\alpha}{2\alpha+1}}$ &  \\ \hline
\multirow{2}{*}{$\epsilon \asymp d^{\frac{1}{2}}n^{-\frac{1}{2}}$} & High privacy & $d^2(n^2\epsilon^2)^{-1}$ & \multirow{2}{*}{Dense rates} \\ 
&  Low privacy & $dn^{-1}$ & \\ \hline
\multirow{2}{*}{$\epsilon \asymp d^{\frac{1}{2}}n^{\frac{-\alpha+1}{2\alpha}}m^{-\frac{1}{2}}$} & Private \& sparse & $d^{\frac{2\alpha+1}{\alpha+1}}(n^2m\epsilon^2)^{-\frac{\alpha}{\alpha+1}}$ & \\ 
&  Non-private \& dense & $dn^{-1}$  &  \\ \hline        
\end{tabular}

\label{table_transition_boundary_cdp_vcm}
\end{table}

\begin{table}[!htbp]
\caption{Phase transition phenomena for federated private VCM estimation. All results are up to poly-logarithmic factors. }
\centering
\begin{tabular}{llll}
\hline
Boundaries & Regimes & Rates & Remarks \\ \hline
\multirow{2}{*}{$m \asymp d^{-\frac{1}{\alpha}}(Sn^2\epsilon^2)^{\frac{1}{\alpha}}$} & Sparse & $d^{\frac{2\alpha+1}{\alpha+1}}(Sn^2m\epsilon^2)^{-\frac{\alpha}{\alpha+1}}$ & \multirow{2}{*}{Private rates} \\ 
& Dense & $d^2(Sn^2\epsilon^2)^{-1}$ & \\ \hline
\multirow{2}{*}{$m \asymp (Sn)^{\frac{1}{2\alpha}}$} & Sparse & $d(Snm)^{-\frac{2\alpha}{2\alpha+1}}$ & \multirow{2}{*}{Non-private rates} \\ 
& Dense & $d(Sn)^{-1}$ &  \\ \hline
\multirow{2}{*}{$\epsilon \asymp d^{\frac{1}{2}}S^{\frac{1}{4\alpha+2}}n^{-\frac{\alpha}{2\alpha+1}}m^{\frac{1}{4\alpha+2}}$} & High privacy & $d^{\frac{2\alpha+1}{\alpha+1}}(Sn^2m\epsilon^2)^{-\frac{\alpha}{\alpha+1}}$ & \multirow{2}{*}{Sparse rates} \\ 
& Low privacy & $d(Snm)^{-\frac{2\alpha}{2\alpha+1}}$ &  \\ \hline
\multirow{2}{*}{$\epsilon \asymp d^{\frac{1}{2}}n^{-\frac{1}{2}}$} & High privacy & $d^2(Sn^2\epsilon^2)^{-1}$ & \multirow{2}{*}{Dense rates} \\ 
& Low privacy & $d(Sn)^{-1}$ & \\ \hline
\multirow{2}{*}{$\epsilon \asymp d^{\frac{1}{2}}S^{\frac{1}{2\alpha}}n^{\frac{-\alpha+1}{2\alpha}}m^{-\frac{1}{2}}$} & Private \& sparse & $d^{\frac{2\alpha+1}{\alpha+1}}(Sn^2m\epsilon^2)^{-\frac{\alpha}{\alpha+1}}$ & \\ 
& Non-private \& dense & $d(Sn)^{-1}$  &  \\  \hline         
\end{tabular}
\label{table_transition_boundary_fdp_vcm}
\end{table}

\section{Details for numerical experiments in \texorpdfstring{\Cref{section_numerical}}{}} \label{section_appendix_numerical}
\subsection{Additional numerical details for \texorpdfstring{\Cref{section_numerical_simulated}}{}} \label{section_appendix_numerical_simulated}
% \subsection{Additional numerical details for \Cref{section_numerical_simulated}}
\noindent \textbf{Simulation setups.} We simulated data from Model \eqref{mean_model_obs} with $\{X_{j}^{(i)}\}_{i=1,j=1}^{n,m} \stackrel{\text{i.i.d.}}{\sim} \text{Uniform}[0,1]$, where~$U$ is a mean-zero Gaussian process with its Mat\'{e}rn covariance function
    \begin{align*}
        \mathbb{E}\{U(x)U(y)\} = (0.5)^2 \frac{2^{-3}}{\Gamma (4)}\Big(\sqrt{8}\frac{|x-y|}{0.8}\Big)^4 \mathcal{B}_{4}\Big(\sqrt{8}\frac{|x-y|}{0.8}\Big), \quad x, y \in [0,1],
    \end{align*}
where $\Gamma(\cdot)$ is the Gamma function and $\mathcal{B}_4$ is the modified Bessel function of the second kind with the parameter of $4$. The measurement errors are sampled as $\{\xi_{ij}\}_{i=1,j=1}^{n,m} \stackrel{\text{i.i.d.}}{\sim} N(0, 0.25)$. The smoothness parameter is set to $\alpha =3$. We remark that though the typical privacy regime of interest is when $\epsilon \in (0,1)$, in real-world applications, values of $\epsilon > 1$ are often chosen \citep[e.g.][]{uscensus, appledp}. Therefore, in our numerical experiments, we vary the privacy budget 
\[\epsilon \in \{0.5,0.6,0.7,0.8,0.9,1\}\cup \{3,4,5,6,7,8\}
\]
and fix the privacy leakage parameter $\delta = 10^{-3}$. 

Two separate functional mean estimation problems are considered and the true mean functions of interest are constructed as 
\begin{align*}
    \mu_1^*(x) = 4/5 + 3/5\cos(2\pi x)+ 2/3\sin(2 \pi x) \quad \text{and} \quad \mu_2^*(x) = 1/7 + 5x^2/7 - 10(1/2-x)^3/7,
\end{align*}
for all $x\in [0,1]$.

For each true mean function, two simulation studies are carried out to investigate the effects of $m$ and $n$ individually, with details below.
\begin{itemize}
    \item Setting 1: effect of $m$. Fix $n =250$, vary $m \in \{2,4,6,8,10,12,14,16,18,20\}$; and
    \item Setting 2: effect of $n$. Fix $m = 10$, vary $n \in \{100,150,200,250,300,350,400,450,500\}$.
\end{itemize}

\medskip 
\noindent \textbf{Evaluation metrics.} We carry out $100$ Monte Carlo experiments for each setting and report the mean and standard error of $\|\widetilde{\mu}-\mu^*\|_{L^2}^2$ across $100$ simulations, where $\widetilde{\mu}$ is the privatized estimator obtained by Algorithm \ref{algorithm_mean}.

\medskip 
\noindent \textbf{Tuning parameters.} There are four tuning parameters involved: $C_R$, the constant in entry-wise truncation rate $\{R_{\ell}\}_{\ell=1}^r$; $C_T$, the constant involved in the number of iterations $T$, such that $T = \lceil C_T\log(n)\rceil$; $C_r$, the constant involved in $r$ the number of selected basis, such that $$r =\Big\lceil C_r \big\{n^{\frac{1}{2\alpha}} \wedge (nm)^{\frac{1}{2\alpha+1}} \wedge (n^2\epsilon^2)^{\frac{1}{2\alpha}} \wedge (n^2m\epsilon^2)^{\frac{1}{2\alpha+2}}\big\} \Big \rceil;$$
and $\rho$, the learning rate. We fix $C_R = 0.75, C_T = 4, C_r = 1.25$ and $\rho = 0.1$. Since $\mu^*_1$ is constructed with the first three Fourier bases, we select the true value and pick $r=3$ in this case. A sensitivity analysis on tuning parameters is collected in Appendix \ref{section_appendix_sensitivity}. Data-driven methods for tuning parameter selection under DP constraint can be found in \citet{chaudhuri2013stability}.

\medskip
\noindent \textbf{Results.} The simulation results for estimating $\mu^*_1$ is presented in \Cref{plot_simulated_mean_1} in \Cref{section_numerical_simulated} while $\mu^*_2$ are collected in Figure \ref{plot_simulated_mean_2}. In addition to the observations discussed in \Cref{section_numerical_simulated}, in Figure \ref{plot_simulated_mean_2}, we further show that even when the true covariance function $\mu^*_2$ is not constructed with Fourier basis, our algorithm could still output satisfactory estimators under privacy constraints.

\begin{figure}[!htbp]
    \centering
    \includegraphics[width = 0.9\linewidth]{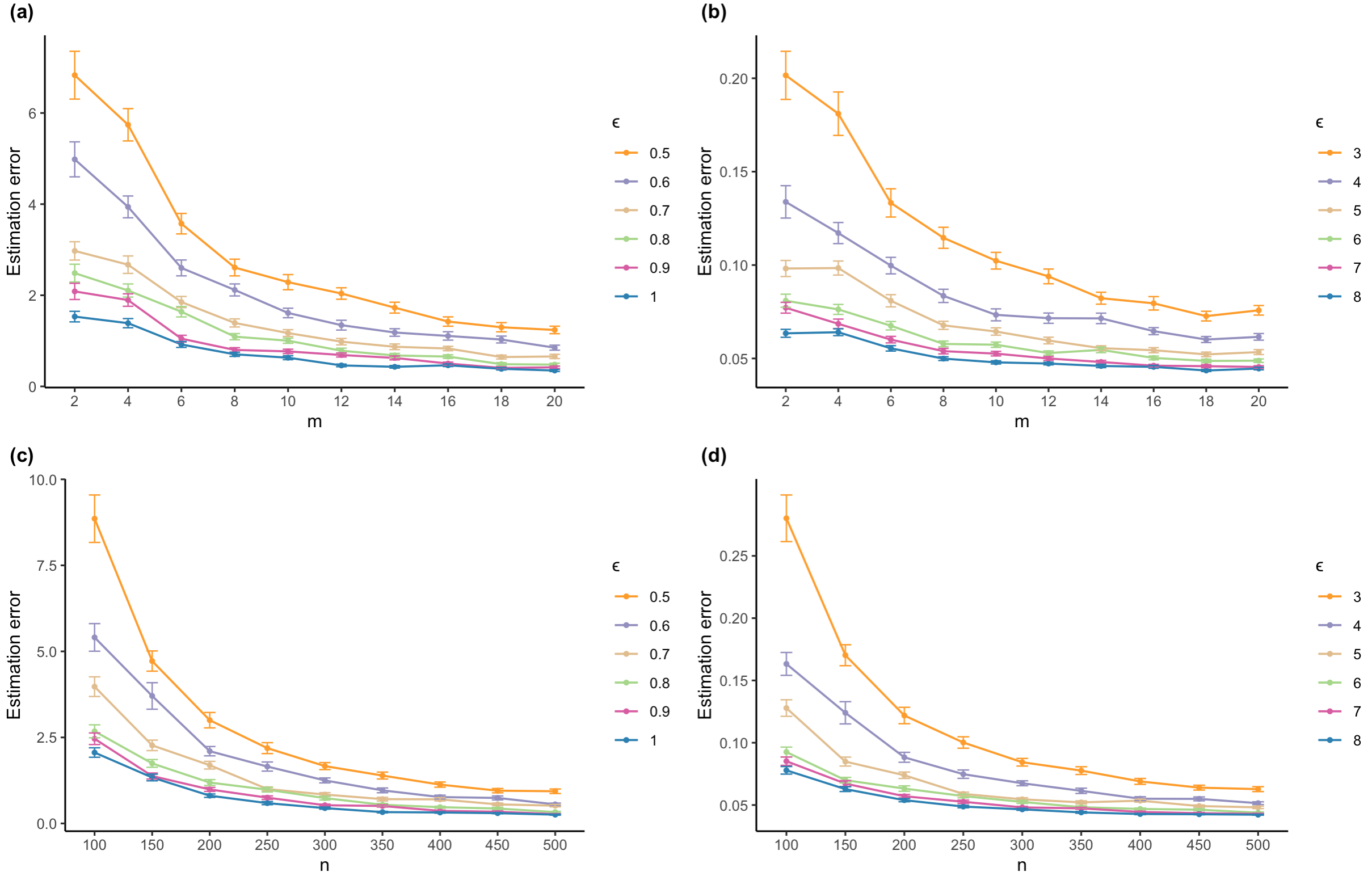}
    \caption{Estimation results for $\mu^*_2$. (a) and (b): Results in Setting 1 as $m$ varies; (c) and (d): Results in Setting 2 as $n$ varies. From left to right: $\epsilon~\in~\{0.5,0.6,0.7,0.8,0.9,1\}$ and $\epsilon~\in~\{3,4,5,6,7,8\}$.}
    \label{plot_simulated_mean_2}
\end{figure}

% \subsection{Additional numerical details for \texorpdfstring{\Cref{section_real_data}}{}} \label{section_appendix_numerical_real}
% To compare the effect of privacy budget $\epsilon$ on the estimation error, we randomly sample $1/3$ of data as training data in each iteration with $n_{\text{train}}= 226$, and leave the rest $2/3$ data as testing data with $n_{\text{test}}= 452$. Non-private estimation with the ordinary gradient descent algorithm is firstly performed on the training data to get a non-private estimator $\widehat{\mu}_{\text{train}}$. Private functional mean estimation \Cref{algorithm_mean} is then summoned to the testing data to obtain a private estimator $\widetilde{\mu}_{\text{test}}$. We report the mean and standard error of  $\|\widetilde{\mu}_{\text{test}} - \widehat{\mu}_{\text{train}}\|_{L^2}$ over $100$ iterations in \Cref{plot_real_error}.
% \begin{figure}[!htbp]
%     \centering
%     \includegraphics[width = 0.9\linewidth]{plot/real_estimation_error.png}
%     \caption{Estimation results in the study of the average level of estradiol over age, with $\epsilon$ being the privacy budget.}
%     \label{plot_real_error}
% \end{figure}

% We can see that in Figure \ref{plot_real_error}, the estimation error decreases as $\epsilon$ increases, indicating the larger $\epsilon$, the smaller the noise induced by privacy preservation. 

\subsection{Additional numerical results on phase transition for \texorpdfstring{$\mu^*_1$}{}}
We carried out additional experiments under settings with $\epsilon \in [0.2, 2]$, $n \in [200, 1500]$, and $m \in \{2, 6, 10, 15\}$. We plot in \Cref{fig_simulated_phase} the heatmaps of the log-transformed average estimation error over 200 iterations for each value of $m$. In these heatmaps, the number of functions $n$ is shown on the horizontal axis, while the privacy parameter $\epsilon$ is shown on the vertical axis. 

In all four plots in \Cref{fig_simulated_phase}, we observe that for a fixed value of $m$,  the estimation error decreases as either $n$ or $\epsilon$ increases. Moreover, comparing the plots for $m=2$ and $m=6$, the estimation error decreases noticeably as $m$ increases, particularly in settings with strong privacy level or limited sample size, i.e.~small values of $\epsilon$ and $n$. However, the improvement becomes more modest when we compare the two plots at the bottom, corresponding to cases when $m=10$ and $m=15$. This empirical behaviour is consistent with our theoretical results in \Cref{thm_mean_upper} and can be further supported by \Cref{fig_plot_m}, where we perform simulations analogous to Figures~\ref{plot_simulated_mean_1}(a) and~\ref{plot_simulated_mean_1}(b) in the main paper on a finer grid of $m \in [2,25]$ and $n \in \{300,800\}$. Once $m$ reaches approximately above $10$, a phase transition occurs, and the convergence rate becomes largely independent of~$m$, especially for larger values of $\epsilon$. 

\begin{figure}[!htbp]
  \centering
  \includegraphics[width=0.68\linewidth]{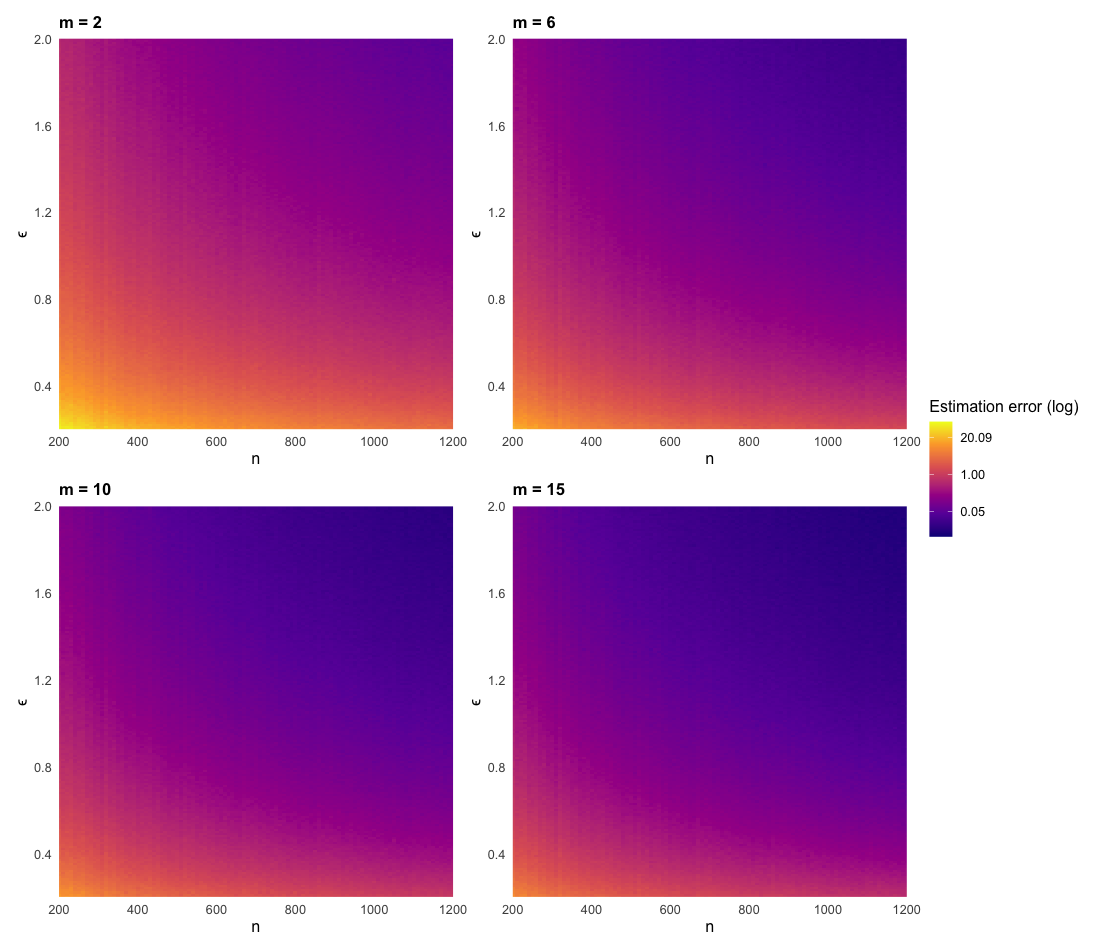}
  \caption{Simulation results over $200$ iterations for functional mean estimation under the same setting described in \Cref{section_numerical_simulated}.}
  \label{fig_simulated_phase}
\end{figure}

\begin{figure}[!htbp]
  \centering
  \includegraphics[width=1\linewidth]{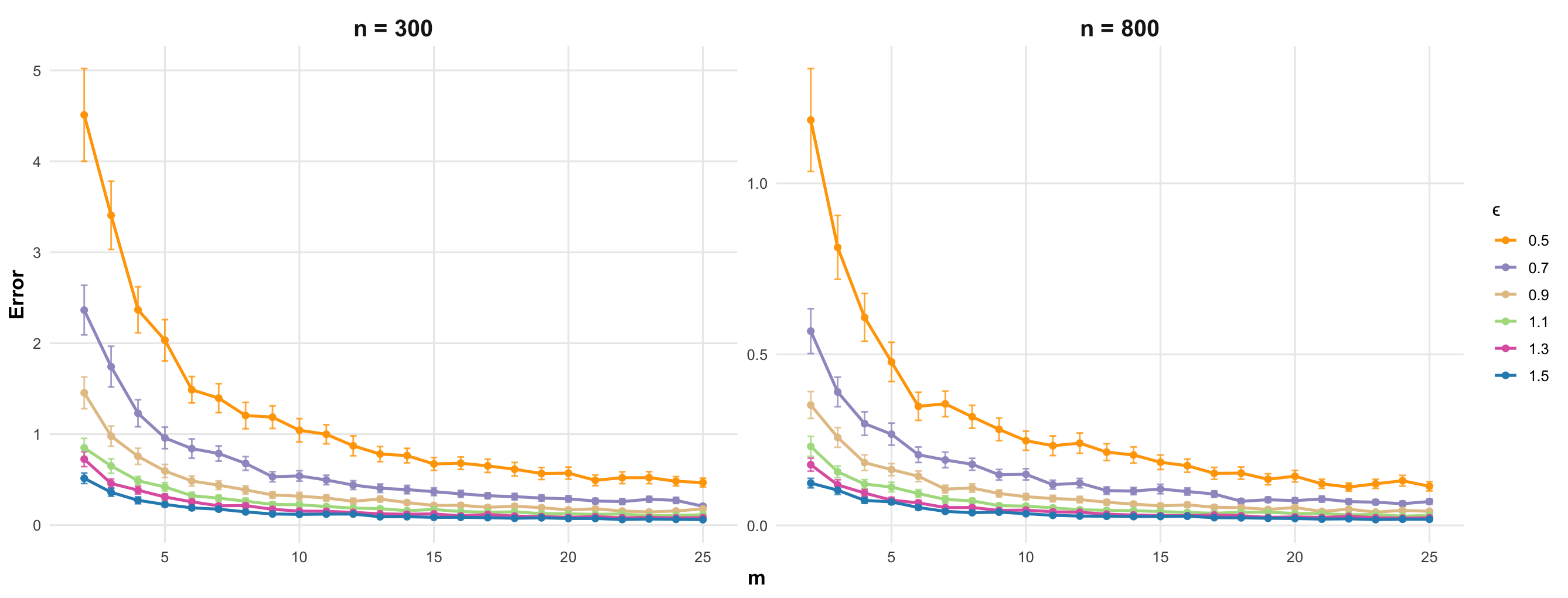}
  \caption{Simulation results over $200$ iterations for functional mean estimation for $m \in [2,25]$.}
  \label{fig_plot_m}
\end{figure}
\subsection{Sensitivity analysis} \label{section_appendix_sensitivity}
In this subsection, we present the result of a sensitivity analysis conducted for estimating $\mu^*_1$ with $n=250$ and $m=10$. 
\begin{figure}[!htbp]
    \centering
    \includegraphics[scale = 0.46]{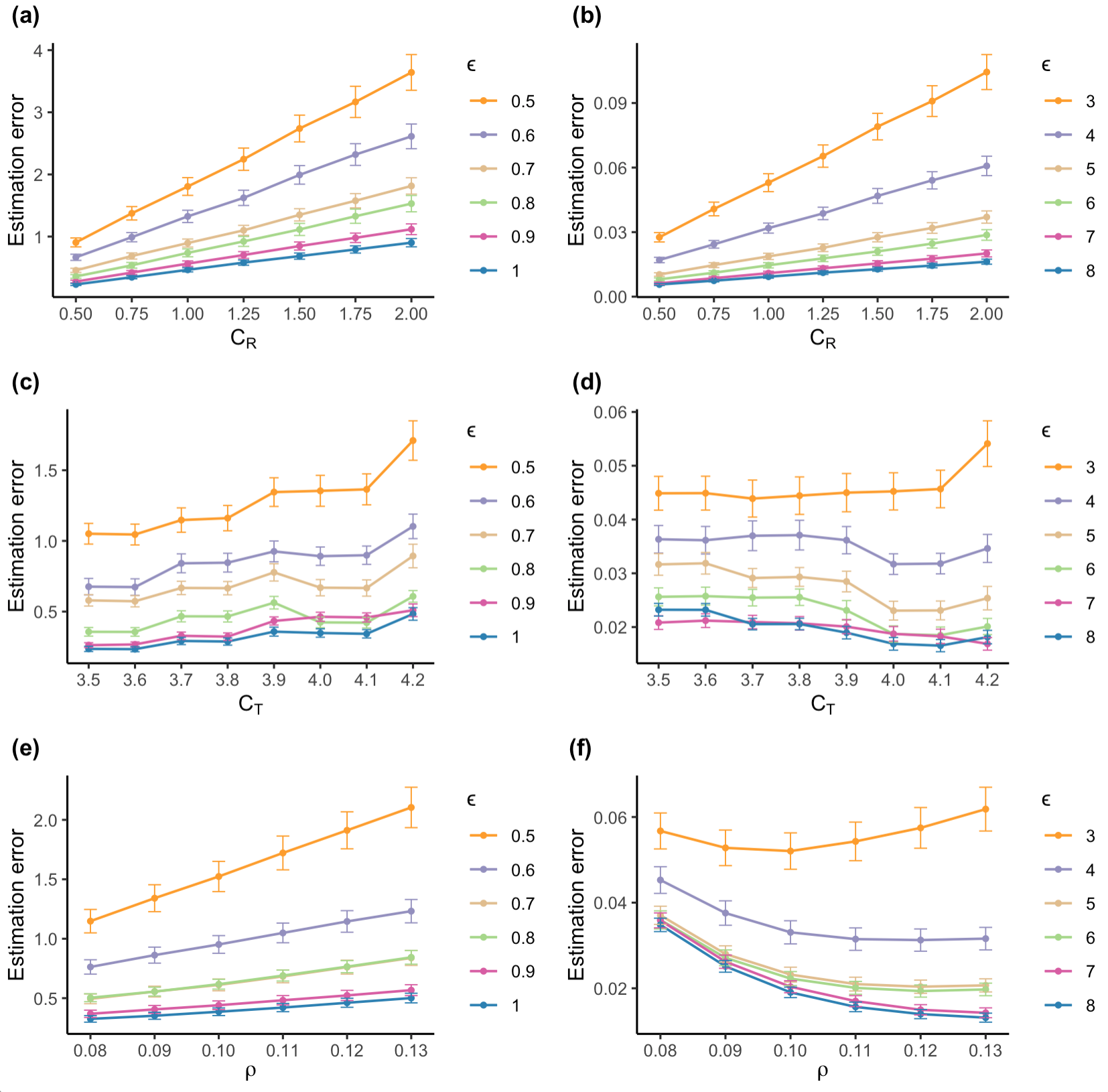}
    \caption{Sensitivity analysis for estimating $\mu^*_1$ with $n=250$ and $m=10$. From top to bottom: sensitivity analysis of constants $C_R$, $C_T$ and learning rate $\rho$. From left to right: different values of $\epsilon$.}
\end{figure}

Firstly, in the top panel, we study the sensitivity associated with the constant $C_R$, which is the constant in the level of entry-wise truncation $\{R_{\ell}\}_{\ell=1}^r$. As expected, on the top panel, the estimation error increases as the value of $C_R$ increases as a higher level of noise is added in each iteration of gradient descent. The effect of $C_R$ diminishes as the value of $\epsilon$ increases. 

In addition, in the middle panel, we examine the performance by considering a collection of $C_T$, which is used to choose the number of iterations. Roughly speaking, the performance of the estimator remains relatively stable across different values of $C_T$.

Finally, in the bottom panel, sensitivity analysis is performed for a range of learning rate values. Though it seems like trends do exist as $\rho$ varies, the variation in estimation error is not significant. In our simulation study, for demonstration purposes, we adopt $\rho =0.1$, which is a common default choice in gradient descent.

\section{Technical details for Section \ref{section_mean_cdp}} \label{section_appendix_mean_cdp}
\subsection{Proof of Lemma \ref{l_gaussian_mechanism}} \label{section_appendix_gaussian_mechanism}
\begin{proof}
    Let $D$ and $D'$ be two given neighboring datasets and consider drawing a $x \sim N_r(0,\Sigma)$, where $\Sigma \in \mathbb{R}^{r\times r}$ is a diagonal matrix with $\mathrm{diag}(\Sigma) = (\sigma^2_1, \ldots, \sigma^2_r)$.

    Without loss of generality, we let $f(D) = f(D')+\widetilde{\Delta} f$. Denote $g$ the probability density function of $M(D)$ and $g'$ the probability density function of $M(D')$, then by the construction of $M$, we have that 
    \begin{align*}
        \log\Big\{\frac{g\{f(D)+x\}}{g'\{f(D)+x\}}\Big\} &= \log\Big\{\frac{\exp(-\frac{1}{2}x^\top\Sigma^{-1}x)}{\exp(-\frac{1}{2}\{x+\widetilde{\Delta} f)^\top\Sigma^{-1}(x+\widetilde{\Delta} f)\}}\Big\}\\
        & = \frac{1}{2}\sum_{\ell=1}^r \sigma_\ell^{-2}\{(\widetilde{\Delta} f_\ell)^2 + 2\widetilde{\Delta} f_\ell x_\ell\},
    \end{align*}
    where $\widetilde{\Delta} f_\ell$ and $x_\ell$ are $\ell$th entries of $\widetilde{\Delta} f$ and $x$ respectively. Note that $x_\ell \sim N(0,\sigma_\ell^2)$ for any $\ell \in [r]$. Following from standard properties of Gaussian random variables, we have that 
    \begin{align*}
        W = \frac{1}{2}\sum_{\ell=1}^r \sigma_\ell^{-2}\{(\widetilde{\Delta} f_\ell)^2 + 2\widetilde{\Delta} f_\ell x_\ell\} \sim N\Big(\frac{1}{2}\sum_{\ell=1}^r\sigma_\ell^{-2}(\widetilde{\Delta} f_\ell)^2, \sum_{\ell=1}^r\sigma_\ell^{-2}(\widetilde{\Delta} f_\ell)^2\Big),
    \end{align*}
    which could also be written as 
    \begin{align*}
        W \stackrel{\mathcal{D}}{=} \frac{1}{2}\sum_{\ell=1}^r\sigma_\ell^{-2}(\widetilde{\Delta} f_\ell)^2 + Z\sqrt{\sum_{\ell=1}^r\sigma_\ell^{-2}(\widetilde{\Delta} f_\ell)^2},
    \end{align*}
    where $Z \sim N(0,1)$. To satisfy DP, we would like to have that
    \begin{align} \label{eq_GM_eq1}
        \mathbb{P}\Bigg\{Z \geq \frac{\epsilon}{\sqrt{\sum_{\ell=1}^r\sigma_\ell^{-2}(\widetilde{\Delta} f_\ell)^2}}- \frac{1}{2}\sqrt{\sum_{\ell=1}^r\sigma_\ell^{-2}(\widetilde{\Delta} f_\ell)^2}\Bigg\} \leq \frac{\delta}{2},
    \end{align}
    which implies that $\mathbb{P}(|W| \geq \epsilon) \leq \delta$.
    The aim now is to satisfy \eqref{eq_GM_eq1} and at the same time minimize the estimation loss incurred by the noise added. To do so, we would like to find the correct level of $\{\sigma_\ell^2\}_{\ell=1}^r$ satisfying the following optimization problem:
    \begin{align}\label{eq_GM_optimisation}
        &\min_{\{\sigma^2_\ell\}_{\ell=1}^r} \quad \sum_{\ell=1}^r \sigma_\ell^2  \quad \text{s.t.} \quad  \sum_{\ell=1}^r \frac{(\widetilde{\Delta} f_\ell)^2}{\sigma_\ell^2} = \frac{\epsilon^2}{4\log(2/\delta)}.
    \end{align}
    If the \eqref{eq_GM_optimisation} is satisfied, then we have that
    \begin{align*}
        \mathbb{P}\Big\{Z \geq 2\sqrt{\log(2/\delta)} - \frac{\epsilon}{4\sqrt{\log(2/\delta)}}\Big\} \leq \mathbb{P}(Z \geq \sqrt{\log(2/\delta)}) \leq \delta/2
    \end{align*}
    whenever $4\log(2/\delta) \geq \epsilon$, hence \eqref{eq_GM_eq1} is satisfied. To solve the optimization problem in \eqref{eq_GM_optimisation}, note that the Lagrangian function is
    \begin{align*}
        \mathcal{L}(\lambda) = \sum_{\ell=1}^r \sigma_\ell^2 + \lambda\Big( \sum_{\ell=1}^r \frac{(\widetilde{\Delta} f_\ell)^2}{\sigma_\ell^2}-\frac{\epsilon^2}{4\log(2/\delta)}\Big).
    \end{align*}
    Hence we have that 
    \begin{equation*}
        \frac{\partial\mathcal{L}}{\partial \sigma_\ell^2} = 1- \frac{\lambda \widetilde{\Delta} f_\ell^2}{(\sigma_\ell^2)^2} = 0 \quad \text{and} \quad \sigma_\ell^2 = \sqrt{\lambda \widetilde{\Delta} f_\ell^2}.
    \end{equation*}
    Substituting the above result into \eqref{eq_GM_optimisation}, we have that
    \begin{equation*}
        \lambda = \frac{16\log^2(2/\delta)\|\widetilde{\Delta} f\|_1^2}{\epsilon^4}, \quad \text{and}\quad \sigma_\ell^2 = \frac{4\log(2/\delta)|\widetilde{\Delta} f_\ell| \|\widetilde{\Delta} f\|_1}{\epsilon^2}.
    \end{equation*}
    % The Lemma is then followed by the new construction of $\widetilde{\Delta} f$ defined before the Lemma, in which case a higher level of noise is added to each entry.

    To generalize the argument to all neighbouring datasets, it then suffices to find 
    \[\sup_{D \sim D'} \sigma_\ell^2 \leq \frac{4\log(2/\delta)|\Delta f_\ell| \|\Delta f\|_1}{\epsilon^2},\]
    % \leq  \sup_{D \sim D'}\frac{4\log(2/\delta)|\widetilde{\Delta} f_\ell| \|\widetilde{\Delta} f\|_1}{\epsilon^2} \leq \frac{4\log(2/\delta)|\Delta f_\ell| \Delta_1(f)}{\epsilon^2} 
    where $\Delta f \in \mathbb{R}^r$ is the vector with the $\ell$-th entry defined as $\Delta f_\ell= \sup_{D\sim D'} |f_\ell(D)-f_\ell(D')|$. The lemma thus follows.
    
\end{proof}

\subsection{Proof of Theorem \ref{thm_mean_upper}} \label{section_appendix_mean_upper}
\begin{proof}[Proof of Theorem \ref{thm_mean_upper}]
To prove the first claim, since in each iteration, we use a disjoint set of independent data, by the parallel composition property \citep[e.g.][]{smith2021making} of the CDP mechanism, it suffices for us to show that $a^t$ in each iteration is guaranteed to be $(\epsilon,\delta)$-CDP. The Gaussian mechanism in Lemma \ref{l_gaussian_mechanism} is used here and we are required to find the sensitivity $\Delta f$ for the gradient function of interest. Note that by the entry-wise projection $\Pi^{\mathrm{entry}}_{R}$ defined in our algorithm, for any two sets of data $(x,y)_{1:m}^{(1:n)}$ and $(x', y')_{1:m}^{(1:n)}$ such that $\sum_{i=1}^n \mathbbm{1}\{(x,y)_{1:m}^{(i)} \neq (x', y')_{1:m}^{(i)}\}=1$, we have that for all $\ell \in [r]$, the $\ell$th entry of the sensitivity vector of gradient satisfies that
    \begin{align*}
        \Big(\big|f\{(x,y)_{1:m}^{(1:n)}\} - f\{(x', y')_{1:m}^{(1:n)}\}\big|\Big)_{\ell} \leq \frac{2R_\ell}{b}.
    \end{align*}
    Hence by choosing $\sigma_\ell^2 = 16\log(2/\delta)R_\ell\sum_{k=1}^r R_k/(b^2\epsilon^2)$, our algorithm is guaranteed to be $(\epsilon,\delta)$-CDP from Lemma \ref{l_gaussian_mechanism} and the post-processing property of DP in Lemma \ref{l_postprocessing}. 

    Denote $\tau_t = bt$ for $t \in \{0\}\cup [T-1]$, and consider the following two events
    \begin{align*}
        \mathcal{E}_1 = \Bigg\{&\Lambda_{\min}\Big\{\frac{1}{bm}\sum_{i=1}^b\sum_{j=1}^m \Phi_r(X^{(\tau_t+i)}_j)\Phi_r^{\top}(X^{(\tau_t+i)}_j)\Big\} \geq 1/(2L)\quad \text{and} \\
        &\Lambda_{\max}\Big\{\frac{1}{bm}\sum_{i=1}^b \sum_{j=1}^m \Phi_r(X^{(\tau_t+i)}_j)\Phi_r^{\top}(X^{(\tau_t+i)}_j)\Big\} \leq 2L, \forall t\in \{0,\ldots,T-1\}\Bigg\}
    \end{align*}
    and
    \begin{align*}
        \mathcal{E}_2 = \Big\{\Pi^{\mathrm{entry}}_{R}&\Big[\frac{1}{m}\sum_{j=1}^m \Phi_r(X^{(\tau_t+i)}_j)\big\{\Phi^{\top}_r(X^{(\tau_t+i)}_j)a^{t}-Y^{(\tau_t+i)}_j \big\}\Big]\\
        &= \frac{1}{m}\sum_{j=1}^m \Phi_r(X^{(\tau_t+i)}_j)\big\{\Phi^{\top}_r(X^{(\tau_t+i)}_j)a^{t}-Y^{(\tau_t+i)}_j \big\}, \; \forall i \in [b],\;  t \in \{0, \ldots, T-1\}\Big\}.
    \end{align*}
    We control the probability of these events happening in Lemma \ref{l_mean_upper_event}. The remainder of the proof is conditional on both of these events happening. 

    For the $t$th iteration, we can rewrite the noisy gradient descent as
    \begin{align*}
        &a^{t} - \rho \Big[\frac{1}{bm}\sum_{i=1}^b\sum_{j=1}^m \Phi_r(X^{(\tau_t+i)}_j)\{\Phi^{\top}_r(X^{(\tau_t+i)}_j)a^{t}-Y^{(\tau_t+i)}_j\} + w_t\Big]\\
        = & a^t - \rho \Big[\frac{1}{bm}\sum_{i=1}^b\sum_{j=1}^m \Phi_r(X^{(\tau_t+i)}_j)\big[\Phi^{\top}_r(X^{(\tau_t+i)}_j)(a^{t}- a^*_r)- \{\mu^*(X^{(\tau_t+i)}_j)- \Phi^{\top}_r(X^{(\tau_t+i)}_j)a^*_r\}\\
        &\hspace{5.8cm} -U^{(\tau_t+i)}(X^{(\tau_t+i)}_j)-\xi_{\tau_t+i,j}\big]+  w_t\Big],
    \end{align*}
    where the equality follows from \eqref{mean_model_obs}. Due to the projection $\Pi^{*}_{\mathcal{A}}$ and the fact that $a_r^* \in \mathcal{A}$, then by Lemma \ref{lemma_projection}, the above equation then implies
    \begin{align} \notag
        \|a^{t+1} - a^*_r\|_2^2 \lesssim &\; \Big\|\Big\{I-\frac{\rho}{bm}\sum_{i=1}^b\sum_{j=1}^m\Phi_r(X^{(\tau_t+i)}_j)\Phi^{\top}_r(X^{(\tau_t+i)}_j)\Big\}(a^t - a_r^*) \Big\|_2^2\\ \notag
        & + \Big\|\frac{\rho}{bm}\sum_{i=1}^b\sum_{j=1}^m \Phi_r(X^{(\tau_t+i)}_j)\Big\{\mu^*(X^{(\tau_t+i)}_j)- \Phi^{\top}_r(X^{(\tau_t+i)}_j)a^*_r\Big\}\Big\|_2^2\\ \notag
        & +\Big\|\frac{\rho}{bm}\sum_{i=1}^b\sum_{j=1}^m \Phi_r(X^{(\tau_t+i)}_j)U^{(\tau_t+i)}(X^{(\tau_t+i)}_j)\Big\|_2^2 \\ \notag
        & +\Big\|\frac{\rho}{bm}\sum_{i=1}^b\sum_{j=1}^m \Phi_r(X^{(\tau_t+i)}_j)\xi_{\tau_t+i,j}\Big\|_2^2 + \|\rho w_t\|_2^2\\ \label{thm_mean_upper_eq1}
        \lesssim & \|(I)\|_2^2+ \|(II)\|_2^2+ \|(III)\|_2^2+ \|(IV)\|^2_2 + \rho^2\|w_t\|^2_2.
    \end{align}
    We will control the estimation error $\|a^{t+1} - a^*_r\|_2^2$ by controlling each term on the right hand side individually.

    In the rest of the proof, we choose $\rho = 4L/(1+4L^2)$ for simplicity, which satisfy the condition $0 < \rho < 1/L$ given in Theorem \ref{thm_mean_upper}. With this choice of $\rho$, to control $(I)$, in the event $\mathcal{E}_1$, it then holds that 
    \begin{equation*}
        \Big\|I-\frac{\rho}{bm}\sum_{i=1}^b\sum_{j=1}^m\Phi_r(X^{(\tau_t+i)}_j)\Phi^{\top}_r(X^{(\tau_t+i)}_j)\Big\|_{\op}\leq \max \Big\{|1-\frac{\rho}{2L}|, |1- 2L\rho|\Big\} = 1 - \frac{2}{1+4L^2}.
    \end{equation*} Therefore, we have that 
    \begin{equation} \label{thm_mean_upper_eq7}
        \|(I)\|_2^2 \leq \Big(1 - \frac{2}{1+4L^2}\Big)^2\|a^t - a_r^*\|_2^2.
    \end{equation}

    For $(II)$, we have that 
    \begin{align*}
        (II) \leq &\; \frac{\rho}{bm}\sum_{i=1}^b\sum_{j=1}^m\Big[\Phi_r(X^{(\tau_t+i)}_j)\Big\{\mu^*(X^{(\tau_t+i)}_j)- \Phi^{\top}_r(X^{(\tau_t+i)}_j)a^*_r\Big\}\\
        &\hspace{2.3cm}- \mathbb{E}\big[\Phi_r(X^{(\tau_t+i)}_j)\big\{\mu^*(X^{(\tau_t+i)}_j)- \Phi^{\top}_r(X^{(\tau_t+i)}_j)a^*_r\big\}\big]\Big]\\
        & +\frac{\rho}{bm}\sum_{i=1}^b\sum_{j=1}^m \mathbb{E}\Big[\Phi_r(X^{(\tau_t+i)}_j)
        \Big\{\mu^*(X^{(\tau_t+i)}_j)- \Phi^{\top}_r(X^{(\tau_t+i)}_j)a^*_r\Big\}\Big]\\
        = &\; (II_1) + (II_2).
    \end{align*}
    To upper bound $\|(II_1)\|_2$, it holds that
    \begin{align} \notag
        \mathbb{E}(\|(II_1)\|^2_2) &= \frac{\rho^2}{b^2m^2}\sum_{\ell=1}^r \mathbb{E}\Bigg[\Big\{\sum_{i=1}^b\sum_{j=1}^m\phi_\ell(X^{(\tau_t+i)}_j)\Big\{\mu^*(X^{(\tau_t+i)}_j)- \Phi^{\top}_r(X^{(\tau_t+i)}_j)a^*_r\Big\} \\ \notag
        &\hspace{2.8cm}-\sum_{i=1}^b\sum_{j=1}^m\mathbb{E}\big[\phi_\ell(X^{(\tau_t+i)}_j)\big\{\mu^*(X^{(\tau_t+i)}_j)- \Phi^{\top}_r(X^{(\tau_t+i)}_j)a^*_r\big\}\big]\Big\}^2\Bigg]\\ \notag
        &= \frac{\rho^2}{bm}\sum_{\ell=1}^r\mathbb{E}\Bigg[\Big\{\phi_\ell(X^{(\tau_t+1)}_1)\Big\{\mu^*(X^{(\tau_t+1)}_1)- \Phi^{\top}_r(X^{(\tau_t+1)}_1)a^*_r\Big\} \\ \notag
        & \hspace{2cm} - \mathbb{E}\big[\phi_\ell(X^{(\tau_t+1)}_1)\big\{\mu^*(X^{(\tau_t+1)}_1)- \Phi^{\top}_r(X^{(\tau_t+1)}_1)a^*_r\big\}\big]\Big\}^2\Bigg]\\ \notag
        & \leq \frac{\rho^2}{bm}\sum_{\ell=1}^r\mathbb{E}\Bigg[\Big\{\phi_\ell(X^{(\tau_t+1)}_1)\Big\{\mu^*(X^{(\tau_t+1)}_1)- \Phi^{\top}_r(X^{(\tau_t+1)}_1)a^*_r\Big\}\Big\}^2\Bigg]\\ \notag
        & \leq \frac{2\rho^2 r }{bm}\mathbb{E}\Big[\Big\{\mu^*(X^{(\tau_t+1)}_1)- \Phi^{\top}_r(X^{(\tau_t+1)}_1)a^*_r\Big\}^2\Big]\\ \label{thm_mean_upper_eq5}
        & \lesssim \frac{r^{1-2\alpha}}{bm} \lesssim r^{-2\alpha}.
    \end{align}
    where the second equality holds since the collection of random variables $\{X^{(\tau_t+i)}_j\}_{i=1, j=1}^{b,m}$ are mutually independent, the second inequality holds from the properties of Fourier basis and the third inequality holds since under Assumptions \ref{a_sample} and \ref{a_model}(a)
    \[\mathbb{E}\Big[\Big\{\mu^*(X^{(\tau_t+1)}_1)- \Phi^{\top}_r(X^{(\tau_t+1)}_1)a^*_r\Big\}^2\Big]\leq L\|\mu^*- \Phi^{\top}_r a^*_r\|_{L^2}^2 \leq Lr^{-2\alpha}.\]
    To upper bound $\|(II)_2\|_2$, it holds that 
    \begin{align} \notag
        \|(II)_2\|_2^2 &= \Big\|\frac{\rho}{bm}\sum_{i=1}^b\sum_{j=1}^m \mathbb{E}\Big[\Phi_r(X^{(\tau_t+i)}_j)
        \Big\{\mu^*(X^{(\tau_t+i)}_j)- \Phi^{\top}_r(X^{(\tau_t+i)}_j)a^*_r\Big\}\Big]\Big\|_2^2 \\ \notag
        & = \rho^2 \Big\|\mathbb{E}\Big[\Phi_r(X^{(\tau_t+1)}_1)
        \Big\{\mu^*(X^{(\tau_t+1)}_1)- \Phi^{\top}_r(X^{(\tau_t+1)}_1)a^*_r\Big\}\Big]\Big\|_2^2\\ \label{thm_mean_upper_eq6}
        & \leq 2\rho^2L^2 \|\mu^*- \Phi^{\top}_r a^*_r\|_{L^2}^2 \leq \rho^2 L^2r^{-2\alpha}
    \end{align}
    where the second equality holds from Assumption \ref{a_sample} and the first inequality follows from Lemma \ref{l_mean_approximation}. Therefore, by Markov's inequality, we have that
    \begin{equation}\label{thm_mean_upper_eq2}
        \mathbb{P}\big\{\|(II)\|^2_2 \lesssim r^{-2\alpha}/\eta_1\big\} \geq 1- \eta_1,
    \end{equation}
    for any $\eta_1 \in (0,16L^2/(4L^2+1)^2)$ as long as $r \lesssim bm$.

To upper bound $\|(III)\|_2^2$, it holds that
    \begin{align*}
        \mathbb{E}(\|(III)\|_2^2) = \;&\frac{\rho^2}{b^2m^2}\sum_{\ell=1}^r \mathbb{E}\Big[\Big\{\sum_{i=1}^b\sum_{j=1}^m \phi_\ell(X^{(\tau_t+i)}_j)U^{(\tau_t+i)}(X^{(\tau_t+i)}_j)\Big\}^2\Big]\\
        =\;& \frac{\rho^2}{bm^2}\sum_{\ell=1}^r \mathbb{E}\Big[\Big\{ \sum_{j=1}^m\phi_\ell(X^{(\tau_t+1)}_j)U^{(\tau_t+1)}(X^{(\tau_t+1)}_j)\Big\}^2\Big]\\
        \lesssim \;&\frac{\rho^2}{bm}\sum_{\ell=1}^r\mathbb{E}\Big[\phi_\ell^2(X^{(\tau_t+1)}_1)\{U^{(\tau_t+1)}(X^{(\tau_t+1)}_1)\}^2\Big]\\
        &+ \frac{\rho^2}{b}\sum_{\ell=1}^r\mathbb{E}\Big[\phi_\ell(X^{(\tau_t+1)}_1)U^{(\tau_t+1)}(X^{(\tau_t+1)}_1)\phi_\ell(X^{(\tau_t+1)}_2)U^{(\tau_t+1)}(X^{(\tau_t+1)}_2)\Big]\\
        \lesssim \;& \frac{r}{bm}+ \frac{1}{b},
    \end{align*}
    where the last inequality follows from Assumption \ref{a_model}(b)~and the fact that
    \begin{align*}
        & \mathbb{E}\Big[\phi_\ell^2(X^{(\tau_t+1)}_1)\{U^{(\tau_t+1)}(X^{(\tau_t+1)}_1)\}^2\Big]
        = \mathbb{E}_U\Big[\int_{0}^1 \big\{U^{(\tau_t+1)}(s)\big\}^2\phi_\ell^2(s)f_X(s) \mathrm{d}s\Big] \\
        \leq & 2L\mathbb{E}_U\Big[\int_{0}^1 U^2(s) \mathrm{d}s\Big] \leq 2LC_U,
    \end{align*}
    and
    \begin{align*}
       &\frac{1}{b}\sum_{\ell=1}^r\mathbb{E}\Big[\phi_\ell(X^{(\tau_t+1)}_1)U^{(\tau_t+1)}(X^{(\tau_t+1)}_1)\phi_\ell(X^{(\tau_t+1)}_2)U^{(\tau_t+1)}(X^{(\tau_t+1)}_2)\Big]\\
       =\;&\frac{1}{b}\sum_{\ell=1}^r \mathbb{E}\Big[\mathbb{E}\Big\{\phi_\ell(X^{(\tau_t+1)}_1)U^{(\tau_t+1)}(X^{(\tau_t+1)}_1)\Big|U\Big\}^2\Big] \\
       =\;& \frac{1}{b}\mathbb{E}\Big[\Big\|\mathbb{E}\Big\{\Phi_\ell(X^{(\tau_t+1)}_1)U^{(\tau_t+1)}(X^{(\tau_t+1)}_1)\Big|U\Big\}\Big\|_2^2\Big]\\
       = \;& \frac{1}{b}\mathbb{E}\Big[\sum_{\ell=1}^r\Big\{\int_{0}^1 \phi_\ell(s)U^{(\tau_t+1)}(s)f_X(s) \; \mathrm{d}s\Big\}^2\Big]\\
       \leq \;& \frac{1}{b}\mathbb{E}\Big[\sum_{\ell=1}^\infty \Big\{\int_{0}^1 \phi_\ell(s)U^{(\tau_t+1)}(s)f_X(s) \; \mathrm{d}s\Big\}^2\Big]\\
       = \;& \frac{1}{b}\mathbb{E}\Big\{\|U^{(\tau_t+1)}f_X\|_{L^2}^2\Big\} \leq \frac{L^2}{b}\mathbb{E}\Big\{\|U^{(\tau_t+1)}\|_{L^2}^2\Big\} \lesssim \frac{1}{b}.
    \end{align*}
    Therefore, the Markov inequality implies that
    \begin{equation}\label{thm_mean_upper_eq3}
        \mathbb{P}\Big\{\|(III)\|^2_2 \lesssim (1/b+r/bm)/\eta_2 \Big\} \geq 1 - \eta_2,
    \end{equation}
    for any $\eta_2 \in (0,16L^3/(4L^2+1)^2)$.
    
    Finally, to find an upper bound on $\|(IV)\|_2^2$, we have that
    \begin{align*}
        \mathbb{E}(\|(IV)\|_2^2)&= \frac{\rho^2}{b^2m^2}\sum_{\ell=1}^r \mathbb{E}\Big[\Big\{\sum_{i=1}^b\sum_{j=1}^m \phi_\ell(X^{(\tau_t+i)}_j)\xi_{\tau_t+i,j}\Big\}^2\Big]  = \frac{\rho^2}{bm}\sum_{\ell=1}^r \mathbb{E}[\phi^2_\ell(X^{(\tau_t+1)}_1)\xi^2_{\tau_t+1,1}]\\
        & = \frac{\rho^2}{bm}\sum_{\ell=1}^r \mathbb{E}_X[\phi^2_\ell(X^{(\tau_t+1)}_1)]\mathbb{E}_\xi[\xi^2_{\tau_t+1,1}] \leq \frac{r \rho^2 L C_{\xi}^2}{bm} \asymp \frac{r}{bm},
    \end{align*}
    where the third equality follows from the fact that $X^{(\tau_t+1)}_1)$ and $\xi_{\tau_t+1,1}$ are independent and the fourth equality follows from Assumption \ref{a_model}(c). Therefore, the Markov inequality implies that 
    \begin{equation}\label{thm_mean_upper_eq4}
        \mathbb{P}\Big\{\|(IV)\|_2^2 \lesssim r/(bm\eta_3)\Big\}\geq 1-\eta_3,
    \end{equation}
    for any $\eta_3 \in (0, 16L^3/(4L^2+1)^2)$.

    For the term $\|w_t\|^2_2$, note that 
    \begin{align*}
        \|w_t\|^2_2 = \sum_{\ell=1}^r\sigma_\ell^2 z_\ell^2 \sim \chi^2_{\sum_{\ell=1}^r\sigma_\ell^2},
    \end{align*}
    where $\{z_\ell\}_{\ell=1}^r$ is a sequence of independent and identically distributed standard Gaussian random variables. Also, note that $\chi^2_{\sum_{\ell=1}^r\sigma_\ell^2}$ is sub-Exponential with parameter $\sum_{\ell=1}^r\sigma_\ell^2$. Also, we have that 
    \begin{align*}
        \mathbb{E}(\|w_t\|^2_2) = \sum_{\ell=1}^r\sigma_\ell^2\mathbb{E}\{ z_\ell^2\} = \sum_{\ell=1}^r\sigma_\ell^2.
    \end{align*}
    Therefore, following from standard properties of sub-Exponential random variables \citep[e.g.~Proposition 2.7.1 in][]{vershynin2018high}, we have that for any $\tau >0$, 
    \begin{align*}
        \mathbb{P}\Big\{\|w_t\|^2_2 \geq \tau - \mathbb{E}(\|w_t\|^2_2)\Big\} \leq \exp\Big\{-\frac{C_1\tau}{\sum_{\ell=1}^r\sigma_\ell^2}\Big\}.
    \end{align*}
    Combining with a union bound on $t \in [T]$, we have with probability at least $1-\eta_4$ that
    \begin{align} \label{thm_mean_upper_eq8}
        & \|w_t\|_{2}^2 \lesssim \log(T/\eta_4)\sum_{\ell=1}^r\sigma_\ell^2 \lesssim \log(T/\eta_4)\log(1/\delta)\frac{(\sum_{\ell=1}^r R_\ell)^2 }{b^2\epsilon^2}\nonumber \\
        \lesssim & \log(T/\eta_4)\log(1/\delta)\Big\{\frac{r^2\log^2(n/\eta)}{b^2m\epsilon^2}+\frac{1}{b^2\epsilon^2}\Big\},
    \end{align}
    where the second inequality follows by construction of $\sigma_\ell^2$ in Algorithm \ref{algorithm_mean} and the last inequality follows from the fact that when $\alpha >1$,
    \begin{align*}
        \sum_{\ell=1}^r R_\ell \lesssim \sum_{\ell=1}^r \sqrt{m^{-1}\log^2(n/\eta)}+\ell^{-\alpha} \lesssim  \sqrt{r^2m^{-1}\log^2(n/\eta)} + 1.
    \end{align*}
    Substituting \eqref{thm_mean_upper_eq7}, \eqref{thm_mean_upper_eq2}, \eqref{thm_mean_upper_eq3}, \eqref{thm_mean_upper_eq4} and \eqref{thm_mean_upper_eq8} into \eqref{thm_mean_upper_eq1}, we have by a union bound argument that with probability at least $1-6\eta- \eta_1- \eta_2 -\eta_3-\eta_4 = 1-10\eta$ that
    \begin{align*}
        \|a^{t+1} - a^*_r\|_2^2 &\lesssim \Big(1 - \frac{2}{1+4L^2}\Big)^2\|a^t - a_r^*\|_2^2+ \frac{1}{\eta}\Big(\frac{1}{b}+\frac{r}{bm}+ r^{-2\alpha}\Big) \\
        & \hspace{2cm}+\log(T/\eta)\log(1/\delta)\Big\{\frac{r^2\log^2(n/\eta)}{b^2m\epsilon^2}+\frac{1}{b^2\epsilon^2}\Big\}\\
        & \lesssim \Big(1 - \frac{2}{1+4L^2}\Big)^{2t+2}\|a_r^*\|_2^2 + \frac{1}{\eta}\Big(\frac{1}{b}+\frac{r}{bm}+ r^{-2\alpha}\Big) \\
        & \hspace{2cm} +\log(T/\eta)\log(1/\delta)\Big\{\frac{r^2\log^2(n/\eta)}{b^2m\epsilon^2}+\frac{1}{b^2\epsilon^2}\Big\}\\
        & \lesssim \exp\Big(-\frac{2(2t+2)}{1+4L^2}\Big)\|a_r^*\|_2^2 + \frac{1}{\eta}\Big(\frac{1}{b}+\frac{r}{bm}+ r^{-2\alpha}\Big) \\
        &\hspace{2cm}+\log(T/\eta)\log(1/\delta)\Big\{\frac{r^2\log^2(n/\eta)}{b^2m\epsilon^2}+\frac{1}{b^2\epsilon^2}\Big\},
    \end{align*}
    where the second inequality follows from an iterative argument. Pick $T \asymp \log(n)$, hence $b \asymp n/\log(n)$ and we have that
    \begin{align*}
        \|a^{T} - a^*_r\|_2^2 &\lesssim \frac{1}{\eta}\Big(\frac{\log(n)}{n}+\frac{r\log(n)}{nm}+ r^{-2\alpha}\Big) \\
        & \hspace{1cm}+\Big(\frac{1}{n^2\epsilon^2}+\frac{r^2\log^2(n/\eta)}{n^2m\epsilon^2}\Big)\log^2(n)\log(\log(n)/\eta)\log(1/\delta)\\
        & \lesssim_{\log} \frac{1}{\eta}\Big(\frac{1}{n}+\frac{r}{nm}+ r^{-2\alpha}\Big) +\frac{1}{n^2\epsilon^2}+\frac{r^2}{n^2m\epsilon^2}.
    \end{align*}
    Hence we have that 
    \begin{align*}
        \|\widetilde{\mu} - \mu^*\|_{L^2}^2 \lesssim \|a^{T} - a^*_r\|_2^2 + \|\Phi_r^\top a^*_r - \mu^*\|_{L^2}^2 \lesssim \|a^{T} - a^*_r\|_2^2 + r^{-2\alpha},
    \end{align*}
    where the last inequality follows from the fact that $\mu^* \in \mathcal{W}(\alpha, C_\alpha)$.

     To prove the third claim, we remark that \Cref{thm_mean_upper}.\ref{thm_mean_upper_rate-1} can be rewritten as
     \begin{align} \label{thm_mean_upper_eq9}
          \|\widetilde{\mu} - \mu^*\|_{L^2}^2 & \lesssim_{\log} O_p\Big(\frac{r}{nm} +\frac{r^2}{n^2m\epsilon^2} + \frac{1}{n} + \frac{1}{n^2\epsilon^2} + r^{-2\alpha}\Big).
     \end{align}
    The third claim follows from selecting $r$ and matching the four variance terms individually with the fifth term (squared bias) in \eqref{thm_mean_upper_eq9}. We provide details of rate-matching below.

    \noindent \textbf{Case 1:} The term $r(nm)^{-1}$ is the largest. In this case, we select $r_{\mathrm{I}} = (nm)^{1/(2\alpha+1)}$ and the final resulting rate is $r_{\mathrm{I}}^{-2\alpha} = (nm)^{-2\alpha/(2\alpha+1)}$. Next, we find conditions when $r_{\mathrm{I}}^{-2\alpha}$ is the dominant term.
    \begin{itemize}
        \item When $(nm)^{-\frac{2\alpha}{2\alpha+1}} \geq n^{-1}$, we require that $m \leq n^{\frac{1}{2\alpha}}$.

        \item  When $(nm)^{-\frac{2\alpha}{2\alpha+1}} \geq (n^2\epsilon^2)^{-1}$, we require that $\epsilon \geq n^{\frac{-1-\alpha}{2\alpha+1}}m^{\frac{\alpha}{2\alpha+1}}$,

        \item  When $(nm)^{-\frac{2\alpha}{2\alpha+1}} \geq r_{\mathrm{I}}^2(n^2m\epsilon^2)^{-1} = n^{-\frac{4\alpha}{2\alpha+1}}m^{\frac{\alpha}{2\alpha+1}}\epsilon^{-2}$, we require that $\epsilon \geq n^{-\frac{\alpha}{2\alpha+1}}m^{\frac{1}{4\alpha+2}}$.
    \end{itemize}

    \noindent \textbf{Case 2:} The term $r^2(n^2m\epsilon^2)^{-1}$ is the largest. In this case, we select $r_{\mathrm{II}} = (n^2m\epsilon^2)^{1/(2\alpha+2)}$ and the final resulting rate is $r_{\mathrm{II}} ^{-2\alpha} = (n^2m\epsilon^2)^{-\alpha/(\alpha+1)}$. Next, we find conditions when $r_{\mathrm{II}} ^{-2\alpha}$ is the dominant term.
    \begin{itemize}
        \item When $(n^2m\epsilon^2)^{-\frac{\alpha}{\alpha+1}} \geq n^{-1}$, we require that $\epsilon \leq n^{\frac{-\alpha+1}{2\alpha}}m^{-\frac{1}{2}}$.

        \item When $(n^2m\epsilon^2)^{-\frac{\alpha}{\alpha+1}} \geq (n^2\epsilon^2)^{-1}$, we require that $\epsilon \geq n^{-1}m^{\frac{\alpha}{2}}$.

        \item When $(n^2m\epsilon^2)^{-\frac{\alpha}{\alpha+1}} \geq r_{\mathrm{II}}(nm)^{-1} = n^{-\frac{\alpha}{\alpha+1}}m^{\frac{-2\alpha-1}{2\alpha+2}}\epsilon^{\frac{1}{\alpha+1}}$, we require that $\epsilon \leq n^{-\frac{\alpha}{2\alpha+1}}m^{\frac{1}{4\alpha+2}}$.
    \end{itemize}

    \noindent \textbf{Case 3:} The term $n^{-1}$ is the largest. In this case, we select $r_{\mathrm{III}} = n^{1/(2\alpha)}$. Next, we find conditions when $n^{-1}$ is the dominant term.
    \begin{itemize}
        \item When $n^{-1} \geq (n^2\epsilon^2)^{-1}$, we require that $\epsilon \geq n^{-\frac{1}{2}}$.

        \item  When $n^{-1} \geq r_{\mathrm{III}}(nm)^{-1} = n^{\frac{-\alpha+1}{\alpha}}m^{\frac{\alpha}{2\alpha+1}}$, we require that $m \geq n^{\frac{1}{2\alpha}}$.

        \item  When $n^{-1} \geq r_{\mathrm{III}}^2(n^2\epsilon m^2)^{-1} = n^{\frac{2-2\alpha}{\alpha}}m^{-1}\epsilon^{\frac{2-2\alpha}{\alpha}}$, we require that $\epsilon \geq n^{\frac{-\alpha+1}{2\alpha}}m^{-\frac{1}{2}}$.
    \end{itemize}

    \noindent \textbf{Case 4:} The term $(n^2\epsilon^2)^{-1}$ is the largest. In this case, we select $r_{\mathrm{IV}} = (n^2\epsilon^2)^{1/(2\alpha)}$. Next, we find conditions when $(n^2\epsilon^2)^{-1}$ is the dominant term.
    \begin{itemize}
        \item When $(n^2\epsilon^2)^{-1} \geq n^{-1}$, we require that $\epsilon \leq n^{-\frac{1}{2}}$.

        \item  When $(n^2\epsilon^2)^{-1} \geq r_{\mathrm{IV}}(nm)^{-1} = n^{\frac{-\alpha+1}{\alpha}}m^{-1}\epsilon^{\frac{1}{\alpha}}$, we require that $\epsilon \leq n^{\frac{-1-\alpha}{2\alpha+1}}m^{\frac{\alpha}{2\alpha+1}}$.

        \item When $(n^2\epsilon^2)^{-1} \geq r_{\mathrm{IV}}(n^2\epsilon m^2)^{-1}= n^{\frac{2-2\alpha}{\alpha}}m^{-1}\epsilon^{\frac{2-2\alpha}{\alpha}}$, we require that $\epsilon \leq n^{-1}m^{\frac{\alpha}{2}}$.
    \end{itemize}
\end{proof}

\subsection{Proof of Theorem \ref{thm_mean_lower}} \label{section_appendix_cdp_mean_low}
\begin{proof}[Proof of Theorem \ref{thm_mean_lower}]
    With the definition of $\mathcal{W}(\alpha, C_\alpha)$ in Definition \ref{def_sobolev}, instead of considering the minimax risk defined in \eqref{eq_minimax_original}, we will reduce the problem to a finite-dimensional subspace. For $r \in \mathbb{N}_+$, denote the $r$-dimensional subspace $\mathcal{W}_r(\alpha,C_\alpha)$ as 
    \begin{align*}
        \mathcal{W}_r(\alpha,C_\alpha) = \Big\{f \in L^2([0,1]):f = \sum_{\ell=1}^r \phi_\ell\langle f,\phi_\ell \rangle_{L^2} = \sum_{\ell=1}^r \phi_\ell a_\ell, \sum_{\ell=1}^r (\tau_\ell)^{2\alpha} a_\ell^2 < C_\alpha^2/\pi^{2\alpha}\Big\},
    \end{align*}
     where $\tau_\ell = \ell$ for even $\ell$ and $\tau_\ell = \ell-1$ for odd $\ell$. Also, denote the Sobolev ellipsoid $\Theta(r, C_\alpha)$ as
    \begin{align}\label{mean_t_lower_eq6}
        \Theta(r, C_\alpha) = \Big\{a \in \mathbb{R}^r: \sum_{\ell=1}^r (\tau_\ell)^{2\alpha} a_\ell^2 < C_\alpha^2/\pi^{2\alpha}\Big\}.
    \end{align}
    Since it holds that $\mathcal{W}_r(\alpha, C_\alpha) \subseteq \mathcal{W}(\alpha, C_\alpha)$ and $\|\widetilde{\mu}- \mu^*\|^2_{L^2}~\geq \sum_{\ell=1}^r (\widetilde{a}_\ell - a^*_\ell)^2$, where $\widetilde{a}_\ell = \langle \widetilde{\mu},\phi_\ell\rangle_{L^2}$ and $a^*_\ell =\langle \mu^* , \phi_\ell\rangle_{L^2}$ for any $\ell \in [r]$. The minimax risk in \eqref{eq_minimax_original} can further be lower bounded by
    \begin{align*}
        \eqref{eq_minimax_original} &\geq \underset{Q \in \mathcal{Q}_{\epsilon, \delta}}{\inf} \underset{\widetilde{\mu}}{\inf} \underset{\mu^* \in \mathcal{W}_r(\alpha,C_\alpha)}{\sup} \mathbb{E}_{P_X, P_Y, Q}\|\widetilde{\mu}- \mu^*\|_{L^2}^2 \\
        & \geq \underset{Q \in \mathcal{Q}_{\epsilon, \delta}}{\inf} \underset{\widetilde{a}}{\inf} \underset{a \in \Theta(r,C_\alpha)}{\sup} \mathbb{E}_{P_X, P_Y, Q}\|\widetilde{a} - a\|_{2}^2.
    \end{align*}
    Thus throughout the rest of the proof, it suffices to consider a minimax lower bound of
    \begin{align*}
        \underset{Q \in \mathcal{Q}_{\epsilon, \delta}}{\inf} \underset{\widetilde{\mu}}{\inf} \underset{a \in \Theta(r,C_\alpha)}{\sup} \mathbb{E}_{P_X, P_Y, Q}\|\widetilde{a} - a\|_{2}^2.
    \end{align*}

    Let $\mathcal{M}_{\epsilon,\delta}$ be the collection of all $(\epsilon, \delta)$-CDP algorithms. To show the above result, it suffices to show that for every estimator $M(\bm{X},\bm{Y}) \in \mathcal{M}_{\epsilon,\delta}$, it holds that 
    \begin{align} \label{mean_t_lower_eq1}
        \underset{M \in \mathcal{M}_{\epsilon,\delta}}{\inf} \underset{a \in \Theta(r,C_\alpha)}{\sup} \mathbb{E}\|M(\bm{X},\bm{Y}) - a\|_{2}^2 \gtrsim r^{-2\alpha} \wedge \frac{r^2}{n^2m\epsilon^2}.
    \end{align}
    Since by considering the random function $\mu^* + U$ as an unknown constant random function, the problem essentially becomes estimating the univariate mean for $n$ independent and identically distributed random variables, and from Theorem 3.1 in \citet{cai2021cost}, the minimax risk of such a problem is $n^{-1} \vee (n\epsilon)^{-2}$ whenever $\delta \lesssim n^{-(1+\omega)}$ for some fixed constant $\omega >0$. The first term in \eqref{mean_t_lower_eq1} follows from the non-private minimax lower bound of mean function estimation \citep[e.g.][]{cai2011optimal}.

    \noindent \textbf{Step 1: Construction of a class of distribution.} To prove \eqref{mean_t_lower_eq1}, we construct the class of distribution as follows. Assume that $\{U^{(i)}\}_{i=1}^n$ is a sequence of mean zero Gaussian processes with covariance function as the Mat\'ern covariance function $K$ with parameter $\alpha+1/2$, $\{X_j^{(i)}\}_{i=1,j=1}^{n,m}$ is a collection of independent uniform random variables distributed over $[0,1]$ and the measurement error $\xi_{ij} \sim N(0, \sigma_0^2)$ with $0 < \sigma_0^2 < \infty$. To construct the prior distribution of $a$, we assume that each coordinate of $a$ follows the uniform distribution between $-B$ and $B$, where $B^2 = \frac{C_1^2}{2\pi^{2\alpha}}\Big(\int_{1}^{r+1}t^{2\alpha}\mathrm{d}t\Big)^{-1}\asymp r^{-(2\alpha+1)}$. To show that the prior distribution is supported in the class we consider, note that
    \begin{align*}
        \sum_{\ell=1}^r \tau_\ell^2a^2_\ell \leq B^2 \sum_{\ell=1}^r\tau_\ell^2\asymp B^2 \sum_{\ell=1}^r \ell^{2\alpha} \leq \frac{C_1^2}{2\pi^{2\alpha}}.
    \end{align*}
    Hence, we have proved that the above construction is valid. 
    
    With the above construction, all of the assumptions in Assumption \ref{a_model} are satisfied. We further have that conditioning on $\{X_{j}^{(i)}\}_{i=1, j=1}^{n,m}$ and $a$, $Y_{i\cdot}=(Y_1^{(i)},\ldots, Y_m^{(i)})^\top$ follows a multivariate Gaussian distribution with mean vector
    \begin{align*}
        \mu_{i\cdot} = \{\mu(X_1^{(i)}), \ldots, \mu(X_m^{(i)})\}^\top = \{a^\top\Phi_r(X_1^{(i)}), 
        \ldots,a^\top\Phi_r(X_m^{(i)})\}^\top,
    \end{align*}
    and covariance matrix $\Sigma_i = C_i + \sigma^2_0I_m$, where $C_i \in \mathbb{R}^{m\times m}$ with $(C_i)_{ab} = K(X_a^{(i)},X_b^{(i)})$.

    \noindent \textbf{Step 2: Construction of Score attack.} For $i \in [n]$, construct the score attack by
    \begin{align*}
        \mathcal{A}(M(\bm{X},\bm{Y}), (X,Y)_{1:m}^{(i)}) = \langle M(\bm{X},\bm{Y})-a, \Phi_{[m],i}\Sigma_i^{-1}(Y_{i\cdot}-\mu_{i\cdot}) \rangle,
    \end{align*}
    where $\Phi_{[m],i} = (\Phi_r(X_1^{(i)}), \ldots, \Phi_r(X_m^{(i)})) \in \mathbb{R}^{r\times m}$.

    Let $(\bm{X}_i',\bm{Y}_i')$ be the dataset obtained by only replacing $(X,Y)^{(i)}_{1:m}$ in $(\bm{X},\bm{Y})$ with an independent copy. Let
    \begin{align} \label{mean_t_lower_eq12}
        A_i =  \langle M(\bm{X},\bm{Y})-a, \Phi_{[m],i}\Sigma_i^{-1}(Y_{i\cdot}-\mu_{i\cdot}) \rangle \quad \text{and} \quad A_i' =  \langle M(\bm{X}_i',\bm{Y}_i')-a, \Phi_{[m],i}\Sigma_i^{-1}(Y_{i\cdot}-\mu_{i\cdot}) \rangle.
    \end{align}
    
   Using a similar idea as the one used in the proof of Proposition 6.1 in \citet{cai2023score}, for $Z\in \mathbb{R}$, let $Z^+ = \max(Z,0)$ and $Z^- = -\min(Z,0)$, then it holds for any $T>0$  that 
    \begin{align} \notag
        \mathbb{E}A_i &= \int_{0}^\infty \mathbb{P}(A_i^+ > t)\; \mathrm{d}t - \int_{0}^\infty \mathbb{P}(A_i^- > t)\; \mathrm{d}t\\ \notag
        & \leq \int_{0}^T \mathbb{P}(A_i^+ > t)\; \mathrm{d}t - \int_{0}^T \mathbb{P}(A_i^- > t)\; \mathrm{d}t + \int_{T}^\infty \mathbb{P}(|A_i| > t)\; \mathrm{d}t\\ \notag
        & \leq \int_{0}^T \{\exp(\epsilon)\mathbb{P}((A'_i)^+ > t) + \delta\}\; \mathrm{d}t - \int_{0}^T \{\exp(-\epsilon)\mathbb{P}((A'_i)^- > t) - \delta\}\; \mathrm{d}t\\\notag
        &\hspace{9cm}+ \int_{T}^\infty \mathbb{P}(|A_i| > t)\; \mathrm{d}t\\ \notag
        & \leq (1+2\epsilon)\int_{0}^T \mathbb{P}((A'_i)^+ > t) \; \mathrm{d}t - (1-2\epsilon)\int_{0}^T \mathbb{P}((A'_i)^- > t) \; \mathrm{d}t + 2\delta T\\ \notag
        & \hspace{9cm}+ \int_{T}^\infty \mathbb{P}(|A_i| > t)\; \mathrm{d}t\\ \label{mean_t_lower_eq2}
        &\leq \mathbb{E}A_i' + 2\epsilon\mathbb{E}|A_i'|+2\delta T + \int_{T}^\infty \mathbb{P}(|A_i|>t)\mathrm{d}t.
    \end{align}
    \noindent \textbf{Step 3: Upper bound on} $\sum_{i\in [n]}\mathbb{E}A_i$. In this step, the following proof is constructed conditioning on $a$ and we will find an upper bound on $\mathbb{E}A_i$ by upper bounding each term in \eqref{mean_t_lower_eq2} individually. For the first term, note that by construction, we have that  
    \begin{align} \notag
        \mathbb{E}_{X,Y|a}[A_i'] &= \mathbb{E}_{X|a}\{\mathbb{E}_{Y|X,a}[A_i'|X]\}\\ \label{mean_t_lower_eq9}
        &= \mathbb{E}_{X|a}\Big\{\langle \mathbb{E}_{Y|X,a}[M(\bm{X}_i',\bm{Y}_i')-a], \mathbb{E}_{Y|X,a}[\Phi_{[m],i}\Sigma_i^{-1}(Y_{i\cdot}-\mu_{i\cdot})] \rangle\Big\} = 0,
    \end{align}
    where the second equality follows from the independence between $M(\bm{X}_i',\bm{Y}_i')$ and $Y_{i\cdot}-\mu_{i\cdot}$ and the last equality follows from the standard property of score function. Conditioning on the discrete grids $\{X_j^{(i)}\}_{i=1,j=1}^{n,m}$, denote the eigenvalues of $C_i$ as $\{\lambda_{i,j}\}_{j=1}^m$, hence the eigenvalue of $\Sigma_i$ are $\lambda_{i,1}+\sigma^2_0, \ldots, \lambda_{i,m}+\sigma^2_0$ respectively by construction. Since $\Sigma_i$ is positive definite, we could further derive that $\Lambda_{\max}(\Sigma_i^{-1}) \leq \sigma_0^{-2}$. Using the above result, it holds that
    \begin{align} \notag
        &(\mathbb{E}_{X,Y|a}|A_i|')^2 \leq \mathbb{E}_{X,Y|a}|A_i'|^2 \\ \notag
        =& \mathbb{E}_{X,Y|a}\Big\{(M(\bm{X}_i',\bm{Y}_i')-a)^\top\Phi_{[m],i}\Sigma_i^{-1}(Y_{i\cdot}-\mu_{i\cdot})(Y_{i\cdot}-\mu_{i\cdot})^\top\Sigma_i^{-1}\Phi_{[m],i}^{\top}(M(\bm{X}_i',\bm{Y}_i')-a)\Big\}\\ \notag
        =& \mathbb{E}_{\bm{X}_i',\bm{Y}_i'|a}\Big[(M(\bm{X}_i',\bm{Y}_i')-a)^\top \\ \notag
        &\hspace{2cm} \mathbb{E}_{X|a}\Big[\Phi_{[m],i}\Sigma_i^{-1} \mathbb{E}_{Y_{i,\cdot}|X,a}\big\{(Y_{i\cdot}-\mu_{i\cdot})(Y_{i\cdot}-\mu_{i\cdot})^\top|X\big\}\Sigma_i^{-1}\Phi^{\top}_{[m],i}\Big](M(\bm{X}_i',\bm{Y}_i')-a)\Big]\\ \notag
        \leq& \mathbb{E}_{\bm{X}_i',\bm{Y}_i'|a}\Big[\|M(\bm{X}_i',\bm{Y}_i')-a\|_2^2\Big]\\ \notag
        & \hspace{2cm}\Big\|\mathbb{E}_{X|a}\Big[\Phi_{[m],i}\Sigma_i^{-1} \mathbb{E}_{Y_{i,\cdot}|X,a}\big\{(Y_{i\cdot}-\mu_{i\cdot})(Y_{i\cdot}-\mu_{i\cdot})^\top|X\big\}\Sigma_i^{-1}\Phi^{\top}_{[m],i}\Big]\Big\|_{\op}\\ \notag
        =& \mathbb{E}_{\bm{X}_i',\bm{Y}_i'|a}\Big[\|M(\bm{X}_i',\bm{Y}_i')-a\|_2^2\Big]\Big\|\mathbb{E}_{X|a}\Big[\Phi_{[m],i}\Sigma_i^{-1}\Phi^{\top}_{[m],i}\Big]\Big\|_{\op}\\ \notag
        \leq& \mathbb{E}_{\bm{X}_i',\bm{Y}_i'|a}\Big[\|M(\bm{X}_i',\bm{Y}_i')-a\|_2^2\Big]\Big\|\mathbb{E}_{X|a}\Big[\|\Sigma_i^{-1}\|_{\op}\Phi_{[m],i}\Phi^{\top}_{[m],i}\Big]\Big\|_{\op}\\ \notag
        \leq& \sigma_0^{-2}\mathbb{E}_{\bm{X}_i',\bm{Y}_i'|a}\Big[\|M(\bm{X}_i',\bm{Y}_i')-a\|_2^2\Big]\Big\|\mathbb{E}_{X|a}\Big[\Phi_{[m],i}\Phi^{\top}_{[m],i}\Big]\Big\|_{\op}\\ \label{mean_t_lower_eq10}
        \leq& m\sigma_0^{-2}\mathbb{E}_{\bm{X}_i',\bm{Y}_i'|a}\Big[\|M(\bm{X}_i',\bm{Y}_i')-a\|_2^2\Big],
        \end{align}
    where the last inequality follows from the fact that
    \begin{align}
        \notag&\Big\|\mathbb{E}_{X|a}\Big[\Phi_{[m],i}\Phi^{\top}_{[m],i}\Big]\Big\|_{\op} = \sup_{v\in \mathbb{R}^r: \|v\|_2 = 1}v^{\top}\mathbb{E}_{X|a}\Big[\Phi_{[m],i}\Phi^{\top}_{[m],i}\Big]v\\ \notag
        =\;& \sup_{v\in \mathbb{R}^r: \|v\|_2 = 1} \sum_{j=1}^m v^{\top}\mathbb{E}_{X|a}\Big[\Phi_{r}(X_j^{(i)})\Phi^{\top}_{r}(X_j^{(i)})\Big]v = \sup_{v\in \mathbb{R}^r: \|v\|_2 = 1} \sum_{j=1}^m \int v^\top \Phi_{r}(s)\Phi^{\top}_{r}(s)v f_X(s)\;\mathrm{d}s\\ \label{mean_t_lower_eq7}
        \leq \;& L\sup_{v\in \mathbb{R}^r: \|v\|_2 = 1}\sum_{j=1}^m \int \Big\{\sum_{\ell=1}^r v_\ell \phi_\ell(s)\Big\}^2 \;\mathrm{d}s \leq Lm \sup_{v\in \mathbb{R}^r: \|v\|_2 = 1} \|v\|_2^2 = Lm.
    \end{align}

    Next to consider the term $\mathbb{P}(|A_i| > t)$, firstly note that
    \begin{align*}
        |A_i| \leq \|M(\bm{X},\bm{Y})-a\|_2\|\Phi_{[m],i}\Sigma_i^{-1}(Y_{i\cdot}-\mu_{i\cdot})\|_2 \leq C_2\|\Phi_{[m],i}\Sigma_i^{-1}(Y_{i\cdot}-\mu_{i\cdot})\|_2,
    \end{align*}
    where the first inequality follows from the Cauchy--Schwarz inequality and the second inequality follows from the fact that $\|a\|_2 \lesssim \sqrt{C_2}$ and assume that $\|M(\bm{X},\bm{Y})\|_2 \lesssim \sqrt{C_2}$ without loss of generality. For the second term $\|\Phi_{[m],i}\Sigma_i^{-1}(Y_{i\cdot}-\mu_{i\cdot})\|_2$, it holds from definition that
    \begin{align*}
        \|\Phi_{[m],i}\Sigma_i^{-1}(Y_{i\cdot}-\mu_{i\cdot})\|_2 \leq \sqrt{r}\underset{k =1, \ldots, r}{\max} \Big|\sum_{h=1}^m \sum_{l=1}^m (\Phi_{[m],i})_{kl}(\Sigma_i^{-1})_{lh}(Y_{i\cdot}-\mu_{i\cdot})_{h}\Big|.
    \end{align*}
    By construction in \textbf{Step 1}, conditioning on the observation grids $X$, we have that for any $h \in [m]$, $(Y_{i\cdot}-\mu_{i\cdot})_{h}$ is sub-Gaussian with parameter $C_U+C_\xi \asymp 1$ under Assumptions \ref{a_model}(b)~and \ref{a_model}(c). Therefore, following from the boundedness property of Fourier basis and the fact that $|(\Sigma_i^{-1})_{lh}| \leq \|\Sigma_i^{-1}\|_{\op} \lesssim 1$ for every $l,h \in [m]$, we have that
    \begin{align} \label{mean_t_lower_eq8}
        \Big\|\sum_{h=1}^m \sum_{l=1}^m (\Phi_{[m],i})_{kl}(\Sigma_i^{-1})_{lh}(Y_{i\cdot}-\mu_{i\cdot})_{h}\Big\|_{\psi_2} \leq  \sum_{h=1}^m \sum_{l=1}^m \Big\|(\Phi_{[m],i})_{kl}(\Sigma_i^{-1})_{lh}(Y_{i\cdot}-\mu_{i\cdot})_{h}\Big\|_{\psi_2} \lesssim m^2.
    \end{align}
    Therefore, it holds from the sub-Gaussian properties \citep[e.g.~Proposition 2.5.2][]{vershynin2018high} that 
    \begin{align*}
        \mathbb{P}\Big\{\Big|\sum_{h=1}^m \sum_{l=1}^m (\Phi_{[m],i})_{kl}(\Sigma_i^{-1})_{lh}(Y_{i\cdot}-\mu_{i\cdot})_{h}\Big| \geq t \Big| X\Big\} \leq \exp(-C_3t^2/m^4).
    \end{align*}
    Taking another expectation over $X$ gives 
    \begin{align*}
        &\mathbb{P}\Big\{\Big|\sum_{h=1}^m \sum_{l=1}^m (\Phi_{[m],i})_{kl}(\Sigma_i^{-1})_{lh}(Y_{i\cdot}-\mu_{i\cdot})_{h}\Big| \geq t \Big\} \\
        =\;& \mathbb{E}_X\Big[\mathbb{P}\Big\{\Big|\sum_{h=1}^m \sum_{l=1}^m (\Phi_{[m],i})_{kl}(\Sigma_i^{-1})_{lh}(Y_{i\cdot}-\mu_{i\cdot})_{h}\Big| \geq t \Big| X\Big\}\Big]\\
        \leq\;& \exp(-C_3t^2/m^4).
    \end{align*}
    Applying a union bound argument, we have that
    \begin{align*}
        \mathbb{P}\Big\{\|\Phi_{[m],i}\Sigma_i^{-1}(Y_{i\cdot}-\mu_{i\cdot})\|_2 > t\Big\} &\leq \mathbb{P}\Big\{\underset{k =1, \ldots, r}{\max} \Big|\sum_{h=1}^m \sum_{l=1}^m (\Phi_{[m],i})_{kl}(\Sigma_i^{-1})_{lh}(Y_{i\cdot}-\mu_{i\cdot})_{h}\Big| > \frac{t}{\sqrt{r}}\Big\}\\
        &\leq \exp\big\{-C_3t^2/(m^4r) + \log(r)\big\}.
    \end{align*}
    Pick $T \asymp m^2\sqrt{r\log(1/\delta)}$, we have that
    \begin{align} \notag
        \int_{T}^\infty \mathbb{P}(|A_i|>t)\mathrm{d}t &\leq \int_{T}^\infty \mathbb{P}\Big\{C_2\|\Phi_{[m],i}\Sigma_i^{-1}(Y_{i\cdot}-\mu_{i\cdot})\|_2 > t\Big\}\;\mathrm{d}t \\ \notag
        &\lesssim \int_{T}^\infty \exp\big\{-C_4t^2/(m^4r) + \log(r)\big\} \; \mathrm{d}t\\ \label{mean_t_lower_eq11}
        & \asymp r^{3/2}m^2\int_{\frac{T}{m^2\sqrt{r}}}^\infty \exp(-u^2) \; \mathrm{d}u \lesssim r^{3/2}m^2\exp\Big(-\frac{T^2}{m^4r}\Big) \lesssim r^{3/2}m^2\delta,
    \end{align}
    where the equality follows from a change of variable from $t$ to $u = t/(m^2\sqrt{r})$.
    Plugging results in \eqref{mean_t_lower_eq9}, \eqref{mean_t_lower_eq10} and \eqref{mean_t_lower_eq11} into \eqref{mean_t_lower_eq2}
    \begin{align*}
        \mathbb{E}_{\bm{X},\bm{Y}|a} A_i &\lesssim \epsilon\sigma_0^{-1}\sqrt{m\mathbb{E}_{\bm{X},\bm{Y}|a}\|M(\bm{X},\bm{Y})-a\|_2^2}+ \delta r^{3/2}m^2 + \delta m^2\sqrt{r\log(1/\delta)}\\
        & \lesssim \epsilon\sigma_0^{-1}\sqrt{m\mathbb{E}_{\bm{X},\bm{Y}|a}\|M(\bm{X},\bm{Y})-a\|_2^2}+ \delta m^2r^{3/2}\sqrt{\log(1/\delta)}.
    \end{align*}
    Summing over $i \in [n]$ and taking another expectation with respect to $a$ guarantees that 
    \begin{align} \label{mean_t_lower_eq3}
        \sum_{i\in [n]} \mathbb{E}_a \mathbb{E}_{\bm{X},\bm{Y}|a} A_i \lesssim n\epsilon\sqrt{m\mathbb{E}_a \mathbb{E}_{\bm{X},\bm{Y}|a}\|M(\bm{X},\bm{Y})-a\|_2^2} + \delta nm^2r^{3/2}\sqrt{\log(1/\delta)}.
    \end{align}

    \noindent \textbf{Step 4: Lower bound on} $\sum_{i\in [n]}\mathbb{E}A_i$.
    To find a lower bound on $\sum_{i\in [n]} \mathbb{E}_a \mathbb{E}_{\bm{X},\bm{Y}|a} A_i$, write $f_{\bm{Y},\bm{X}|a}$ as the joint density of $\bm{Y}$ and $\bm{X}$ conditioning on $a$, we have that
    \begin{align*}
        f_{\bm{X},\bm{Y}|a}(\bm{x},\bm{y}) &=f_{\bm{Y}|\bm{X},a}(\bm{y})f_{\bm{X}|a}(\bm{x}) \\
        &= (2\pi)^{-mn/2} \Big\{\prod_{i=1}^n \text{det}(\Sigma_i)^{-1/2}\Big\}\exp\Big\{-\frac{1}{2}\sum_{i\in [n]} (\bm{y}-\mu_{i\cdot})^{\top}\Sigma_i^{-1}(\bm{y}-\mu_{i\cdot})\Big\},
    \end{align*}
    and
    \begin{align*}
        \frac{\partial f_{\bm{X},\bm{Y}|a}(\bm{x},\bm{y})}{\partial a} = f_{\bm{X},\bm{Y}|a}(\bm{x},\bm{y})\sum_{i\in [n]}\Phi_{[m],i}\Sigma_i^{-1}(\bm{y}-\mu_{i\cdot}).
    \end{align*}
    
    Observe that by construction of $A_i$ in \eqref{mean_t_lower_eq12}, we have 
    \begin{align*}
        \sum_{i\in [n]} \mathbb{E}_{\bm{X},\bm{Y}|a} A_i &= \sum_{\ell=1}^r\mathbb{E}_{\bm{X},\bm{Y}|a}\Big\{(M(\bm{X},\bm{Y}))_\ell\sum_{i\in [n]}\{\Phi_{[m],i}\Sigma_i^{-1}(Y_{i\cdot}-\mu_{i\cdot})\}_\ell\Big\}\\
        & =\sum_{\ell=1}^r\mathbb{E}_{\bm{X},\bm{Y}|a}\Big\{(M(\bm{X},\bm{Y}))_\ell\frac{1}{f_{\bm{X},\bm{Y}|a}(\bm{x},\bm{y})}\frac{\partial f_{\bm{X},\bm{Y}|a}}{\partial a_\ell}\Big\}\\
        & = \sum_{\ell=1}^r \frac{\partial}{\partial a_\ell}\mathbb{E}_{\bm{X},\bm{Y}|a}\Big\{(M(\bm{X},\bm{Y}))_\ell\Big\}.
    \end{align*}
    
    Suppose that $M$ being an estimator of $a$ satisfies that
    \begin{align} \label{mean_t_lower_eq4}
        \underset{a \in \Theta(r,C_\alpha)}{\sup} \mathbb{E}\|M(\bm{X},\bm{Y})-a\|_2^2 \leq \frac{rB^2}{20}.
    \end{align}
    with $B = C_1^2(\int_{1}^{r+1}t^{2\alpha}\mathrm{d}t)^{-1}/(2\pi^{2\alpha})$ defined in \textbf{Step 1}.
    Then by the proof of Proposition 6.2 in \citet{cai2023score}, it holds that 
    \begin{align} \label{mean_t_lower_eq5}
        \sum_{i\in [n]} \mathbb{E}_a \mathbb{E}_{\bm{X},\bm{Y}|a} A_i \geq  \frac{r}{10}.
    \end{align}
    
    Combining results in \eqref{mean_t_lower_eq3} and \eqref{mean_t_lower_eq5}, it holds that 
    \begin{align*}
        r\lesssim \sum_{i\in [n]} \mathbb{E}_a \mathbb{E}_{\bm{X},\bm{Y}|a} A_i \lesssim n\epsilon\sqrt{m\mathbb{E}_a \mathbb{E}_{\bm{X},\bm{Y}|a}\|M(\bm{X},\bm{Y})-a\|_2^2} + \delta nm^2r^{3/2}\sqrt{\log(1/\delta)},
    \end{align*}
    which implies
    \begin{align*}
        \mathbb{E}_a \mathbb{E}_{\bm{X},\bm{Y}|a}\|M(\bm{X},\bm{Y})-a\|_2^2 \geq \frac{r^2}{n^2m\epsilon^2},
    \end{align*}
    if $\delta \sqrt{\log(1/\delta)}\lesssim r^{-1/2}n^{-1}m^{-2}$. For those $M \in \mathcal{M}_{\epsilon,\delta}$ which fail to satisfy \eqref{mean_t_lower_eq4}, we automatically have a lower bound of $rB^2 \asymp r^{-2\alpha}$. As the Bayes risk always lowers the bounds of the supremum risk, the proof is concluded. The fourth term in the result can then be obtained by setting $r \asymp (n^2m\epsilon^2)^{1/(2\alpha+2)}$.
\end{proof}

\subsection{Auxiliary results}
\begin{lemma}\label{l_mean_upper_event}
    Consider the events of interest in the proof of Theorem \ref{thm_mean_upper}:
    \begin{align*}
        \mathcal{E}_1 = \Bigg\{&\Lambda_{\min}\Big\{\frac{1}{bm}\sum_{i=1}^b\sum_{j=1}^m \Phi_r(X^{(\tau_t+i)}_j)\Phi_r^{\top}(X^{(\tau_t+i)}_j)\Big\} \geq 1/(2L)\quad \text{and} \\
        &\Lambda_{\max}\Big\{\frac{1}{bm}\sum_{i=1}^b \sum_{j=1}^m \Phi_r(X^{(\tau_t+i)}_j)\Phi_r^{\top}(X^{(\tau_t+i)}_j)\Big\} \leq 2L, \forall t\in \{0,\ldots,T-1\}\Bigg\}
    \end{align*}
    and
    \begin{align*}
        \mathcal{E}_2 = \Big\{&\Pi^{\mathrm{entry}}_{R}\Big[\frac{1}{m}\sum_{j=1}^m \Phi_r(X^{(\tau_t+i)}_j)\big\{\Phi^{\top}_r(X^{(\tau_t+i)}_j)a^{t}-Y^{(\tau_t+i)}_j \big\}\Big]\\
        &= \frac{1}{m}\sum_{j=1}^m \Phi_r(X^{(\tau_t+i)}_j)\big\{\Phi^{\top}_r(X^{(\tau_t+i)}_j)a^{t}-Y^{(\tau_t+i)}_j \big\}, \; \forall i \in [b],\;  t \in \{0, \ldots, T-1\}\Big\}.
    \end{align*}
    Suppose that $nm \geq C_1 L^2\{r+ \log(T/\eta)\}\log(n)$, we have 
    \begin{align*}
        \mathbb{P}(\mathcal{E}_1 \cap \mathcal{E}_2) \geq 1- 6\eta,
    \end{align*}
    for a small enough $\eta \in (0,1/6)$.
\end{lemma}
\begin{proof}
    In the proof, we control the probability of each event individually.

     To control event $\mathcal{E}_1$, we firstly bound the largest and the smallest eigenvalues of 
    \[\Sigma = \frac{1}{bm}\sum_{i=1}^b\sum_{j=1}^m \mathbb{E}\big[\Phi_r(X^{(\tau_t+i)}_j)\Phi_r^{\top}(X^{(\tau_t+i)}_j)\big] = \mathbb{E}\big[\Phi_r(X^{(\tau_t+1)}_1)\Phi_r^{\top}(X^{(\tau_t+1)}_1)\big].\]
    
    Since $\Sigma$ is symmetric and positive semi-definite, we have that
    \begin{align}\notag
        \Lambda_{\max}(\Sigma) &= \underset{v \in \mathbb{R}^r: \|v\|_2 =1}{\sup} v^{\top}\Sigma v= \underset{v \in \mathbb{R}^r: \|v\|_2 =1}{\sup} v^{\top}\int_0^1 \Phi_r(s)\Phi_r^{\top}(s)f_X(s) \mathrm{d}s\;v\\ \label{l_mean_upper_event_eq2}
        & \leq L \underset{v \in \mathbb{R}^r: \|v\|_2 =1}{\sup} \int_0^1 v^{\top}\Phi_r(s)\Phi_r^{\top}(s)v\mathrm{d}s = L,
    \end{align}
    where the last equality is due to the orthonormality of the Fourier basis. Similarly, we could also lower bound the smaller eigenvalue by 
    \begin{equation} \label{l_mean_upper_event_eq3}
        \Lambda_{\min}(\Sigma) = \underset{v \in \mathbb{R}^r: \|v\|_2 =1}{\inf} v^{\top}\Sigma v \geq \frac{1}{L} \underset{v \in \mathbb{R}^r: \|v\|_2 =1}{\inf} \int_0^1 v^{\top}\Phi_r(s)\Phi_r^{\top}(s)v\mathrm{d}s = \frac{1}{L}.
    \end{equation}
    Denote $A = (\Phi_r(X_{\tau_t+1}^{(1)}),\Phi_r(X_{\tau_t+1}^{(2)}), \ldots, \Phi_r(X_{\tau_t+b}^{(m)}))^{\top} \in \mathbb{R}^{bm\times r}$ the matrix formed by joining the row vectors $\{\Phi_r^{\top}(X_j^{(\tau_t+i)})\}_{i=1, j=1}^{b,m}$. Under Assumption \ref{a_sample}, $A$ is a matrix whose rows are independent sub-Gaussian random vectors in $\mathbb{R}^r$ with second moment matrix $\Sigma$. The matrix Bernstein’s inequality \citep[e.g.~Remark 5.40 in][]{vershynin2010introduction} implies that
    \begin{align*}
        \mathbb{P}\Big\{\big\|(bm)^{-1}A^{\top}A - \Sigma\big\|_{\op} \leq C_1\sqrt{(bm)^{-1}\{r+\log(1/\eta_1)\}}\Big\} \geq 1-\eta_1.
    \end{align*}
    Thus a union bound argument on $t$ and Weyl’s inequality further imply that $\mathbb{P}(\mathcal{E}_1) \geq 1- \eta_1$ as long as 
    \begin{align*}
        bm \geq 4C_1 L^2\{r+ \log(T/\eta_1)\}.
    \end{align*}

    To control the event of $\mathcal{E}_2$, firstly noting that under Model \eqref{mean_model_obs}, we could rewrite the summation of interest as 
    \begin{align*}
        &\frac{1}{m}\Big|\sum_{j=1}^m \phi_\ell(X^{(\tau_t+i)}_j)\big\{\Phi^{\top}_r(X^{(\tau_t+i)}_j)a^{t}-Y^{(\tau_t+i)}_j \big\}\Big|\\
        \lesssim \; & \Big|\frac{1}{m}\sum_{j=1}^m \phi_\ell(X^{(\tau_t+i)}_j)\Phi^{\top}_r(X^{(\tau_t+i)}_j)a^t\Big|+\Big|\frac{1}{m}\sum_{j=1}^m \phi_\ell(X^{(\tau_t+i)}_j)\mu^*(X^{(\tau_t+i)}_j)\Big|\\
        &+ \Big|\frac{1}{m}\sum_{j=1}^m \phi_\ell(X^{(\tau_t+i)}_j)U^{(\tau_t+i)}(X^{(\tau_t+i)}_j)\Big|+ \Big|\frac{1}{m}\sum_{j=1}^m \phi_\ell(X^{(\tau_t+i)}_j)\xi_{\tau_t+i,j}\Big|\\
        = \;& (I) + (II) + (III) +(IV),
    \end{align*}
    and we will consider each term individually. For $(I)$, using the properties of Fourier basis, we have that for any fixed $a^t$ after truncating with $\Pi^{*}_{\mathcal{A}}$, 
    \begin{align*}
        |\phi_\ell(X^{(\tau_t+i)}_j)\Phi^{\top}_r(X^{(\tau_t+i)}_j)a^t| \leq |\phi_\ell(X^{(\tau_t+i)}_j)||\Phi^{\top}_r(X^{(\tau_t+i)}_j)a^t| < C_2,
    \end{align*}
    where the last inequality follows from the fact that $\Phi_r^\top a^t \in \mathcal{W}(\alpha, C_\alpha)$ and 
    \begin{align*}
        \Phi^{\top}_r(X^{(\tau_t+i)}_j)a^t = \sum_{\ell=1}^r a^t_\ell\phi_{\ell}(X^{(\tau_t+i)}_j) \leq \sqrt{2} \sum_{\ell=1}^r |a^t_\ell| \lesssim \sum_{\ell=1}^r \ell^{-\alpha} \lesssim 1,
    \end{align*}
    for any $\alpha > 1$. Moreover, $\{\phi_\ell(X^{(\tau_t+i)}_j)\Phi^{\top}_r(X^{(\tau_t+i)}_j)a^t\}_{j=1}^m$ are mutually independent under Assumption~\ref{a_sample}. Therefore, Hoeffding’s inequality for general bounded random variables \citep[e.g.~Theorem 2.2.6 in][]{vershynin2018high} implies that for any $\tau>0$,
    \begin{align*}
        \mathbb{P}\Big[\Big|\sum_{j=1}^m \phi_\ell(X^{(\tau_t+i)}_j)\Phi^{\top}_r(X^{(\tau_t+i)}_j)a^t -m \mathbb{E}\{\phi_\ell(X^{(\tau_t+i)}_1)\Phi^{\top}_r(X^{(\tau_t+i)}_1)a^t\}\Big| \geq \tau \Big] \leq \exp\Big(\frac{-C_3\tau^2}{m}\Big).
    \end{align*}
    To upper bound the expectation term, note the fact that for any function $h(s)\in \mathcal{W}(\alpha, C_\alpha)$, we have that $\langle h,\phi_{\ell}\rangle_{L^2} \lesssim \ell^{-\alpha}$ as shown in Definition \ref{def_sobolev}. Hence, it holds that
    \begin{align}\label{l_mean_upper_event_eq1}
        \Big|\mathbb{E}\Big\{\phi_\ell(X^{(\tau_t+i)}_1)\Phi^{\top}_r(X^{(\tau_t+i)}_1)a^t\Big\} \Big|= \Big|\int_{0}^1 \phi_\ell(s)\Phi^{\top}_r(s)a^t f_X(s)\;\mathrm{d}s\Big| =|\langle\phi_\ell, \Phi^{\top}_r a^t f_X\rangle_{L^2}| \lesssim \ell^{-\alpha},
    \end{align}
    where the last inequality follows from the fact that $\sum_{h=1}^{r} a^t_h\phi_h(s) \in \mathcal{W}(\alpha, C_\alpha)$ by the projection $\Pi^*_{\mathcal{A}}$ and $f_X(s) \in \mathcal{W}(\alpha, C_\alpha)$ in Assumption \ref{a_sample}, hence Theorem \ref{thm_product_sobolev} implies that
    $$\Phi^{\top}_r(s)a^t f_X(s) = f_X(s)\sum_{h=1}^{r} a^t_h\phi_h(s) \in \mathcal{W}(\alpha, C_\alpha).$$ 
    Therefore, by applying a union bound argument, we have that with probability $1-\eta_2$ that 
    \begin{align*}
        &\max_{t,i}\Big|\frac{1}{m}\sum_{j=1}^m \phi_\ell(X^{(\tau_t+i)}_j)\Phi^{\top}_r(X^{(\tau_t+i)}_j)a^t - \mathbb{E}\{\phi_\ell(X^{(\tau_t+i)}_1)\Phi^{\top}_r(X^{(\tau_t+i)}_1)a^t\}\Big| \\
        \leq\;& \sqrt{\log(n/\eta_2)/m} + \ell^{-\alpha}.
    \end{align*}

    A similar justification could also be applied to $(II)$ and we have that with a probability at least $1-\eta_3$ that 
    \begin{align*}
        \max_{t,i}\Big|\frac{1}{m}\sum_{j=1}^m \phi_\ell(X^{(\tau_t+i)}_j)\mu^*(X^{(\tau_t+i)}_j) \Big| \lesssim \sqrt{\log(n/\eta_3)/m} + \ell^{-\alpha}.
    \end{align*}
    
    To consider $(III)$, we construct the event $\mathcal{E}_3 = \{\sup_{i\in [b]}|U^{(i)}(s)| \leq C_4\sqrt{\log(b/\eta_4)}, \;\forall s\in [0,1]\}$. Note that under Assumption \ref{a_model}(b), using standard properties of sub-Gaussian random variables \citep[e.g.~Proposition 2.5.2 in][]{vershynin2018high} and a union bound argument, we could show that $\mathbb{P}(\mathcal{E}_3) > 1-\eta_4$. The rest of the proof is constructed conditioning on $\mathcal{E}_3$. To remove the dependence resulting from the random process of $U$ in $(III)$, we construct the following proof conditioning on $U$. Note that using standard boundedness properties of Fourier basis, we have that 
    \begin{align*}
        |\phi_\ell(X^{(\tau_t+i)}_j)U^{(\tau_t+i)}(X^{(\tau_t+i)}_j)|\leq \sqrt{2}|U^{(\tau_t+i)}(X^{(\tau_t+i)}_j)| \leq C_5\sqrt{\log(b/\eta_4)}.
    \end{align*}
    Hence, Hoeffding’s inequality for general bounded random variables \citep[e.g.~Theorem 2.2.6 in][]{vershynin2018high} implies that for any $\tau>0$, 
    \begin{align*}
        &\mathbb{P}\Big[\Big|\sum_{j=1}^m \phi_\ell(X^{(\tau_t+i)}_j)U^{(\tau_t+i)}(X^{(\tau_t+i)}_j) - m\mathbb{E}\{\phi_\ell(X^{(\tau_t+i)}_1)U^{(\tau_t+i)}(X^{(\tau_t+i)}_1)|U\}\Big|\geq\tau\Big|U\Big] \\
        & \hspace{8cm} \leq \exp\Big(\frac{-C_6\tau^2}{m\log(b/\eta_4)}\Big).
    \end{align*}
    Taking another expectation with respect to $U$, we have that 
    \begin{align*}
        &\mathbb{P}\Big[\Big|\sum_{j=1}^m \phi_\ell(X^{(\tau_t+i)}_j)U^{(\tau_t+i)}(X^{(\tau_t+i)}_j) - m\mathbb{E}\{\phi_\ell(X^{(\tau_t+i)}_1)U^{(\tau_t+i)}(X^{(\tau_t+i)}_1)|U\}\Big|\geq\tau\Big]\\
        & \hspace{8cm} \leq \exp\Big(\frac{-C_6\tau^2}{m\log(b/\eta_4)}\Big).
    \end{align*}
    To control the expectation, we have that 
    \begin{align*}
        \Big|\mathbb{E}\Big\{\phi_\ell(X^{(\tau_t+i)}_1)U^{(\tau_t+i)}(X^{(\tau_t+i)}_1)\Big|U\Big\} \Big|= \Big|\int_{0}^1 \phi_\ell(s)U^{(\tau_t+i)}(s) f_X(s)\;\mathrm{d}s\Big| = |\langle\phi_\ell, U^{(\tau_t+i)}f_X\rangle_{L^2}| \lesssim \ell^{-\alpha},
    \end{align*}
    where the last inequality follows from a similar argument as the one leads to \eqref{l_mean_upper_event_eq1} and Assumption \ref{a_model}(b) Taking another union bound argument, we have that with probability at least $1-\eta_4- \eta_5$ that
    \begin{align*}
        \max_{t,i}\Big|\frac{1}{m}\sum_{j=1}^m \phi_\ell(X^{(\tau_t+i)}_j)U^{(\tau_t+i)}(X^{(\tau_t+i)}_j)\Big| \lesssim \sqrt{\log(b/\eta_4)\log(n/\eta_5)/m}+\ell^{-\alpha}.
    \end{align*}

    Finally, for $(IV)$, note that under Assumption \ref{a_model}(c), we have that the terms of interest inside the summation are independent mean zero sub-Gaussian random variables with parameter $\sqrt{2}C_{\xi}$. Hence, Hoeffding inequality \citep[e.g.~Theorem 2.6.2 in][]{vershynin2018high} implies that
    \begin{align*}
        \mathbb{P}\Big\{\Big|\sum_{j=1}^m \phi_\ell(X^{(\tau_t+i)}_j)\xi_{\tau_t+i,j}\Big|\geq \tau\Big\} \leq \exp\Big(\frac{-C_7\tau^2}{m}\Big).
    \end{align*}
    Taking another union bound argument, we have that with probability at least $1-\eta_6$ that
    \begin{align*}
         \max_{t,i}\Big|\frac{1}{m}\sum_{j=1}^m \phi_\ell(X^{(\tau_t+i)}_j)\xi_{\tau_t+i,j}\Big| \leq \sqrt{\log(n/\eta_6)/m}.
    \end{align*}
    
    Combine all results together and apply another union bound argument over $\ell \in [k]$, we have that $\mathbb{P}(\mathcal{E}_2)\geq 1 -r(\eta_2+\eta_3+\eta_4+\eta_5+\eta_6)$. Overall, we have that $\mathbb{P}(\mathcal{E}_1 \cap \mathcal{E}_2) \leq 1- \eta_1 -5r\eta_3 = 1-6\eta$ by picking $\eta_1= \eta$ and $\eta_2 = \eta_3 = \eta_4=\eta_5=\eta_6= \eta/r$.
\end{proof}

\section{Technical details for Section \ref{section_mean_fdp}}
\subsection{Proof of Theorem \ref{fdp_thm_mean_up}} \label{section_appendix_fdp_mean_up}
\begin{proof}[Proof of Theorem \ref{fdp_thm_mean_up}]
    The first part of the theorem follows from a similar argument as the one used in the proof of the first part of Theorem \ref{thm_mean_upper}. For every server $s \in [S]$, we ensure the output in each of the $T$ iterations satisfies $(\epsilon_s,\delta_s)$-CDP, thus we have proved that Algorithm \ref{algorithm_fdp_mean} is $(\bm{\epsilon}, \bm{\delta},T)$-FDP. 
    
    The second claim is a consequence of Proposition \ref{fdp_prop_mean_up}. Set the weights 
    \begin{align*}
        \nu_s = \frac{u_s}{\sum_{s=1}^S u_s} \quad \text{where} \quad u_s \asymp_{\log} (r^2n_s) \wedge (r^2n_s^2\epsilon_s^2) \wedge (rn_sm) \wedge (n_s^2m\epsilon_s^2),
    \end{align*}
    and $T\asymp \log(\sum_{s=1}^S n_s)$, it then holds that
    \begin{align*}
        \|a^{T}-a_r^*\|_2^2 =_{\log} O_P\Bigg\{\frac{r^2}{\sum_{s=1}^S (r^2n_s) \wedge (r^2n_s^2\epsilon_s^2) \wedge (rn_sm) \wedge (n_s^2m\epsilon_s^2)} + r^{-2\alpha}\Bigg\}.
    \end{align*}
    Hence, we have 
    \begin{align*}
        \|\widetilde{\mu} - \mu^*\|_{L^2}^2 \lesssim \|a^{T} - a^*_r\|_2^2 + \|\Phi_r^\top a^*_r - \mu^*\|_{L^2}^2 \lesssim \|a^{T} - a^*_r\|_2^2 + r^{-2\alpha},
    \end{align*}
    where the last inequality follows from the fact that $\mu^* \in \mathcal{W}(\alpha, C_\alpha)$.

    The third claim follows from a direct application of the second claim. 
\end{proof}

\begin{proposition} \label{fdp_prop_mean_up}
        Initializing Algorithm \ref{algorithm_fdp_mean} with $a^0 = 0$ and step size $\rho = 4L/(1+4L^2)$. Suppose that
        \begin{align*}
            r\log^2(Tr/\eta) \lesssim \Big(\sum_{s=1}^S \frac{T\nu_s^2}{n_sm}\Big)^{-1},  \quad r\log(Tr/\eta) \lesssim \Big(\sup_{s\in[S]} \frac{T\nu_s}{n_sm}\Big)^{-1},
        \end{align*}
        then under the same condition as the one in Theorem \ref{fdp_thm_mean_up}, it holds with probability at least $1-10\eta$ that 
        \begin{align*}
            \|\widetilde{\mu} - \mu^*\|_{L^2}^2 \lesssim &\exp\Big(-\frac{4T}{1+4L^2}\Big)+ \frac{1}{\eta}\Big(\sum_{s=1}^S\frac{Tr\nu_s^2}{n_sm}+ \sum_{s=1}^S\frac{T\nu_s^2}{n_s}+r^{-2\alpha}\Big)\\
            & \;+ \log(T/\eta)\log(1/\delta_s)\sum_{s=1}^S \Big(\frac{T^2\nu_s^2}{n_s^2\epsilon_s^2} + \frac{T^2r^2\nu_s^2\log^2(N/\eta)}{n_s^2m\epsilon_s^2}\Big),
        \end{align*}
        for some small enough $\eta\in (0, 1/10)$.
\end{proposition}

\begin{proof}[Proof of Proposition \ref{fdp_prop_mean_up}]
    Denote $\tau_{s,t} = b_st$ for $t \in \{0\} \cup [T-1]$ and consider the following two events
    \begin{align*}
        \mathcal{E}'_1 = \Bigg\{&\Lambda_{\min}\Big\{\sum_{s=1}^S\sum_{i=1}^{b_s} \sum_{j=1}^m\frac{\nu_s}{b_sm} \Phi_r(X^{(s,\tau_{s,t}+i)}_j)\Phi_r^{\top}(X^{(s,\tau_{s,t}+i)}_j)\Big\} \geq 1/(2L)\quad \text{and} \\
        &\Lambda_{\max}\Big\{\sum_{s=1}^S\sum_{i=1}^{b_s} \sum_{j=1}^m\frac{\nu_s}{b_sm} \Phi_r(X^{(s,\tau_{s,t}+i)}_j)\Phi_r^{\top}(X^{(s,\tau_{s,t}+i)}_j)\Big\} \leq 2L, \forall t \in \{0\} \cup [T-1]\Bigg\},
    \end{align*}
    \begin{align*}
        \mathcal{E}'_2 = \Big\{\Pi^{\mathrm{entry}}_{R}\Big[&\frac{1}{m}\sum_{j=1}^m \Phi_r(X^{(s,\tau_{s,t}+i)}_j)\big\{\Phi^{\top}_r(X^{(s,\tau_{s,t}+i)}_j)a^{t}-Y^{(s,\tau_{s,t}+i)}_j \big\}\Big]\\
        &= \frac{1}{m}\sum_{j=1}^m \Phi_r(X^{(s,\tau_{s,t}+i)}_j)\big\{\Phi^{\top}_r(X^{(s,\tau_{s,t}+i)}_j)a^{t}-Y^{(s,\tau_{s,t}+i)}_j \big\}, \\
        & \hspace{-1cm} \forall i \in [b_s],\;  t \in \{0\} \cup [T-1], \; s\in [S]\Big\}.
    \end{align*}
    We control the probability of these events happening in Lemma \ref{lemma_fdp_thm_up}. The rest of the proof is conditioned on these two events.

    For the $t$th iteration, we can rewrite the noisy gradient descent as
    \begin{align*}
        &a^{t} - \sum_{s=1}^S\sum_{i=1}^{b_s}\sum_{j=1}^m \frac{\rho \nu_s}{b_sm} \Phi_r(X^{(s,\tau_{s,t}+i)}_j)\{\Phi^{\top}_r(X^{(s,\tau_{s,t}+i)}_j)a^{t}-Y^{(s,\tau_{s,t}+i)}_j\} - \sum_{s=1}^S \rho \nu_s w_{s,t}\\
        =\;& a^t - \sum_{s=1}^S\sum_{i=1}^{b_s}\sum_{j=1}^m \frac{\rho \nu_s}{b_sm} \Phi_r(X^{(s,\tau_{s,t}+i)}_j)\big[\Phi^{\top}_r(X^{(s,\tau_{s,t}+i)}_j)(a^{t}- a^*_r)\\
        &\hspace{6cm} - \{\mu^*(X^{(s,\tau_{s,t}+i)}_j)- \Phi^{\top}_r(X^{(s,\tau_{s,t}+i)}_j)a^*_r\}\\
        &\hspace{6cm} -U^{(s,\tau_{s,t}+i)}(X^{(s,\tau_{s,t}+i)}_j)-\xi_{s,\tau_{s,t}+i,j}\big]- \sum_{s=1}^S \rho\nu_s w_{s,t},
    \end{align*}
    where the equality follows from \eqref{fdp_mean_obs}. Due to the projection operator $\Pi^{*}_{\mathcal{A}}$ and the fact that $a_r^* \in \mathcal{A}$, Lemma~\ref{lemma_projection} leads to
    \begin{align} \notag
        \|a^{t+1} - a^*_r\|_2^2 \lesssim &\; \Big\|\Big\{I-\sum_{s=1}^S\sum_{i=1}^{b_s}\sum_{j=1}^m  \frac{\rho \nu_s}{b_sm}\Phi_r(X^{(s,\tau_{s,t}+i)}_j)\Phi^{\top}_r(X^{(s,\tau_{s,t}+i)}_j)\Big\}(a^t - a_r^*)\Big\|_2^2\\ \notag
        & + \Big\|\sum_{s=1}^S\sum_{i=1}^{b_s}\sum_{j=1}^m \frac{\rho \nu_s}{b_sm} \Phi_r(X^{(s,\tau_{s,t}+i)}_j)\Big\{\mu^*(X^{(\tau_{s,t}+i)}_j)- \Phi^{\top}_r(X^{(\tau_{s,t}+i)}_j)a^*_r\Big\}\Big\|_2^2\\ \notag
        & +\Big\|\sum_{s=1}^S\sum_{i=1}^{b_s}\sum_{j=1}^m \frac{\rho \nu_s}{b_sm} \Phi_r(X^{(s,\tau_{s,t}+i)}_j)U^{(\tau_{s,t}+i)}(X^{(\tau_{s,t}+i)}_j)\Big\|_2^2 \\ \notag
        &+\Big\|\sum_{s=1}^S\sum_{i=1}^{b_s}\sum_{j=1}^m  \frac{\rho \nu_s}{b_sm}\Phi_r(X^{(s,\tau_{s,t}+i)}_j)\xi_{\tau_{s,t}+i,j}\Big\|_2^2 +\Big\| \sum_{s=1}^S \rho\nu_sw_{s,t}\Big\|_2^2\\ \label{l_fdp_mean_upper_eq1}
        = &\; \|(I)\|^2_2+\|(II)\|^2_2+\|(III)\|^2_2+\|(IV)\|^2_2 +\rho^2\Big\|\sum_{s=1}^S \nu_s w_{s,t}\Big\|_2^2.
    \end{align}
    We will control the estimation error $\|a^{t+1} - a^*_r\|_2^2$ by controlling each term in \eqref{l_fdp_mean_upper_eq1} individually. 
    
    For $(I)$, note that in the event $\mathcal{E}'_1$, it holds that
    \begin{align*}
        &\Big\|I-\sum_{s=1}^S\sum_{i=1}^{b_s}\sum_{j=1}^m  \frac{\rho \nu_s}{b_sm}\Phi_r(X^{(s,\tau_{s,t}+i)}_j)\Phi^{\top}_r(X^{(s,\tau_{s,t}+i)}_j)\Big\|_{\op}\\
        \leq \;& \max \Big\{|1-\frac{\rho}{2L}|, |1- 2L\rho|\Big\} =1 - \frac{2}{1+4L^2},
    \end{align*}
    by the choice of $\rho = 4L/(1+4L^2)$. Therefore, we have that
    \begin{align}\label{l_fdp_mean_upper_eq2}
    	\|(I)\|^2_2 \leq \Big(1 - \frac{2}{1+4L^2}\Big)^2\|a^{t}-a_r^*\|_2^2.
    \end{align}
    Note that here we choose $\rho = 4L/(1+4L^2)$ for simplicity of the proof and we remark that any choices of $\rho \in (0,1/L)$ will work. To control $\|(II)\|^2_2, \|(III)\|^2_2$ and $\|(IV)\|^2_2$, we will use a similar argument as the one used to derive \eqref{thm_mean_upper_eq2}, \eqref{thm_mean_upper_eq3} and \eqref{thm_mean_upper_eq4} with the only difference to control the summation over $s\in [S]$. For $(II)$, we have that 
    \begin{align*}
            (II) \leq &\; \sum_{s=1}^S\sum_{i=1}^{b_s}\sum_{j=1}^m  \frac{\rho\nu_s}{b_sm} \Big[\Phi_r(X^{(s,\tau_{s,t}+i)}_j)\Big\{\mu^*(X^{(s,\tau_{s,t}+i)}_j)- \Phi^{\top}_r(X^{(s,\tau_{s,t}+i)}_j)a^*_r\Big\}\\
            &\hspace{3cm}- \mathbb{E}\big[\Phi_r(X^{(s,\tau_{s,t}+i)}_j)\big\{\mu^*(X^{(s,\tau_{s,t}+i)}_j)- \Phi^{\top}_r(X^{(s,\tau_{s,t}+i)}_j)a^*_r\big\}\big]\Big]\\
            & +\sum_{s=1}^S\sum_{i=1}^{b_s}\sum_{j=1}^m \frac{\rho\nu_s}{b_sm}\mathbb{E}\big[\Phi_r(X^{(s,\tau_{s,t}+i)}_j)\big\{\mu^*(X^{(s,\tau_{s,t}+i)}_j)- \Phi^{\top}_r(X^{(s,\tau_{s,t}+i)}_j)a^*_r\big\}\big]\Big]\\
            = &\; (II_1) + (II_2).
        \end{align*}
    To upper bound $\|(II_1)\|_2$, it holds that
    \begin{align*}
        \mathbb{E}(\|(II_1)\|^2_2) &= \sum_{\ell=1}^r \mathbb{E}\Bigg[\Big\{\sum_{s=1}^S\sum_{i=1}^{b_s}\sum_{j=1}^m\frac{\rho\nu_s}{b_sm}\phi_\ell(X^{(s,\tau_{s,t}+i)}_j)\big\{\mu^*(X^{(s,\tau_{s,t}+i)}_j)- \Phi^{\top}_r(X^{(s,\tau_{s,t}+i)}_j)a^*_r\big\}\\
        &\hspace{0.5cm}-\sum_{s=1}^S\sum_{i=1}^{b_s}\sum_{j=1}^m\frac{\rho\nu_s}{b_sm}\mathbb{E}\big[\phi_\ell(X^{(s,\tau_{s,t}+i)}_j)\big\{\mu^*(X^{(s,\tau_{s,t}+i)}_j)- \Phi^{\top}_r(X^{(s,\tau_{s,t}+i)}_j)a^*_r\big\}\big]\Big\}^2\Bigg]\\
        &= \sum_{\ell=1}^r\sum_{s=1}^S\frac{\rho^2\nu_s^2}{b_sm} \mathbb{E}\Bigg[\Big\{\phi_\ell(X^{(s,\tau_{s,t}+1)}_1)\Big\{\mu^*(X^{(s,\tau_{s,t}+1)}_1)- \Phi^{\top}_r(X^{(s,\tau_{s,t}+1)}_1)a^*_r\Big\} \\
        & \hspace{2cm} - \mathbb{E}\big[\phi_\ell(X^{(s,\tau_{s,t}+1)}_1)\big\{\mu^*(X^{(s,\tau_{s,t}+1)}_1)- \Phi^{\top}_r(X^{(s,\tau_{s,t}+1)}_1)a^*_r\big\}\big]\Big\}^2\Bigg]\\
        & \leq \sum_{\ell=1}^r\sum_{s=1}^S \frac{\rho^2\nu_s^2}{b_sm} \mathbb{E}\Bigg[\Big\{\phi_\ell(X^{(s,\tau_{s,t}+1)}_1)\Big\{\mu^*(X^{(s,\tau_{s,t}+1)}_1)- \Phi^{\top}_r(X^{(s,\tau_{s,t}+1)}_1)a^*_r\Big\}\Big\}^2\Bigg]\\
        & \lesssim \sum_{s=1}^S \frac{\rho^2 r\nu_s^2}{b_sm}\mathbb{E}\Big[\Big\{\mu^*(X^{(s,\tau_{s,t}+1)}_1)- \Phi^{\top}_r(X^{(s,\tau_{s,t}+1)}_1)a^*_r\Big\}^2\Big]\\
        & \lesssim \sum_{s=1}^S\frac{r^{1-2\alpha}\nu_s^2}{b_sm}\lesssim r^{-2\alpha},
    \end{align*}
    where across the proof in addition to the previous arguments used to derive \eqref{thm_mean_upper_eq5}, we further use the fact that we have independence across $S$ sources and the last inequality holds for all selections of $r$ such that $r \lesssim \{\sum_{s=1}^S \nu_s^2/(b_sm)\}^{-1}$. 
    
    To upper bound $\|(II)_2\|_2$, it holds that
    \begin{align*}
        \|(II)_2\|_2 &= \Big\|\sum_{s=1}^S\sum_{i=1}^{b_s}\sum_{j=1}^m \frac{\rho \nu_s}{b_sm} \mathbb{E}\big[\Phi_r(X^{(s,\tau_{s,t}+i)}_j)\big\{\mu^*(X^{(s,\tau_{s,t}+i)}_j)- \Phi^{\top}_r(X^{(s,\tau_{s,t}+i)}_j)a^*_r\Big\}\Big]\Big\|_2 \\
        & = \Big\|\sum_{s=1}^S \rho \nu_s\mathbb{E}\Big[\Phi_r(X^{(s,\tau_{s,t}+1)}_1)
        \Big\{\mu^*(X^{(s,\tau_{s,t}+1)}_1)- \Phi^{\top}_r(X^{(s,\tau_{s,t}+1)}_1)a^*_r\Big\}\Big]\Big\|_2\\
        &\leq\sum_{s=1}^S \rho \nu_s\Big\|\mathbb{E}\Big[\Phi_r(X^{(s,\tau_{s,t}+1)}_1)
        \Big\{\mu^*(X^{(s,\tau_{s,t}+1)}_1)- \Phi^{\top}_r(X^{(s,\tau_{s,t}+1)}_1)a^*_r\Big\}\Big]\Big\|_2\\
        & \lesssim \rho \|\mu^*- \Phi^{\top}_r a^*_r\|_{L^2} \lesssim r^{-\alpha},
    \end{align*}
    where in addition to the arguments used to derive \eqref{thm_mean_upper_eq6}, we also used the fact that $\sum_{s=1}^S \nu_s = 1$. Therefore, by Markov's inequality, we have that with probability at least $1-\eta_1$ that
    \begin{align}\label{l_fdp_mean_upper_eq3}
        \mathbb{P}(\|(II)\|_2^2 \lesssim r^{-2\alpha}/\eta_1) \geq 1-\eta_1
    \end{align}
    for any $\eta_1\in (0,1/2)$ as long as $r \lesssim \{\sum_{s=1}^S \nu_s^2/(b_sm)\}^{-1}$. To upper bound $\|(III)\|_2^2$, it holds that
    \begin{align*}
        \mathbb{E}(\|(III)\|_2^2) = \;&\sum_{\ell=1}^r \mathbb{E}\Big[\Big\{\sum_{s=1}^S\sum_{i=1}^{b_s}\sum_{j=1}^m \frac{\rho\nu_s}{b_sm}\phi_\ell(X^{(s,\tau_{s,t}+i)}_j)U^{(s,\tau_{s,t}+i)}(X^{(s,\tau_{s,t}+i)}_j)\Big\}^2\Big]\\
        =\;& \sum_{\ell=1}^r\sum_{s=1}^S \frac{\rho^2\nu_s^2}{b_sm^2}\mathbb{E}\Big[\Big\{ \sum_{j=1}^m\phi_\ell(X^{(s,\tau_{s,t}+1)}_j)U^{(s,\tau_{s,t}+1)}(X^{(s,\tau_{s,t}+1)}_j)\Big\}^2\Big]\\
        \lesssim \;&\sum_{s=1}^S\frac{\rho^2\nu_s^2}{b_sm}\sum_{\ell=1}^r\mathbb{E}\Big[\phi_\ell^2(X^{(s,\tau_{s,t}+1)}_1)\{U^{(s,\tau_{s,t}+1)}(X^{(s,\tau_{s,t}+1)}_1)\}^2\Big]\\
        &+ \sum_{s=1}^S\frac{\rho^2\nu_s^2}{b_s}\sum_{\ell=1}^r\mathbb{E}\Big[\phi_\ell(X^{(s,\tau_{s,t}+1)}_1)U^{(s,\tau_{s,t}+1)}(X^{(s,\tau_{s,t}+1)}_1)\\
        & \hspace{5cm}\phi_\ell(X^{(s,\tau_{s,t}+1)}_2)U^{(s,\tau_{s,t}+1)}(X^{(s,\tau_{s,t}+1)}_2)\Big]\\
        \lesssim \;& \sum_{s=1}^S\frac{r\nu_s^2}{b_sm}+ \sum_{s=1}^S\frac{\nu_s^2}{b_s},
    \end{align*}
    where the last inequality follows from Assumption \ref{a_model}(b)~and the fact that
    \begin{align*}
        \mathbb{E}\Big[\phi_\ell^2(X^{(s,\tau_{s,t}+1)}_1)\{U^{(s,\tau_{s,t}+1)}(X^{(s,\tau_{s,t}+1)}_1)\}^2\Big]
        &= \mathbb{E}_U\Big[\int_{0}^1 \big\{U^{(s,\tau_{s,t}+1)}(s)\big\}^2\phi_\ell^2(s)f_X(s) \mathrm{d}s\Big]\\
        & \leq 2L\mathbb{E}_U\Big[\int_{0}^1 U^2(s) \mathrm{d}s\Big]
        \leq 2LC_U,
    \end{align*}
    and
    \begin{align*}
       &\sum_{\ell=1}^r\mathbb{E}\Big[\phi_\ell(X^{(s,\tau_{s,t}+1)}_1)U^{(s,\tau_{s,t}+1)}(X^{(s,\tau_{s,t}+1)}_1)\phi_\ell(X^{(s,\tau_{s,t}+1)}_2)U^{(s,\tau_{s,t}+1)}(X^{(s,\tau_{s,t}+1)}_2)\Big]\\
       =\;&\sum_{\ell=1}^r \mathbb{E}\Big[\mathbb{E}\Big\{\phi_\ell(X^{(s,\tau_{s,t}+1)}_1)U^{(s,\tau_{s,t}+1)}(X^{(s,\tau_{s,t}+1)}_1)\Big|U\Big\}^2\Big] \\
       =\;& \frac{1}{b}\mathbb{E}\Big[\Big\|\mathbb{E}\Big\{\Phi_\ell(X^{(s,\tau_{s,t}+1)}_1)U^{(\tau_{s,t}+1)}(X^{(s,\tau_{s,t}+1)}_1)\Big|U\Big\}\Big\|_2^2\Big]\\
       \leq \;& \mathbb{E}\Big\{\|U^{(s,\tau_{s,t}+1)}f_X\|_{L^2}^2\Big\}] \leq L^2\mathbb{E}\Big\{\|U^{(s,\tau_{s,t}+1)}\|_{L^2}^2\Big\} \lesssim 1.
    \end{align*}
    Therefore, the Markov inequality implies that
    \begin{equation}\label{l_fdp_mean_upper_eq4}
        \mathbb{P}\Big\{\|(III)\|^2_2 \lesssim \Big(\sum_{s=1}^S\frac{r\nu_s^2}{b_sm}+ \sum_{s=1}^S\frac{\nu_s^2}{b_s}\Big)/\eta_2 \Big\} \geq 1 - \eta_2,
    \end{equation}
    for any $\eta_2 \in (0,16L^3/(4L^2+1)^2)$.
    
    Finally, to find an upper bound on $\|(IV)\|_2^2$, we have that
    \begin{align*}
        \mathbb{E}(\|(IV)\|_2^2)&= \sum_{\ell=1}^r \mathbb{E}\Big[\Big\{\sum_{i=1}^b\sum_{j=1}^m \sum_{s=1}^S \frac{\rho\nu_s}{b_sm} \phi_\ell(X^{(s,\tau_{s,t}+i)}_j)\xi_{s,\tau_{s,t}+i,j}\Big\}^2\Big] \\
        &= \sum_{\ell=1}^r \sum_{s=1}^S\frac{\rho^2\nu_s^2}{b_sm} \mathbb{E}[\phi^2_\ell(X^{(s,\tau_{s,t}+1)}_1)\xi^2_{s,\tau_{s,t}+1,1}]\\
        & = \sum_{\ell=1}^r \sum_{s=1}^S\frac{\rho^2\nu_s^2}{b_sm}\mathbb{E}_X[\phi^2_\ell(X^{(\tau_{s,t}+1)}_1)]\mathbb{E}_\xi[\xi^2_{\tau_{s,t}+1,1}]\\
        &\leq \sum_{s=1}^S\frac{r\nu_s^2 \rho^2 L C_{\xi}^2}{b_sm} \asymp \sum_{s=1}^S\frac{r\nu_s^2}{b_sm},
    \end{align*}
    where the third equality follows from the fact that $X^{(\tau_{s,t}+1)}_1$ and $\xi_{\tau_{s,t}+1,1}$ are independent and the fourth equality follows from Assumption \ref{a_model}(c). Therefore, Markov's inequality implies that 
    \begin{equation}\label{l_fdp_mean_upper_eq5}
        \mathbb{P}\Big\{\|(IV)\|_2^2 \lesssim \Big(\sum_{s=1}^S\frac{r\nu_s^2}{b_sm}\Big)/\eta_3\Big\}\geq 1-\eta_3,
    \end{equation}
    for any $\eta_3 \in (0,1/2)$. For the term $\|\sum_{s=1}^S \nu_s w_{s,t}\|_2^2$, firstly, we note that for any $\ell \in [r]$, the $\ell$th entry of $\sum_{s=1}^S \nu_s w_{s,t}$ follows a Gaussian distribution with mean zero and variance $\sum_{s=1}^S \nu_s^2\sigma^2_{s,\ell}$. Therefore, using standard properties of Gaussian random variables, we have that 
    \begin{align*}
        \Big\|\sum_{s=1}^S \nu_s w_{s,t}\Big\|_2^2\sim \chi^2_{\sum_{\ell=1}^r\sum_{s=1}^S \nu_s^2\sigma^2_{s,\ell}},
    \end{align*}
    which is a sub-Exponential distribution with the sub-Exponential parameter $\sum_{\ell=1}^r\sum_{s=1}^S \nu_s^2\sigma^2_{s,\ell}$ and mean
    \begin{align*}
        \mathbb{E}\Big(\Big\|\sum_{s=1}^S \nu_s w_{s,t}\Big\|_2^2\Big) = \sum_{\ell=1}^r\sum_{s=1}^S \nu_s^2\sigma_{s,\ell}^2.
    \end{align*}
    Therefore, following from standard properties of sub-Exponential random variables \citep[e.g.~Proposition 2.7.1 in][]{vershynin2018high}, we have that for any $\tau >0$,
    \begin{align*}
        \mathbb{P}\Big\{\Big\|\sum_{s=1}^S \nu_s w_{s,t}\Big\|_2^2 \geq \tau - \mathbb{E}\Big(\Big\|\sum_{s=1}^S \nu_s w_{s,t}\Big\|_2^2\Big)\Big\} \leq \exp\Big(-\frac{C_1\tau}{\sum_{\ell=1}^r\sum_{s=1}^S \nu_s^2\sigma^2_{s,\ell}}\Big).
    \end{align*}
    Combining with a union bound argument on $t \in [T-1]$, we have with probability at least $1-\eta_4$ that
    \begin{align}\notag
        \Big\|\sum_{s=1}^S \nu_s w_{s,t}\Big\|_2^2 &\leq \log(T/\eta_4)\sum_{\ell=1}^r\sum_{s=1}^S \nu_s^2\sigma^2_{s,\ell}\lesssim \log(T/\eta_4)\sum_{\ell=1}^r\sum_{s=1}^S\nu_s^2\log(1/\delta_s)R_\ell\sum_{k=1}^r R_k/(b_s^2\epsilon_s^2)\\ \notag
        & = \log(T/\eta_4)\log(1/\delta_s)\sum_{s=1}^S \nu_s^2\Big(\sum_{k=1}^r R_k\Big)^2/(b_s^2\epsilon_s^2)\\ \notag
        &\asymp \log(T/\eta_4)\log(1/\delta_s)\sum_{s=1}^S \nu_s^2\Big\{\sum_{\ell=1}^r \sqrt{m^{-1}\log^2(N/\eta)}+\ell^{-\alpha}\Big\}^2/(b_s^2\epsilon_s^2)\\ \notag
        &\lesssim \log(T/\eta_4)\log(1/\delta_s)\sum_{s=1}^S \nu_s^2\Big\{r^2m^{-1}\log^2(N/\eta)+1\Big\}/(b_s^2\epsilon_s^2)\\ \label{l_fdp_mean_upper_eq6}
        & = \log(T/\eta_4)\log(1/\delta_s)\sum_{s=1}^S \Big(\frac{\nu_s^2}{b_s^2\epsilon_s^2} + \frac{r^2\nu_s^2\log^2(N/\eta)}{b_s^2m\epsilon_s^2}\Big).
    \end{align}
    Therefore, substituting results in \eqref{l_fdp_mean_upper_eq2}, \eqref{l_fdp_mean_upper_eq3}, \eqref{l_fdp_mean_upper_eq4}, \eqref{l_fdp_mean_upper_eq5} and \eqref{l_fdp_mean_upper_eq6} into \eqref{l_fdp_mean_upper_eq1} and applying a union bound argument, it holds with probability at least $1-6\eta-\eta_1-\eta_2-\eta_3-\eta_4 = 1-10
    \eta$ that 
    \begin{align} \notag
        \|a^{t+1}-a_r^*\|_2^2 &\lesssim \Big(1 - \frac{2}{1+4L^2}\Big)^2\|a^t - a_r^*\|_2^2+ \frac{1}{\eta}\Big(\sum_{s=1}^S\frac{r\nu_s^2}{b_sm}+ \sum_{s=1}^S\frac{\nu_s^2}{b_s}+r^{-2\alpha}\Big) \\ \notag
        & \hspace{4cm}+ \log(T/\eta)\log(1/\delta_s)\sum_{s=1}^S \Big(\frac{\nu_s^2}{b_s^2\epsilon^2} + \frac{r^2\nu_s^2\log^2(N/\eta)}{b_s^2m\epsilon_s^2}\Big)\\ \notag
        & \leq \Big(1 - \frac{2}{1+4L^2}\Big)^{2t+2}\|a_r^*\|_2^2+ \frac{1}{\eta}\Big(\sum_{s=1}^S\frac{r\nu_s^2}{b_sm}+ \sum_{s=1}^S\frac{\nu_s^2}{b_s}+r^{-2\alpha}\Big) \\ \notag
        & \hspace{4cm}+ \log(T/\eta)\log(1/\delta_s)\sum_{s=1}^S \Big(\frac{\nu_s^2}{b_s^2\epsilon_s^2} + \frac{r^2\nu_s^2\log^2(N/\eta)}{b_s^2m\epsilon_s^2}\Big)\\ \notag
        &\leq \exp\Big(-\frac{2(2t+2)}{1+4L^2}\Big)\|a_r^*\|_2^2+ \frac{1}{\eta}\Big(\sum_{s=1}^S\frac{r\nu_s^2}{b_sm}+ \sum_{s=1}^S\frac{\nu_s^2}{b_s}+r^{-2\alpha}\Big) \\ \label{l_fdp_mean_upper_eq7}
        & \hspace{4cm} + \log(T/\eta)\log(1/\delta_s)\sum_{s=1}^S \Big(\frac{\nu_s^2}{b_s^2\epsilon_s^2} + \frac{r^2\nu_s^2\log^2(N/\eta)}{b_s^2m\epsilon_s^2}\Big),
    \end{align}
    where the second inequality follows from an iterative argument. The proposition then follows from the fact that
    \begin{align*}
        \|\widetilde{\mu} - \mu^*\|_{L^2}^2 \lesssim \|a^{T} - a^*_r\|_2^2 + \|\Phi_r^\top a^*_r - \mu^*\|_{L^2}^2 \lesssim \|a^{T} - a^*_r\|_2^2 + r^{-2\alpha},
    \end{align*}
    where the last inequality follows from the fact that $\mu^* \in \mathcal{W}(\alpha, C_\alpha)$.
    
\end{proof}

\subsection{Proof of Theorem \ref{fdp_thm_mean_low}} \label{appendix_section_thm_mean_low}
\begin{proof}[Proof of Theorem \ref{fdp_thm_mean_low}]
    The first part of Theorem \ref{fdp_thm_mean_low} is a consequence of Proposition \ref{prop_mean_fdp} by selecting the optimal value of $\{b_s^t\}_{t=1}^{T}$ for any $s\in [S]$ to maximize the value in the right-hand side of \eqref{prop_mean_fdp_eq1}. By considering the Lagrangian of two optimization problems
    \begin{align*}
        \mathcal{L}_1(\lambda) = -\sum_{t=1}^T \{b_s^t \wedge (b_s^t)^2\epsilon_s^2\} + \lambda\Big\{n_s-\sum_{t=1}^T b_s^t\Big\},
    \end{align*}
    and
    \begin{align*}
        \mathcal{L}_2(\lambda) = -\sum_{t=1}^T \{(b_s^t)^2m\epsilon_s^2 \wedge r_0b_s^tm\} + \lambda\Big\{n_s -\sum_{t=1}^T b_s^t\Big\},
    \end{align*}
    it holds that the value of $b_s^t$ should be shared across all iterations, i.e.~$(b_s^t)^* = n_s/T$ to maximize the lower bound. The second part is a direct calculation from the first part.
\end{proof}

\begin{proposition} \label{prop_mean_fdp}
    Denote $\mathcal{P}_X$ the class of sampling distributions satisfying Assumption \ref{a_sample} and $\mathcal{P}_Y$ the class of distributions for noisy observations satisfying Assumptions \ref{a_model}. Let  $r_0 >0$ be the number solving the equation $r^{2\alpha+2} = \sum_{s=1}^S\sum_{t=1}^T \{(b_s^t)^{2}m\epsilon_s^2 \wedge rb^t_sm\}$. Suppose that $\delta_s \sqrt{\log(1/\delta_s)}  \lesssim r_0^{-1}m^{-3}\epsilon_s^2$ for any $s \in [S]$, then it holds that 
        \begin{align}  \notag
            \underset{Q \in \mathcal{Q}^T_{\bm{\epsilon}, \bm{\delta}}}{\inf} \underset{\widetilde{\mu}}{\inf} &\underset{P_X \in \mathcal{P}_X, P_Y \in \mathcal{P}_Y}{\sup} \mathbb{E}_{P_X, P_Y, Q}\|\widetilde{\mu}- \mu^*\|_{L^2}^2 \\ \label{prop_mean_fdp_eq1}
            & \gtrsim \frac{1}{\sum_{s=1}^S\sum_{t=1}^T \{b_s^t \wedge (b_s^t)^2\epsilon_s^2\}} \vee \frac{r_0^2}{\sum_{s=1}^S \sum_{t=1}^T \{(b_s^t)^2m\epsilon_s^2 \wedge r_0b_s^tm\}} ,
        \end{align}
        where $\mathcal{Q}^T_{\bm{\epsilon}, \bm{\delta}}$ the collection of all mechanisms satisfying $(\bm{\epsilon},\bm{\delta},T)$-FDP defined in Definition~\ref{def_fdp}.
\end{proposition}

\begin{proof}[Proof of Proposition \ref{prop_mean_fdp}]
    Throughout the section, denote $L_t^{(s)} = \{L_t^{(s,i)}\}_{i=1}^{b_s^t} = \{(X^{(s,i)}_j, Y^{(s,i)}_j)\\ \}_{i=1,j=1}^{b_s^t,m}$, the $b_s^t$ pairs of realizations on the $s$th server in iteration $t$ and $Z_t^{(s)}$ the privatized transcript released from server $s$ to the central server in iteration $t$. For any $s\in [S]$, we further denote $L^{(s)} = \{L_t^{(s)}\}_{t=1}^T$ and $Z^{(s)} = \{Z_t^{(s)}\}_{t=1}^T$. We prove the minimax lower bound by considering two separate constructions. 
    
    \noindent \textbf{Case 1.} Similar to the case of CDP, we firstly reduce the problem to a finite-dimensional subspace $\Theta(r, C_\alpha)$ as defined in \eqref{mean_t_lower_eq6}. Following a similar argument as the one used in the proof of Theorem 4.1 in \citet{cai2024optimal} and applying the multivariate version of the Van-Trees inequality \citep[e.g.][]{gill1995applications}, we can bound the average $\ell^2$ risk of interest by
    \begin{align} \label{fdp_t_mean_low_eq1}
        \int \mathbb{E}\Big\{\sum_{\ell=1}^r (\widetilde{a}_\ell-a_\ell)^2\Big\}\pi(a)\;\mathrm{d}a \geq \frac{r^2}{\int \tr\{I_{Z^{(1)},\ldots,Z^{(S)}}(a)\}\pi(a)\;\mathrm{d}a + J(\pi)}, 
    \end{align}
    where $I_{Z^{(1)},\ldots,Z^{(S)}}(a)$ is the Fisher information associated with $\bm{Z} = (Z^{(1)},\ldots,Z^{(S)})$, $\pi(a)$ is a prior for the parameter $a$ and $J(\pi)$ is the Fisher information associated with the prior $\pi$.

    \noindent \textbf{Case 1 - Step 1: Construction of a class of distribution.}
    To construct the prior distribution of $a$, for any $\ell \in [r]$, denote $\pi_{\ell}$ the prior distribution of the $\ell$th entry of $a$ and we further assume that $\pi(a) = \prod_{\ell=1}^r \pi_\ell(a_\ell)$. We will take $\pi_\ell$ as a rescaled version of the density $t \mapsto \cos^2(\pi t/2)\mathbf{1}\{|t|\leq 1\}$ such that it is supported on $[-B, B]$ where $B^2 = \frac{C_1^2}{2\pi^{2\alpha}}\Big(\int_{1}^{r+1}t^{2\alpha}\mathrm{d}t\Big)^{-1}\asymp r^{-(2\alpha+1)}$. We obtain that the prior distribution is supported in $\Theta(r, C_\alpha)$ since
    \begin{align*}
        \sum_{\ell=1}^r \ell^{2\alpha}a_\ell^2 \leq \frac{C_1^2}{2\pi^{2\alpha}}\Big(\int_{1}^{r+1}t^{2\alpha}\mathrm{d}t\Big)^{-1}\sum_{\ell=1}^r \ell^{2\alpha} \lesssim \frac{C_1^2}{\pi^{2\alpha}}.
    \end{align*}
    
    To calculate the Fisher information $J(\pi)$, following the same argument as the one in the proof of Lemma 4.3 in \citet{cai2024optimal}, we have that 
    \begin{align} \label{fdp_t_mean_low_eq2}
        J(\pi) =\pi^2 r B^{-2} \asymp \pi^2r^{2\alpha+2}.
    \end{align}
    For the term $I_{Z^{(1)},\ldots, Z^{(S)}}(a)$, using the chain rule of Fisher information, it then holds that for any $a \in \Theta(r, C_{\alpha})$,
    \begin{align} \label{fdp_t_mean_low_eq9}
        I_{Z^{(1)},\ldots, Z^{(S)}}(a) = \sum_{s=1}^S \sum_{t=1}^T I_{Z_t^{(s)}|M^{(t-1)}}(a),
    \end{align}
    where $M^{(t-1)}$ is all the private information generated in the previous $(t-1)$ rounds, with $M^{0} = \emptyset$ and the equality follows from the fact that for any $t \in [T]$, the collection of random variables $\{Z^{s,t}\}_{s=1}^S$ are mutually independent conditioning on $M^{(t-1)}$.
   
    Therefore, substituting \eqref{fdp_t_mean_low_eq2} and \eqref{fdp_t_mean_low_eq9} into \eqref{fdp_t_mean_low_eq1}, we have that the minimax risk is bounded by
    \begin{align} \notag
        \sup_{a \in \Theta(r, C_\alpha)} \mathbb{E}\|\widetilde{a} - a\|_2^2 &\gtrsim \frac{r^2}{\sup_{a}\tr\{I_{Z^{(1)},\ldots,Z^{(S)}}(a)\} +r^{2\alpha+2}}\\
        \label{fdp_t_mean_low_eq3}
        & = \frac{r^2}{\sup_{a} \sum_{s=1}^S \sum_{t=1}^T \tr\{I_{Z_t^{(s)}|M^{(t-1)}}(a)\} + r^{2\alpha+2}},
    \end{align}
    where the first inequality follows from the fact that the supremum risk is lower bounded by Bayes risk. With the same construction and notation as the one used in \textbf{Step 1} in the proof of Theorem \ref{thm_mean_lower} (the extra $s,t$ in the notation denote the server and iteration), for any $i \in [b_s^t]$, define
    \begin{align*}
        S_a (L_t^{(s,i)}) = \Phi_{[m],s,i,t}\Sigma_{s,i,t}^{-1}(Y^{(s)}_{i\cdot,t}-\mu^{(s)}_{i\cdot,t}) \in \mathbb{R}^r,
    \end{align*}
    where for $i \in b_s^t$, $\Phi_{[m],s,i,t} = \{\Phi_r(X_1^{(s,i)}), \ldots, \Phi_r(X_m^{(s,i)})\} \in \mathbb{R}^{r\times m}$. Note that $S_a (L_t^{(s,i)})$ is the score function of $a$ based on the data $L_t^{(s,i)}$, hence we have that $\mathbb{E}\{S_a (L_t^{(s,i)})\} = 0$. Write the score function for local data $L_t^{(s)}$ on the $s$th server in iteration $t$ as $S_a(L_t^{(s)}) = \sum_{i=1}^{b_s^t} S_a (L_t^{(s,i)})$.
    Let 
    \begin{align*}
        C_a(Z_t^{(s)}|M^{(t-1)}) = \mathbb{E}\Big\{S_a(L_t^{(s)})\Big| Z_t^{(s)}, M^{(t-1)}\Big\}\mathbb{E}\Big\{S_a(L_t^{(s)})\Big| Z_t^{(s)}, M^{(t-1)}\Big\}^\top \in \mathbb{R}^{r\times r},
    \end{align*}
    write $C_{a,t}^{(s)} = \mathbb{E}\{S_a(L_t^{(s)})S_a(L_t^{(s)})^\top\}$ for the unconditional version covariance matrix of $S_a(L_t^{(s)})$ and let $C^{(s,i)}_{a,t} = \mathbb{E}\{S_a(L_t^{(s,i)})S_a(L_t^{(s,i)})^\top\}$ such that $C_{a,t}^{(s)} = \sum_{i=1}^{b^s_t} C^{(s,i)}_{a,t}$. With the above construction, by definition of conditional expectation, we then have that  
    \begin{align}\notag
        &\mathbb{E}\Big\{S_a(L_t^{(s)})\Big| Z_t^{(s)} = z^{(s)}_t, M^{(t-1)}\Big\}\\ \notag
        =\;& \int f(l_t^{(s)}|z_t^{(s)}, M^{(t-1)})\frac{\partial }{\partial a}\log f(l^{(s)}_t|M^{(t-1)}) \;\mathrm{d}l^{(s)}_t\\ \notag
        =\; & \int \frac{f(l^{(s)}_t, z^{(s)}_t|M^{(t-1)})}{f(z_t^{(s)}|M^{(t-1)})} \frac{\frac{\partial }{\partial a} f(l^{(s)}_t|M^{(t-1)})}{f(l_t^{(s)}|M^{(t-1)})}\;\mathrm{d}l^{(s)}_t \\ \notag
        = \;&\frac{1}{f(z_t^{(s)}|M^{(t-1)})} \frac{\partial}{\partial a} \int f(l_t^{(s)}|M^{(t-1)})f(z_t^{(s)}|l_t^{(s)},M^{(t-1)})\;\mathrm{d}l^{(s)}_t\\ \notag
        =\;& \frac{1}{f(z_t^{(s)}|M^{(t-1)})} \frac{\partial}{\partial a} \int f(z_t^{(s)}, l_t^{(s)}|M^{(t-1)})\;\mathrm{d}l^{(s)}_t \\ \label{fdp_t_mean_low_eq8}
        = \;& \frac{1}{f(z_t^{(s)}|M^{(t-1)})} \frac{\partial}{\partial a} f(z_t^{(s)} |M^{(t-1)}) =  \frac{\partial}{\partial a} \log f(z_t^{(s)}|M^{(t-1)}),
    \end{align}
    where the first equality follows from the fact that $L_t^{(s)}$ and $M^{(t-1)}$ are conditionally independent. Hence, we have that $I_{Z_t^{(s)}|M^{(t-1)}}(a) = \mathbb{E}[\mathbb{E}\{C_a(Z_t^{(s)}|M^{(t-1)})|M^{(t-1)}\}]$ and substituting into \eqref{fdp_t_mean_low_eq3}, we have that
    \begin{align}\label{fdp_t_mean_low_eq7}
         \sup_{a \in \Theta(r, C_\alpha)} \mathbb{E}\|\widetilde{a} - a\|_2^2 \geq \frac{r^2}{\sup_{a} \sum_{s=1}^S \sum_{t=1}^T\tr\big[\mathbb{E}[\mathbb{E}\{C_a(Z_t^{(s)}|M^{(t-1)})|M^{(t-1)}\}]\big] + r^{2\alpha+2}}.
    \end{align}
    \noindent \textbf{Case 1 - Step 2: Upper bound on} $\tr[\mathbb{E}[\mathbb{E}\{C_a(Z_t^{(s)}|M^{(t-1)})|M^{(t-1)}\}]]$.
    Next, we will find an upper bound on $\tr[\mathbb{E}[\mathbb{E}\{C_a(Z_t^{(s)}|M^{(t-1)})|M^{(t-1)}\}]]$. Denote 
    \begin{align*}
        G_t^{(s,i)} = \langle \mathbb{E}\big\{S_a(L_t^{(s)})\big| Z_t^{(s)}, M^{(t-1)}\big\}, S_a(L_t^{(s,i)})\rangle ,
    \end{align*}
    and 
    \begin{align*}
        \breve{G}_t^{(s,i)} = \langle \mathbb{E}\big\{S_a(L_t^{(s)})\big| Z_t^{(s)}, M^{(t-1)}\big\}, S_a(\breve{L}_t^{(s,i)})\rangle,
    \end{align*}
    where $\breve{L}_t^{(s,i)}$ is an independent copy of $L_t^{(s,i)}$. By Lemma \ref{lemma_orthogonal} and linearity of the inner product, we have that
    \begin{align} \notag
        &\tr\Big[\mathbb{E}\big[\mathbb{E}\{C_a(Z_t^{(s)}|M^{(t-1)})|M^{(t-1)}\}\big]\Big] = \mathbb{E}\Big[\mathbb{E}\Big\{\Big\|\mathbb{E}\Big\{S_a(L_t^{(s)})\Big| Z_t^{(s)}, M^{(t-1)}\Big\}\Big\|_2^2\Big|M^{(t-1)}\Big\}\Big]\\ \notag
        =\;& \mathbb{E}\Big[\mathbb{E}\Big\{\Big \langle \mathbb{E}\Big\{S_a(L_t^{(s)})\Big| Z_t^{(s)}, M^{(t-1)}\Big\}, \mathbb{E}\Big\{S_a(L_t^{(s)})\Big| Z_t^{(s)}, M^{(t-1)}\Big\}\Big\rangle \Big| M^{(t-1)}\Big\}\Big]\\ \notag
        =\;& \sum_{i=1}^{b_s^t} \mathbb{E}\Big[\mathbb{E}\Big\{\Big\langle S_a(L_t^{(s,i)}), \mathbb{E}\Big\{S_a(L_t^{(s)})\Big| Z_t^{(s)}, M^{(t-1)}\Big\}\Big\rangle \Big| M^{(t-1)}\Big\}\Big]\\ \notag
        =\;& \sum_{i=1}^{b_s^t} \mathbb{E}\Big\langle S_a(L_t^{(s,i)}), \mathbb{E}\Big\{S_a(L_t^{(s)})\Big| Z_t^{(s)}, M^{(t-1)}\Big\}\Big\rangle = \sum_{i=1}^{b_s^t}\mathbb{E}(G_t^{(s,i)})\\ \label{fdp_t_mean_low_eq4}
        \lesssim \; &  \sum_{i=1}^{b_s^t} \mathbb{E}(\breve{G}_t^{(s,i)})+ \epsilon_s \mathbb{E}|\breve{G}_t^{(s,i)}|+ W\delta_s + \int_{W}^\infty \mathbb{P}\{|G_t^{(s,i)}| \geq w\}\; \mathrm{d}w,
    \end{align}
    where the last inequality follows from a similar argument as the one used in the proof of Lemma 4.2 in \citet{cai2024optimal}.
    We will upper bound \eqref{fdp_t_mean_low_eq4} by upper bound the four terms inside it individually. Note that for the first term, we have that
    \begin{align*}
        &\mathbb{E}\Big[\big\langle \mathbb{E}\big\{S_a(L_t^{(s)})\big| Z_t^{(s)}, M^{(t-1)}\big\}, S_a(\breve{L}_t^{(s,i)})\big\rangle\Big]\\ 
        =\;& \mathbb{E}\Big[\mathbb{E}\big\{S_a(L_t^{(s)})\big| Z_t^{(s)}, M^{(t-1)}\big\}\Big]^\top \mathbb{E}\{S_a(\breve{L}_t^{(s,i)})\} = 0,
    \end{align*}
    where the second equality follows from the independence between $\breve{L}_t^{(s,i)}$ and $(Z_t^{(s)}, M^{(t-1)})$ and the last equality follows from the property of the score function. For the second term, firstly note that
    \begin{align*}
        \Lambda_{\max}(C_{a,t}^{(s,i)})  &= \Big\|\mathbb{E}_{X}\Big[\Phi_{[m],s,i,t }\Sigma_{s,i,t}^{-1} \mathbb{E}_{Y|X}\big\{(Y^{(s)}_{i\cdot,t}-\mu^{(s)}_{i\cdot,t})(Y^{(s)}_{i\cdot,t}-\mu^{(s)}_{i\cdot,t})^\top|X \big\}\Sigma_{s,i,t}^{-1}\Phi^{\top}_{[m],s,i,t}\Big]\Big\|_{\op}\\
        & = \Big\|\mathbb{E}_{X}\Big(\Phi_{[m],s,i,t}\Sigma_{s,i,t}^{-1}\Phi^{\top}_{[m],s,i,t}\Big)\Big\|_{\op} \lesssim \sigma_0^{-2}\Big\|\mathbb{E}_{X}\Big(\Phi_{[m],s,i,t}\Phi^{\top}_{[m],s,i,t}\Big)\Big\|_{\op} \lesssim m,
    \end{align*}
    where the last inequality follows from \eqref{mean_t_lower_eq7}. Hence, we have that
    \begin{align} \notag
        \mathbb{E}|\breve{G}_t^{(s,i)}| &\leq \sqrt{\mathbb{E}\{(\breve{G}_t^{(s,i)})^2\}} = \sqrt{\mathbb{E}\Big[\big\langle \mathbb{E}\big\{S_a(L_t^{(s)})\big| Z_t^{(s)}, M^{(t-1)}\big\}, S_a(\breve{L}_t^{(s,i)})\big\rangle^2\Big]}\\ \notag
        & = \sqrt{\mathbb{E}\Big[\mathbb{E}\big\{S_a(L_t^{(s)})\big| Z_t^{(s)}, M^{(t-1)}\big\}^\top \cov\{S_a(\breve{L}_t^{(s,i)})\}\mathbb{E}\big\{S_a(L_t^{(s)})\big| Z_t^{(s)}, M^{(t-1)}\big\}\Big]}\\ \notag
        & \leq \sqrt{\Lambda_{\max}(C_{a,t}^{(s,i)})} \sqrt{\mathbb{E}\Big[\mathbb{E}\Big\{\Big\|\mathbb{E}\big\{S_a(L_t^{(s)})\big| Z_t^{(s)}, M^{(t-1)}\big\}\Big\|_2^2\Big| M^{(t-1)}\Big\}\Big]}\\ \label{fdp_t_mean_low_eq10}
        & = \sqrt{\Lambda_{\max}(C_{a,t}^{(s,i)})} \sqrt{\tr\Big[\mathbb{E}\big[\mathbb{E}\{C_a(Z_t^{(s)}|M^{(t-1)})|M^{(t-1)}\}\big]\Big] }\\ \label{fdp_t_mean_low_eq12}
        &\leq \sqrt{m}\sqrt{\tr\Big[\mathbb{E}\big[\mathbb{E}\{C_a(Z_t^{(s)}|M^{(t-1)})|M^{(t-1)}\}\big]\Big]},
    \end{align}
    where the first inequality follows from the Cauchy--Schwarz inequality and the third inequality follows from standard properties of eigenvalues. To control the last term, by a similar argument as the one in the proof of Lemma B.9 in \citet{cai2024optimal}, we have for any $t \in \mathbb{R}$ that 
    \begin{align} \notag
        & \mathbb{E}\Big\{\exp\big(t|G_t^{(s,i)}|\big)\Big\} = \mathbb{E}\Big[\exp\Big\{t\big|\langle \mathbb{E}\big\{S_a(L_t^{(s)})\big| Z_t^{(s)}, M^{(t-1)}\big\}, S_a(L_t^{(s,i)})\rangle\big|\Big\}\Big]\\ \notag
        = \;& \int\int \exp\Big\{t\big|\langle \mathbb{E}\big\{S_a(L_t^{(s)})\big|Z_t^{(s)} =z, M^{(t-1)}= M\big\}, S_a(l)\rangle\big|\Big\} \\ \notag
        & \hspace{8cm}\;\mathrm{d}\mathbb{P}_{Z_t^{(s)},M^{(t-1)}|L_t^{(s,i)}}(z,M)\;\mathrm{d}\mathbb{P}_{L_t^{(s,i)}}(l)\\ \notag
        \leq \;& \int\int \mathbb{E}\Big[\exp\Big\{t\big|\langle S_a(L_t^{(s)}), S_a(l)\rangle\big|\Big\}\Big|Z_t^{(s)} =z, M^{(t-1)}= M\Big]\\ \notag
        & \hspace{8cm}\;\mathrm{d}\mathbb{P}_{Z_t^{(s)},M^{(t-1)}|L_t^{(s,i)}}(z,M)\; \mathrm{d}\mathbb{P}_{L_t^{(s,i)}}(l)\\ \label{fdp_t_mean_low_eq11}
        =\;&  \mathbb{E}\Big[\exp\Big\{t\big|\langle S_a(L_t^{(s)}), S_a(L_t^{(s,i)})\rangle\big|\Big\}\Big],
    \end{align}
    where the inequality follows from Jensen's inequality. Note that the above moment generating function is well defined as each entry inside $\Phi_{[m],s,i,t}$ is bounded and each entry in $Y^{(s)}_{i\cdot,t}-\mu^{(s)}_{i\cdot,t}$ is sub-Gaussian. Without loss of generality, take $i =1$ in $S_a(L_t^{(s,i)})$, then we have
    \begin{align*}
        \langle S_a(L_t^{(s)}),S_a(L_t^{(s,i)}) \rangle &= S_a(L_t^{(s)})^\top S_a(L_t^{(s,i)})\\
        &= \sum_{i=1}^{b_s^t}  (Y^{(s)}_{i\cdot,t}-\mu^{(s)}_{i\cdot,t})^\top \Sigma_{s,i,t}^{-1}\Phi^\top_{[m],s,i,t}\Phi_{[m],s,1,t}\Sigma_{s,1,t}^{-1}(Y^{(s)}_{1\cdot,t}-\mu^{(s)}_{1\cdot,t}).
    \end{align*}
    By a similar argument used to get \eqref{mean_t_lower_eq8}, we have that each entry inside $(Y^{(s)}_{i\cdot,t}-\mu^{(s)}_{i\cdot,t})^\top \Sigma_{s,i,t}^{-1}\\\Phi^\top_{[m],s,i,t}$ follows a sub-Gaussian distribution with parameter $C_2m^2$, hence using standard properties of sub-Gaussian random variables \citep[e.g.~Lemma 2.7.7 in][]{vershynin2018high} and triangle inequality of sub-Exponetial norms, we have that for any $i \in [n_s]$, $(Y^{(s)}_{i\cdot,t}-\mu^{(s)}_{i\cdot,t})^\top \Sigma_{s,i,t}^{-1}\\\Phi^\top_{[m],s,i,t}\Phi_{[m],s,1,t}\Sigma_{s,1,t}^{-1}(Y^{(s)}_{1\cdot,t}-\mu^{(s)}_{1\cdot,t})$ follows sub-Exponential distribution with parameter of $C_3rm^4$. By applying another triangle inequality, we have that $\langle S_a(L_t^{(s)}),S_a(L_t^{(s,i)}) \rangle$ follows a sub-Exponential distribution with parameter of order $C_3b_s^tm^4r$. Therefore, we have from sub-Exponential properties \citep[e.g.~Proposition 2.7.1][]{vershynin2018high} that 
    \begin{align*}
        &\mathbb{E}\Big\{\exp\big(t|\langle S_a(L_t^{(s)}),S_a(L_t^{(s,i)}) \rangle|\big)\Big\} \leq  \exp\{C_3b_s^tm^4r t \},\\
        & \hspace{6cm}\; \text{for any} \;t\; \text{such that}\; 0\leq t\leq (C_3b_s^tm^4r)^{-1}.
    \end{align*}
    Pick $t  = (C_3b_s^tm^4r)^{-1}/2$ and by Markov's inequality, we have for any $\tau >0$ that 
    \begin{align*}
        \mathbb{P}(|G_t^{(s,i)}|\geq \tau) &= \mathbb{P}\{\exp(t|G_t^{(s,i)}|)\geq \exp(t\tau)\}  \leq \exp(-t\tau)\mathbb{E}\{\exp(t|G_t^{(s,i)}|)\}\\
        & \lesssim \exp(1/2)\exp\Big(-\frac{\tau}{b_s^tm^4r}\Big),
    \end{align*}
    which means that pick $W \asymp b_s^tm^4r\log(1/\delta_s)$, we have that 
    \begin{align} \label{fdp_t_mean_low_eq13}
        \int_{W}^\infty \mathbb{P}\{|G_t^{(s,i)}| \geq w\}\; \mathrm{d}w  \lesssim \int_{W}^\infty \exp\Big(-\frac{w}{b_s^tm^4r}\Big) \; \mathrm{d}w = b_s^tm^4r \exp\Big(-\frac{W}{b_s^tm^4r}\Big) \asymp \delta_s b_s^tm^4r. 
    \end{align}
    Putting \eqref{fdp_t_mean_low_eq12} and \eqref{fdp_t_mean_low_eq13} into \eqref{fdp_t_mean_low_eq4}, we have that 
    \begin{align*}
        &\tr\Big[\mathbb{E}\big[\mathbb{E}\{C_a(Z_t^{(s)}|M^{(t-1)})|M^{(t-1)}\}\big]\Big]\\
        \lesssim \;& b_s^t\epsilon_s\sqrt{m}\sqrt{\tr\Big[\mathbb{E}\big[\mathbb{E}\{C_a(Z_t^{(s)}|M^{(t-1)})|M^{(t-1)}\}\big]\Big]} + (b_s^t)^2m^4r\delta_s\log(1/\delta_s)+ (b_s^t)^2\delta_s m^4r\\
        \lesssim \;& b_s^t\epsilon_s\sqrt{m}\sqrt{\tr\Big[\mathbb{E}\big[\mathbb{E}\{C_a(Z_t^{(s)}|M^{(t-1)})|M^{(t-1)}\}\big]\Big]} + (b_s^t)^2m^4r\delta_s\log(1/\delta_s).
    \end{align*}
    Next, we show that $\tr[\mathbb{E}[\mathbb{E}\{C_a(Z_t^{(s)}|M^{(t-1)})|M^{(t-1)}\}]] \lesssim (b_s^t)^2m\epsilon_s^2$. If this does not hold, then when $\delta_s$ is small enough such that $\delta_s\log(1/\delta_s) \lesssim \epsilon_s^2m^{-3}r^{-1}$, we have that 
    \begin{align}\label{fdp_t_mean_low_eq5}
        \sqrt{\tr\Big[\mathbb{E}\big[\mathbb{E}\{C_a(Z_t^{(s)}|M^{(t-1)})|M^{(t-1)}\}\big]\Big]} \lesssim b_s^t\epsilon_s\sqrt{m} + \frac{(b_s^t)^2m^4r}{b_s^t\sqrt{m}\epsilon_s}\delta_s\log(1/\delta_s) \lesssim b_s^t\epsilon_s\sqrt{m}.
    \end{align}
    Hence, we can conclude that $\tr[\mathbb{E}[\mathbb{E}\{C_a(Z_t^{(s)}|M^{(t-1)})|M^{(t-1)}\}]] \lesssim (b_s^t)^2m\epsilon_s^2$. 
    
    \noindent \textbf{Case 1 - Step 3: Another upper bound on} $\tr[\mathbb{E}[\mathbb{E}\{C_a(Z_t^{(s)}|M^{(t-1)})|M^{(t-1)}\}]]$. To get another upper bound on $\tr[\mathbb{E}[\mathbb{E}\{C_a(Z_t^{(s)}|M^{(t-1)})|M^{(t-1)}\}]]$, we will follow a similar argument used in the proof of Theorem 4.1 in \citet{cai2024optimal}. By standard matrix algebra, we have that 
    \begin{align*}
        \tr(C_{a,t}^{(s,i)}) = \sum_{\ell=1}^r \Lambda_\ell(C_{a,t}^{(s,i)}) \leq r\Lambda_{\max}(C_{a,t}^{(s,i)}) \lesssim rm,
    \end{align*}
    where for any $\ell \in [r]$, $\Lambda_\ell(C_{a,t}^{(s,i)})$ denote the $\ell$th eigenvalue of $C_{a,t}^{(s,i)}$. Using the property of conditional expectations we have that for random vectors $U, V$ and $W$, $\mathbb{E}[\cov\{\mathbb{E}(U|V, W)\}]\\ \preceq \mathbb{E}[\cov(U|W)]$ where $\preceq$ denotes the positive definite inequality. Taking $W = M^{(t-1)}$, $U =    S_a(L^{(s)}_t)$ and $V = Z^{(s)}$, this then implies that $\mathbb{E}[\mathbb{E}\{C_a(Z_t^{(s)}|M^{(t-1)})|M^{(t-1)}\}] \preceq C_{a,t}^{(s)}$ using the fact that $\mathbb{E}\{S_a(L^{(s)}_t)\} = 0$. Hence, we have that 
    \begin{align}\label{fdp_t_mean_low_eq6}
       \tr\Big[\mathbb{E}\big[\mathbb{E}\{C_a(Z_t^{(s)}|M^{(t-1)})|M^{(t-1)}\}\big]\Big] \leq \tr(C_{a,t}^{(s)}) = \sum_{i=1}^{b_s^t} \tr(C_{a,t}^{(s,i)}) \lesssim b_s^tmr.
    \end{align}
    Hence, substituting the results in \eqref{fdp_t_mean_low_eq5} and \eqref{fdp_t_mean_low_eq6} into \eqref{fdp_t_mean_low_eq7}, it holds that
    \begin{align*}
        \sup_{a \in \Theta(r, C_\alpha)} \mathbb{E}\|\widetilde{a} - a\|_2^2
        & \geq \frac{r^2}{\sum_{s=1}^S \sum_{t=1}^T \{(b_s^t)^2m\epsilon_s^2 \wedge rb_s^tm\} + r^{2\alpha+2}}.
    \end{align*}

     \noindent \textbf{Case 2.} To derive the term of $1/(\sum_{s=1}^S\sum_{t=1}^T b_s^t \wedge (b_s^t)^2\epsilon_s^2)$, we follow a similar and even simpler argument as above. By considering the random function $\mu^*+U$ as a constant random function, we are back to the problem of univariate mean estimation over independent samples, i.e.~we will consider the model
    \begin{equation*}
        Y^{(s,i)} = \mu^* + \xi^{(s,i)}, \quad \text{where}\quad \xi^{(s,i)}\sim N(0,\sigma_0^2).
    \end{equation*}
    Pick the prior distribution of $\mu$ as the truncated standard Gaussian distribution over $[-1,1]$ and construct $S_a (L_t^{(s,i)}) = \sigma_0^{-2}(Y_{i,t}-\mu^*_{i,t})$, we can get the desired result following a similar argument as the one that derives \eqref{fdp_t_mean_low_eq5} and \eqref{fdp_t_mean_low_eq6} as long as $\delta_s\log(1/\delta_s) \lesssim (b_s^t)^{1/2}\epsilon_s^2$ for any $s\in [S]$.

\end{proof}

\subsection{Auxiliary results}
\begin{lemma} \label{lemma_fdp_thm_up}
    Consider the events of interest in the proof of Theorem \ref{fdp_thm_mean_up}:
    \begin{align*}
        \mathcal{E}'_1 = \Bigg\{&\Lambda_{\min}\Big\{\sum_{s=1}^S\sum_{i=1}^{b_s} \sum_{j=1}^m\frac{\nu_s}{b_sm} \Phi_r(X^{(s,\tau_{s,t}+i)}_j)\Phi_r^{\top}(X^{(s,\tau_{s,t}+i)}_j)\Big\} \geq 1/(2L)\quad \text{and} \\
        &\Lambda_{\max}\Big\{\sum_{s=1}^S\sum_{i=1}^{b_s} \sum_{j=1}^m\frac{\nu_s}{b_sm} \Phi_r(X^{(s,\tau_{s,t}+i)}_j)\Phi_r^{\top}(X^{(s,\tau_{s,t}+i)}_j)\Big\} \leq 2L, \forall t \in \{0\} \cup [T-1]\Bigg\},
    \end{align*}
    \begin{align*}
        \mathcal{E}'_2 = \Big\{\Pi^{\mathrm{entry}}_{R}\Big[&\frac{1}{m}\sum_{j=1}^m \Phi_r(X^{(s,\tau_{s,t}+i)}_j)\big\{\Phi^{\top}_r(X^{(s,\tau_{s,t}+i)}_j)a^{t}-Y^{(s,\tau_{s,t}+i)}_j \big\}\Big]\\
        &= \frac{1}{m}\sum_{j=1}^m \Phi_r(X^{(s,\tau_{s,t}+i)}_j)\big\{\Phi^{\top}_r(X^{(s,\tau_{s,t}+i)}_j)a^{t}-Y^{(s,\tau_{s,t}+i)}_j \big\},\\
        & \hspace{5cm}\; \forall i \in [b_s],\;  t \in \{0\} \cup [T-1], \; s\in [S]\Big\}.
    \end{align*}
    Suppose that 
    \begin{align*}
        r\log^2(Tr/\eta) \lesssim \Big(\sum_{s=1}^S \frac{\nu_s^2}{b_sm}\Big)^{-1},  \quad r\log(Tr/\eta) \lesssim \Big(\sup_{s\in[S]} \frac{\nu_s}{b_sm}\Big)^{-1},
    \end{align*}
    we have that 
    \begin{equation*}
        \mathbb{P}(\mathcal{E}_1 \cap \mathcal{E}_2) \geq 1-6\eta,
    \end{equation*}
    for a small enough $\eta \in (0,1/6)$.
\end{lemma}

\begin{proof}
    Note that $\mathcal{E}'_1$ and $\mathcal{E}'_2$ are the counter parts of $\mathcal{E}_1$ and $\mathcal{E}_2$ considered in Lemma~\ref{l_mean_upper_event}. The proofs are a generalized version of the proof of Lemma \ref{l_mean_upper_event} and most of the arguments in the proof of Lemma~\ref{l_mean_upper_event} could be adapted here.

     To control $\mathcal{E}'_1$, for any $i \in [b_s], j\in [m]$ and $s\in [S]$, denote 
    \begin{align*}
        T_{s,i,j} = \frac{\nu_s}{b_sm} \Phi_r(X^{(s,\tau_{s,t}+i)}_j)\Phi_r^{\top}(X^{(s,\tau_{s,t}+i)}_j).
    \end{align*}
    We firstly bound the largest and the smallest eigenvalues of $\mathbb{E}(T_{s,i,j})$. Note that under Assumption \ref{a_sample}, we have that 
    \begin{align*}
        \mathbb{E}(T_{s,i,j})& = \frac{\nu_s}{b_sm}\mathbb{E}\Big\{\Phi_r(X^{(s,\tau_{s,t}+i)}_j)\Phi_r^{\top}(X^{(s,\tau_{s,t}+i)}_j)\Big\}\\
        &= \frac{\nu_s}{b_sm}\mathbb{E}\Big\{\Phi_r(X^{(s,\tau_{s,t}+1)}_1)\Phi_r^{\top}(X^{(s,\tau_{s,t}+1)}_1)\Big\} =  \mathbb{E}(T_{s,1,1}).
    \end{align*}
    For any $s\in [S]$, since $\mathbb{E}(T_{s,1,1})$ is symmetric and positive semi-definite, we can upper bound its largest eigenvalue by
    \begin{align} \notag
        \Lambda_{\max}\{\mathbb{E}(T_{s,1,1})\} &= \sup_{u\in \mathbb{R}^r: \|u\|_2 =1} u^{\top}\mathbb{E}(T_{s,1,1})u = \frac{\nu_s}{b_sm}u^\top \mathbb{E}\Big\{\Phi_r(X^{(s,\tau_{s,t}+1)}_1)\Phi_r^{\top}(X^{(s,\tau_{s,t}+1)}_1)\Big\}u\\ \label{l_fdp_mean_upper_event_eq1}
        & \leq \frac{L\nu_s}{b_sm}\sup_{u\in \mathbb{R}^r: \|u\|_2 =1}\|u\|_2^2 = \frac{L\nu_s}{b_sm},
    \end{align}
    where the inequality follows from \eqref{l_mean_upper_event_eq2}. Similarly, it holds that
    \begin{align} \label{l_fdp_mean_upper_event_eq2}
        \Lambda_{\min}\{\mathbb{E}(T_{s,1,1})\} = \inf_{u\in \mathbb{R}^r: \|u\|_2 =1}u^{\top}\mathbb{E}(T_{s,1,1})u \geq \frac{\nu_s}{Lb_sm}, 
    \end{align}
    where the inequality follows from \eqref{l_mean_upper_event_eq3}. In the next step, we find an upper bound on the operator norm of $T_{s,i,j} - \mathbb{E}(T_{s,i,j})$. Observe that
    \begin{align*}
        &\sup_{s\in [S], i\in [b_s], j \in[m]}\|T_{s,i,j} - \mathbb{E}(T_{s,i,j})\|_{\op} \leq \sup_{s\in [S], i\in [b_s], j \in[m]}\|T_{s,i,j}\|_{\op} + \sup_{s\in [S]}\|\mathbb{E}(T_{s,1,1})\|_{\op}\\
        \leq \;& \sup_{s\in [S], i\in [b_s], j \in[m]} \sup_{u \in \mathbb{R}^r:\|u\|_2=1} u^\top T_{s,i,j}u + \frac{L\nu_s}{b_sm} \\
        =\;& \sup_{s\in [S], i\in [b_s], j \in[m]} \sup_{u \in \mathbb{R}^r: \|u\|_2=1} \frac{\nu_s}{b_sm} \Big\{\sum_{\ell=1}^r u_\ell\phi_{\ell}(X^{(s,i)}_j)\Big\}^2 + \frac{L\nu_s}{b_sm}\\
        \leq \; & \sup_{s\in [S], i\in [b_s], j \in[m]} \sup_{u \in \mathbb{R}^r:\|u\|_2=1} \frac{\nu_s}{b_sm}\Big(\sum_{\ell=1}^r u_{\ell}^2\Big)\Big\{\sum_{\ell=1}^r \phi^2_{\ell}(X^{(s,i)}_j)\Big\} + \frac{L\nu_s}{b_sm} \lesssim \sup_{s\in [S]} \frac{r\nu_s}{b_sm},
    \end{align*}
    where the third inequality follows from the Cauchy--Schwarz inequality and the last inequality follows from the boundedness property of Fourier basis. Next, we find an upper bound of the matrix variance statistic of the sum we are interested in. We have that
    \begin{align*}
        &\Big\|\sum_{s=1}^S\sum_{i=1}^{b_s} \sum_{j=1}^m\mathbb{E}\Big[\big\{T_{s,i,j} - \mathbb{E}(T_{s,i,j})\big\}\big\{T_{s,i,j} - \mathbb{E}(T_{s,i,j})\big\}^\top\Big]\Big\|_{\op}\\
        =\;& \Big\|\sum_{s=1}^S\sum_{i=1}^{b_s} \sum_{j=1}^m\mathbb{E}\Big[\big\{T_{s,i,j} - \mathbb{E}(T_{s,i,j})\big\}^\top\big\{T_{s,i,j} - \mathbb{E}(T_{s,i,j})\big\}\Big]\Big\|_{\op}\\
        = \;& \sup_{v \in \mathbb{R}^r: \|v\|_2=1} \sum_{s=1}^S\sum_{i=1}^{b_s} \sum_{j=1}^m v^\top\mathbb{E}\Big[\big\{T_{s,i,j} - \mathbb{E}(T_{s,i,j})\big\}\big\{T_{s,i,j} - \mathbb{E}(T_{s,i,j})\big\}^\top\Big]v\\
        \leq\;& \sup_{v \in \mathbb{R}^r: \|v\|_2=1} \left[\sum_{s=1}^S \left\{b_sm v^\top \mathbb{E}\{T_{s,1,1}T_{s,1,1}^\top\}v + b_smv^\top\mathbb{E}(T_{s,1,1})\mathbb{E}(T_{s,1,1})^\top v\right\}\right]\\
        = \;& \sup_{v \in \mathbb{R}^r: \|v\|_2=1} \bigg\{\sum_{s=1}^S \bigg[\frac{\nu_s^2}{b_sm}\mathbb{E}\Big\{v^\top\Phi_r(X^{(s,\tau_{s,t}+1)}_1)\Phi_r^{\top}(X^{(s,\tau_{s,t}+1)}_1)\Phi_r(X^{(s,\tau_{s,t}+1)}_1)\Phi_r^{\top}(X^{(s,\tau_{s,t}+1)}_1)v\Big\} \\
        & \hspace{1cm}+ b_sm\|v^\top\mathbb{E}(T_{s,1,1})\|_2^2\bigg\}\bigg]\\
        = \; & \sup_{v \in \mathbb{R}^r: \|v\|_2=1} \bigg[\sum_{s=1}^S \bigg\{\frac{\nu_s^2}{b_sm}\mathbb{E}\Big[v^\top\Phi_r(X^{(s,\tau_{s,t}+1)}_1)\Big\{\sum_{\ell=1}^r\phi^2_\ell(X^{(s,\tau_{s,t}+1)}_1)\Big\}\Phi_r^{\top}(X^{(s,\tau_{s,t}+1)}_1)v\Big] \\
        & \hspace{1cm} + b_sm\|v^\top\mathbb{E}(T_{s,1,1})\|_2^2\bigg\}\bigg]\\
        \leq \;& \sup_{v \in \mathbb{R}^r: \|v\|_2=1}\left\{ \sum_{s=1}^S \left[\frac{r\nu_s^2}{b_sm}\mathbb{E}\Big\{v^\top\Phi_r(X^{(s,\tau_{s,t}+1)}_1)\Phi_r^{\top}(X^{(s,\tau_{s,t}+1)}_1)v\Big\} + b_sm\|v\|_2^2\Lambda^2_{\max}\{\mathbb{E}(T_{s,1,1})\} \right]\right\} \\
        \lesssim \; &\sum_{s=1}^S\frac{r\nu_s^2}{b_sm},
    \end{align*}
    where the first inequality follows from the independence across $i\in[b_s]$ and $j \in [m]$, the second inequality follows from the boundedness property of Fourier basis and the last inequality follows from \eqref{l_mean_upper_event_eq2} and \eqref{l_fdp_mean_upper_event_eq1}. Hence the matrix Bernstein's inequality \citep[e.g.~Theorem 6.1.1 in][]{tropp2015introduction} implies that for any $\tau>0$,
    \begin{align*}
        &\mathbb{P}\Big[\Big\|\sum_{s=1}^S\sum_{i=1}^{b_s} \sum_{j=1}^m \big\{T_{s,i,j} - \mathbb{E}(T_{s,i,j})\big\}\Big\|_{\op} \geq \tau\Big] \\
        & \hspace{3cm}\leq 2r\exp\Big\{-\frac{-\tau^2/2}{\sum_{s=1}^S r\nu_s^2/(b_sm)+ \tau\sup_{s\in[S]}C_1r\nu_s/(3b_sm)}\Big\}.
    \end{align*}
    Note that by \eqref{l_fdp_mean_upper_event_eq1} and \eqref{l_fdp_mean_upper_event_eq2}, it holds that
    \begin{align*}
        \Lambda_{\max}\Big\{\sum_{s=1}^S\sum_{i=1}^{b_s} \sum_{j=1}^m \mathbb{E}(T_{s,i,j})\Big\} \leq \sum_{s=1}^S\sum_{i=1}^{b_s} \sum_{j=1}^m \frac{L\nu_s}{b_sm} = L,
    \end{align*}
    and 
    \begin{align*}
        \Lambda_{\min}\Big\{\sum_{s=1}^S\sum_{i=1}^{b_s} \sum_{j=1}^m \mathbb{E}(T_{s,i,j})\Big\} \geq \sum_{s=1}^S\sum_{i=1}^{b_s} \sum_{j=1}^m \frac{\nu_s}{Lb_sm} = \frac{1}{L}.
    \end{align*}
    Hence, applying a union bound argument on $t\in \{0\}\cup [T-1]$ and choosing 
    $$\tau = \log(Tr/\eta_1)\Big[\Big\{\sum_{s=1}^S \frac{r\nu_s^2}{b_sm}\Big\}^{1/2}+\sup_{s\in[S]} \frac{r\nu_s}{b_sm}\Big] \leq \frac{1}{2L},$$
    we have that $\mathbb{P}(\mathcal{E}'_1) \geq 1-\eta_1$ by further assuming 
    \begin{align}\label{l_fdp_mean_upper_event_eq3}
        r\log^2(Tr/\eta_1) \lesssim \Big(\sum_{s=1}^S \frac{\nu_s^2}{b_sm}\Big)^{-1}, \quad r\log(Tr/\eta_1) \lesssim \Big(\sup_{s\in[S]} \frac{\nu_s}{b_sm}\Big)^{-1}
    \end{align}
    and applying the Weyl's inequality. 
    
    For $\mathcal{E}'_2$, the same argument used to control $\mathcal{E}_2$ in Lemma \ref{l_mean_upper_event} still works. The only adjustment is to replace $n$ with $N$ in the argument to accommodate the effect of a union bound argument over $N$ random variables.
\end{proof}

\begin{remark} \label{remark_fdp_mean_r}
    Note that the optimal choice of $r$ in Theorem \ref{fdp_thm_mean_up} is the solution to the equation 
    \begin{equation}\label{l_fdp_mean_upper_event_eq4}
        r \asymp \Big\{\sum_{s=1}^S  (r^2n_s)\wedge (r^2n_s^2\epsilon_s^2) \wedge (rn_sm) \wedge (n_s^2m\epsilon_s^2)\Big\}^{1/(2\alpha+2)}.
    \end{equation}
    In this remark, all the calculations are up to a poly-logarithmic factor. Denote $s^* = \argmax_{s\in[S]} \nu_s/(b_sm)$. The first assumption is automatically satisfied by our choice of $r$ in equation \eqref{l_fdp_mean_upper_eq7}. For the second assumption, in the homogeneous setting, the assumption reduces to $r \lesssim Snm$, which trivially holds. In the heterogeneous setting, we discuss it in four different cases. Firstly, when when $rn_{s^*}m$ is the smallest term
    , with the choice of 
    \begin{align*}
        \nu_s = \frac{u_s}{\sum_{s=1}^S u_s} \quad \text{where} \quad u_s \asymp_{\log} (r^2n_s) \wedge (r^2n_s^2\epsilon_s^2) \wedge (rn_sm) \wedge (n_s^2m\epsilon_s^2),
    \end{align*}
    the second assumption in \eqref{l_fdp_mean_upper_event_eq3} is equivalent to say
    \begin{align*}
        r &\lesssim  \frac{\sum_{s=1}^S (r^2n_s) \wedge (r^2n_s^2\epsilon_s^2) \wedge (rn_sm) \wedge (n_s^2m\epsilon_s^2)}{r} \asymp \frac{r^{2\alpha+2}}{r}.
    \end{align*}
    and the assumption is satisfied with our choice of $r$ in \eqref{l_fdp_mean_upper_event_eq4} for any $\alpha >0$. In the second case, when $n_{s^*}^2m\epsilon_s^2$ is the smallest term, we have that \eqref{l_fdp_mean_upper_event_eq3} is equivalent to
    \begin{align*}
        r \lesssim \frac{\Big\{\sum_{s=1}^S (r^2n_s) \wedge (r^2n_s^2\epsilon_s^2) \wedge (rn_sm) \wedge (n_s^2m\epsilon_s^2)\Big\}b_{s^*}m}{n_{s^*}^2m\epsilon_s^2} \asymp \frac{r^{2\alpha+2}}{n_{s^*}\epsilon_s^2},
    \end{align*}
    which is equivalent to selecting $r$ such that $r \gtrsim (n_{s^*}\epsilon_s^2)^{1/(2\alpha+1)}$. In the third case, when $r^2n_{s^*}$ is the smallest term, the assumption is equivalent to say 
    \begin{align*}
        r \lesssim \frac{m\sum_{s=1}^S (r^2n_s) \wedge (r^2n_s^2\epsilon_s^2) \wedge (rn_sm) \wedge (n_s^2m\epsilon_s^2)\Big\}}{r^2}\asymp \frac{r^{2\alpha+2}m}{r^2},
    \end{align*}
    which is satisfied whenever $\alpha \geq 1/2$. Finally, when $r^2n_{s^*}^2\epsilon_s^2$ is the smallest term (the case when $r^2 \lesssim m$), we have that \eqref{l_fdp_mean_upper_event_eq3} is equivalent to
    \begin{align*}
        r \lesssim \frac{\Big\{\sum_{s=1}^S (r^2n_s) \wedge (r^2n_s^2\epsilon_s^2) \wedge (rn_sm) \wedge (n_s^2m\epsilon_s^2)\Big\}b_{s^*}m}{r^2n_{s^*}^2\epsilon_s^2} \asymp \frac{r^{2\alpha+2}m}{r^2n_{s^*}\epsilon_s^2}.
    \end{align*}
    If we select $r$ such that $r \gtrsim (n_{s^*}\epsilon_s^2)^{1/(2\alpha+1)}$, then we have that $r^{2\alpha-1}\gtrsim r^{-2}n_{s^*}\epsilon_s^2 \gtrsim n_{s^*}m^{-1}\epsilon_s^2$ and the assumption holds. Combine the above together, the overall assumption on $r$ is that $r \gtrsim (n_{s^*}\epsilon_s^2)^{1/(2\alpha+1)}$.
    \end{remark}

\section{Technical details for Section \ref{section_vcm_fdp}} \label{section_appendix_vcm_fdp}
\subsection{Proof of Theorem \ref{fdp_thm_vcm_up}} \label{section_appendix_vcm_fdp_up}
\begin{proof}[Proof of Theorem \ref{fdp_thm_vcm_up}]
    The first part of the theorem follows from a similar argument as the one used in the proof of the first part of Theorem \ref{thm_mean_upper}. For every server $s \in [S]$, we ensure in each of the $T$ iteration, the output satisfies $(\epsilon_s,\delta_s)$-CDP, thus we have proved that Algorithm \ref{algorithm_fdp_vcm} is $(\bm{\epsilon}, \bm{\delta},T)$-FDP. 
    
    The second claim is a consequence of Proposition \ref{fdp_prop_vcm_up}. Selecting 
    \begin{align*}
        \nu_s = \frac{u_s}{\sum_{s=1}^S u_s} \quad \mathrm{where}\quad u_s\asymp_{\log} (dr^2n_s) \wedge (drn_sm) \wedge (r^2n_s^2\epsilon_s^2) \wedge (n_s^2m\epsilon_s^2),
    \end{align*}
    and $T\asymp \log(\sum_{s=1}^S n_s)$,  it then holds that 
    \begin{align*}
        \|B^T - B_r^*\|_2^2 =_{\log} O_p\Bigg\{\frac{d^2r^2}{\sum_{s=1}^S(dr^2n_s) \wedge (drn_sm) \wedge (r^2n_s^2\epsilon_s^2) \wedge (n_s^2m\epsilon_s^2)}+ dr^{-2\alpha}\Bigg\}.
    \end{align*}
    Hence, we have that 
    \begin{align*}
        \|\widetilde{\bm{\beta}}- \bm{\beta}^*\|_{L^2}^2 &\lesssim \|B^T - B_r^*\|_2^2 + \sum_{k=0}^d\|\Phi_r^\top b_k^*- \beta_k^*\|_{L^2}^2 \lesssim \|B^T - B_r^*\|_2^2 + dr^{-2\alpha},
    \end{align*}
    where the last inequality follows from the fact that $\beta_k^* \in \mathcal{W}(\alpha, C_{\alpha})$ for any $k \in \{0\}\cup [d]$.

    The third claim follows from a direct calculation of the second claim.
\end{proof}

\begin{proposition} \label{fdp_prop_vcm_up}
    Initialising Algorithm \ref{algorithm_fdp_vcm} with $B^0 = 0$ and step size $\rho = 4C_{\lambda}L/(1+4C_{\lambda}^2L^2)$.
    Suppose that
    \begin{align*}
        r\log^2(Tr/\eta)\lesssim \Big(\sum_{s=1}^S \frac{T\nu_s^2d}{n_sm}\Big)^{-1}, \log^2(Tr/\eta)\sum_{s=1}^S \frac{T\nu_s^2d}{n_s} \lesssim 1 ,\; r\log(Tr/\eta)\lesssim \Big(\sup_{s\in [S] }\frac{T\nu_s d}{n_s}\Big)^{-1},
    \end{align*}
    then under the same condition as the one in Theorem \ref{fdp_thm_vcm_up}, it holds with probability at least $1-7\eta$ that
    \begin{align*}
        \|\widetilde{\bm{\beta}}- \bm{\beta}^*\|_{L^2}^2 \lesssim &d\exp\Big(-\frac{4T}{1+4C_{\lambda}^2L^2}\Big)+ \frac{1}{\eta}\Big(\sum_{s=1}^S \frac{Td\nu_s^2}{n_s}+ \sum_{s=1}^S\frac{Tdr\nu_s^2}{n_sm} +dr^{-2\alpha}\Big) \\
        & \hspace{2cm} + \log(T/\eta)\log(1/\delta_s)\sum_{s=1}^S\Big(\frac{T^2
        d^2\nu_s^2}{n_s^2\epsilon_s^2} + \frac{T^2d^2r^2\nu_s^2\log(N/\eta)}{n_s^2m\epsilon_s^2}\Big),
    \end{align*}
    for a small enough $\eta \in (0,1/7)$.
\end{proposition}
\begin{proof}[Proof of Proposition \ref{fdp_prop_vcm_up}]
    The proof of the first claim follows from a similar argument as the one used in the proof of the first claim of Theorem \ref{fdp_thm_mean_up}. The third claim follows directly from the second claim. It therefore suffices to show the second claim.

    Denote $\tau_{s,t} = b_st$ for $t \in \{0\}\cup [T-1]$ and consider the following two events:
    \begin{align*}
        \mathcal{E}_3 = \Bigg\{&\Lambda_{\min}\Big\{\sum_{s=1}^S\sum_{i=1}^{b_s}\sum_{j=1}^m\frac{\nu_s}{b_sm}\widetilde{\Phi}_r(X_{j}^{(s,\tau_{s,t}+i)})\bm{G}^{(s,\tau_{s,t}+i)}\bm{G}^{(s,\tau_{s,t}+i)\top}\widetilde{\Phi}_r^\top(X_{j}^{(s,\tau_{s,t}+i)})\Big\}\\
        & \hspace{9cm}\geq 1/(2LC_\lambda) \\
        &\text{and} \; \Lambda_{\max}\Big\{\sum_{s=1}^S\sum_{i=1}^{b_s}\sum_{j=1}^m\frac{\nu_s}{b_sm} \widetilde{\Phi}_r(X_{j}^{(s,\tau_{s,t}+i)})\bm{G}^{(s,\tau_{s,t}+i)}\bm{G}^{(s,\tau_{s,t}+i)\top}\widetilde{\Phi}_r^\top(X_{j}^{(s,\tau_{s,t}+i)})\Big\} \\
        & \hspace{9cm} \leq 2LC_\lambda, \; \forall t\in \{0\}\cup [T-1]\Bigg\},
    \end{align*}
    and
    \begin{align*}
        \mathcal{E}_4 = \Bigg\{&\Pi^{\mathrm{entry}}_{R}\Big[\frac{1}{m}\sum_{j=1}^m \widetilde{\Phi}_r(X_{j}^{(s,\tau_{s,t}+i)})\bm{G}^{(s,\tau_{s,t}+ i)}\Big\{\bm{G}^{(s,\tau_{s,t}+ i)\top}\widetilde{\Phi}_r^\top(X_{j}^{(s,\tau_{s,t}+i)})B^t-Y_{j}^{(s,\tau_{s,t}+i)} \Big\}\Big] \Big\}\\
        & = \frac{1}{m}\sum_{j=1}^m \widetilde{\Phi}_r(X_{j}^{(s,\tau_{s,t}+i)})\bm{G}^{(s,\tau_{s,t}+ i)}\Big\{\bm{G}^{(s,\tau_{s,t}+ i)\top}\widetilde{\Phi}_r^\top(X_{j}^{(s,\tau_{s,t}+i)})B^t-Y_{j}^{(s,\tau_{s,t}+i)} \Big\}, \\
        & \forall i \in [b_s],\;  t \in \{0\}\cup [T-1], \; s\in [S]\Bigg\}.
    \end{align*}
    We control the probability of these events happening in Lemma \ref{fdp_lemma_vcm_up}. The remainder of the proof is conditional on both of these events happening.

    For the $t$th iteration, we can rewrite the noisy gradient descent as
    \begin{align*}
        &B^{t} - \rho \Big[\sum_{s=1}^S\sum_{i=1}^{b_s}\sum_{j=1}^m\frac{\nu_s}{b_sm} \widetilde{\Phi}_r(X_{j}^{(s,\tau_{s,t}+i)})\bm{G}^{(s,\tau_{s,t}+ i)}\Big\{\bm{G}^{(s,\tau_{s,t}+i)\top}\widetilde{\Phi}_r^\top(X_{j}^{(s,\tau_{s,t}+i)})B^t\\
        & \hspace{3cm} -Y_{j}^{(s,\tau_{s,t}+i)} \Big\}
         - \rho \sum_{s=1}^S \nu_s w_{s,t}\\
        =\;&  B^{t} - \Big[\sum_{s=1}^S\sum_{i=1}^{b_s}\sum_{j=1}^m \frac{\rho\nu_s}{b_sm}\widetilde{\Phi}_r(X_{j}^{(s,\tau_{s,t}+i)})\bm{G}^{(s,\tau_{s,t}+ i)}\Big[\bm{G}^{(s,\tau_{s,t}+i)\top}\widetilde{\Phi}_r^\top(X_{j}^{(s,\tau_{s,t}+i)})(B^t-B^*_r)\\
        & \hspace{3.5cm}-\bm{G}^{(s,\tau_{s,t}+i)\top}\Big\{\bm{\beta}^*(X_{j}^{(s,\tau_{s,t}+i)})-\widetilde{\Phi}_r^\top(X_{j}^{(s,\tau_{s,t}+i)})B^*_r\Big\}-\xi_{s,\tau_{s,t}+i,j}\Big]\Big]\\
        & \hspace{0.5cm} - \rho \sum_{s=1}^S \nu_s w_{s,t},
    \end{align*}
    where the equality inequality follows from \eqref{fdp_vcm_obs}. By the definition of the projection $\Pi^*_{\mathcal{B}}$ defined in Remark \ref{remark_projection_vcm} and the fact that $B_r^* \in \mathcal{B}$, we have from Lemma \ref{lemma_projection} that
    \begin{align} \notag
         &\|B^{t+1} - B_r^*\|_2^2\\ \notag
         \leq\; & \Big\|\Big\{I-\sum_{s=1}^S\sum_{i=1}^{b_s}\sum_{j=1}^m\frac{\rho\nu_s}{b_sm}\widetilde{\Phi}_r(X_{j}^{(s,\tau_{s,t}+i)})\bm{G}^{(s,\tau_{s,t}+i)}\bm{G}^{(s,\tau_{s,t}+i)\top}\widetilde{\Phi}_r^\top(X_{j}^{(s,\tau_{s,t}+i)})\Big\}(B^t-B^*_r)\Big\|_2^2\\ \notag
         & + \Big\|\sum_{s=1}^S\sum_{i=1}^{b_s}\sum_{j=1}^m \frac{\rho\nu_s}{b_sm}\widetilde{\Phi}_r(X_{j}^{(s,\tau_{s,t}+i)})\bm{G}^{(s,\tau_{s,t}+ i)}\bm{G}^{(s,\tau_{s,t}+i)\top}\\ \notag
         & \hspace{6cm} \Big\{\bm{\beta}^*(X_{j}^{(s,\tau_{s,t}+i)})-\widetilde{\Phi}_r^\top(X_{j}^{(s,\tau_{s,t}+i)})B^*_r\Big\}\Big\|_2^2\\ \notag
         & + \Big\|\sum_{s=1}^S\sum_{i=1}^{b_s}\sum_{j=1}^m\frac{\rho\nu_s}{b_sm}\widetilde{\Phi}_r(X_{j}^{(s,\tau_{s,t}+i)})\bm{G}^{(s,\tau_{s,t}+ i)}\xi_{s,\tau_{s,t}+i,j}\Big\|_2^2 + \Big\|\sum_{s=1}^S\rho\nu_s w_{s,t}\Big\|_2^2\\ \label{t_fdp_vcm_upper_eq7}
         =\;&\|(I)\|_2^2+\|(II)\|_2^2+\|(III)\|_2^2+\rho^2\Big\|\sum_{s=1}^S\nu_s w_{s,t}\Big\|_2^2,
    \end{align}
    where the last inequality follows from the triangle inequality. We will control the estimation error $\|B^{t+1} - B_r^*\|_2^2$ by controlling each term on the right hand side individually. For $(I)$, note that in the event $\mathcal{E}_3$, it holds that
    \begin{align*}
        &\Big\|I-\sum_{s=1}^S\sum_{i=1}^{b_s}\sum_{j=1}^m\frac{\rho\nu_s}{b_sm}\widetilde{\Phi}_r(X_{j}^{(s,\tau_{s,t}+i)})\bm{G}^{(s,\tau_{s,t}+i)}(X_{j}^{(s,\tau_{s,t}+i)})\\
        & \hspace{5cm}\bm{G}^{(s,\tau_{s,t}+i)\top}(X_{j}^{(s,\tau_{s,t}+i)})\widetilde{\Phi}_r^\top(X_{j}^{(s,\tau_{s,t}+i)})\Big\|_{\op}\\
        \leq \; &\max \Big\{|1-\frac{\rho}{2LC_{\lambda}}|, |1- 2LC_{\lambda}\rho|\Big\} = 1 - \frac{2}{1+4C_{\lambda}^2L^2},
    \end{align*}
    by the choice of $\rho = 4LC_{\lambda}/(1+4C_{\lambda}^2L^2)$. Therefore, we have that 
    \begin{align} \label{t_fdp_vcm_upper_eq1}
        \|(I)\|_2^2 \leq \Big(1 - \frac{2}{1+4C_{\lambda}^2L^2}\Big)^2\|B^t-B^*_r\|_2^2.
    \end{align}
    Note that here we choose $\rho = 4LC_{\lambda}/(1+4C_{\lambda}^2L^2)$ for simplicity of the proof and we remark that any choices of $\rho \in (0,(C_\lambda L)^{-1})$ will work. To control the second term, we will control  $ \|(II) - \mathbb{E}\{(II)\}\|_2^2$ and $ \|\mathbb{E}\{(II)\}\|_2^2$ individually. Note that to control the first term, we have that 
    \begin{align} \notag
        &\mathbb{E}\{\|(II) - \mathbb{E}\{(II)\}\|_2^2\}\\ \notag
        \lesssim\; &  \sum_{s=1}^S \frac{\rho^2\nu_s^2}{b_sm}\sum_{k=0}^d\sum_{\ell=1}^r \mathbb{E}\Big[\phi_{\ell}^2(X_{1}^{(s,\tau_{s,t}+1)})G_{k}^{(s,\tau_{s,t}+1)^2} \Big\{\bm{G}^{(s,\tau_{s,t}+1)\top}\\ \notag
        & \hspace{7cm}\Big\{\bm{\beta}^*(X_{1}^{(s,\tau_{s,t}+1)})-\widetilde{\Phi}_r^\top(X_{1}^{(s,\tau_{s,t}+1)})B^*_r\Big\}\Big\}^2\Big]\\ \notag
        & + \sum_{s=1}^S\frac{\rho^2\nu_s^2}{b_s}\sum_{k=0}^d\sum_{\ell=1}^r \mathbb{E}\Big[\phi_{\ell}(X_{1}^{(s,\tau_{s,t}+1)})G_{k}^{(s,\tau_{s,t}+1)}\bm{G}^{(s,\tau_{s,t}+1)\top}\\ \notag
        & \hspace{7cm} \Big\{\bm{\beta}^*(X_{1}^{(s,\tau_{s,t}+1)})-\widetilde{\Phi}_r^\top(X_{1}^{(s,\tau_{s,t}+1)})B^*_r\Big\}\\ \notag
        &\hspace{4cm}\phi_{\ell}(X_{2}^{(s,\tau_{s,t}+1)})G_{k}^{(s,\tau_{s,t}+1)}\bm{G}^{(s,\tau_{s,t}+1)\top}\\ \notag
        & \hspace{7cm}\Big\{\bm{\beta}^*(X_{2}^{(s,\tau_{s,t}+1)})-\widetilde{\Phi}_r^\top(X_{2}^{(s,\tau_{s,t}+1)})B^*_r\Big\} \Big]\\ \notag
        & + \sum_{s=1}^S\frac{\rho^2\nu_s^2}{b_s}\sum_{k=0}^d\sum_{\ell=1}^r  \mathbb{E}\Big[\phi_{\ell}(X_{1}^{(s,\tau_{s,t}+1)})G_{k}^{(s,\tau_{s,t}+1)}\bm{G}^{(s,\tau_{s,t}+1)\top}\\\notag
        & \hspace{7cm} \Big\{\bm{\beta}^*(X_{1}^{(s,\tau_{s,t}+1)})-\widetilde{\Phi}_r^\top(X_{1}^{(s, \tau_{s,t}+1)})B^*_r\Big\}\Big]\\ \notag
        &\hspace{3.3cm} \mathbb{E}\Big[\phi_{\ell}(X_{2}^{(s, \tau_{s,t}+1)})G_{k}^{(s,\tau_{s,t}+1)}\bm{G}^{(s,\tau_{s,t}+1)\top}\\ \notag
        & \hspace{7cm}\Big \{\bm{\beta}^*(X_{2}^{(s,\tau_{s,t}+1)})-\widetilde{\Phi}_r^\top(X_{2}^{(s,\tau_{s,t}+1)})B^*_r\Big\} \Big]\\ \notag
        =\;& (II)_1 + (II)_2 + (II)_3,
    \end{align}
    where the inequality follows from the independence across $i \in [b_s]$. For the term $(II)_1$, by selecting $r$ such that $dr^{-2\alpha} < 1$,
    we then have that
    \begin{align} \notag
        (II)_1 &\lesssim \sum_{s=1}^S \frac{\rho^2\nu_s^2}{b_sm}\sum_{k=0}^d\sum_{\ell=1}^r  \mathbb{E}\Big[\Big\{\bm{G}^{(s,\tau_{s,t}+1)\top}\Big\{\bm{\beta}^*(X_{1}^{(s,\tau_{s,t}+1)})-\widetilde{\Phi}_r^\top(X_{1}^{(s,\tau_{s,t}+1)})B^*_r\Big\}\Big\}^2\Big]\\ \notag
        & = \sum_{s=1}^S \frac{\rho^2\nu_s^2}{b_sm}\sum_{k=0}^d\sum_{\ell=1}^r \int_{0}^1 \Big\{\bm{\beta}^*(s)-\widetilde{\Phi}_r^\top(s)B^*_r\Big\}^\top \\ \notag
        & \hspace{3cm} \mathbb{E}\Big\{\bm{G}^{(s,\tau_{s,t}+1)}(s)\bm{G}^{(s,\tau_{s,t}+1)\top}(s)\Big\}\Big\{\bm{\beta}^*(s)-\widetilde{\Phi}_r^\top(s)B^*_r\Big\}f_X(s) \;\mathrm{d}s\\ \label{t_fdp_vcm_upper_eq13}
        & \lesssim \sum_{s=1}^S \frac{\rho^2\nu_s^2dr}{b_sm}\int_{0}^1 \Big\{\bm{\beta}^*(s)-\widetilde{\Phi}_r^\top(s)B^*_r\Big\}^\top \Big\{\bm{\beta}^*(s)-\widetilde{\Phi}_r^\top(s)B^*_r\Big\} \;\mathrm{d}s \\ \notag
        &\lesssim \sum_{s=1}^S \frac{\rho^2\nu_s^2d^2r^{1-2\alpha}}{b_sm} \lesssim \sum_{s=1}^S \frac{\rho^2\nu_s^2dr}{b_sm},
    \end{align}
    where the first inequality follows from Assumption \ref{simple_vcm_a_model}(b)~and properties of Fourier basis, the second inequality follows from Assumptions \ref{a_sample} and \ref{simple_vcm_a_model}(a)~and the third inequality follows from Assumption \ref{simple_vcm_a_model}(c).

    For the term $(II)_2$, we have that 
    \begin{align}\notag
        (II)_2 &= \sum_{s=1}^S\frac{\rho^2\nu_s^2}{b_s}\sum_{k=0}^d\mathbb{E}\Big[\sum_{\ell=1}^r \mathbb{E}\Big[\phi_{\ell}(X_{1}^{(s,\tau_{s,t}+1)})G_{k}^{(s,\tau_{s,t}+1)}\bm{G}^{(s,\tau_{s,t}+1)\top}\\ \notag
        & \hspace{5.5cm}\Big\{\bm{\beta}^*(X_{1}^{(s,\tau_{s,t}+1)})-\widetilde{\Phi}_r^\top(X_{1}^{(s,\tau_{s,t}+1)})B^*_r\Big\}\Big|\bm{G}\Big]^2\Big]\\ \notag
        & = \sum_{s=1}^S\frac{\rho^2\nu_s^2}{b_s}\sum_{k=0}^d \mathbb{E}\Big[\Big\|\mathbb{E}\Big\{\Phi_r(X_{1}^{(s,\tau_{s,t}+1)})G_{k}^{(s,\tau_{s,t}+1)}\bm{G}^{(s,\tau_{s,t}+1)\top}\\ \notag
        & \hspace{5.5cm}\Big\{\bm{\beta}^*(X_{1}^{(s,\tau_{s,t}+1)})-\widetilde{\Phi}_r^\top(X_{1}^{(s,\tau_{s,t}+1)})B^*_r\Big\} \Big|\bm{G}\Big\}\Big\|_2^2\Big]\\ \notag
        & \leq \sum_{s=1}^S\frac{\rho^2\nu_s^2}{b_s}\sum_{k=0}^d \mathbb{E}\Big[\Big\|G_{k}^{(s,\tau_{s,t}+1)}\bm{G}^{(s,\tau_{s,t}+1)\top}\Big\{\bm{\beta}^*-\widetilde{\Phi}_r^\top B^*_r\Big\}f_X\Big\|_{L^2}^2\Big] \\ \notag
        & \lesssim  \sum_{s=1}^S\frac{\rho^2\nu_s^2d}{b_s}\mathbb{E}\Big[\Big\|\bm{G}^{(s,\tau_{s,t}+1)\top}\Big\{\bm{\beta}^*-\widetilde{\Phi}_r^\top B^*_r\Big\}\Big\|_{L^2}^2\Big]\\ \label{t_fdp_vcm_upper_eq8}
        & \lesssim  \sum_{s=1}^S\frac{\rho^2\nu_s^2d^2r^{-2\alpha}}{b_s} \lesssim \sum_{s=1}^S\frac{\nu_s^2d}{b_s},
    \end{align}
    where the first equality follows from the fact that given $\bm{G}$, the random variables 
    $$\Big\{\phi_{\ell}(X_{j}^{(s,\tau_{s,t}+1)})G_{k}^{(s,\tau_{s,t}+1)}\bm{G}^{(s,\tau_{s,t}+1)\top}\Big\{\bm{\beta}^*(X_{j}^{(s,\tau_{s,t}+1)})-\widetilde{\Phi}_r^\top(X_{j}^{(s,\tau_{s,t}+1)})B^*_r\Big\}\Big\}_{j\in [2]}$$
    are independent, the first inequality follows from the projection to the Fourier basis, the second inequality follows from Assumptions \ref{a_sample} and \ref{simple_vcm_a_model}(b)~and the third inequality follows from Assumption \ref{simple_vcm_a_model}(c).

    For the term $(II)_3$, we have that 
    \begin{align} \notag
        (II)_3 &=\sum_{s=1}^S\frac{\rho^2\nu_s^2}{b_s}\sum_{k=0}^d\sum_{\ell=1}^r  \mathbb{E}\Big[\mathbb{E}\Big[\phi_{\ell}(X_{1}^{(s,\tau_{s,t}+1)})G_{k}^{(s,\tau_{s,t}+1)}\bm{G}^{(s,\tau_{s,t}+1)\top}\\ \notag
        & \hspace{5.5cm}\Big\{\bm{\beta}^*(X_{1}^{(s,\tau_{s,t}+1)})-\widetilde{\Phi}_r^\top(X_{1}^{(s, \tau_{s,t}+1)})B^*_r\Big\}\Big|\bm{G}\Big]\Big]\\ \notag
        &\hspace{3.3cm} \mathbb{E}\Big[\mathbb{E}\Big[\phi_{\ell}(X_{2}^{(s, \tau_{s,t}+1)})G_{k}^{(s,\tau_{s,t}+1)}\bm{G}^{(s,\tau_{s,t}+1)\top}\\ \notag
        & \hspace{5.5cm}\Big\{\bm{\beta}^*(X_{2}^{(s,\tau_{s,t}+1)})-\widetilde{\Phi}_r^\top(X_{2}^{(s,\tau_{s,t}+1)})B^*_r\Big\} \Big|\bm{G}\Big]\Big]\\ \notag
        & =\sum_{s=1}^S\frac{\rho^2\nu_s^2}{b_s}\sum_{k=0}^d\sum_{\ell=1}^r  \mathbb{E}\Big[\mathbb{E}\Big[\phi_{\ell}(X_{1}^{(s,\tau_{s,t}+1)})G_{k}^{(s,\tau_{s,t}+1)}\bm{G}^{(s,\tau_{s,t}+1)\top}\\ \notag
        & \hspace{5.5cm} \Big\{\bm{\beta}^*(X_{1}^{(s,\tau_{s,t}+1)})-\widetilde{\Phi}_r^\top(X_{1}^{(s, \tau_{s,t}+1)})B^*_r\Big\}\Big|\bm{G}\Big]\Big]^2\\ \notag
        & \leq \sum_{s=1}^S\frac{\rho^2\nu_s^2}{b_s}\sum_{k=0}^d\sum_{\ell=1}^r  \mathbb{E}\Big[\mathbb{E}\Big[\phi_{\ell}(X_{1}^{(s,\tau_{s,t}+1)})G_{k}^{(s,\tau_{s,t}+1)}\bm{G}^{(s,\tau_{s,t}+1)\top}\\ \notag
        & \hspace{5.5cm}\Big\{\bm{\beta}^*(X_{1}^{(s,\tau_{s,t}+1)})-\widetilde{\Phi}_r^\top(X_{1}^{(s, \tau_{s,t}+1)})B^*_r\Big\}\Big|\bm{G}\Big]^2\Big]\\ \label{t_fdp_vcm_upper_eq10}
        & \lesssim \sum_{s=1}^S\frac{\nu_s^2d}{b_s},
    \end{align}
    where the first inequality follows from the Cauchy--Schwarz inequality and the last inequality follows from a similar argument as the one used to derive \eqref{t_fdp_vcm_upper_eq8}.

    To control the term $\|\mathbb{E}\{(II)\}\|_2^2$, since $\mathbb{E}\{(II)\} \in \mathbb{R}^{r(d+1)}$, by definition of operator norm for vectors, it is equivalent to control $\|\mathbb{E}\{(II)\}\|_{\op}^2$. By the definition of operator norm, we have that 
    \begin{align}\notag
        &\|\mathbb{E}\{(II)\}\|_{\op} \\ \notag
        =\;&  \sup_{\|v\|_2 =1}  \mathbb{E}\Big[\sum_{s=1}^S \rho\nu_s v^{\top}\widetilde{\Phi}_r(X_{1}^{(s,\tau_{s,t}+1)})\bm{G}^{(s,\tau_{s,t}+ 1)}(X_{1}^{(s,\tau_{s,t}+1)})\bm{G}^{(s,\tau_{s,t}+ 1)\top}\\ \notag
        & \hspace{7cm}\Big\{\bm{\beta}^*(X_{1}^{(s,\tau_{s,t}+1)})-\widetilde{\Phi}_r^\top(X_{1}^{(s,\tau_{s,t}+1)})B^*_r\Big\}\Big]\\ \notag
        =\;& \sum_{s=1}^S \rho\nu_s\sup_{\|v\|_2 =1} \mathbb{E}\Big[\Big\{v^{\top}\widetilde{\Phi}_r(X_{1}^{(s,\tau_{s,t}+1)})\bm{G}^{(s,\tau_{s,t}+ 1)}\Big\}\\ \notag
        & \hspace{5cm}\Big\{\bm{G}^{(s,\tau_{s,t}+ 1)\top}\Big\{\bm{\beta}^*(X_{1}^{(s,\tau_{s,t}+1)})-\widetilde{\Phi}_r^\top(X_{1}^{(s,\tau_{s,t}+1)})B^*_r\Big\}\Big\}\Big]\\  \notag
        \leq\; &\sum_{s=1}^S \rho\nu_s\sup_{\|v\|_2 =1} \mathbb{E}\Big[\Big\{v^{\top}\widetilde{\Phi}_r(X_{1}^{(s,\tau_{s,t}+1)})\bm{G}^{(\tau_{s,t}+ 1)}\Big\}^2\Big]^{1/2}\\ \label{t_fdp_vcm_upper_eq11}
        & \hspace{3cm}\mathbb{E}\Big[\Big\{\bm{G}^{(s,\tau_{s,t}+ 1)\top}\Big\{\bm{\beta}^*(X_{1}^{(s,\tau_{s,t}+1)})-\widetilde{\Phi}_r^\top(X_{1}^{(s,\tau_{s,t}+1)})B^*_r\Big\}\Big\}^2\Big]^{1/2}, 
    \end{align}
    where the inequality follows from the Cauchy-Scwarz inequality. Also, we have that 
    \begin{align} \notag
        &\mathbb{E}\Big[\Big\{v^{\top}\widetilde{\Phi}_r(X_{1}^{(s,\tau_{s,t}+1)})\bm{G}^{(s,\tau_{s,t}+1)}\Big\}^2\Big]\\ \notag
        =\;& \mathbb{E}\Big\{v^\top\widetilde{\Phi}_r(X_{1}^{(s,\tau_{s,t}+1)})\bm{G}^{(s,\tau_{s,t}+1)}\bm{G}^{(s,\tau_{s,t}+1)\top}\widetilde{\Phi}^\top_r(X_{1}^{(s,\tau_{s,t}+1)}v\Big\}\\ \notag
        = \; &  \int_{0}^1 v^\top \widetilde{\Phi}_r(s)\mathbb{E}\Big\{\bm{G}^{(s,\tau_{s,t}+1)}\bm{G}^{(s,\tau_{s,t}+1)\top}\Big\}\widetilde{\Phi}^\top_r(s)v f_X(s)\;\mathrm{d}s \lesssim \|v\|_2^2,
    \end{align} 
    where the last inequality follows from Assumptions \ref{a_sample} and \ref{simple_vcm_a_model}(a)~and
    \begin{align} \notag
        &\mathbb{E}\Big[\Big\{\bm{G}^{(s,\tau_{s,t}+1)\top}\Big\{\bm{\beta}^*(X_{1}^{(s,\tau_{s,t}+1)})-\widetilde{\Phi}_r^\top(X_{1}^{(s,\tau_{s,t}+1)})B^*_r\Big\}\Big\}^2\Big]\\ \notag
        \lesssim \; & \mathbb{E}\Big[\Big\{\bm{\beta}^*(X_{j}^{(s,\tau_{s,t}+1)})-\widetilde{\Phi}_r^\top(X_{j}^{(s,\tau_{s,t}+1)})B^*_r\Big\}^\top\bm{G}^{(s,\tau_{s,t}+1)}\\
        & \hspace{4cm}\bm{G}^{(s,\tau_{s,t}+1)\top}\Big\{\bm{\beta}^*(X_{j}^{(s,\tau_{s,t}+1)})-\widetilde{\Phi}_r^\top(X_{j}^{(s,\tau_{s,t}+1)})B^*_r\Big\}\Big]\\ \notag
        \lesssim \; & \sum_{k=0}^d\|\beta_{k} - \Phi_r b^*_k\|_{L^2}^2 \lesssim dr^{-2\alpha},
    \end{align}
    where the last inequality follows from Assumptions \ref{a_sample} and \ref{simple_vcm_a_model}(a). Plugging the above results into \eqref{t_fdp_vcm_upper_eq11}, we have that 
    \begin{align} \label{t_fdp_vcm_upper_eq12}
        \|\mathbb{E}\{(II)\}\|_{\op} \lesssim \sum_{s=1}^S \rho\nu_s \sqrt{dr^{-2\alpha}} \asymp \sqrt{dr^{-2\alpha}}.
    \end{align}
    
    Combine the results in \eqref{t_fdp_vcm_upper_eq13}, \eqref{t_fdp_vcm_upper_eq8}, \eqref{t_fdp_vcm_upper_eq10} and \eqref{t_fdp_vcm_upper_eq12}, we have that 
    \begin{align*}
        \|(II)\|_2^2 \lesssim  \sum_{s=1}^S \frac{d\nu_s^2}{b_s}+ \sum_{s=1}^S\frac{dr\nu_s^2}{b_sm} +dr^{-2\alpha}.
    \end{align*}
    Therefore, the Markov inequality implies that 
    \begin{align} \label{t_fdp_vcm_upper_eq4}
        \mathbb{P}\Big\{\|(II)\|_2^2 \lesssim \Big(\sum_{s=1}^S \frac{d\nu_s^2}{b_s}+ \sum_{s=1}^S\frac{dr\nu_s^2}{b_sm} +dr^{-2\alpha}\Big)/\eta_1\Big\} \geq 1-\eta_1,
    \end{align}
    for some small enough $\eta_1 \in (0,1/7)$. To find an upper bound on $\|(III)\|_2^2$, we have that
    \begin{align*}
        &\mathbb{E}(\|(III)\|_2^2) \\
        \lesssim\;& \sum_{s=1}^S\frac{\nu_s^2}{b_sm}\sum_{k=0}^d\sum_{\ell=1}^r \mathbb{E}\Big\{\phi^2_{\ell}(X_{1}^{(s,\tau_{s,t}+1)})G_k^{(s,\tau_{s,t}+1)^2}\xi^2_{s,\tau_{s,t}+1,1}\Big\}\\
        & + \sum_{s=1}^S\frac{\nu_s^2}{b_sm}\sum_{k=0}^d\sum_{\ell=1}^r \mathbb{E}\Big\{\phi_{\ell}(X_{1}^{(s,\tau_{s,t}+1)})G_k^{(s,\tau_{s,t}+1)}\xi_{s,\tau_{s,t}+1,1}\phi_{\ell}(X_{2}^{(s,\tau_{s,t}+1)})G_k^{(s,\tau_{s,t}+1)}\xi_{s,\tau_{s,t}+1,2}\Big\}\\
        \lesssim \;& \sum_{s=1}^S\frac{\nu_s^2}{b_sm}\sum_{k=0}^d\sum_{\ell=1}^r \mathbb{E}\Big(\xi^2_{s,\tau_{s,t}+1,1}\Big) \\
        &+ \sum_{s=1}^S\frac{\nu_s^2}{b_sm}\sum_{k=0}^d\sum_{\ell=1}^r \mathbb{E}\Big\{\phi_{\ell}(X_{1}^{(s,\tau_{s,t}+1)})G_k^{(s,\tau_{s,t}+1)}\phi_{\ell}(X_{2}^{(s,\tau_{s,t}+1)})G_k^{(s,\tau_{s,t}+1)}\Big\}\\
        & \hspace{9cm} \mathbb{E}(\xi_{s,\tau_{s,t}+1,1})\mathbb{E}(\xi_{s,\tau_{s,t}+1,2})\\
        \lesssim\;& \sum_{s=1}^S\frac{dr\nu_s^2}{b_sm}, 
    \end{align*}
    where the second inequality follows from Assumption \ref{simple_vcm_a_model}(b) and the last inequality follows from Assumption \ref{simple_vcm_a_model}(d). Therefore, the Markov inequality implies that
    \begin{align} \label{t_fdp_vcm_upper_eq5}
        \mathbb{P}\Big\{\|(III)\|_2^2 \lesssim \Big(\sum_{s=1}^S\frac{dr\nu_s^2}{b_sm}\Big)/\eta_2\Big\} \geq 1-\eta_2,
    \end{align}
    for any $\eta_2 \in (0,1/7)$. For the term $\|\sum_{s=1}^S \nu_s w_{s,t}\|_2^2$, firstly note that for any $h \in [r(d+1)]$, the $h$th entry of $\sum_{s=1}^S \nu_s w_{s,t}$ follows a Gaussian distribution with mean $0$ and variance $\sum_{s=1}^S \nu_s^2\sigma^2_{s,h}$. Therefore, using the standard property of Gaussian random variables, we have that
    \begin{align*}
        \Big\|\sum_{s=1}^S \nu_s w_{s,t}\Big\|_2^2 \sim \chi^2_{\sum_{h=1}^{r(d+1)}\sum_{s=1}^S \nu_s^2\sigma^2_{s,h}},
    \end{align*}
    which follows a sub-Exponential distribution with parameter $\sum_{h=1}^{r(d+1)}\sum_{s=1}^S \nu_s^2\sigma^2_{s,h}$ and mean
    \begin{align*}
        \mathbb{E}\Big(\Big\|\sum_{s=1}^S \nu_s w_{s,t}\Big\|_2^2\Big) = \sum_{h=1}^{r(d+1)}\sum_{s=1}^S \nu_s^2\sigma^2_{s,h}. 
    \end{align*}
    Hence standard properties of sub-Exponential random variables \citep[e.g.~Proposition 2.7.1 in][]{vershynin2018high} implies that for any $\tau >0$, 
    \begin{align*}
        \mathbb{P}\Big\{\Big\|\sum_{s=1}^S \nu_s w_{s,t}\Big\|_2^2 \geq \tau - \mathbb{E}\Big(\Big\|\sum_{s=1}^S \nu_s w_{s,t}\Big\|_2^2\Big)\Big\} \leq \exp\Big(-\frac{C_1\tau}{\sum_{h=1}^{r(d+1)}\sum_{s=1}^S \nu_s^2\sigma^2_{s,h}}\Big).
    \end{align*}
    Combining with a union bound argument on $t \in [T-1]$, we have with probability at least $1-\eta_3$ that 
    \begin{align} \notag
        \Big\|\sum_{s=1}^S \nu_s w_{s,t}\Big\|_2^2 &\lesssim \log(T/\eta_3)\sum_{h=1}^{r(d+1)}\sum_{s=1}^S \nu_s^2\sigma^2_{s,h} \\ \notag
        & \lesssim \log(T/\eta_3)\sum_{h=1}^{r(d+1)}\sum_{s=1}^S \nu_s^2\log(1/\delta_s)R_{h}\sum_{a=1}^{r(d+1)} R_{a}/(b_s^2\epsilon_s^2)\\ \notag
        &= \log(T/\eta_3)\log(1/\delta_s)\sum_{s=1}^S\nu_s^2\Big(\sum_{a=1}^{r(d+1)} R_{a}\Big)^2/(b_s^2\epsilon_s^2)\\ \notag
        & \asymp \log(T/\eta_3)\log(1/\delta_s)d^2\sum_{s=1}^S\nu_s^2\Big\{\sum_{\ell=1}^{r} \sqrt{m^{-1}\log(N/\eta)}+\ell^{-\alpha}\Big\}^2/(b_s^2\epsilon_s^2)\\ \notag
        & \lesssim \log(T/\eta_3)\log(1/\delta_s)d^2\sum_{s=1}^S\nu_s^2\Big\{r^2m^{-1}\log(N/\eta)+ 1\Big\}/(b_s^2\epsilon_s^2)\\ \label{t_fdp_vcm_upper_eq6}
        & = \log(T/\eta_3)\log(1/\delta_s)\sum_{s=1}^S\Big(\frac{d^2\nu_s^2}{b_s^2\epsilon_s^2} + \frac{d^2r^2\nu_s^2\log(N/\eta)}{b_s^2m\epsilon_s^2}\Big),
    \end{align}
    where the second equality follows from the construction of $\{R_a\}_{a=1}^{r(d+1)}$. Therefore, substituting results in \eqref{t_fdp_vcm_upper_eq1}, \eqref{t_fdp_vcm_upper_eq4}, \eqref{t_fdp_vcm_upper_eq5} and \eqref{t_fdp_vcm_upper_eq6} into \eqref{t_fdp_vcm_upper_eq7} and applying a union bound argument, it holds with probability at least $1-4\eta-\eta_1-\eta_2-\eta_3 = 1-7\eta$ that
    \begin{align} \notag
        &\|B^{t+1} - B_r^*\|_2^2 \\ \notag
        \lesssim \; &\Big(1 - \frac{2}{1+4C_{\lambda}^2L^2}\Big)^2\|B^t - B_r^*\|_2^2 +\frac{1}{\eta}\Big(\sum_{s=1}^S \frac{d\nu_s^2}{b_s}+ \sum_{s=1}^S\frac{dr\nu_s^2}{b_sm} +dr^{-2\alpha}\Big) \\ \notag
        & \hspace{4cm} + \log(T/\eta)\log(1/\delta_s)\sum_{s=1}^S\Big(\frac{d^2\nu_s^2}{b_s^2\epsilon_s^2} + \frac{d^2r^2\nu_s^2\log(N/\eta)}{b_s^2m\epsilon_s^2}\Big)\\ \notag
        \lesssim\; & \Big(1 - \frac{2}{1+4C_{\lambda}^2L^2}\Big)^{2t+2}\|B_r^*\|_2^2  +\frac{1}{\eta}\Big(\sum_{s=1}^S \frac{d\nu_s^2}{b_s}+ \sum_{s=1}^S\frac{dr\nu_s^2}{b_sm} +dr^{-2\alpha}\Big) \\ \notag
        & \hspace{4cm} + \log(T/\eta)\log(1/\delta_s)\sum_{s=1}^S\Big(\frac{d^2\nu_s^2}{b_s^2\epsilon_s^2} + \frac{d^2r^2\nu_s^2\log(N/\eta)}{b_s^2m\epsilon_s^2}\Big)\\ \notag
        \lesssim\;& \exp\Big(-\frac{2(2t+2)}{1+4C_{\lambda}^2L^2}\Big)\|B^*_r\|_2^2+\frac{1}{\eta}\Big(\sum_{s=1}^S \frac{d\nu_s^2}{b_s}+ \sum_{s=1}^S\frac{dr\nu_s^2}{b_sm} +dr^{-2\alpha}\Big) \\ \label{t_fdp_vcm_upper_eq9}
        & \hspace{4cm} + \log(T/\eta)\log(1/\delta_s)\sum_{s=1}^S\Big(\frac{d^2\nu_s^2}{b_s^2\epsilon_s^2} + \frac{d^2r^2\nu_s^2\log(N/\eta)}{b_s^2m\epsilon_s^2}\Big),
    \end{align}
    where the second inequality follows from an iterative argument. The Proposition then follows from the fact that
    \begin{align*}
        \|\widetilde{\bm{\beta}}- \bm{\beta}^*\|_{L^2}^2 &\lesssim \|B^T - B_r^*\|_2^2 + \sum_{k=0}^d\|\Phi_r^\top b_k^*- \beta_k^*\|_{L^2}^2 \lesssim \|B^T - B_r^*\|_2^2 + dr^{-2\alpha},
    \end{align*}
    where the last inequality follows from the fact that $\beta_k^* \in \mathcal{W}(\alpha, C_{\alpha})$ for any $k \in \{0\}\cup [d]$.
    
\end{proof}

\subsection{Proof of Theorem \ref{fdp_thm_vcm_low}} \label{section_appendix_vcm_fdp_low}
\begin{proof}[Proof of Theorem \ref{fdp_thm_vcm_low}]
    The first claim in Theorem \ref{fdp_thm_vcm_low} is a consequence of Proposition \ref{fdp_prop_vcm_low} by selecting the optimal value of $\{b_s^t\}_{t=1}^T$ for any $s\in [S]$ using a similar argument as the one used in the proof of Theorem \ref{fdp_thm_mean_low}. The second claim is a direct calculation from the first part.
\end{proof}

\begin{proposition} \label{fdp_prop_vcm_low}
    Denote $\mathcal{P}_X$ the class of sampling distributions satisfying Assumption \ref{a_sample} and $\mathcal{P}_Y$, $\mathcal{P}_G$ the classes of distributions for observations satisfying Assumption \ref{simple_vcm_a_model}. Let $r_0$ be the number solving the equation $dr^{2\alpha+2}= \sum_{s=1}^S\sum_{t=1}^T \{rdb^t_sm \wedge (b^t_s)^2m\epsilon_s^2\}$. Suppose that $\delta_s\log(1/\delta_s) \lesssim r_0^{-1}d^{-1}m^{-1}\epsilon_s^2$ for all $s\in [S]$, then it holds that
    \begin{align*}
        \underset{Q \in \mathcal{Q}^T_{\bm{\epsilon}, \bm{\delta}}}{\inf} \underset{\widetilde{\bm{\beta}}}{\inf} \underset{\substack{P_X \in \mathcal{P}_X, P_Y \in \mathcal{P}_Y\\P_G \in \mathcal{P}_G}}{\sup}& \mathbb{E}_{P_X, P_Y, P_G, Q}\|\widetilde{\bm{\beta}}- \bm{\beta}^*\|_{L^2}^2\\
        & \gtrsim \frac{d^2}{\sum_{s=1}^S\sum_{t=1}^T \{db^t_s\wedge (b^t_s)^2\epsilon_s^2\}} \vee \frac{r_0^2d^2}{\sum_{s=1}^S\sum_{t=1}^T \{r_0db^t_sm \wedge (b^t_s)^2m\epsilon_s^2\}}.
    \end{align*}
\end{proposition}

\begin{proof}[Proof of Proposition \ref{fdp_prop_vcm_low}]
    Throughout the section, denote $L_t^{(s)} = \{L_t^{(s,i)}\}_{i=1}^{b_s^t} = \{(X^{(s,i)}_j, \bm{G}^{(s,i)}\\, Y^{(s,i)}_j)\}_{i=1,j=1}^{b_s^t,m}$ the $b_s^t$ collections of realizations on the $s$th server in iteration $t$ and $Z_t^{(s)}$ the privatized transcript released to the central server from server $s$ to the central server in iteration $t$. For any $s\in [S]$, we further denote $L^{(s)} = \{L_t^{(s)}\}_{t=1}^T$ and $Z^{(s)} = \{Z_t^{(s)}\}_{t=1}^T$. We prove the minimax lower bound by considering two separate constructions. 
    
    \noindent \textbf{Case 1.} Note that by considering the coefficient function $\bm{\beta}$ as unknown constant functions, we arrive at the case of $d$-dimensional linear regression with the model defined as
    \begin{align*}
        Y_t^{(s,i)} = \bm{G}_t^{(s,i)\top}\bm{\beta} + \xi_{s,t,i}, \quad \text{for any}\quad  s\in [S], t \in [T], i \in [b_s^t].
    \end{align*}
    
    \noindent \textbf{Case 1 - Step 1: Construction of a class of distribution.} We construct the class of distributions as follows. Assume the measurement error $\xi_{s,i}\sim N(0,\sigma^2)$. Also, assume that $\bm{G}$ has independent entries with each entry $G_k$ following a uniform distribution over the interval $[-1,1]$. Denote $N_{[-1,1]}$ the truncated standard Gaussian distribution over $[-1,1]$. To construct the prior distribution on $\bm{\beta}$, for any $k \in \{0\} \cup [d]$, let each entry, $\beta_k$, follow a scaled version of $N_{[-1,1]}$ i.e.~we firstly sample $\beta_0 \sim N_{[-1,1]}$ and let $\beta_k = \beta_0/d$. We further assume that entries of $\bm{\beta}$ are mutually independent. We will verify that the above construction satisfies Assumptions \ref{a_sample} and \ref{simple_vcm_a_model}.
    
    Firstly note that by construction, we have that
    \begin{align*}
        |G_k| \leq 1 \quad \text{and} \quad \mathbb{E}\{\bm{G}\bm{G}^\top\} = \frac{1}{3}I_{d+1} \quad \text{for any} \; k \in \{0\} \cup [d].
    \end{align*}

    Moreover, to show the construction satisfies Assumption \ref{simple_vcm_a_model}(c), note that since for any $k \in \{0\}\cup [d]$, $\beta_k$ is constant function, hence $\beta_k \in \mathcal{W}(\alpha, C_{\alpha})$ trivially, also, it holds that 
    \begin{align*}
        \|\bm{G}^\top\bm{\beta}^*\|_{\infty} = \sum_{k=0}^d G_k\beta_k \leq \sum_{k=0}^d \frac{1}{d} \leq 1.
    \end{align*}
    To show $G^\top\bm{\beta}^* \in \mathcal{W}(\alpha, C_{\alpha})$, note that
    \begin{align*}
        \sum_{\ell=1}^\infty \ell^{2\alpha}\Big\{\int_{0}^1 \bm{G}^\top\bm{\beta}^*\phi_{\ell}(s)\;\mathrm{d}s\Big\}^2 = 1^{2\alpha} \left\{\int_{0}^1 \bm{G}^\top\bm{\beta}^*\phi_{1}(s)\;\mathrm{d}s\right\}^2  \lesssim 1,
    \end{align*}
    where the first equality follows from the fact that for all $\ell >1$,
    \begin{align*}
        \langle \bm{G}^\top\bm{\beta}^*, \phi_\ell \rangle_{L^2} = \bm{G}^\top\bm{\beta}^*\int_{0}^1 \phi_{\ell}(s)\;\mathrm{d}s  = 0.
    \end{align*}
    Hence, we have shown that $\bm{G}^\top\bm{\beta}^* \in \mathcal{W}(\alpha, C_{\alpha})$ and we can conclude that the above construction is valid.
    
    Following a similar argument as the one used in the proof of Theorem \ref{fdp_thm_mean_low} and applying the multivariate version of the Van-Trees inequality \citep[e.g.][]{gill1995applications}, we can bound the average $\ell^2$ risk of interest by
    \begin{align} \label{fdp_t_vcm_low_eq1}
        \int \mathbb{E}\Big\{\sum_{k=0}^d (\widetilde{\beta}_k-\beta_k)^2\Big\}\pi(\bm{\beta})\;\mathrm{d}\bm{\beta} \geq \frac{(d+1)^2}{\int \tr\{I_{Z^{(1)},\ldots,Z^{(S)}}(\bm{\beta})\}\pi(\bm{\beta})\;\mathrm{d}\bm{\beta} + J(\pi)}, 
    \end{align}
    where $I_{Z^{(1)},\ldots,Z^{(S)}}(\bm{\beta})$ is the Fisher information associated with $\bm{Z} = (Z^{(1)},\ldots,Z^{(S)})$ and $\pi(\bm{\beta}) = \Pi_{k=0}^d \pi_k(\beta_k)$ is a prior for the parameter $\bm{\beta}$ and $J(\pi)$ is the Fisher information associated with the prior $\pi$. To calculate the Fisher information $J(\pi)$, using Lemma \ref{l_fisher_rescale} and a direct calculation of Fisher information of truncated Gaussian random variables \citep[e.g.~Theorem 11 in][]{mihoc2003fisher}, we have that 
    \begin{align} \label{fdp_t_vcm_low_eq2}
        J(\pi) = \sum_{k=0}^d \Big\{1+\frac{2\exp(-1/2)}{\Phi(1) - \Phi(-1)}\Big\} \asymp d,
    \end{align}
    where $\Phi(\cdot)$ is the CDF for standard Gaussian random variables. Therefore, substituting \eqref{fdp_t_vcm_low_eq2} into \eqref{fdp_t_vcm_low_eq1}, we have that the minimax risk is bounded by
    \begin{align} \notag
        \sup_{\substack{\bm{\beta} \in \mathbb{R}^{d+1}: \|\bm{\beta}\|_2^2 = d+1\\ \|\bm{G}^\top\bm{\beta}\|_{\infty} \lesssim 1,\bm{G}^\top\bm{\beta}\in \mathcal{W}(\alpha,C_{\alpha})}} \mathbb{E}\|\widetilde{\bm{\beta}}-\bm{\beta}\|_2^2 &\gtrsim \frac{(d+1)^2}{\sup_{\bm{\beta}}\tr\{I_{Z^{(1)},\ldots,Z^{(S)}}(\bm{\beta})\} + d}\\
        \label{fdp_t_vcm_low_eq3}
        & \gtrsim \frac{d^2}{\sup_{\bm{\beta}} \sum_{s=1}^S \sum_{t=1}^T \tr\{I_{Z_t^{(s)}|M^{(t-1)}}(\bm{\beta})\}+ d},
    \end{align}
    where the first inequality follows from the fact that the supremum risk is lower bounded by Bayes risk and the last equality follows from a similar argument as the one used in \eqref{fdp_t_mean_low_eq9}. For any $i \in [b_s^t]$, define
    \begin{align*}
        S_{\bm{\beta}}(L_t^{(s,i)}) = \sigma^{-2}\bm{G}_t^{(s,i)}\big(Y_t^{(s,i)}-\bm{G}_t^{(s,i)\top}\bm{\beta}\big).
    \end{align*}
    Note that $S_{\bm{\beta}} (L_t^{(s,i)})$ is the score function of $\bm{\beta}$ based on $L^{(s,i)}$, hence we have that $\mathbb{E}\{S_{\bm{\beta}} (L_t^{(s,i)})\\\} =~0$. Write the score function for local data $L_t^{(s)}$ on the $s$th server in iteration $t$ as $S_{\bm{\beta}}(L_t^{(s)}) = \sum_{i=1}^{b_s^t} S_{\bm{\beta}}(L_t^{(s,i)})$. Let 
    \begin{align*}
        C_{\bm{\beta}}(Z_t^{(s)}|M^{(t-1)}) = \mathbb{E}\Big\{S_{\bm{\beta}}(L_t^{(s)})\Big| Z_t^{(s)},  M^{(t-1)}\Big\}\mathbb{E}\Big\{S_{\bm{\beta}}(L_t^{(s)})\Big| Z_t^{(s)},  M^{(t-1)}\Big\}^\top \in \mathbb{R}^{(d+1)\times (d+1)},
    \end{align*}
    write $C^{(s)}_{\bm{\beta},t}  = \mathbb{E}\{S_{\bm{\beta}}(L_t^{(s)})S_{\bm{\beta}}(L_t^{(s)})^\top\}$ for the unconditional version covariance matrix of $S_{\bm{\beta}}(L_t^{(s)})$ and let $C^{(s,i)}_{\bm{\beta},t} = \mathbb{E}\{S_{\bm{\beta}}(L_t^{(s,i)})S_{\bm{\beta}}(L_t^{(s,i)})^\top\}$ such that $C^{(s)}_{\bm{\beta},t} = \sum_{i=1}^{b_s^t}C^{(s,i)}_{\bm{\beta},t}$. With the above notation and following a similar argument as the one in \eqref{fdp_t_mean_low_eq8}, we have that $I_{Z_t^{(s)}|M^{(t-1)}}(\bm{\beta}) = \mathbb{E}[\mathbb{E}\{C_{\bm{\beta}}(Z_t^{(s)}|M^{(t-1)})|M^{(t-1)}\}]$. Substituting into \eqref{fdp_t_vcm_low_eq3}, it then holds that
    \begin{align}\label{fdp_t_vcm_low_eq4} \notag
         \sup_{\substack{\bm{\beta} \in \mathbb{R}^{d+1}: \|\bm{\beta}\|_2^2 = d+1\\ \|\bm{G}^\top\bm{\beta}\|_{\infty} \lesssim 1,\bm{G}^\top\bm{\beta}\in \mathcal{W}(\alpha,C_{\alpha})}}&\mathbb{E}\|\widetilde{\bm{\beta}}-\bm{\beta}\|_2^2 \\
         &\gtrsim \frac{d^2}{\sup_{\bm{\beta}} \sum_{s=1}^S \sum_{t=1}^T\tr\big[\mathbb{E}[\mathbb{E}\{C_{\bm{\beta}}(Z_t^{(s)}|M^{(t-1)})|M^{(t-1)}\}]\big] + d}.
    \end{align}
     \noindent \textbf{Case 1 - Step 2: Upper bound on} $\tr\big[\mathbb{E}[\mathbb{E}\{C_{\bm{\beta}}(Z_t^{(s)}|M^{(t-1)})|M^{(t-1)}\}]\big]$. Next, we will find an upper bound on $\tr\big[\mathbb{E}[\mathbb{E}\{C_{\bm{\beta}}(Z_t^{(s)}|M^{(t-1)})|M^{(t-1)}\}]\big]$. Denote 
    \begin{align*}
        G_t^{(s,i)} = \langle \mathbb{E}\big\{S_{\bm{\beta}}(L_t^{(s)})\big| Z_t^{(s)},M^{(t-1)}\big\}, S_{\bm{\beta}}(L_t^{(s,i)})\rangle ,
    \end{align*}
    and
    \begin{align*}
        \breve{G}_t^{(s,i)} = \langle \mathbb{E}\big\{S_{\bm{\beta}}(L_t^{(s)})\big| Z_t^{(s)}, M^{(t-1)}\big\}, S_{\bm{\beta}}(\breve{L}_t^{(s,i)})\rangle.
    \end{align*}
    where $\breve{L}_t^{(s,i)}$ is an independent copy of $L_t^{(s,i)}$. Using a similar argument as the one leading to \eqref{fdp_t_mean_low_eq4}, it holds that
    \begin{align} \notag
        & \tr\big[\mathbb{E}[\mathbb{E}\{C_{\bm{\beta}}(Z_t^{(s)}|M^{(t-1)})|M^{(t-1)}\}]\big] = \sum_{i=1}^{b_s^t}\mathbb{E}(G_t^{(s,i)})\\ \label{fdp_t_vcm_low_eq5}
        \lesssim \;&  \sum_{i=1}^{b_s^t} \left\{\mathbb{E}(\breve{G}_t^{(s,i)})+ \epsilon_s \mathbb{E}|\breve{G}_t^{(s,i)}|+ W\delta_s + \int_{W}^\infty \mathbb{P}\{|G_t^{(s,i)}| \geq w\}\; \mathrm{d}w\right\}.
    \end{align}
    We will upper bound \eqref{fdp_t_vcm_low_eq5} by upper bounding the four terms inside it individually. Note that for the first term, we have that
    \begin{align*}
        \mathbb{E}\langle \mathbb{E}\big\{S_{\bm{\beta}}(L_t^{(s)})\big| Z_t^{(s)}, M^{(t-1)}\big\}, S_{\bm{\beta}}(\breve{L}_t^{(s,i)})\rangle = \mathbb{E}\Big[\mathbb{E}\big\{S_{\bm{\beta}}(L_t^{(s)})\big| Z_t^{(s)}, M^{(t-1)}\big\}\Big]^\top \mathbb{E}\{S_{\bm{\beta}}(\breve{L}_t^{(s,i)})\} = 0,
    \end{align*}
    where the second equality follows from the independence between $\breve{L}_t^{(s,i)}$ and $(Z_t^{(s)}, M^{(t-1)})$ and the last equality follows from the property of the score function. For the second term, firstly note that
    \begin{align*}
         \Lambda_{\max}(C_{\bm{\beta},t}^{(s,i)}) &=\sigma^{-4} \Big\| \mathbb{E}_{\bm{G}}\Big[\bm{G}_t^{(s,i)}\mathbb{E}_{Y|\bm{G}}\Big\{\big(Y_t^{(s,i)}-\bm{G}_t^{(s,i)\top}\bm{\beta}\big)^2\Big| \bm{G}\Big\}\bm{G}_t^{(s,i)\top}\Big]\Big\|_{\op}\\
         &= \sigma^{-2}\Big\| \mathbb{E}_{\bm{G}}\Big[\bm{G}_t^{(s,i)}\bm{G}_t^{(s,i)\top}\Big]\Big\|_{\op} \lesssim 1.
    \end{align*}
    Hence using a similar argument as the one leads to  \eqref{fdp_t_mean_low_eq10}, we have that
    \begin{align*}
        \mathbb{E}|\breve{G}^{(s,i)}| &\leq \sqrt{\tr\big[\mathbb{E}[\mathbb{E}\{C_{\bm{\beta}}(Z_t^{(s)}|M^{(t-1)})|M^{(t-1)}\}]\big]} \sqrt{\Lambda_{\max}(C_{\bm{\beta},t}^{(s,i)})}\\
        & \lesssim \sqrt{\tr\big[\mathbb{E}[\mathbb{E}\{C_{\bm{\beta}}(Z_t^{(s)}|M^{(t-1)})|M^{(t-1)}\}]\big]}.
    \end{align*}
    To control the last term, by a similar argument leading to \eqref{fdp_t_mean_low_eq11}, we have for any $t \in \mathbb{R}$ that 
    \begin{align*}
        \mathbb{E}\Big\{\exp\big(t|G_t^{(s,i)}|\big)\Big\} \leq \mathbb{E}\Big\{\exp\big(t|\langle S_{\bm{\beta}}(L_t^{(s)}),S_{\bm{\beta}}(L_t^{(s,i)}) \rangle|\big)\Big\}.
    \end{align*}
    Note that the moment generating function is well defined here from the fact that each entry of $\bm{G}_t^{(s,i)}$ is bounded and $\xi_{s,t,i}$ is sub-Gaussian. Without loss of generality, take $i =1$ for $S_{\bm{\beta}}(L_t^{(s,i)})$, then it holds
    \begin{align*}
        \langle S_{\bm{\beta}}(L_t^{(s)}),S_{\bm{\beta}}(L_t^{(s,1)}) \rangle &= S_{\bm{\beta}}(L_t^{(s)})^\top S_{\bm{\beta}}(L_t^{(s,1)})\\
        &\asymp \sum_{i=1}^{b_s^t} \big(Y_t^{(s,i)}-\bm{G}_t^{(s,i)\top}\bm{\beta})\big(Y_t^{(s,1)}-\bm{G}_t^{(s,1)\top}\bm{\beta}\big)\bm{G}_t^{(s,i)^\top}\bm{G}_t^{(s,1)}\\
        & = \sum_{i=1}^{b_s^t}\xi_{s,t,i}\xi_{s,t,1}\bm{G}_t^{(s,i)^\top}\bm{G}_t^{(s,1)}.
    \end{align*}
    With the above, by taking $t = C_1/(d\sqrt{b_s^t})$, we have that
    \begin{align*}
        &\mathbb{E}\Big\{\exp\big(t|\langle S_{\bm{\beta}}(L_t^{(s)}),S_{\bm{\beta}}(L_t^{(s,i)}) \rangle|\big)\Big\} \\
        \lesssim \;& \mathbb{E}\Big[\exp\Big\{t\big|\xi_{s,t,1}^2\bm{G}_t^{(s,1)^\top}\bm{G}_t^{(s,1)}\big|+ t|\xi_{s,t,1}|\Big|\sum_{i=2}^{b_s^t}\xi_{s,t,i}\bm{G}_t^{(s,i)^\top}\bm{G}_t^{(s,1)}\Big|\Big\}\Big]\\
        = \;& \mathbb{E}\Big[\exp\Big\{t\big|\xi_{s,t,1}^2\bm{G}_t^{(s,1)^\top}\bm{G}_t^{(s,1)}\big|\Big\}\mathbb{E}\Big[\exp\Big\{t|\xi_{s,t,1}|\Big|\sum_{i=2}^{b_s^t}\xi_{s,t,i}\bm{G}_t^{(s,i)^\top}\bm{G}_t^{(s,1)}\Big|\Big\}\Big| \xi_{s,t,1}, \{\bm{G}_t^{(s,i)}\}_{i=1}^{b_s^t}\Big]\Big]\\
        \lesssim \;& \mathbb{E}\Big[\exp\Big\{t\big|\xi_{s,t,1}^2\bm{G}_t^{(s,1)^\top}\bm{G}_t^{(s,1)}\big|\Big\}\exp\Big\{b_s^td^2t^2\xi_{s,t,1}^2\Big\}\Big] \\
        \lesssim \;& \mathbb{E}\Big[\exp\Big\{\xi_{s,t,1}^2\Big(td + t^2b^t_sd^2\Big)\Big\}\Big] \lesssim \exp\Big(td + t^2b_s^td^2\Big),
    \end{align*}
    where the equality follows from the tower property of conditional expectation, the first inequality follows from the fact that $\sum_{i=2}^{b^t_s}\xi_{s,t,i}\bm{G}_t^{(s,i)^\top}\bm{G}_t^{(s,1)}| \{\bm{G}_t^{(s,i)}\}_{i=1}^{b_s^t} \sim N(0, \sigma^2\sum_{i=2}^{b_s^t}(\bm{G}_t^{(s,i)^\top}\\ \bm{G}_t^{(s,1)})^2)$ and $\bm{G}_t^{(s,i)^\top}\bm{G}_t^{(s,1)} \lesssim~d$ for any $i \in [b_s^t]$ and the last inequality follows from the standard property of Gaussian random variable $\xi_{s,t,1}$ \citep[e.g.~Proposition 2.5.2 in][]{vershynin2018high} as with the choice of $t$ we have that $td + t^2b^t_sd^2 = C_1/\sqrt{b_s^t} + C_1^2 \lesssim 1$. it then holds that for any $\tau > 0$, 
    \begin{align*}
        \mathbb{P}(|G_t^{(s,i)}| \geq \tau) &\leq \mathbb{P}\{\exp(t|G_t^{(s,i)}|) \geq \exp(t\tau)\} \\ 
        &\leq \exp(-t\tau)\mathbb{E}\big\{\exp\big(t|G_t^{(s,i)}|\big)\big\} \lesssim \exp\Big(-\frac{\tau}{d\sqrt{b_s^t}}\Big),
    \end{align*}
    where the second inequality follows from Markov's inequality. This means that if we pick $T~\asymp~d\sqrt{b_s^t}\log(1/\delta_s)$, then it holds that
    \begin{align*}
        \int_{W}^\infty \mathbb{P}\{|G_t^{(s,i)}| \geq w\}\; \mathrm{d}t \lesssim \int_{W}^\infty\exp\Big(-\frac{w}{d\sqrt{b_s^t}}\Big)\; \mathrm{d}w \leq d\sqrt{b_s^t}\exp\{-\log(1/\delta_s)\} = d\sqrt{b_s^t}\delta_s.
    \end{align*}
    Putting everything together into \eqref{fdp_t_vcm_low_eq5}, we have that 
    \begin{align*}
        &\tr\big[\mathbb{E}[\mathbb{E}\{C_{\bm{\beta}}(Z_t^{(s)}|M^{(t-1)})|M^{(t-1)}\}]\big] \\
        \lesssim \;& b_s^t\epsilon_s \sqrt{\tr\big[\mathbb{E}[\mathbb{E}\{C_{\bm{\beta}}(Z_t^{(s)}|M^{(t-1)})|M^{(t-1)}\}]\big]} + (b_s^t)^{3/2}d\delta_s\log(1/\delta_s)+(b_s^t)^{3/2}d\delta_s\\
        \lesssim \;& b_s^t\epsilon_s \sqrt{\tr\big[\mathbb{E}[\mathbb{E}\{C_{\bm{\beta}}(Z_t^{(s)}|M^{(t-1)})|M^{(t-1)}\}]\big]} + (b_s^t)^{3/2}d\delta_s\log(1/\delta_s).
    \end{align*}
    Next, we show that $\tr\big[\mathbb{E}[\mathbb{E}\{C_{\bm{\beta}}(Z_t^{(s)}|M^{(t-1)})|M^{(t-1)}\}]\big] \lesssim (b_s^t)^2\epsilon_s^2$. Suppose this is not true, then when $\delta$ is small enough such that $\delta_s\log(1/\delta_s) \lesssim (b_s^t)^{1/2}\epsilon_s^2d^{-1}$, we have that
    \begin{align}\label{fdp_t_vcm_low_eq6}
        \sqrt{\tr\big[\mathbb{E}[\mathbb{E}\{C_{\bm{\beta}}(Z_t^{(s)}|M^{(t-1)})|M^{(t-1)}\}]\big]} \lesssim b^t_s\epsilon_s + \frac{(b^t_s)^{3/2}d\delta_s\log(1/\delta_s)}{b_s^t\epsilon_s} \lesssim b_s^t\epsilon_s.
    \end{align}
    Hence we can conclude that $\tr\big[\mathbb{E}[\mathbb{E}\{C_{\bm{\beta}}(Z_t^{(s)}|M^{(t-1)})|M^{(t-1)}\}]\big] \lesssim (b_s^t)^2\epsilon_s^2$. 
    
    \noindent \textbf{Case 1 - Step 3: Another upper bound on} $\tr\big[\mathbb{E}[\mathbb{E}\{C_{\bm{\beta}}(Z_t^{(s)}|M^{(t-1)})|M^{(t-1)}\}]\big]$. To get another upper bound on $\tr\big[\mathbb{E}[\mathbb{E}\{C_{\bm{\beta}}(Z_t^{(s)}|M^{(t-1)})|M^{(t-1)}\}]\big]$, we will follow a similar argument as the one leads to \eqref{fdp_t_mean_low_eq6}. By standard matrix algebra, we have that 
    \begin{align*}
        \tr(C_{\bm{\beta},t}^{(s,i)}) = \sum_{k=0}^d \Lambda_k(C_{\bm{\beta},t}^{(s,i)}) \leq (d+1)\Lambda_{\max}(C_{\bm{\beta},t}^{(s,i)}) \lesssim d,
    \end{align*}
    where for any $k \in \{0\} \cup [d]$, $\Lambda_k(C_{\bm{\beta},t}^{(s,i)})$ denote the $k$th eigenvalue of $C_{\bm{\beta},t}^{(s,i)}$. Hence, we have that 
    \begin{align}\label{fdp_t_vcm_low_eq7}
        \mathbb{E}[\mathbb{E}\{C_{\bm{\beta}}(Z_t^{(s)}|M^{(t-1)})|M^{(t-1)}\}] \leq \tr(C_{\bm{\beta},t}^{(s)}) = \sum_{i=1}^{b_s^t} \tr(C_{\bm{\beta},t}^{(s,i)}) \lesssim b_s^t d.
    \end{align}
    Substituting the result in \eqref{fdp_t_vcm_low_eq6} and \eqref{fdp_t_vcm_low_eq7} into \eqref{fdp_t_vcm_low_eq4}, we have that 
    \begin{align} \notag
        \sup_{\substack{\bm{\beta} \in \mathbb{R}^{d+1}: \|\bm{\beta}\|_2^2 = d+1\\ \|\bm{G}^\top\bm{\beta}\|_{\infty} \lesssim 1,\bm{G}^\top\bm{\beta}\in \mathcal{W}(\alpha,C_{\alpha})}} \mathbb{E}\|\widetilde{\bm{\beta}}-\bm{\beta}\|_2^2 &\gtrsim \frac{d^2}{ \sum_{s=1}^S \sum_{t=1}^T \{(b_s^t)^2\epsilon_s^2 \wedge b^t_sd\} + d}\\ \label{fdp_t_vcm_low_eq8}
        & \gtrsim  \frac{d^2}{\sum_{s=1}^S\sum_{t=1}^T \{b^t_sd \wedge (b_s^t)^2\epsilon_s^2\}},
    \end{align}
    where the last inequality follows from the fact that $\sum_{s=1}^S \sum_{t=1}^T \{(b_s^t)^2\epsilon_s^2 \wedge b^t_sd\} \gtrsim d $.

    \noindent \textbf{Case 2.} In the second case, when $\bm{\beta}$ are considered to be functional regression coefficients, following a similar argument as the one in the proof of Theorem \ref{thm_mean_lower}, we will reduce the problem to a finite-dimensional subspace $\Theta(r, C_\alpha)$ as defined in \eqref{mean_t_lower_eq6}. In this case, we have that for any $k \in \{0\} \cup [d]$, $b_k \in \Theta(r, C_\alpha)$. 
    
    \noindent \textbf{Case 2 - Step 1: Construction of class of distribution.} Denote $B = (b_0^\top, b_1^\top, \ldots, b_d^\top)^\top$, we construct the class of distributions as follows. With the same construction of $\xi$ and $\bm{G}$ in the first case, we further assume the sampling distribution of $X$ follows a uniform distribution over $[0,1]$. To construct the prior distribution on $B$, for any $k \in \{0\}\cup [d]$ and $\ell \in [r]$, denote $\pi_{k,\ell}$ the prior distribution for the $\ell$th entry of $b_k$ and we assume that $\pi(B) = \prod_{k=0}^d \prod_{\ell=1}^r \pi_{k,\ell}\{(b_{k})_\ell\}$. To construct the prior, denote $\pi^*$ as a rescaled version of the density $t \mapsto \cos^2(\pi t/2)\bm{1}\{|t|\leq 1\}$ such that it is supported on $[-Q, Q]$ where 
    $$Q^2 = \frac{C_2^2}{2\pi^{2\alpha}}\Big(\int_{1}^{r+1}t^{2\alpha}\mathrm{d}t\Big)^{-1}\asymp r^{-(2\alpha+1)}.$$
    For any $k \in \{0\}\cup [d]$ and $\ell \in [r]$, in order to get $b_{k,\ell}$, we will firstly sample $b^0_{k ,\ell} \sim \pi^*$ and then take $b_{k,\ell} = b^0_{k ,\ell}/d$.
    To show that the prior distribution is supported in the class we consider, note that for any $k \in \{0\}\cup [d]$ 
    \begin{align*}
        \sum_{\ell=1}^r \tau_\ell^2(b_{k})^2_\ell \leq \frac{Q^2}{d^2} \sum_{\ell=1}^r\tau_\ell^2\asymp \frac{Q^2}{d^2} \sum_{\ell=1}^r \ell^{2\alpha} \leq \frac{C_2^2}{2\pi^{2\alpha}}.
    \end{align*}
    To further show that Assumption \ref{simple_vcm_a_model}(c)~is satisfied note that we have
    \begin{align*}
        \|\bm{G}^\top\bm{\beta}^*\|_{\infty} &= \sup_{s\in [0,1]} \Big|\sum_{k=0}^d G_k \Phi_r^\top(s)b_k\Big| =\sup_{s\in [0,1]} \Big|\sum_{k=0}^d G_k \sum_{\ell=1}^r (b_k)_\ell\phi_{\ell}(s)\Big|\\
        &\lesssim \sum_{k=0}^d |G_k|\frac{r^{-\alpha+1/2}}{d}\lesssim r^{-\alpha+1/2} \lesssim 1,
    \end{align*}
    where the last inequality follows from the fact that $\alpha >1$. Also, to show $\bm{G}^\top\bm{\beta}^* \in \mathcal{W}(\alpha, C_{\alpha})$, we have that 
    \begin{align*}
        \sum_{\ell=1}^r \ell^{-2\alpha}\Big\{\int_{0}^1 \sum_{k=0}^d G_k \Phi_r^\top(s)b_k\phi_{\ell}(s)\; \mathrm{d}s\Big\}^2 &=\sum_{\ell=1}^r \ell^{-2\alpha}\Big\{\sum_{k=0}^d G_k (b_k)_\ell\Big\}^2 \lesssim \sum_{\ell=1}^r \ell^{-2\alpha}\Big\{\sum_{k=0}^d  \frac{Q}{d}\Big\}^2\\
        & \lesssim \sum_{\ell=1}^r \ell^{-2\alpha} Q^2 \leq \frac{C_2^2}{2\pi^{2\alpha}}.
    \end{align*}
    Hence, we have shown that the above construction is valid.
    
    Following a similar argument as the one used in \eqref{fdp_t_vcm_low_eq1} and using the same notation, by the multivariate Van-Trees \citep[e.g.][]{gill1995applications} inequality, we can bound the average $\ell^2$ risk of interest by
    \begin{align} \label{fdp_t_vcm_low_eq9}
        \int \mathbb{E}\Big[\sum_{k=0}^d \sum_{\ell=1}^r \Big\{(\widetilde{b}_{k})_{\ell}-(b_k)_\ell\Big\}^2\Big]\pi(\bm{\beta})\;\mathrm{d}\bm{\beta} \geq \frac{r^2(d+1)^2}{\int \tr\{I_{Z^{(1)},\ldots,Z^{(S)}}(B)\}\pi(B)\;\mathrm{d}B + J(\pi)}.
    \end{align}
    To calculate the Fisher information $J(\pi)$, by Lemma \ref{l_fisher_rescale} and  a similar argument as the one used in the proof Lemma 4.3 in \citet{cai2024optimal}, we have that
    \begin{align} \label{fdp_t_vcm_low_eq10}
        J(\pi) = \pi^2\sum_{k=0}^d r Q^{-2} \asymp \pi^2 dr^{2\alpha+2} \asymp dr^{2\alpha+2}.
    \end{align}
    Therefore, substituting \eqref{fdp_t_vcm_low_eq10} into \eqref{fdp_t_vcm_low_eq9}, we have that the minimax risk is bounded by
    \begin{align} \notag
        \underset{\substack{b_k \in \Theta(r,C_\alpha),  \|\bm{G}^\top\widetilde{\Phi}^\top B\|_{\infty} \lesssim 1\\ \bm{G}^\top\widetilde{\Phi}^\top B \in \mathcal{W}(\alpha, C_{\alpha}), \\\forall k \in \{0\} \cup [d]}}{\sup} \mathbb{E}\|\widetilde{B}-B\|_2^2 &\gtrsim \frac{r^2d^2}{\sup_{B}\tr\{I_{Z^{(1)},\ldots,Z^{(S)}}(B)\} +  dr^{2\alpha+2}}\\ \label{fdp_t_vcm_low_eq11}
        & = \frac{r^2d^2}{\sup_{B} \sum_{s=1}^S \sum_{t=1}^T \tr\{I_{Z_t^{(s)}|M^{(t-1)}}(B)\} + dr^{2\alpha+2}},
    \end{align}
    where the first inequality follows from the fact that the supremum risk is lower bounded by Bayes risk and the last equality follows from the same argument that leads to \eqref{fdp_t_vcm_low_eq3}.
    
    With the above construction, note that for any $i \in [b_s^t]$, conditioning on $\{X_{t,j}^{(s,i)}\}_{j=1}^{m}$ and $\bm{G}_t^{(s, i)}$, we have that $Y^{(s)}_{i\cdot,t}=(Y^{(s,i)}_{t,1},\ldots, Y^{(s,i)}_{t,m})^\top$ follows a multivariate Gaussian distribution with mean vector
    \begin{align*}
        \mu_{i\cdot,t}^{(s)} &= \Big(\bm{G}_t^{(s,i)\top}\widetilde{\Phi}_r(X^{(s,i)}_{t,1})B, \ldots, \bm{G}_t^{(s,i)\top}\widetilde{\Phi}_r(X^{(s,i)}_{t,m})B\Big)^\top\\
        & = \Big(\sum_{k=0}^d G^{(s,i)}_{t,k}\Phi_r^\top(X^{(i)}_{t,1})b_k, \ldots, \sum_{k=0}^d G^{(s,i)}_{t,k}\Phi_r^\top(X^{(s,i)}_{t,m})b_k\Big)^{\top},
    \end{align*}
    and covariance matrix $\sigma^2I_m$. For any $i \in [n_s]$, define 
    \begin{align*}
        S_{B}(L_t^{(s,i)}) = \sigma^{-2}\widetilde{\Phi}_{[m],s,i,t}(Y^{(s)}_{i\cdot,t}-\mu^{(s)}_{i\cdot,t}),
    \end{align*}
    where $\widetilde{\Phi}_{[m],s,i,t} = (\widetilde{\Phi}_r(X^{(s,i)}_{t,1})\bm{G}_t^{(s,i)},\ldots, \widetilde{\Phi}_r(X^{(s,i)}_{t,m})\bm{G}_t^{(s,i)}) \in \mathbb{R}^{r(d+1)\times m}$. Note that $S_{B} (L_t^{(s,i)})$ is the score function of $B$ based on $L_t^{(s,i)}$, hence we have that $\mathbb{E}\{S_{B} (L_t^{(s,i)})\} = 0$. Write the score function for local data $L_t^{(s)}$ on the $s$th server in the $t$th iteration as $S_{B}(L_t^{(s)}) = \sum_{i=1}^{b_s^t} S_{B}(L_t^{(s,i)})$. Let 
    \begin{align*}
        C_{B}(Z_t^{(s)}) = \mathbb{E}\Big\{S_{B}(L_t^{(s)})\Big| Z_t^{(s)}, M^{(t-1)}\Big\}\mathbb{E}\Big\{S_{B}(L_t^{(s)})\Big| Z_t^{(s)}, M^{(t-1)}\Big\}^\top \in \mathbb{R}^{r(d+1)\times r(d+1)},
    \end{align*}
    write $C^{(s)}_{B,t} = \mathbb{E}\{S_{B}(L_t^{(s)})S_{B}(L_t^{(s)})^\top\}$ for the unconditional version covariance matrix of $S_{B}(L_t^{(s)})$ and let $C^{(s,i)}_{B,t} = \mathbb{E}\{S_{B}(L_t^{(s,i)})S_{B}(L_t^{(s,i)})^\top\}$ such that $C^{(s)}_{B,t} = \sum_{i=1}^{b_s^t} C^{(s,i)}_{B,t}$. With the above notation and following a similar argument as the one in \eqref{fdp_t_mean_low_eq8}, we have that $I_{Z^{(s)}|M^{(t-1)}}(B) = \mathbb{E}[\mathbb{E}\{C_{B}(Z_t^{(s)}|M^{(t-1)})|M^{(t-1)}\}]$. Substituting into \eqref{fdp_t_vcm_low_eq11}, it holds that
    \begin{align}\label{fdp_t_vcm_low_eq12}
        \underset{\substack{b_k \in \Theta(r,C_\alpha),  \|\bm{G}^\top\widetilde{\Phi}^\top B\|_{\infty} \lesssim 1\\ \bm{G}^\top\widetilde{\Phi}^\top B \in \mathcal{W}(\alpha, C_{\alpha}), \\\forall k \in \{0\} \cup [d]}}{\sup}
        &\mathbb{E}\|\widetilde{B}-B\|_2^2 \\ \notag
        &\gtrsim \frac{r^2d^2}{\sup_{\bm{\beta}} \sum_{s=1}^S \sum_{t=1}^T \tr\big[\mathbb{E}[\mathbb{E}\{C_{B}(Z_t^{(s)}|M^{(t-1)})|M^{(t-1)}\}]\big] + dr^{2\alpha+2}}.
    \end{align}
    
    \noindent \textbf{Case 2 - Step 2: Upper bound on} $\tr\big[\mathbb{E}[\mathbb{E}\{C_{B}(Z_t^{(s)}|M^{(t-1)})|M^{(t-1)}\}]\big]$. Next, we will find upper bounds on $\tr\big[\mathbb{E}[\mathbb{E}\{C_{B}(Z_t^{(s)}|M^{(t-1)})|M^{(t-1)}\}]\big]$. Denote 
    \begin{align*}
        G_t^{(s,i)} = \langle \mathbb{E}\big\{S_{B}(L_t^{(s)})\big| Z_t^{(s)}, M^{(t-1)}\big\}, S_{B}(L_t^{(s,i)})\rangle,
    \end{align*}
    and
    \begin{align*}
        \breve{G}_t^{(s,i)} = \langle \mathbb{E}\big\{S_{B}(L_t^{(s)})\big| Z_t^{(s)}, M^{(t-1)}\big\}, S_{B}(\breve{L}_t^{(s,i)})\rangle,
    \end{align*}
    where $\breve{L}_t^{(s,i)}$ is an independent copy of $L_t^{(s,i)}$. Using a similar argument as the one leading to \eqref{fdp_t_vcm_low_eq5}, it holds that
    \begin{align} \notag
        &\tr\big[\mathbb{E}[\mathbb{E}\{C_{B}(Z_t^{(s)}|M^{(t-1)})|M^{(t-1)}\}]\big] =  \sum_{i=1}^{b^t_s}\mathbb{E}(G_t^{(s,i)})\\ \label{fdp_t_vcm_low_eq13}
        \lesssim\;& \sum_{i=1}^{b^t_s}\Big\{\mathbb{E}(\breve{G}_t^{(s,i)})+ \epsilon_s \mathbb{E}|\breve{G}_t^{(s,i)}|+ W\delta_s + \int_{W}^\infty \mathbb{P}\{|G_t^{(s,i)}| \geq w\}\; \mathrm{d}w\Big\}.
    \end{align}
    We will upper bound \eqref{fdp_t_vcm_low_eq13} by upper bounding the four terms inside it individually. Note that for the first term, we have that
    \begin{align*}
        \mathbb{E}\langle \mathbb{E}\big\{S_{B}(L_t^{(s)})\big| Z_t^{(s)}, M^{(t-1)}\big\}, S_{B}(\breve{L}_t^{(s,i)})\rangle = \mathbb{E}\Big[\mathbb{E}\big\{S_{B}(L_t^{(s)})\big| Z_t^{(s)}, M^{(t-1)}\big\}\Big]^\top \mathbb{E}\{S_{B}(\breve{L}_t^{(s,i)})\} = 0,
    \end{align*}
    where the second equality follows from the independence between $\breve{L}_t^{(s,i)}$ and $(Z_t^{(s)}, M^{(t-1)})$ and the last equality follows from the property of the score function. For the second term, firstly note that 
    \begin{align} \notag
        \Lambda_{\max}(C_{B,t}^{(s,i)}) &= \sigma^{-4}\Big\|\mathbb{E}_{X,G}\Big[\widetilde{\Phi}_{[m],s,i,t}\mathbb{E}_{Y|X,G}\Big\{\big(Y^{(s)}_{i\cdot,t}-\mu^{(s)}_{i\cdot,t}\big)^2\Big|X\Big\}\widetilde{\Phi}^\top_{[m],s,i,t}\Big]\Big\|_{\op}\\ \notag
        & \asymp \Big\|\mathbb{E}_{X,G}\Big\{\widetilde{\Phi}_{[m],s,i,t}\widetilde{\Phi}^\top_{[m],s,i,t}\Big\}\Big\|_{\op}\\ \notag
        &= \sup_{v \in \mathbb{R}^{r(d+1)}: \|v\|_2=1} v^{\top}\mathbb{E}_{X,G}\Big[\widetilde{\Phi}_{[m],s,i,t}\widetilde{\Phi}^{\top}_{[m],s,i,t}\Big]v\\ \notag
        & = \sup_{v \in \mathbb{R}^{r(d+1)}: \|v\|_2=1} \sum_{j=1}^m v^\top \mathbb{E}_{X,G}\Big\{\widetilde{\Phi}_r(X^{(s,i)}_{t,j})\bm{G}_t^{( s,i)}\bm{G}_t^{(s,i)\top}\widetilde{\Phi}_r^{\top}(X^{(s,i)}_{t,j})\Big\}v\\ \notag
        & =\sup_{v \in \mathbb{R}^{r(d+1)}: \|v\|_2=1} \sum_{j=1}^m \mathbb{E}_{X,G}\Big[\Big\{\sum_{k=0}^d\sum_{\ell=1}^r v_{\ell(k+1)}\phi_\ell(X^{(s,i)}_{t,j})G_{t,k}^{(s,i)}\Big\}^2\Big]\\ \notag
        & =\sup_{v \in \mathbb{R}^{r(d+1)}: \|v\|_2=1} \sum_{j=1}^m \int_{0}^1 J(s)\mathbb{E}\{\bm{G}_t^{(s,i)}\bm{G}_t^{(s,i)\top}\}J^{\top}(s)f_X(s)\mathrm{d}s\\ \notag
        & \lesssim \sup_{v \in \mathbb{R}^{r(d+1)}: \|v\|_2=1} m\int_{0}^1 J(s)J^{\top}(s)\mathrm{d}s\\ \label{fdp_t_vcm_low_eq17}
        &=\inf_{v \in \mathbb{R}^{r(d+1)}: \|v\|_2=1} m \sum_{k=0}^d \sum_{\ell=1}^r v_{\ell(k+1)}^2= m,
    \end{align}
    where $J^\top(s) = \{\sum_{\ell=1}^r v_{\ell}\phi_\ell(s), \ldots, \sum_{\ell=1}^r v_{\ell(d+1)}\phi_\ell(s)\}^\top \in \mathbb{R}^{d+1}$. Hence, by a similar argument as the one leading to \eqref{fdp_t_mean_low_eq10}, we have that
    \begin{align} \notag
        \mathbb{E}|\breve{G}^{(s,i)}| &\leq \sqrt{\tr\big[\mathbb{E}[\mathbb{E}\{C_{B}(Z_t^{(s)}|M^{(t-1)})|M^{(t-1)}\}]\big]} \sqrt{\Lambda_{\max}(C_{B,t}^{(s,i)})}\\ \label{fdp_t_vcm_low_eq18}
        & \lesssim \sqrt{m}\sqrt{\tr\big[\mathbb{E}[\mathbb{E}\{C_{B}(Z_t^{(s)}|M^{(t-1)})|M^{(t-1)}\}]\big]}.
    \end{align}
    To control the last term, by a similar argument as the one leading to \eqref{fdp_t_mean_low_eq11}, we have for any $t \in \mathbb{R}$ that 
    \begin{align*}
        \mathbb{E}\Big\{\exp\big(t|G_t^{(s,i)}|\big)\Big\} \leq \mathbb{E}\Big\{\exp\big(t|\langle S_{B}(L_t^{(s)}),S_{B}(L_t^{(s,i)}) \rangle|\big)\Big\}.
    \end{align*}
    The above moment generating function is well defined as each entry inside $\widetilde{\Phi}_{[m],s,i,t}$ is bounded and each entry in $Y^{(s)}_{i\cdot,t}-\mu^{(s)}_{i\cdot,t}$ is sub-Gaussian. Without loss of generality, take $i =1$ for $S_a(L_t^{(s,i)})$, then it holds
    \begin{align*}
        \langle S_{B}(L_t^{(s)}),S_{B}(L_t^{(s,1)}) \rangle &= S_{B}(L_t^{(s)})^\top S_{B}(L_t^{(s,1)})\\
        &\asymp \sum_{i=1}^{b^t_s}(Y^{(s)}_{i\cdot,t}-\mu^{(s)}_{i\cdot,t})^\top \widetilde{\Phi}^\top_{[m],s,i,t}\widetilde{\Phi}_{[m],s,1,t}(Y^{(s)}_{1\cdot,t}-\mu^{(s)}_{1\cdot,t}).
    \end{align*}
    Note that, for any $i \in [b^t_s]$, each entry inside $\widetilde{\Phi}_{[m],s,i,t}(Y^{(s)}_{i\cdot,t}-\mu^{(s)}_{i\cdot,t})$ follows a sub-Gaussian distribution with parameter $C_3m$, hence using standard properties of sub-Gaussian random variables \citep[e.g.~Lemma 2.7.7 in][]{vershynin2018high} and triangle inequality of sub-Exponetial norms, we have that for any $i \in [b^t_s]$, the term $(Y^{(s)}_{i\cdot,t}-\mu^{(s)}_{i\cdot,t})^\top \widetilde{\Phi}^\top_{[m],s,i,t}\widetilde{\Phi}_{[m],s,1,t}(Y^{(s)}_{1\cdot,t}-\mu^{(s)}_{1\cdot,t})$ follows sub-Exponential distribution with parameter of $C_4rdm^2$. By applying another triangle inequality, we have that $\langle S_B(L_t^{(s)}),S_B(L_t^{(s,i)}) \rangle$ follows a sub-Exponential distribution with parameter of order $C_4rdb_s^tm^2$. Therefore, we have from sub-Exponential properties \citep[e.g.~Proposition 2.7.1][]{vershynin2018high} that 
    \begin{align*}
        &\mathbb{E}\Big\{\exp\big(t|\langle S_B(L_t^{(s)}),S_B(L_t^{(s,i)}) \rangle|\big)\Big\} \leq  \exp\{C_4rdb_s^tm^2t\},\\
        & \hspace{6cm} \; \text{for any} \;t\; \text{such that}\; 0\leq t\leq (C_4rdb_s^tm^2)^{-1}.
    \end{align*}
    Pick $t  = (C_3rdb^t_sm^2)^{-1}/2$ and by Markov's inequality, we have for any $\tau >0$ that 
    \begin{align*}
        \mathbb{P}(|G_t^{(s,i)}|\geq \tau) &= \mathbb{P}\{\exp(t|G_t^{(s,i)}|)\geq \exp(t\tau)\}  \leq \exp(-t\tau)\mathbb{E}\{\exp(t|G_t^{(s,i)}|)\}\\
        & \lesssim \exp(1/2)\exp\Big(-\frac{\tau}{rdb^t_sm^2}\Big),
    \end{align*}
    which suggests that if we pick $T \asymp rdb^t_sm^2\log(1/\delta_s)$, then we have that 
    \begin{align} \notag 
        \int_{W}^\infty \mathbb{P}\{|G_t^{(s,i)}| \geq w\}\; \mathrm{d}w  & \lesssim  \int_{W}^\infty\exp\Big(-\frac{w}{rdb^t_sm^2}\Big) \; \mathrm{d}w \\ \label{fdp_t_vcm_low_eq19}
        &\lesssim rdb^t_sm^2\exp\{-\log(1/\delta_s)\} = rdb^t_sm^2\delta_s.
    \end{align}
    Substituting \eqref{fdp_t_vcm_low_eq18} and \eqref{fdp_t_vcm_low_eq19} together into \eqref{fdp_t_vcm_low_eq13}, we have that 
    \begin{align*}
        &\tr\big[\mathbb{E}[\mathbb{E}\{C_{B}(Z_t^{(s)}|M^{(t-1)})|M^{(t-1)}\}]\big] \\
        \lesssim\;& b^t_s\epsilon_s\sqrt{m\tr\big[\mathbb{E}[\mathbb{E}\{C_{B}(Z_t^{(s)}|M^{(t-1)})|M^{(t-1)}\}]\big]} + rd(b^t_s)^2m^2\delta_s\log(1/\delta_s) +rd(b^t_s)^2m^2\delta_s\\
        \lesssim \;& b^t_s\epsilon_s\sqrt{m\tr\big[\mathbb{E}[\mathbb{E}\{C_{B}(Z_t^{(s)}|M^{(t-1)})|M^{(t-1)}\}]\big]} + rd(b^t_s)^2m^2\delta_s\log(1/\delta_s).
    \end{align*}
    Next, we show that $\tr\big[\mathbb{E}[\mathbb{E}\{C_{B}(Z_t^{(s)}|M^{(t-1)})|M^{(t-1)}\}]\big] \lesssim (b^t_s)^2m\epsilon_s^2$. Suppose this is not true, then when $\delta_s$ is small enough such that $\delta_s\log(1/\delta_s)\lesssim r^{-1}d^{-1}m^{-1}\epsilon_s^2$, we have that
    \begin{align} \label{fdp_t_vcm_low_eq14}
        \sqrt{\tr\big[\mathbb{E}[\mathbb{E}\{C_{B}(Z_t^{(s)}|M^{(t-1)})|M^{(t-1)}\}]\big]} \lesssim b^t_s\epsilon_s\sqrt{m} + \frac{rd(b^t_s)^2m^2\delta_s\log(1/\delta_s)}{b^t_sm^{1/2}\epsilon_s} \lesssim b^t_s\epsilon_s\sqrt{m}.
    \end{align}
    Hence we can conclude that $\tr\big[\mathbb{E}[\mathbb{E}\{C_{B}(Z_t^{(s)}|M^{(t-1)})|M^{(t-1)}\}]\big] \lesssim (b_s^t)^2m\epsilon^2$. 
    
    \noindent \textbf{Case 2 - Step 3: Upper bound on} $\tr\big[\mathbb{E}[\mathbb{E}\{C_{B}(Z_t^{(s)}|M^{(t-1)})|M^{(t-1)}\}]\big]$. To get another upper bound on $\tr\big[\mathbb{E}[\mathbb{E}\{C_{B}(Z_t^{(s)}|M^{(t-1)})|M^{(t-1)}\}]\big]$, we will follow a similar argument as the one leading to \eqref{fdp_t_mean_low_eq6}. By standard matrix algebra, we have that 
    \begin{align*}
        \tr(C_{B,t}^{(s,i)}) = \sum_{\ell=1}^r\sum_{k=0}^d \Lambda_{\ell(k+1)}(C_{B,t}^{(s,i)}) \leq r(d+1)\Lambda_{\max}(C_{B,t}^{(s,i)}) \lesssim rdm,
    \end{align*}
    where for any $\ell \in [r]$ and $k \in \{0\} \cup [d]$, $\Lambda_{\ell(k+1)}(C_{B,t}^{(s,i)})$ denote the $\ell(k+1)$th eigenvalue of $C_{B,t}^{(s,i)}$. Hence, we have that 
    \begin{align}\label{fdp_t_vcm_low_eq15}
        \tr\big[\mathbb{E}[\mathbb{E}\{C_{B}(Z_t^{(s)}|M^{(t-1)})|M^{(t-1)}\}]\big] \leq \tr(C_{B,t}^{(s)}) = \sum_{i=1}^{b^t_s} \tr(C_{B,t}^{(s,i)}) \lesssim rdb^t_sm.
    \end{align}
    Substituting results in \eqref{fdp_t_vcm_low_eq14} and \eqref{fdp_t_vcm_low_eq15} into \eqref{fdp_t_vcm_low_eq12}, we have that
    \begin{align} \label{fdp_t_vcm_low_eq16}
       \underset{\substack{b_k \in \Theta(r,C_\alpha),  \|\bm{G}^\top\widetilde{\Phi}^\top B\|_{\infty} \lesssim 1\\ \bm{G}^\top\widetilde{\Phi}^\top B \in \mathcal{W}(\alpha, C_{\alpha}), \\\forall k \in \{0\} \cup [d]}}{\sup} \mathbb{E}\|\widetilde{B}-B\|_2^2 &\gtrsim \frac{r^2d^2}{ \sum_{s=1}^S \sum_{t=1}^T \{(b_s^t)^2m\epsilon_s^2 \wedge rdb^t_sm\} + dr^{2+2\alpha}}.
    \end{align}
    The Theorem holds by combining \eqref{fdp_t_vcm_low_eq8} and \eqref{fdp_t_vcm_low_eq16} together.
    \end{proof}

\subsection{Auxiliary results} \label{section_appendix_vcm_auxiliary}
\begin{lemma} \label{fdp_lemma_vcm_up}
    Consider the events of interest in the proof of Theorem \ref{fdp_thm_vcm_up}:
    \begin{align*}
        \mathcal{E}_3 = \Bigg\{&\Lambda_{\min}\Big\{\sum_{s=1}^S\sum_{i=1}^{b_s}\sum_{j=1}^m\frac{\nu_s}{b_sm}\widetilde{\Phi}_r(X_{j}^{(s,\tau_{s,t}+i)})\bm{G}^{(s,\tau_{s,t}+i)}\bm{G}^{(s,\tau_{s,t}+i)\top}\widetilde{\Phi}_r^\top(X_{j}^{(s,\tau_{s,t}+i)})\Big\}\\
        & \hspace{6cm}\geq 1/(2LC_\lambda) \\
        &\text{and} \; \Lambda_{\max}\Big\{\sum_{s=1}^S\sum_{i=1}^{b_s}\sum_{j=1}^m\frac{\nu_s}{b_sm} \widetilde{\Phi}_r(X_{j}^{(s,\tau_{s,t}+i)})\bm{G}^{(s,\tau_{s,t}+i)}\bm{G}^{(s,\tau_{s,t}+i)\top}\widetilde{\Phi}_r^\top(X_{j}^{(s,\tau_{s,t}+i)})\Big\}\\
        & \hspace{6cm}\leq 2LC_\lambda, \forall t\in \{0\}\cup [T-1]\Bigg\},
    \end{align*}
    and
    \begin{align*}
        \mathcal{E}_4 = \Bigg\{&\Pi^{\mathrm{entry}}_{R}\Big[\frac{1}{m}\sum_{j=1}^m \widetilde{\Phi}_r(X_{j}^{(s,\tau_{s,t}+i)})\bm{G}^{(s,\tau_{s,t}+ i)}\Big\{\bm{G}^{(s,\tau_{s,t}+ i)\top}\widetilde{\Phi}_r^\top(X_{j}^{(s,\tau_{s,t}+i)})B^t-Y_{j}^{(s,\tau_{s,t}+i)} \Big\}\Big] \Big\}\\
        & = \frac{1}{m}\sum_{j=1}^m \widetilde{\Phi}_r(X_{j}^{(s,\tau_{s,t}+i)})\bm{G}^{(s,\tau_{s,t}+ i)}\Big\{\bm{G}^{(s,\tau_{s,t}+ i)\top}\widetilde{\Phi}_r^\top(X_{j}^{(s,\tau_{s,t}+i)})B^t-Y_{j}^{(s,\tau_{s,t}+i)} \Big\}, \\
        & \forall i \in [b_s],\;  t \in \{0\}\cup [T-1], \; s\in [S]\Bigg\}.
    \end{align*}
    Suppose that
    \begin{align*}
        r\log^2(Tr/\eta)\lesssim \Big(\sum_{s=1}^S \frac{\nu_s^2d}{b_sm}\Big)^{-1}, \; \log^2(Tr/\eta)\sum_{s=1}^S \frac{\nu_s^2d}{b_s} \lesssim 1,\; r\log(Tr/\eta)\lesssim \Big(\sup_{s\in [S] }\frac{\nu_s d}{b_s}\Big)^{-1},
    \end{align*}
    then we have 
    \begin{align*}
        \mathbb{P}(\mathcal{E}_3 \cap \mathcal{E}_4) \geq 1-4\eta,
    \end{align*}
    for some small enough $\eta \in (0,1/4)$.
\end{lemma}

\begin{proof}
    To control $\mathcal{E}_3$, for any $s \in [S]$, $i \in [b_s]$ and $j\in [m]$, denote
    \begin{align*}
        T_{s,i} = \frac{\nu_s}{b_sm}\sum_{j=1}^m \widetilde{\Phi}_r(X_{j}^{(s,\tau_{s,t}+i)})\bm{G}^{(s,\tau_{s,t}+i)}\bm{G}^{(s,\tau_{s,t}+i)\top}\widetilde{\Phi}_r^\top(X_{j}^{(s,\tau_{s,t}+i)}),
    \end{align*}
    and note that $T_{s,i}$ are mutually independent across $s\in [S]$ and $i \in [b_s]$ under Assumption \ref{a_sample}. We firstly bound the largest and the smallest eigenvalues of $\mathbb{E}(T_{s,i})$. Note that under Assumption \ref{a_sample}, we have that
    \begin{align*}
        \mathbb{E}(T_{s,i}) &= \frac{\nu_s}{b_s}\mathbb{E}\Big[\frac{1}{m}\sum_{j=1}^m \widetilde{\Phi}_r(X_{j}^{(s,\tau_{s,t}+i)})\bm{G}^{(s,\tau_{s,t}+i)}\bm{G}^{(s,\tau_{s,t}+i)\top}\widetilde{\Phi}_r^\top(X_{j}^{(s,\tau_{s,t}+i)})\Big]\\
        & = \frac{\nu_s}{b_s}\mathbb{E}\Big\{\widetilde{\Phi}_r(X_{1}^{(s,\tau_{s,t}+1)})\bm{G}^{(s,\tau_{s,t}+1)}\bm{G}^{(s,\tau_{s,t}+1)\top}\widetilde{\Phi}_r^\top(X_{1}^{(s,\tau_{s,t}+1)})\Big\}.
    \end{align*}
    Using a similar argument as the one used to get \eqref{fdp_t_vcm_low_eq17}, we have that for any $s\in [S]$ and $i \in [b_s]$
    \begin{align} \label{l_fdp_vcm_upper_event_eq1}
        \Lambda_{\max}\{\mathbb{E}(T_{s,i})\} &= \frac{\nu_s}{b_s}\Lambda_{\max}\Big[\mathbb{E}\Big\{\widetilde{\Phi}_r(X_{1}^{(s,\tau_{s,t}+1)})\bm{G}^{(s,\tau_{s,t}+1)}\bm{G}^{(s,\tau_{s,t}+1)\top}\widetilde{\Phi}_r^\top(X_{1}^{(s,\tau_{s,t}+1)})\Big\}\Big]\\ \notag
        & = \frac{\nu_s}{b_s}  \underset{v \in \mathbb{R}^{r(d+1)}: \|v\|_2 =1}{\sup} v^{\top}\mathbb{E}\Big\{\widetilde{\Phi}_r(X_{1}^{(s,\tau_{s,t}+1)})\bm{G}^{(s,\tau_{s,t}+1)}\bm{G}^{(s,\tau_{s,t}+1)\top}\widetilde{\Phi}_r^\top(X_{1}^{(s,\tau_{s,t}+1)})\Big\}v\\ \notag
        & = \frac{\nu_s}{b_s} \underset{v \in \mathbb{R}^{r(d+1)}: \|v\|_2 =1}{\sup} \int_{0}^1 v^\top\widetilde{\Phi}_r(s)\mathbb{E}\Big\{\bm{G}^{(s,\tau_{s,t}+1)}\bm{G}^{(s,\tau_{s,t}+1)\top}\Big\}\widetilde{\Phi}_r^\top(s) f_X(x)v\; \mathrm{d}s\\ \notag
        &\leq \frac{\nu_sLC_{\lambda}}{b_s},
    \end{align}
    and similarly, we have that 
    \begin{align}\label{l_fdp_vcm_upper_event_eq2}
        \Lambda_{\min}\{\mathbb{E}(T_{s,i})\} = \frac{\nu_s}{b_s}\Lambda_{\min}\Big[\mathbb{E}\Big\{\widetilde{\Phi}_r(X_{1}^{(s,\tau_{s,t}+1)})\bm{G}^{(s,\tau_{s,t}+1)}\bm{G}^{(s,\tau_{s,t}+1)\top}\widetilde{\Phi}_r^\top(X_{1}^{(s,\tau_{s,t}+1)})\Big\}\Big] \geq \frac{\nu_s}{b_sLC_{\lambda}}.
    \end{align}
    In the next step, we find the operator norm of $T_{s,i} - \mathbb{E}(T_{s,i})$ for each $s\in [S]$ and $i \in [b_s]$. Observe that 
    \begin{align*}
        &\|T_{s,i} - \mathbb{E}(T_{s,i})\|_{\op} \leq \|T_{s,i}\|_{\op} + \|\mathbb{E}(T_{s,i})\|_{\op}\\
        \lesssim \; &   \sup_{u \in \mathbb{R}^{r(d+1)}:\|u\|_2=1} u^\top T_{s,i}u + \sup_{s\in [S]} \frac{\nu_s}{b_s}\\
        =\;&  \sup_{u \in \mathbb{R}^{r(d+1)}:\|u\|_2=1} \frac{\nu_s}{b_sm}\sum_{j=1}^m \Big\{G^{(s,\tau_{s,t}+i)}_k\sum_{\ell=1}^ru_{(k+1)\ell}\phi_{\ell}(X_{j}^{(s,\tau_{s,t}+i)})\Big\}^2 + \frac{\nu_s}{b_s}\\
        \leq \;& \sup_{u \in \mathbb{R}^{r(d+1)}:\|u\|_2=1} \frac{\nu_s}{b_sm}\sum_{j=1}^m\Big\{\sum_{k=0}^d G_k^{(s,\tau_{s,t} + i)^2}\Big\}\Big\{\sum_{k=0}^d\sum_{\ell=1}^r u_{(k+1)l}^2\Big\}\Big\{\sum_{\ell=1}^r \phi_\ell^2(X^{(s,\tau_{s,t} + i)}_j)\Big\}+ \frac{\nu_s}{b_s}\\
        \lesssim \; & \frac{\nu_s}{b_sm}mdr+ \frac{\nu_s}{b_s} \lesssim \frac{\nu_sdr}{b_s},
    \end{align*}
    where the second inequality follows from the Cauchy--Schwarz inequality and the third inequality follows from Assumption \ref{simple_vcm_a_model}(b) and the boundedness of Fourier basis. Therefore, we have that 
    \begin{equation*}
        \sup_{s\in [S], i \in [b_s]} \|T_{s,i} - \mathbb{E}(T_{s,i})\|_{\op} \lesssim \sup_{s\in [S]}\frac{\nu_sdr}{b_s}.
    \end{equation*}
    Next, we find an upper bound of the matrix variance
    statistic of the sum we are interested in. Denote $M_{1,s}(x) = \widetilde{\Phi}_r(x)\bm{G}^{(s,\tau_{s,t}+1)}$, then we have that 
    \begin{align*}
        &\Big\|\mathbb{E}\Big[\sum_{s=1}^S \sum_{i=1}^{b_s} \Big\{T_i - \mathbb{E}(T_{s,i})\Big\}\Big\{T_{s,i} - \mathbb{E}(T_{s,i})\Big\}^{\top}\Big]\Big\|_{\op}\\
        =\;&\Big\|\mathbb{E}\Big[\sum_{s=1}^S\sum_{i=1}^{b_s} \Big\{T_{s,i} - \mathbb{E}(T_{s,i})\Big\}^{\top}\Big\{T_{s,i} - \mathbb{E}(T_{s,i})\Big\}\Big]\Big\|_{\op}\\
        \lesssim \;& \sum_{s=1}^S \frac{\nu_s^2}{b_sm}\sup_{v\in \mathbb{R}^{r(d+1)}: \|v\|_2 =1} \Big|\mathbb{E}\Big\{v^\top M_{1,s}(X^{(s,\tau_{s,t} + 1)}_1)M^{\top}_{1,s}(X^{(s,\tau_{s,t} + 1)}_1)\\
        & \hspace{5cm} M_{1,s}(X^{(s,\tau_{s,t} + 1)}_1)M^{\top}_{1,s}(X^{(s,\tau_{s,t} + 1)}_1)v\Big\} - v^{\top} \mathbb{E}(T_{s,1})\mathbb{E}(T_{s,1})^\top v\Big|\\
        &+ \sum_{s=1}^S \frac{\nu_s^2}{b_s} \sup_{v\in \mathbb{R}^{r(d+1)}: \|v\|_2 =1}\Big|\mathbb{E}\Big\{v^\top M_{1,s}(X^{(s,\tau_{s,t} + 1)}_1)M^{\top}_{1,s}(X^{(s,\tau_{s,t} + 1)}_1)\\
        & \hspace{5cm} M_{1,s}(X^{(s,\tau_{s,t} + 1)}_2)M^{\top}_{1,s}(X^{(s,\tau_{s,t} + 1)}_2)v\Big\} - v^{\top} \mathbb{E}(T_{s,1})\mathbb{E}(T_{s,1})^\top v\Big|\\
        \lesssim \;& \sum_{s=1}^S \frac{\nu_s^2rd}{b_sm} + \sum_{s=1}^S \frac{\nu_s^2d}{b_s},
    \end{align*}
    where the last inequality follows from the following calculations:
    \begin{align*}
        &\Big|\mathbb{E}\Big\{v^\top M_1(X^{(s,\tau_{s,t} + 1)}_1)M^{\top}_1(X^{(s,\tau_{s,t} + 1)}_1)M_1(X^{(s,\tau_{s,t} + 1)}_1)M^{\top}_1(X^{(s,\tau_{s,t} + 1)}_1)v\Big\} \\
        & \hspace{10cm} - v^{\top} \mathbb{E}(T_{s,1})\mathbb{E}(T_{s,1})^\top v\Big| \\
        \leq \;& \Big|\mathbb{E}\Big[v^\top M_1(X^{(s,\tau_{s,t} + 1)}_1)\Big\{\sum_{k=0}^d\sum_{\ell=1}^r G_k^{(s,\tau_{s,t} + 1)^2}\phi_\ell^2(X^{(s,\tau_{s,t} + 1)}_1)\Big\}M^{\top}_1(X^{(s,\tau_{s,t} + 1)}_1)v\Big]\Big|\\
        &\hspace{10cm} + \|v^\top\mathbb{E}(T_{s,1})\|_2^2\\
        \lesssim \;& rdv^\top\mathbb{E}\Big[ M_1(X^{(s,\tau_{s,t} + 1)}_1)M^{\top}_1(X^{(s,\tau_{s,t} + 1)}_1)\Big]v + \frac{\nu_s}{b_s}\|v\|_2^2\\
        =\;& rd\int_{0}^1 v^\top \widetilde{\Phi}_r(s)\mathbb{E}\{\bm{G}^{(s,\tau_{s,t} + 1)}\bm{G}^{(s,\tau_{s,t} + 1)\top}\}\widetilde{\Phi}^\top_r(s)v f_X(s)\;\mathrm{d}s +\frac{\nu_s}{b_s}\|v\|_2^2\\
        \lesssim \;& rd\|v\|_2
    \end{align*}
    and
    \begin{align*}
        &\Big|\mathbb{E}\Big\{v^\top M_1(X^{(s,\tau_{s,t} + 1)}_1)M^{\top}_1(X^{(s,\tau_{s,t} + 1)}_1)M_1(X^{(s,\tau_{s,t} + 1)}_2)M^{\top}_1(X^{(s,\tau_{s,t} + 1)}_2)v\Big\} \\
        & \hspace{10cm} - v^{\top} \mathbb{E}(T_{s,1})\mathbb{E}(T_{s,1})^\top v\Big|\\
        \lesssim \;& \Big|\sum_{h=0}^d\sum_{s=1}^r\mathbb{E}\Big[G_h^{(s,\tau_{s,t} + 1)}\phi_{s}(X^{(s,\tau_{s,t} + 1)}_1)\Big\{\sum_{k=0}^d\sum_{\ell=1}^r v_{\ell(k+1)}G_k^{(s,\tau_{s,t} + 1)}\phi_{\ell}(X^{(s,\tau_{s,t} + 1)}_1)\Big\}\\
        &\hspace{2cm}G_h^{(s,\tau_{s,t} + 1)}\phi_{s}(X^{(s,\tau_{s,t} + 1)}_2)\Big\{\sum_{k=0}^d\sum_{\ell=1}^r v_{\ell(k+1)}G_k^{(s,\tau_{s,t} + 1)}\phi_{\ell}(X^{(s,\tau_{s,t} + 1)}_2)\Big\}\Big]\Big| + \frac{\nu_s}{b_s}\|v\|_2^2\\
        = \;& \Big|\sum_{h=0}^d \mathbb{E}\Big[\sum_{s=1}^r\mathbb{E}\Big\{\phi_{s}(X^{(s,\tau_{s,t} + 1)}_1)G_h^{(s,\tau_{s,t}+1)}v^\top M_1(X^{(s,\tau_{s,t} + 1)}_1)\Big|\bm{G}\Big\}^2\Big]\Big| + \frac{\nu_s}{b_s}\|v\|_2^2\\
        = \;&  \sum_{h=0}^d \mathbb{E}\Big[\Big\|\mathbb{E}\Big\{G_h^{(s,\tau_{s,t}+1)}\bm{G}^{(s,\tau_{s,t}+1)\top}\widetilde{\Phi}^\top_r(X^{(s,\tau_{s,t} + 1)}_1)v\Phi_r(X^{(\tau_{s,t} + 1)}_1)\Big|\bm{G}\Big\}\Big\|_2^2\Big]+ \frac{\nu_s}{b_s}\|v\|_2^2\\
        \leq \; &\sum_{h=0}^d  \mathbb{E}\Big\{\Big\|G_h^{(s,\tau_{s,t}+1)}\bm{G}^{(s,\tau_{s,t}+1)\top}\widetilde{\Phi}^\top_r vf_X\Big\|_{L^2}^2\Big\}+\frac{\nu_s}{b_s}\|v\|_2^2\\ \lesssim \;& \sum_{k=0}^d \mathbb{E}\Big\{\Big\|\bm{G}^{(s,\tau_{s,t}+1)\top}\widetilde{\Phi}^\top_r v\Big\|_{L^2}^2\Big\}+\frac{\nu_s}{b_s}\|v\|_2^2\\
        = \;& d \int_{0}^1 v^\top \widetilde{\Phi}_r(s)\mathbb{E}\{\bm{G}^{(s,\tau_{s,t} + 1)}\bm{G}^{(s,\tau_{s,t} + 1)\top}\}\widetilde{\Phi}^\top_r(s)v\;\mathrm{d}s + \frac{\nu_s}{b_s}\|v\|_2^2 \lesssim d\|v\|_2.
    \end{align*}

    Hence the matrix Bernstein inequality \citep[e.g.~Theorem 6.1.1 in][]{tropp2015introduction} implies that for any $\tau>0$,
    \begin{align*}
        \mathbb{P}\Big[\Big\|\sum_{s=1}^S\sum_{i=1}^{b_s}& \{T_{s,i} - \mathbb{E}(T_{s,i})\}\Big\|_{\op} \geq \tau\Big] \\
        &\leq 2r(d+1)\exp\Big\{\frac{-\tau^2/2}{\sum_{s=1}^S \nu_s^2rd/(b_sm) + \sum_{s=1}^S \nu_s^2d/(b_s)+ \tau \sup_{s\in [S]}\nu_sdr/(3b_s)}\Big\}.
    \end{align*}
    Note that by \eqref{l_fdp_vcm_upper_event_eq1} and \eqref{l_fdp_vcm_upper_event_eq2}, it holds that
    \begin{align*}
        \Lambda_{\max}\Big\{\sum_{s=1}^S\sum_{i=1}^{b_s}\mathbb{E}(T_{s,i})\Big\} \leq \sum_{s=1}^S\sum_{i=1}^{b_s} \frac{\nu_sLC_{\lambda}}{b_s} \leq LC_{\lambda},
    \end{align*}
    and
    \begin{align*}
        \Lambda_{\min}\Big\{\sum_{s=1}^S\sum_{i=1}^{b_s}\mathbb{E}(T_{s,i})\Big\} \leq \sum_{s=1}^S\sum_{i=1}^{b_s} \frac{\nu_s}{b_sLC_{\lambda}} \leq \frac{1}{LC_{\lambda}}.
    \end{align*}
    Hence, applying a union bound argument and the Weyl's inequality and pick
    \begin{align*}
        \tau = \log(Tr/\eta_1)\Big[\Big\{\sum_{s=1}^S \frac{\nu_s^2rd}{b_sm} + \sum_{s=1}^S \frac{\nu_s^2d}{b_s}\Big\}^{1/2} + \sup_{s\in [S] }\frac{\nu_srd}{b_s}\Big] \leq \frac{1}{2LC_{\lambda}},
    \end{align*}
    we have that $\mathbb{P}(\mathcal{E}_3) \geq 1-\eta_1$ by further assuming 
    \begin{align} \label{l_fdp_vcm_upper_event_eq3}
        r\log^2(Tr/\eta_1)\lesssim \Big(\sum_{s=1}^S \frac{\nu_s^2d}{b_sm}\Big)^{-1}, \; \log^2(Tr/\eta_1)\sum_{s=1}^S \frac{\nu_s^2d}{b_s} \lesssim 1, \; r\log(Tr/\eta_1)\lesssim \Big(\sup_{s\in [S] }\frac{\nu_s d}{b_s}\Big)^{-1}.
    \end{align}

    To control $\mathcal{E}_4$, we will follow a similar argument as the one used in the proof of Lemma \ref{l_mean_upper_event}. Under Model \eqref{simple_vcm_model_obs}, we could rewrite the summation of interests as
    \begin{align*}
        &\Big|\frac{1}{m}\sum_{j=1}^m \phi_\ell(X_{j}^{(s, \tau_{s,t}+i)})G_k^{(s,\tau_{s,t}+ i)}\Big\{\bm{G}^{(s,\tau_{s,t}+ i)\top}\widetilde{\Phi}_r^\top(X_{j}^{(s, \tau_{s,t}+i)})B^t-Y_{j}^{(\tau_{s,t}+i)}\Big\}\Big|\\
        \leq \; & \Big|\frac{1}{m}\sum_{j=1}^m \phi_\ell(X_{j}^{(s, \tau_{s,t}+i)})G_k^{(s,\tau_{s,t}+ i)}\bm{G}^{(s,\tau_{s,t}+ i)\top}\widetilde{\Phi}_r^\top(X_{j}^{(s, \tau_{s,t}+i)})B^t\Big|\\
        & +\Big|\frac{1}{m}\sum_{j=1}^m \phi_\ell(X_{j}^{(s, \tau_{s,t}+i)})G_k^{(s,\tau_{s,t}+ i)}\bm{G}^{(s,\tau_{s,t}+ i)\top} \bm{\beta}^*(X_{j}^{(s, \tau_{s,t}+i)})\Big|\\
        & +\Big|\frac{1}{m}\sum_{j=1}^m \phi_\ell(X_{j}^{(s,\tau_{s,t}+i)})G_k^{(s,\tau_{s,t}+ i)}\xi_{s,(\tau_{s,t}+ i),j}\Big|\\
        =\;& (I)+(II)+(III),
    \end{align*}
    and we will consider each of the terms individually. To remove the effect of the dependence resulting from the random vector $\bm{G}$, we construct the rest of the proof conditioning on $\bm{G}$ and all terms inside the summation are mutually independent across $j \in [m]$. 

    To control $(I)$, note that 
    \begin{align*}
        &\Big|\phi_\ell(X_{j}^{(s, \tau_{s,t}+i)})G_k^{(s,\tau_{s,t}+ i)}\bm{G}^{(s,\tau_{s,t}+ i)\top}\widetilde{\Phi}_r^\top(X_{j}^{(s, \tau_{s,t}+i)})B^t\Big| \lesssim \Big|\bm{G}^{(s,\tau_{s,t}+ i)\top}\widetilde{\Phi}_r^\top(X_{j}^{(s, \tau_{s,t}+i)})B^t\Big|\\
        = \;& \Big|\sum_{k=0}^d  G_k^{(s,\tau_{s,t}+ i)} \sum_{\ell=1}^r (b^t_k)_\ell \phi_{\ell}(X_{j}^{(s, \tau_{s,t}+i)})\Big|  \lesssim 1,
    \end{align*}
    where the first inequality follows from the boundedness property of Fourier basis and Assumption \ref{simple_vcm_a_model}(b) and the second inequality follows from the truncation $\Pi^*_{\mathcal{B}}$. Therefore, Hoeffding’s inequality for general bounded random variables \citep[e.g.~Theorem 2.2.6 in][]{vershynin2018high} implies that for any $\tau>0$,
    \begin{align*}
        \mathbb{P}\Big[\Big|&\frac{1}{m}\sum_{j=1}^m \phi_\ell(X_{j}^{(s, \tau_{s,t}+i)})G_k^{(s,\tau_{s,t}+ i)}\bm{G}^{(s,\tau_{s,t}+ i)\top}\widetilde{\Phi}_r^\top(X_{j}^{(s, \tau_{s,t}+i)})B^t\\
        &-m\mathbb{E}\Big\{\phi_\ell(X_{1}^{(s,\tau_{s,t}+i)})G_k^{(s,\tau_{s,t}+ i)}\bm{G}^{(s,\tau_{s,t}+ i)\top}\widetilde{\Phi}_r^\top(X_{1}^{(s,\tau_{s,t}+i)})B^t\Big|\bm{G}\Big\}\Big|\geq \tau\Big|\bm{G}\Big]\leq \exp\Big(\frac{C_2\tau^2}{m}\Big).
    \end{align*}
    Taking another expectation with respect to $\bm{G}$, we have that 
    \begin{align*}
        \mathbb{P}\Big[\Big|&\sum_{j=1}^m \phi_\ell(X_{j}^{(s, \tau_{s,t}+i)})G_k^{(s,\tau_{s,t}+ i)}\bm{G}^{(s,\tau_{s,t}+ i)\top}\widetilde{\Phi}_r^\top(X_{j}^{(s, \tau_{s,t}+i)})B^t\\
        &-m\mathbb{E}\Big\{\phi_\ell(X_{1}^{(s,\tau_{s,t}+i)})G_k^{(s,\tau_{s,t}+ i)}\bm{G}^{(s,\tau_{s,t}+ i)\top}\widetilde{\Phi}_r^\top(X_{1}^{(s,\tau_{s,t}+i)})B^t\Big|\bm{G}\Big\}\Big|\geq \tau\Big]\leq \exp\Big(\frac{C_2\tau^2}{m}\Big).
    \end{align*}
    To upper bound the expectation term, note that by definition, we have that
    \begin{align} \notag
        &\Big|\mathbb{E}\Big\{\phi_\ell(X_{j}^{(s, \tau_{s,t}+i)})G_k^{(s,\tau_{s,t}+ i)}\bm{G}^{(s,\tau_{s,t}+ i)\top}\widetilde{\Phi}_r^\top(X_{j}^{(s, \tau_{s,t}+i)})B^t\Big|\bm{G}\Big\}\Big|\\ \notag
        =\;& \Big| \mathbb{E}\Big[\phi_\ell(X_{j}^{(s, \tau_{s,t}+i)})\Big\{\sum_{h=0}^d G_k^{(s,\tau_{s,t}+ i)}G_h^{(s, \tau_{s,t}+ i)} \Phi_{\ell}^{\top}(X_{j}^{(s, \tau_{s,t}+i)})b^t_h\Big\}\Big|\bm{G}\Big]\Big|\\ \label{l_vcm_upper_event_eq2}
        = \; & \Big| G_k^{(s,\tau_{s,t}+ i)}\mathbb{E}\Big[\phi_\ell(X_{j}^{(s, \tau_{s,t}+i)})\sum_{h=0}^dG_h^{(s, \tau_{s,t}+ i)}\Phi_{r}^{\top}(X_{j}^{(s, \tau_{s,t}+i)})b^t_h\Big]\Big|\\ \notag
        \lesssim \;&  \Big|\int_{0}^1\phi_\ell(s) f_X(s)\Big\{\sum_{h=0}^d G_h^{(s, \tau_{s,t}+ i)}\Phi_{r}^{\top}(s)b^t_h\Big\}\; \mathrm{d}s\Big| \lesssim \ell^{-\alpha},
    \end{align}
    where the last inequality follows from a similar argument as the one that leads to \eqref{l_mean_upper_event_eq1} together with Assumption \ref{simple_vcm_a_model}(b) and the fact that $\sum_{h=0}^d G_h^{(s, \tau_{s,t}+ i)}\Phi_{r}^{\top}b^t_h \in \mathcal{W}(\alpha,C_\alpha)$ as it holds that
    \begin{align*}
        &\sum_{\ell=1}^\infty \ell^{2\alpha}\Big\{\int_{0}^1 \sum_{h=0}^d G_h^{(s, \tau_{s,t}+ i)}\Phi_{r}^{\top}b^t_h\phi_{\ell}(s)\; \mathrm{d}s\Big\}^2 = \sum_{\ell=1}^\infty \ell^{2\alpha}\Big\{\sum_{h=0}^d G_h^{(s, \tau_{s,t}+ i)}\int_{0}^1 \Phi_{r}^{\top}b^t_h\phi_{\ell}(s)\; \mathrm{d}s\Big\}^2\\
        =\;& \sum_{\ell=1}^r \ell^{2\alpha}\Big\{\sum_{h=0}^d G_h^{(s, \tau_{s,t}+ i)}(b_h^t)_\ell\Big\}^2  \lesssim 1,
    \end{align*}
    where the last inequality follows from the projection $\Pi^*_{\mathcal{B}}$. Applying a union bound argument, we have that with probability at least $1-\eta_2$ that
    \begin{align}\label{l_vcm_upper_event_eq3}
        \max_{t,i}\Big|\frac{1}{m}\sum_{j=1}^m \phi_\ell(X_{j}^{(s, \tau_{s,t}+i)})G_k^{(s,\tau_{s,t}+ i)}\bm{G}^{(s,\tau_{s,t}+ i)\top}\widetilde{\Phi}_r^\top(X_{j}^{(s, \tau_{s,t}+i)})B^t\Big|\lesssim\sqrt{\log(N/\eta_2)/m}+\ell^{-\alpha}.
    \end{align}
    
    For $(II)$, note that 
     \begin{align*}
        &\Big|\phi_\ell(X_{j}^{(s, \tau_{s,t}+i)})G_k^{(s,\tau_{s,t}+ i)}\bm{G}^{(s,\tau_{s,t}+ i)\top}\bm{\beta}^*(X_{j}^{(s, \tau_{s,t}+i)})\Big| \lesssim \Big|\bm{G}^{(s,\tau_{s,t}+ i)\top}\bm{\beta}^*(X_{j}^{(s, \tau_{s,t}+i)})\Big|\lesssim 1,
    \end{align*}
    where the last inequality follows from Assumption \ref{simple_vcm_a_model}(c). To upper bound the expectation term, note that by definition, we have that
    \begin{align*}
        &\Big|\mathbb{E}\Big\{\phi_\ell(X_{j}^{(s, \tau_{s,t}+i)})G_k^{(s,\tau_{s,t}+ i)}\bm{G}^{(s,\tau_{s,t}+ i)\top}\bm{\beta}^*(X_{j}^{(s, \tau_{s,t}+i)})\Big|\bm{G}\Big\}\Big| \\
        =& \;\Big| G_k^{(s,\tau_{s,t}+ i)}\int_{0}^1 \phi_\ell(s)\bm{G}^{(s,\tau_{s,t}+ i)\top}\bm{\beta}^*(X_{j}^{(s, \tau_{s,t}+i)})f_X(s)\;\mathrm{d}s \Big| \lesssim \ell^{-\alpha},
    \end{align*}
    where the last inequality follows from a similar argument as the one used in \eqref{l_mean_upper_event_eq1} and Assumption \ref{simple_vcm_a_model}(c). Thus, using a similar argument as the one leads to \eqref{l_vcm_upper_event_eq3}, it holds with probability at least $1-\eta_3$ that
    \begin{align}\label{l_vcm_upper_event_eq4}
        \max_{t,i}\Big|\frac{1}{m}\sum_{j=1}^m \phi_\ell(X_{j}^{(s, \tau_{s,t}+i)})G_k^{(s,\tau_{s,t}+ i)}\bm{G}^{(s,\tau_{s,t}+ i)\top} \bm{\beta}^*(X_{j}^{(s, \tau_{s,t}+i)})\Big|\lesssim \sqrt{\log(N/\eta_3)/m}+\ell^{-\alpha}.
    \end{align}

    To control $(III)$, under Assumptions \ref{simple_vcm_a_model}(b) and \ref{simple_vcm_a_model}(d), we have that the terms inside the summation are independent mean zero sub-Gaussian random variables with parameter $\sqrt{2}C_{g}C_{\xi}$. Hence Hoeffding inequality \citep[e.g.~Theorem 2.6.2 in][]{vershynin2018high} implies that
    \begin{align*}
        \mathbb{P}\Big\{\Big|\sum_{j=1}^m \phi_\ell(X_{j}^{(s, \tau_{s,t}+i)})G_k^{(s,\tau_{s,t}+ i)}\xi_{s,(\tau_{s,t}+ i),j}\Big| \geq \tau\Big\}\leq \exp\Big(\frac{C_3\tau^2}{m}\Big).
    \end{align*}
    Applying a union bound argument, we have with probability at least $1-\eta_4$ that 
    \begin{align}\label{l_vcm_upper_event_eq5}
        \max_{t,i}\Big|\frac{1}{m}\sum_{j=1}^m \phi_\ell(X_{j}^{(s, \tau_{s,t}+i)})G_k^{(s,\tau_{s,t}+ i)}\xi_{s,(\tau_{s,t}+ i),j}\Big| \lesssim \sqrt{\log(N/\eta_4)/m}.
    \end{align}

    Taking another union bound over entries, $\mathcal{E}_1$, \eqref{l_vcm_upper_event_eq3}, \eqref{l_vcm_upper_event_eq4} and \eqref{l_vcm_upper_event_eq5}, we have that 
    \begin{align*}
        \mathbb{P}(\mathcal{E}_1\cap\mathcal{E}_2)\geq 1-\eta.
    \end{align*}
    
\end{proof}

\begin{remark} \label{remark_vcm_fdp_r}
    The optimal choice of $r$, $r_{\mathrm{opt}}$, is the solution to the equation
    \begin{align*}
        dr^{2+2\alpha} = \sum_{s=1}^S(dr^2n_s) \wedge (drn_sm) \wedge (r^2n_s^2\epsilon^2) \wedge (n_s^2m\epsilon^2).
    \end{align*}
    In this remark, all the calculations are up to a poly-logarithmic factor. Denote $s^* = \argmax_{s\in[S]} \nu_sd/b_s$. The first and the second assumptions in \eqref{l_fdp_vcm_upper_event_eq3} are automatically satisfied by our choice of $r$ in equation \eqref{t_fdp_vcm_upper_eq9} to ensure the overall estimation error is less than one. For the third assumption, the assumption reduces to $rd \lesssim sn$ in a homogeneous setting, which trivially holds for our optimal choice of $r$. In the heterogeneous setting, we discuss the validity of this assumption in four different cases. When $dr^2n_{s^*}$ is the smallest term, we have that the third assumption in \eqref{l_fdp_vcm_upper_event_eq3} is equivalent to say
    \begin{align*}
        r \lesssim \frac{\sum_{s=1}^S(dr^2n_s) \wedge (drn_sm) \wedge (r^2n_s^2\epsilon^2) \wedge (n_s^2m\epsilon^2)}{d^2r^2} \asymp \frac{dr^{2+2\alpha}}{d^2r^2}.
    \end{align*}
    Therefore, the assumption holds by selecting $r$ such that $r \gtrsim d^{1/(2\alpha-1)}$. In the second case, when $drn_{s^*}m$ is the smallest term (the case when $r \gtrsim m$), we have that the third assumption in \eqref{l_fdp_vcm_upper_event_eq3} is equivalent to say
    \begin{align*}
        r \lesssim \frac{\sum_{s=1}^S(dr^2n_s) \wedge (drn_sm) \wedge (r^2n_s^2\epsilon^2) \wedge (n_s^2m\epsilon^2)}{d^2rm} \asymp \frac{dr^{2+2\alpha}}{d^2rm}.
    \end{align*}
    If we select $r \gtrsim d^{1/(2\alpha-1)}$, then we have that $r^{2\alpha} = r^{2\alpha-1}r \gtrsim dr \gtrsim dm$ and the assumption holds. In the third case, when $r^2n_{s^*}^2\epsilon^2$ is the smallest term, we have that
    \begin{align*}
        r \lesssim \frac{\sum_{s=1}^S(dr^2n_s) \wedge (drn_sm) \wedge (r^2n_s^2\epsilon^2) \wedge (n_s^2m\epsilon^2)}{dr^2n_{s^*}\epsilon^2} \asymp \frac{dr^{2+2\alpha}}{dr^2n_{s^*}\epsilon^2},
    \end{align*}
    which further requires selecting $r$ such that $r \gtrsim (n_{s^*}\epsilon^2)^{1/(2\alpha-1)}$. Finally, when $n_{s^*}^2m\epsilon^2$ is the smallest term (the case when $r \gtrsim m^2$), then we have that 
    \begin{align*}
        r \lesssim \frac{\sum_{s=1}^S(dr^2n_s) \wedge (drn_sm) \wedge (r^2n_s^2\epsilon^2) \wedge (n_s^2m\epsilon^2)}{dn_{s^*}m\epsilon^2} \asymp \frac{dr^{2+2\alpha}}{dn_{s^*}m\epsilon^2}.
    \end{align*}
    If we select $r$ such that $r \gtrsim (n_{s^*}\epsilon^2)^{1/(2\alpha-1)}$, then we have that $r^{2\alpha+1}\gtrsim n_{s^*}\epsilon^2r^2 \gtrsim n_{s^*}m\epsilon^2$ and the assumption holds. Combine the above together, the overall assumption on $r$ is that 
    \begin{align*}
        r \gtrsim d^{1/(2\alpha-1)} \vee (n_{s^*}\epsilon^2)^{1/(2\alpha-1)}.
    \end{align*}
\end{remark}

\section{Technical lemmas} \label{section_appendix_technical}
\begin{lemma}[\citealp{dwork2014algorithmic}, Post-processing property of DP] \label{l_postprocessing}
    Let $M$ be an $(\epsilon, \delta)$-CDP algorithm, and $g$ be an arbitrary, deterministic mapping that takes $M$ as an unput; then $g(M)$ is $(\epsilon, \delta)$-CDP.
\end{lemma}

\begin{lemma}[Lemma 2 in \citealp{quan2024optimal}] \label{l_mean_approximation}
    For $T \in [0,1]$, denote $\Phi(T) = \{\phi_1(T), \ldots,\\ \phi_r(T)\}^\top$ the collection of Fourier basis, then for any function $g \in \mathcal{L}^2([0,1])$, it holds that 
    $$\|\mathbb{E}_f\{\Phi(T)g(T)\}\|^2_2 \leq (\inf_{t\in [0,1]} |f(t)|)^{-1}\|f\|_{\infty}\|g\|^2_{L^2},$$ where $f$ is the sampling density.
\end{lemma}

\begin{theorem}[Theorem 7.4 in \citealp{behzadan2021multiplication}] \label{thm_product_sobolev}
    For any $\alpha > 1/2$, if $u \in \mathcal{W}(\alpha, C_{\alpha})$ and $v \in \mathcal{W}(\alpha, C_{\alpha})$, then we have that $uv \in \mathcal{W}(\alpha, C_\alpha)$. 
\end{theorem}

\begin{lemma} \label{l_fisher_rescale}
    Let $X$ be a random variable with density $f_\theta$, where $\theta \in \mathbb{R}$ is the model parameter. Denote $I_X$ the Fisher information of $\theta$ associated with $X$. Let $Y=X/d$ for any $d >0$, then we have that $I_X = I_Y$, where $I_Y$ is the Fisher information of $\theta$ associated with the scaled random variable $Y$.
\end{lemma}
\begin{proof}
    Denote $g_{\theta}$ the density of the scaled random variable $Y$. Using the standard property of density functions, we have that $g_{\theta}(y) = df_\theta(dy)$ and $\log g_{\theta}(y) = \log(d)+ \log f_\theta(dy)$. Taking partial derivatives with respect to $\theta$ on both sides of equation, we then have that
    \begin{align}\label{l_fisher_rescale_eq1}
        \frac{\partial \log g_{\theta} (y)}{\partial \theta} = \frac{\partial \log f_{\theta}(dy)}{\partial \theta}.
    \end{align}
    By definition of Fisher information, we have that 
    \begin{align*}
        I_Y = \mathbb{E}_Y\Big\{\Big(\frac{\partial \log g_{\theta} (y)}{\partial \theta}\Big)^2\Big\} -\mathbb{E}_Y\Big\{\frac{\partial \log g_{\theta} (y)}{\partial \theta}\Big\}^2.
    \end{align*}
    Note that for the first term, we have that 
    \begin{align*}
        \mathbb{E}_Y\Big\{\Big(\frac{\partial \log g_{\theta} (y)}{\partial \theta}\Big)^2\Big\} &= \int \Big(\frac{\partial\log  g_{\theta} (y)}{\partial \theta}\Big)^2 g_{\theta}(y)\;\mathrm{d}y = \int \Big(\frac{\partial \log f_{\theta} (dy)}{\partial \theta}\Big)^2 df_{\theta}(dy)\;\mathrm{d}y\\
        & = \int \Big(\frac{\partial \log f_{\theta} (z)}{\partial \theta}\Big)^2 f_{\theta}(z)\;\mathrm{d}z = \mathbb{E}_X\Big\{\Big(\frac{\partial \log f_{\theta} (f)}{\partial \theta}\Big)^2\Big\},
    \end{align*}
    where the second equality follows from \eqref{l_fisher_rescale_eq1} and the third equality follows from the change of variable $z=dy$. Similarly, we can also show that 
    \begin{align*}
        \mathbb{E}_Y\Big\{\frac{\partial \log g_{\theta} (y)}{\partial \theta}\Big\} = \mathbb{E}_X\Big\{\frac{\partial \log f_{\theta} (x)}{\partial \theta}\Big\}.
    \end{align*}
    Hence we conclude the proof.    
\end{proof}

\begin{lemma} \label{lemma_projection}
    For any $r \in \mathbb{N}_{+}$ and $v \in \mathbb{R}^r$, suppose $\mathcal{A} \subset \mathbb{R}^r$ is a closed, compact and convex space and denote $\Pi^*(v)$ the projection of $v$ to the space $\mathcal{A}$ with the shortest distance (in $\ell_2$ sense) . Then for any $c \in \mathbb{R}^r$ and $c^* \in \mathcal{A}$, we have that
    \begin{align*}
        \|\Pi^*(c) - c^*\|_2^2 \lesssim \|c - c^*\|_2^2.
    \end{align*}
\end{lemma}

\begin{proof}
    By definition of $\Pi^*$, we have that $\|c -\Pi^*(c)\|_{2}^2 \leq \|c-c^*\|_2^2$. Therefore, we have that 
    \begin{align*}
        \|c - c^*\|_2^2 &\geq \frac{3}{4}\|\Pi^*(c) - c^*\|_2^2 - 3\|c - \Pi^*(c)\|_2^2 \geq \frac{3}{4}\|\Pi^*(c) - c^*\|_2^2 - 3\|c - c^*\|_2^2,
    \end{align*}
    where the first inequality follows from the fact that $(a-b)^2 \geq 3a^2/4 - 3b^2$ for all $a,b \in \mathbb{R}$. Hence by rearranging, we have that $\|c - c^*\|_2^2 \geq 3\|\Pi^*(c)- c^*\|_2^2/16$.
\end{proof}

\begin{lemma}\label{lemma_orthogonal}
    For random variables $W,V,U$, it holds that 
    \begin{align*}
        \mathbb{E}[\mathbb{E}\{W\mathbb{E}(W|U,V)|V\}] = \mathbb{E}[\mathbb{E}\{\mathbb{E}(W|U,V)\mathbb{E}(W|U,V)|V\}].
    \end{align*}
\end{lemma}

\begin{proof}
    For the left-hand side, it holds that 
    \begin{align*}
        \mathbb{E}[\mathbb{E}\{W\mathbb{E}(W|U,V)|V\}] &=  \mathbb{E}\big[\mathbb{E}\big\{\{W -\mathbb{E}(W|U,V)+\mathbb{E}(W|U,V)\}\mathbb{E}(W|U,V)|V\big\}\big]\\
        & = \mathbb{E}\big[\mathbb{E}\big\{\{W -\mathbb{E}(W|U,V)\}\mathbb{E}(W|U,V)|V\big\}\big] \\
        & \hspace{4cm} + \mathbb{E}\big[\mathbb{E}\big\{\mathbb{E}(W|U,V)\mathbb{E}(W|U,V)|V\big\}\big]\\
        & = \mathbb{E}\big[\mathbb{E}\big\{\mathbb{E}(W|U,V)\mathbb{E}(W|U,V)|V\big\}\big],
    \end{align*}
    where the last equality follows from orthogonality between $W - \mathbb{E}(W|U,V)$ and $\mathbb{E}(W|U,V)$.
\end{proof}

\end{document}